\RequirePackage{scrlfile}
\BeforePackage{amsthm}{\RequirePackage{amsmath}}

\documentclass[a4paper,twoLevelNum]{higherStructures}
\newif\ifarxiv
\arxivtrue % change for HS/arXiv
\ifarxiv
  \renewcommand\banner{}
  \makeatletter
  \def\@articleRef{}
  \def\@numLang{3}
  \def\@keywordsname{Keywords}
  \makeatother
\fi

\usepackage{amssymb}
\usepackage{lmodern}
\usepackage{quiver}
\usepackage{extpfeil}
\usepackage[T1]{fontenc}
\usepackage[utf8]{inputenc}

\usepackage[style=numeric,maxnames=10,backend=biber,
useprefix=true,hyperref=true]{biblatex}
\usepackage{todonotes}
\usepackage{float}
\usepackage{tikz-cd}
\usepackage{etoolbox}
\usepackage{mathtools}
\usepackage{pgfplots}
\usepackage[capitalize,noabbrev]{cleveref}
\usepackage{mymacros_hs}
\usepackage[mathscr]{euscript}
\usepackage[shortcuts]{extdash}

\title[Synthetic fibered $\inftyone$-category theory]{Synthetic fibered $\inftyone$-category theory}

\pubYear{2022}
\author[Buchholtz]{Ulrik Buchholtz}
\email{ulrik.buchholtz@nottingham.ac.uk}
\address{Functional Programming Lab, School of Computer Science, University of Nottingham, Jubilee Campus, Nottingham, NG7 1BB, UK}
\author[Weinberger]{Jonathan Weinberger}
\address{Dept.~of Mathematics, Krieger School of Arts and Sciences, Johns Hopkins U., 3400 N Charles St., Baltimore, MD 21218, USA}
\email{jweinb20@jhu.edu}

\theoremstyle{plain}
\newtheorem{thm}[theorem]{Theorem}
\newtheorem{lem}[theorem]{Lemma}
\newtheorem{prop}[theorem]{Proposition}
\newtheorem{cor}[theorem]{Corollary}
\newtheorem{ax}[theorem]{Axiom}
\theoremstyle{definition}
\newtheorem{defn}[theorem]{Definition}
\newtheorem{rem}[theorem]{Remark}
\newtheorem{expl}[theorem]{Example}
\newtheorem{obs}[theorem]{Observation}
\theoremstyle{remark}

\numberwithin{equation}{section}

\hypersetup{bookmarksdepth=2,bookmarksopenlevel=2}

\begin{document}

\maketitle

\begin{abstract}
  We study cocartesian fibrations in the setting of the
  synthetic \inftyone-category theory developed in
  simplicial type theory introduced by Riehl and Shulman.
  Our development culminates in a Yoneda Lemma
  for cocartesian fibrations.
\end{abstract}

\section{Introduction}\label{sec:intro}

\subsection{Motivation and overview}

\subsubsection{Synthetic mathematics in homotopy type theory}

So far, homotopy type theory (HoTT) has served as a convenient framework for a lot of homotopical mathematics. It allows one to reason synthetically about homotopy types, for example to study homotopy or cohomology groups~\cite{hottbook,SymmetryBook,RijPhd}.
Indeed, it functions as an internal language for higher toposes.
This had been conjectured for some time by Awodey, and then recently established
by Shulman~\cite{Shu19}.
Alas, developing specifically the field of higher category theory in standard Book HoTT has not been as fruitful yet. For example, it is still an important open problem to give a definition of the notion of $(\infty,1)$-category completely internal to HoTT. Currently, one has to rely on some extension such as two-level type theory~\cite{2ltt}.

It is also not possible to directly interpret types as higher categories, since not all functors are exponentiable. At the level of model structures presenting the homotopy theory of $(\infty,1)$-categories (such as the Joyal or Rezk model structure), this manifests itself in their failure to be right proper.

\subsubsection{Synthetic higher category theory in simplicial homotopy type theory}

Outside the realm of Book HoTT, there do exist various approaches to reason type-theoretically about higher categorical structues, cf.~\cite{B19} for an overview and discussion.
As one solution, simplicial (homotopy) type theory (sHoTT) has been suggested by Riehl--Shulman~\cite{RS17}. Also independently proposed by Joyal, the idea is to work internally to simplicial spaces, or more generally, simplicial objects in any higher topos. This extension of HoTT comes with two new kinds of gadgets: shapes and extension types, both with \emph{judgmental equalities}. The strict shape layer encompasses e.g.~the standard $n$-simplices $\Delta^n$, boundaries $\partial \Delta^n$, and horns $\Lambda^n_k$. The strict extension types correspond to subtypes of dependent function types whose elements \emph{judgmentally} restrict to a fixed section defined on a subshape of a larger shape. These extension types make it possibly to define \eg,~the directed hom-types $\hom_A(x,y)$, for types $A$ and elements $x,y:A$, derived from the arrow type $A^{\Delta^1}$. Importantly, extension types are homotopically and computationally well-behaved: they allow for strict computations while maintaining fibrancy.

The prime model of sHoTT is the $\inftyone$-topos $\sSpace := \left[\Simplex^{\Op}, \Space\right]$ of simplicial $\infty$-groupoids, hence sHoTT can be understood as a synthetic theory of simplicial $\infty$-groupoids. The internal simplicial structure of each type $A$ can be probed by investigating the function type $\Delta^n \to A$. This allows for reasoning about higher categories in a convenient way because one can state the conditions for a type to be groupoidal (\emph{aka discrete}) or (complete) Segal in relatively simple and \emph{finitary} terms using extension types, as shown by Joyal in the classical set-theoretic setting. Semantically, these, in fact correspond to the desired properties, cf.~\cite[Appendix A]{RS17}.

\subsubsection{Synthetic fibered higher category theory}

In this paper, we complement Riehl--Shulman's work on \emph{covariant families}, \ie,~functorial type families with groupoidal fibers, by a development of a synthetic notion of \emph{cocartesian family}, \ie,~functorial type families with categorical fibers. That is, cocartesian families represent functors to the category of categories. Dually, we obtain a synthetic notion of \emph{cartesian family} as well, representing contravariant functors to the category of categories.

Specifically, we give characterizations of cocartesian fibrations via certain adjointness conditions, one of which being the so-called \emph{Chevalley criterion} which traditionally is known from $2$-category theory due to Street~\cite{StrBicat} and Gray~\cite{GrayFib}, and in the past years has been recovered in Riehl--Verity's model-independent higher category theory~\cite{RV}. We furthermore prove a type-theoretic $2$-Yoneda Lemma, which for $\infty$-cosmoses has been established by~\cite{RVyoneda}, \cite[Section 5.7]{RV} (in turn being inspired by~\cite{StrYon}). This implies the discrete case as a corollary, which previously has been formulated and proven in the type-theoretic setting in~\cite{RS17}, and semantically for (complete) Segal spaces in~\cite{KV14,RasYon}.

The type-theoretic Yoneda Lemma can also be understood as an induction principle for \emph{directed} arrows. Furthermore, we establish various theorems about cocartesian functors, analogous to~\cite{RV}, including characterization theorems in terms of invertibility of mates, as well as several closure properties inspired by $\infty$-cosmos theory. In fact, these closure properties could be understood as completeness results about the $\inftyone$-category of small cocartesian fibrations w.r.t.~$\inftyone$-categorical limits. However, in the work at hand the universal properties will be spelled out explicitly since the universe type in consideration will not itself be an $\inftyone$-category. For previous discussions on more structured universes and the Directed Univalence Axiom, cf.~Cavallo--Riehl--Sattler~\cite{CRS} and Weaver--Licata~\cite{WL19}.
In the context of the latter, Licata has an Agda formalization of the definition of cocartesian families~\cite{LicataMoreFibs}.

Prior to these main developments, we also briefly touch on non-functorial families of synthetic (pre-)$\inftyone$-categories. These notions, admittedly of somewhat auxiliary character, primarily serve to conceptually systematize our developments in the realm of simplicial types. At the end of the paper, we also provide some tie-ins with the pre-established theory of covariant families from~\cite{RS17}, recovering the familiar characterization of covariant families among cocartesian families.

\subsection{Contributions}

We re-develop the basic theory of fibered $\inftyone$-categories in (a mild extension of) the synthetic setting of Riehl--Shulman's \emph{simplicial homotopy type theory}~\cite{RS17}. Specifically, we prove several closure properties, characterization theorems, and a Yoneda Lemma for cocartesian fibrations and cocartesian functors. Along the way, we also discuss non-functorial families of (pre-)$\inftyone$-categories as auxiliary notions.

Notably, Riehl--Verity's model-independent higher category theory in $\infty$-cosmoses~\cite{RV} has been serving as a principal guiding stone for our developments. Tying in with previous work by Riehl--Shulman, we establish additional characterizations of discrete covariant fibrations and left adjoint right inverse (LARI) adjunctions.

To argue that our synthetic theory in fact captures the well-known analytic theory of fibered $\inftyone$-categories, we briefly point out the semantics in (a suitable presentation of) the $\inftyone$-topos of simplicial spaces (or, more generally, simplicial objects in an $\inftyone$-topos). However, a detailed semantic discussion is omitted in this text.

\subsection{Structure of the paper}

After the introductory \cref{sec:intro}, we give in \cref{sec:expo} an exposition of Riehl--Shulman's simplicial type theory~\cite{RS17} and our concrete type-theoretic setup.

\begin{figure}
	\centering
% https://q.uiver.app/?q=WzAsOSxbMiwwLCJcXHRleHR7U2VjLn4xfSJdLFsyLDEsIlxcdGV4dHtTZWMufjJ9Il0sWzEsMiwiXFx0ZXh0e1NlYy5+M30iXSxbMSwzLCJcXHRleHR7U2VjLn40fSJdLFsyLDQsIlxcdGV4dHtTZWMufjV9Il0sWzAsMiwiXFx0ZXh0e0ExLCBBMn0iXSxbMCwzLCJcXHRleHR7QTN9Il0sWzIsNSwiXFx0ZXh0e1NlYy5+N30iXSxbMyw0LCJcXHRleHR7U2VjLn42fSJdLFsxLDJdLFsyLDNdLFszLDRdLFsyLDUsIiIsMSx7InN0eWxlIjp7ImJvZHkiOnsibmFtZSI6ImRhc2hlZCJ9fX1dLFs1LDMsIiIsMSx7InN0eWxlIjp7ImJvZHkiOnsibmFtZSI6ImRhc2hlZCJ9fX1dLFszLDYsIiIsMSx7InN0eWxlIjp7ImJvZHkiOnsibmFtZSI6ImRhc2hlZCJ9fX1dLFs2LDQsIiIsMSx7InN0eWxlIjp7ImJvZHkiOnsibmFtZSI6ImRhc2hlZCJ9fX1dLFs0LDcsIiIsMSx7Im9mZnNldCI6LTIsInN0eWxlIjp7ImJvZHkiOnsibmFtZSI6InNxdWlnZ2x5In19fV0sWzQsOF0sWzgsN10sWzAsMSwiIiwwLHsib2Zmc2V0IjotMiwic3R5bGUiOnsiYm9keSI6eyJuYW1lIjoic3F1aWdnbHkifX19XSxbMCwxLCIiLDAseyJvZmZzZXQiOjJ9XSxbMSw0LCIiLDEseyJvZmZzZXQiOi0yLCJzdHlsZSI6eyJib2R5Ijp7Im5hbWUiOiJzcXVpZ2dseSJ9fX1dXQ==
\begin{tikzcd}
	&& {\text{Section 1}} \\
	&& {\text{Section 2}} \\
	{\text{A1, A2}} & {\text{Section 3}} \\
	{\text{A3}} & {\text{Section 4}} \\
	&& {\text{Section 5}} & {\text{Section 6}} \\
	&& {\text{Section 7}}
	\arrow[from=2-3, to=3-2]
	\arrow[from=3-2, to=4-2]
	\arrow[from=4-2, to=5-3]
	\arrow[dashed, from=3-2, to=3-1]
	\arrow[dashed, from=3-1, to=4-2]
	\arrow[dashed, from=4-2, to=4-1]
	\arrow[dashed, from=4-1, to=5-3]
	\arrow[shift left=2, squiggly, from=5-3, to=6-3]
	\arrow[from=5-3, to=5-4]
	\arrow[from=5-4, to=6-3]
	\arrow[shift left=2, squiggly, from=1-3, to=2-3]
	\arrow[shift right=2, from=1-3, to=2-3]
	\arrow[shift left=2, squiggly, from=2-3, to=5-3]
\end{tikzcd}
\caption{Dependency of the sections}
\label{fig:sec-dep}
\end{figure}
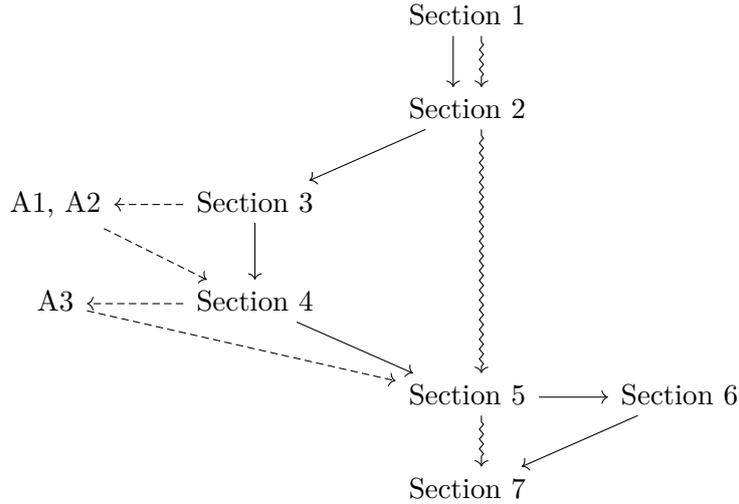

In the subsequent two sections we set the stage for our study of synthetic cocartesian fibrations which makes up the main part of the paper. Namely, in \cref{sec:closure} we discuss closure properties of two kinds of ``fibrations'' which formally generalize cocartesian fibrations and their discrete analogue. \cref{sec:isoinner-fams} captures \emph{non-functorial} families of synthetic (pre-)$\inftyone$-categories. Similar to (for instance) Joyal and Lurie's \emph{inner} (or \emph{mid}) \emph{fibrations}, we discuss these somewhat auxiliary notions of fibrations since at the most general level, types are not necessarily categories.\footnote{Rather, they are (Reedy fibrant) simplicial objects.}

Afterwards, in \cref{sec:cocart-fams}, cocartesian fibrations are introduced and discussed, along with cocartesian functors. The development closely parallels Riehl--Verity's theory of cocartesian fibrations in $\infty$-cosmoses~\cite[Chapter 5]{RV}. The relations of our setup to discrete covariant fibrations, previously introduced by Riehl--Shulman, are discussed in \cref{sec:cov-fam}. As an application of our theory we present a Yoneda Lemma for cocartesian fibrations in \cref{sec:yoneda}, which type-theoretically can be viewed as a principle of \emph{directed arrow induction} for functorial families of Rezk types. This is a direct generalization of Riehl--Shulman's discrete Yoneda Lemma in~\cite[Section 9]{RS17}, and at the same time a type-theoretic analogue of Riehl--Verity's Yoneda Lemma in~\cite[Section 5.7]{RV}.

Appendices~\ref{app:lax-sq} and~\ref{app:sec:adj} cover basic results on left adjoint right inverse (LARI) adjunctions and mates, inspired by~\cite[Section 11]{RS17} and~\cite[Appendix B.3/4]{RV}.

In \cref{fig:sec-dep} we outline the logical dependency of the sections with suggested reading orders. Readers primarily interested in the results about cocartesian fibrations and the Yoneda Lemma are invited to follow the path indicated by squiggly arrows ``$\rightsquigarrow$''. This is also recommended for a first skim through the paper focusing on the main parts of the theory, avoiding some of the more technical aspects.
A complete tour through the main text is given by straight arrows ``$\rightarrow$'', extendable by side trips through the appendices as indicated by dashed arrows ``$\dashrightarrow$''.

\subsection{Related work}

This article extends the theory of discrete fibrations in the setting of
Riehl--Shulman's type theory for synthetic $\inftyone$-categories~\cite[Section~8]{RS17}
to cover various more general notions of fibered $\inftyone$-categories,
building on Riehl--Verity's theory of co-/cartesian fibrations in $\infty$-cosmoses,
\cf~\cite[Chapter~5]{RV} and \cite{RVscratch,RVyoneda}. These, in turn, generalize work by Gray~\cite{GrayFib} and Street~\cite{StrYon,StrBicat,StrBicatCorr}. An extensive analytical and model categorical account to (internal) co-/cartesian fibrations is laid out by Rasekh in~\cite{RasQCatSegCart,rasekh2021cartesian,Ras17Cart,RasYonDSp}.

Segal and Rezk spaces have been studied by Rezk~\cite{rez01}, Joyal--Tierney~\cite{joyal2007quasi}, Lurie~\cite{lurie2009goodwillie}, Kazhdan--Varshavsky~\cite{KV14}, and Rasekh~\cite{RasQCatSegCart,rasekh2021cartesian}. Segal objects are treated by Boavida de~Brito~\cite{dB16segal}, Stenzel~\cite{SteSegObj}, and Rasekh~\cite{Ras17Cart,RasUniv}. For broad developments of these and related techniques, see also the monographs by Simpson~\cite{simpson2011homotopy}, Bergner~\cite{bergner2018homotopy}, and Paoli~\cite{paoli2019simplicial}.

Originally, the theory of Grothendieck \aka~cartesian fibrations of $\inftyone$-categories (implemented as quasi-categories) has been developed by Joyal~\cite{joyal2007quasi} as well as Lurie~\cite{LurHTT}. Notable and fundamental follow-up work on fibrations of (internal) $\inftyone$-categories has been done by~Ayala--Francis~\cite{AFfib}, Boavida de~Brito~\cite{dB16segal}, Barwick--Dotto--Glasman--Nardin--Shah~\cite{barwick2016}, Barwick--Shah~\cite{BarwickShahFib}, Mazel-Gee~\cite{MazGee-UserCart}, Rezk~\cite{rezk2017stuff}, Cisinski~\cite{CisInfBook}, and Nguyen~\cite{NguyPhD}.

Recently, a model-agnostic and more semantics-minded  theory of internal $\inftyone$-categories and co-/cartesian fibrations internal to an $\inftyone$-topos has been under development by Martini~\cite{mar-yon,MarCocart} and Martini--Wolf~\cite{mw-lim}.

A range of \emph{directed homotopy type theories} has been proposed in the works of Warren~\cite{War-DTT-IAS}, Licata--Harper~\cite{LH2DTT}, North~\cite{NorthDHoTT}, Nuyts~\cite{NuyMSc,NuyPhD}, Kavvos~\cite{KavQuant}, and Altenkirch--Sestini~\cite{ASNatFree}. Going back to Voevodsky~\cite{VV-HTS} (originally under the name of \emph{homotopy type system (HTS)}) is the idea of \emph{two-level} type theories, which have been further treated by Capriotti~\cite{CapriottiPhD} and Annenkov--Capriotti--Kraus--Sattler~\cite{2ltt}. For a conceptual discussion, see the essay~\cite{B19} by the first named author.

Specifically in the case of simplicial type theory an account to directed univalence has been given by Cavallo--Riehl--Sattler~\cite{CRS18}, and in a bicubical setting by Weaver--Licata~\cite{WL19}. A development of limits and colimits in simplicial HoTT is now available due to Bardomiano Mart\'{i}nez~\cite{martinez2022limits}.

The theory of synthetic co-/cartesian fibrations in the article at hand has later been developed further by the second named author in his doctoral dissertation~\cite{jw-phd} and in subsequent articles to include so-called \emph{Beck--Chevalley} and \emph{Moens} or \emph{extensive fibrations}~\cite{Wei22-intlsums} (after \cite{MoePhD}, \cite[Sections~5 and 6]{streicher2020fibered}, and \cite{LietzDip}; \cf~also~\cite[\protect{[003C]} Theorem 1.1$\cdot$b]{SteGeoUnivTop}) as well as \emph{sliced} and \emph{two-sided} cartesian fibrations~\cite{W22-2sCart} (after \cite[Chapter~7]{RV}). A semantic treatment of extension types, complementing \cite[Appendix~A]{RS17}, is given in~\cite{Wei-StrExt}.

A proof assistant for simplicial type theory is being implemented by Kudasov~\cite{KudRzk}.

Another expansive treatment of simplicial homotopy type theory, both from a syntactic and semantic point of view, including cocartesian fibrations and other material discussed in the present paper is given in the master's thesis of Bakke~\cite{BakkeMSc}.

\section{Exposition of Riehl--Shulman's synthetic
  \texorpdfstring{$\inftyone$}{(∞,1)}-category theory}\label{sec:expo}

We recall some basic features and results from Riehl--Shulman's synthetic $\inftyone$-category theory~\cite{RS17}, at a very brief and informal level. A significantly more thorough treatment is provided in the original paper.

\subsection{Simplicial type theory}

\subsubsection{Shapes}
In terms of homotopy theory, the shape layer enables us to reason about generating anodyne cofibrations using strict equalities.

In simplicial HoTT, next to the familiar layer of (univalent) intensional Martin-L\"of type theory, there are new ``non-fibrant'' layers added that provide a logical calculus of geometric shapes. We start of from the \emph{cube layer}, \ie, a Lawvere theory generated by a single bi-pointed object $0,1:\I$, the \emph{standard $1$-cube}. A cube context
\[ \Xi \jdeq[ I_1, \ldots,  I_k] \]
is thus a finite list of cubes
\[ I_m ~\mathrm{cube}\]
for $1 \le m \le k$.

On top of the cube layer, we can form \emph{topes} through logical comprehension via (intuitionistic) conjunction $\land$, disjunction $\lor$, and equality $\jdeq$. The \emph{tope layer} hence captures sub-polytopes of $n$-cubes (with explicit embedding). A tope formula $\varphi(t_1,\ldots,t_k)$ together with a cube context $\Xi \jdeq [I_1, \ldots, I_k, \vec{J}]$ gives rise to a tope
\[ \Xi \vdash \varphi ~\mathrm{tope}.\]

The interval $\I$ under consideration shall also come equipped with an \emph{inequality} tope
\[ x,y: \I \vdash x \le y ~ \mathrm{tope}\]
making it a total order (w.r.t.~the \emph{strict} equality tope $\jdeq$) with $0$ and $1$ as the bottom and top element, respectively.\footnote{For a comparison with the setup of \emph{cubical type theory}~\cite{CCHM2018} cf.~\cite[Remark~3.2]{RS17}.}
In particular, we also sometimes make use of \emph{connections} on the cube terms as discussed in~\cite[Proposition~3.5]{RS17}.

A cube together with a tope is called a \emph{shape}:
\[
\frac{I\,\cube \qquad t : I \vdash \varphi\,\tope}%
{\set{t : I}{\varphi}\,\type}
\]

As an addition to the original theory by Riehl--Shulman, we will moreover coerce all shapes to be types, cf.~\cref{ssec:fib-shapes}. This is still in accordance with the intended class of models.

\subsubsection{Extension types}

In addition to the strict layers, the other new feature of simplicial type theory is a new type former called the \emph{extension type}, the idea of which originally was due to Lumsdaine and Shulman. Given a shape $\Psi$ and a type family $P: \Gamma \to \Psi \to  \UU$ together with a \emph{partial} section $a: \prod_{\Gamma \times \Phi} P$, where $\Phi \subseteq \Psi$ denotes a subshape, we can form the corresponding family of extension types
\[ \Gamma \vdash \exten{t:\Psi}{P(t)}{\Phi}{a}\]
which is interpreted as a (strict) pullback, \cf~\cite[Theorem~A.16]{RS17}:
% https://q.uiver.app/?q=WzAsNixbMCwwLCJcXEJpZ1xcbGFuZ2xlIFxccHJvZF97dDpcXFBzaX0gUCh0KSBcXEJpZ3xfYV5cXFBoaVxcQmlnIFxccmFuZ2xlIl0sWzIsMSwiXFx3aWRldGlsZGV7UH1eXFxQaGkgXFx0aW1lc197KFxcR2FtbWEgXFx0aW1lcyBcXFBzaSleXFxQaGl9IChcXEdhbW1hIFxcdGltZXMgXFxQc2kpXlxcUHNpIl0sWzMsMl0sWzAsMSwiXFxHYW1tYSJdLFszLDFdLFsyLDAsIlxcd2lkZXRpbGRle1B9XlxcUHNpIl0sWzAsMywiIiwwLHsic3R5bGUiOnsiaGVhZCI6eyJuYW1lIjoiZXBpIn19fV0sWzAsNV0sWzUsMSwiIiwwLHsic3R5bGUiOnsiaGVhZCI6eyJuYW1lIjoiZXBpIn19fV0sWzMsMSwiXFxsYW5nbGUgYSxcXGlvdGFfXFxQc2lcXHJhbmdsZSIsMl0sWzAsMSwiIiwxLHsic3R5bGUiOnsibmFtZSI6ImNvcm5lciJ9fV1d
\[\begin{tikzcd}
	{\exten{t:\Psi}{P(t)}{\Phi}{a}} && {\widetilde{P}^\Psi} \\
	\Gamma && {\widetilde{P}^\Phi \times_{(\Gamma \times \Psi)^\Phi} (\Gamma \times \Psi)^\Psi} & {} \\
	&&& {}
	\arrow[two heads, from=1-1, to=2-1]
	\arrow[from=1-1, to=1-3]
	\arrow[two heads, from=1-3, to=2-3]
	\arrow["{\pair{\overline{a}}{\overline{\id_{\Gamma \times \Psi}}}}"' swap, from=2-1, to=2-3]
	\arrow["\lrcorner"{anchor=center, pos=0.125}, draw=none, from=1-1, to=2-3]
\end{tikzcd}\]
This means, the elements of $\exten{t:\Psi}{P(t)}{\Phi}{a}$ are total sections $b: \prod_{\Gamma \times \Psi} P$ such that $b|_\Phi \jdeq a$ holds judgmentally:
% https://q.uiver.app/?q=WzAsNCxbMCwwLCJcXFBoaV4qXFx3aWRldGlsZGV7UH0iXSxbMCwxLCJcXEdhbW1hIFxcdGltZXMgXFxQaGkiXSxbMiwxLCJcXEdhbW1hIFxcdGltZXMgXFxQc2kiXSxbMiwwLCJcXHdpZGV0aWxkZXtQfSJdLFswLDEsIiIsMCx7InN0eWxlIjp7ImhlYWQiOnsibmFtZSI6ImVwaSJ9fX1dLFsxLDIsIiIsMCx7InN0eWxlIjp7InRhaWwiOnsibmFtZSI6Imhvb2siLCJzaWRlIjoidG9wIn19fV0sWzAsM10sWzMsMiwiIiwyLHsic3R5bGUiOnsiaGVhZCI6eyJuYW1lIjoiZXBpIn19fV0sWzAsMiwiIiwxLHsic3R5bGUiOnsibmFtZSI6ImNvcm5lciJ9fV0sWzIsMywiYiIsMix7ImN1cnZlIjozLCJzdHlsZSI6eyJib2R5Ijp7Im5hbWUiOiJkb3R0ZWQifX19XSxbMSwwLCJhIiwwLHsiY3VydmUiOi0zLCJzdHlsZSI6eyJib2R5Ijp7Im5hbWUiOiJkb3R0ZWQifX19XV0=
\[\begin{tikzcd}
	{\Phi^*\widetilde{P}} && {\widetilde{P}} \\
	{\Gamma \times \Phi} && {\Gamma \times \Psi}
	\arrow[two heads, from=1-1, to=2-1]
	\arrow[hook, from=2-1, to=2-3]
	\arrow[from=1-1, to=1-3]
	\arrow[two heads, from=1-3, to=2-3]
	\arrow["\lrcorner"{anchor=center, pos=0.125}, draw=none, from=1-1, to=2-3]
	\arrow["b"', curve={height=18pt}, dotted, from=2-3, to=1-3]
	\arrow["a", curve={height=-18pt}, dotted, from=2-1, to=1-1]
\end{tikzcd}\]
The type-theoretic rules are analogous to the familiar rules of $\Pi$-types, but with the desired judgmental equalities added, \cf~\cite[Figure~4]{RS17}.

In particular, non-dependent instances give rise to function types $A^\Phi$ where $\Phi$ is a shape rather than a type (even though, later on all of our shapes are assumed to be fibrant, cf.~Subsection~\ref{ssec:fib-shapes}). Semantically, this reflects the fact that the intended model is cotensored over simplicial sets, cf.~also the discussion in~\cite[Appendix~A]{RS17}.

From the given rules one can show that the extension types interact well with the usual $\Pi$- and $\Sigma$-types, as shown in~\cite[Subsections~4.1, 4.2]{RS17}. In particular, there is a version of the type-theoretic axiom of choice involving extension types that will be used a lot.

\begin{thm}[Type-theoretic axiom of choice for extension types, {\protect\cite[Theorem~4.2]{RS17}}]\label{thm:choice}
Let $\Phi \subseteq \Psi$ be a shape inclusion. Suppose we are given families $P: \Psi \to \UU$, $Q: \prod_{t:\Psi} (P(t) \to \UU)$ and sections $a:\prod_{t:\Phi} P(t)$, $b:\prod_{t:\Phi} Q(t,a(t))$. Then there is an equivalence
\[ \exten{t:\Psi}{\sum_{x:P(t)} Q(t,x)}{\Phi}{\lambda t.\pair{a(t)}{b(t)}} \equiv \sum_{f:\exten{t:\Psi}{P(t)}{\Phi}{a}} \exten{t:\Psi}{Q(t,f(t))}{\Phi}{b}.\]
\end{thm}

A further important principle is \emph{relative function extensionality}, which is added as an axiom:
\begin{ax}[Relative function extensionality, {\protect\cite[Axiom 4.6]{RS17}}]\label{ax:relfunext}
	Let $\Phi \subseteq \Psi$ be a shape inclusion.
  Given a family $P: \Psi \to \UU$ such that each $P(t)$ is contractible,
  and a partial section $a:\prod_{t: \Phi} P(t)$, then the extension type $\exten{t:\Psi}{P(t)}{\Phi}{a}$ is contractible.
\end{ax}
An important consequence is the \emph{homotopy extension property (HEP)}:
\begin{prop}[Homotopy extension property (HEP), {\protect\cite[Proposition 4.10]{RS17}}]\label{prop:hep}
Fix a shape inclusion $\Phi \subseteq \Psi$. Let $P: \Psi \to \UU$ be family, $b:\prod_{t:\Psi} P(t)$ a total section, and $a:\prod_{t:\Phi} A(t)$ a partial section. Then, given a homotopy $H:\prod_{t:\Phi} a(t) = b(t)$, there exist totalizations $a':\exten{t:\Psi}{P(t)}{\Phi}{a}$ and $H':\exten{t:\Psi}{a'(t)=b(t)}{\Phi}{H}$.
 \end{prop}

\subsubsection{Semantics in simplicial spaces}

A model of simplicial type theory is given by the Reedy model structure on bisimplicial sets, which presents the $\inftyone$-topos of simplicial spaces. The main steps in proving this are discussed in~\cite[Appendix A]{RS17}, with previous work done in~\cite{ShuReedyUniv,CisUniv}. The splitting of the extension type former is carried out in~\cite{Wei-StrExt}. In fact, it is conceivable that one can replace the base by an arbitrary (Grothendieck) $\inftyone$-topos $\mathscr E$ so that the results developed synthetically will hold for Rezk objects (\ie, internal $\inftyone$-categories, \cf~\eg~\cite{lurie2009goodwillie,dB16segal,SteSegObj,RasQCatSegCart} and~\cite[Proposition~E.3.7]{RV}) in $\mathscr E$, also as explained in~\emph{ibid}. One should note that the definitions and constructions we are presenting always have characterizations in terms of basic notions, such as (fibered) weak equivalences, (LARI) adjunctions, etc. (See for instance the characterizations of cocartesian fibrations and functors via \cref{thm:cocart-fams-intl-char,thm:cocart-fam-intl-char-fib,thm:cocart-fun-intl-char})
By Riehl--Verity's results on model-independence and Rasekh's work on simplicial and (complete) Segal spaces one can systematically argue that, in essence, all of our internal notions externalize to their intended semantic counterparts---at least when restricting to the Rezk types, which are our objects of primary interest after all.

\subsection{Synthetic higher categories}\label{ssec:syn-higher-cats}
\begin{figure}
	% % https://q.uiver.app/?q=WzAsMjQsWzUsMSwiXFxsYW5nbGUgMCwwIFxccmFuZ2xlIl0sWzYsMSwiXFxsYW5nbGUgMSwwIFxccmFuZ2xlIl0sWzYsMCwiXFxsYW5nbGUgMSwxIFxccmFuZ2xlIl0sWzQsMCwiXFxsYW5nbGUgMSwxIFxccmFuZ2xlIl0sWzQsMSwiXFxsYW5nbGUgMSwwIFxccmFuZ2xlIl0sWzMsMSwiXFxsYW5nbGUgMCwwIFxccmFuZ2xlIl0sWzMsMl0sWzUsMl0sWzMsMCwiXFxsYW5nbGUgMCwxIFxccmFuZ2xlIl0sWzUsMCwiXFxsYW5nbGUgMCwxIFxccmFuZ2xlIl0sWzEsMSwiMCJdLFsyLDEsIjEiXSxbNywyXSxbNywxLCJcXGxhbmdsZSAwLDAgXFxyYW5nbGUiXSxbOCwxLCJcXGxhbmdsZSAxLDAgXFxyYW5nbGUiXSxbNywwLCJcXGxhbmdsZSAwLDEgXFxyYW5nbGUiXSxbOCwwLCJcXGxhbmdsZSAxLDEgXFxyYW5nbGUiXSxbMCwxLCJcXGJ1bGxldCJdLFswLDIsIlxcRGVsdGFeMCJdLFsxLDJdLFsyLDJdLFs0LDJdLFs2LDJdLFs4LDJdLFswLDFdLFsxLDJdLFswLDJdLFs0LDNdLFs1LDRdLFsxMCwxMV0sWzEzLDE0XSxbMTMsMTVdLFsxNCwxNl0sWzE1LDE2XSxbMTUsMTQsIiIsMSx7ImxldmVsIjoyLCJzdHlsZSI6eyJoZWFkIjp7Im5hbWUiOiJub25lIn19fV0sWzE5LDIwLCJcXERlbHRhXjEiLDEseyJzdHlsZSI6eyJib2R5Ijp7Im5hbWUiOiJub25lIn0sImhlYWQiOnsibmFtZSI6Im5vbmUifX19XSxbNiwyMSwiXFxMYW1iZGFfMV4yIiwxLHsic3R5bGUiOnsiYm9keSI6eyJuYW1lIjoibm9uZSJ9LCJoZWFkIjp7Im5hbWUiOiJub25lIn19fV0sWzcsMjIsIlxcRGVsdGFeMiIsMSx7InN0eWxlIjp7ImJvZHkiOnsibmFtZSI6Im5vbmUifSwiaGVhZCI6eyJuYW1lIjoibm9uZSJ9fX1dLFsxMiwyMywiXFxEZWx0YV4xIFxcdGltZXMgXFxEZWx0YV4xIiwxLHsic3R5bGUiOnsiYm9keSI6eyJuYW1lIjoibm9uZSJ9LCJoZWFkIjp7Im5hbWUiOiJub25lIn19fV0sWzI2LDEsIiIsMix7InNob3J0ZW4iOnsic291cmNlIjoyMH0sInN0eWxlIjp7ImhlYWQiOnsibmFtZSI6Im5vbmUifX19XV0=
\[\begin{tikzcd}[sep=7mm]
    &[2mm] &&[2mm]
    {\langle 0,1 \rangle} & {\langle 1,1 \rangle} &[2mm]
    {\langle 0,1 \rangle} & {\langle 1,1 \rangle} &[2mm]
    {\langle 0,1 \rangle} & {\langle 1,1 \rangle} \\
	\bullet & 0 & 1 & {\langle 0,0 \rangle} & {\langle 1,0 \rangle} & {\langle 0,0 \rangle} & {\langle 1,0 \rangle} & {\langle 0,0 \rangle} & {\langle 1,0 \rangle} \\[-5mm]
	{\textstyle\Delta^0} & {} & {} & {} & {} & {} & {} & {} & {}
	\arrow[from=2-6, to=2-7]
	\arrow[from=2-7, to=1-7]
	\arrow[""{name=0, anchor=center, inner sep=0}, from=2-6, to=1-7]
	\arrow[from=2-5, to=1-5]
	\arrow[from=2-4, to=2-5]
	\arrow[from=2-2, to=2-3]
	\arrow[from=2-8, to=2-9]
	\arrow[from=2-8, to=1-8]
	\arrow[from=2-9, to=1-9]
	\arrow[from=1-8, to=1-9]
	\arrow[Rightarrow, no head, from=1-8, to=2-9]
	\arrow["{\textstyle\Delta^1}"{description}, draw=none, from=3-2, to=3-3]
	\arrow["{\textstyle\Lambda_1^2}"{description}, draw=none, from=3-4, to=3-5]
	\arrow["{\textstyle\Delta^2}"{description}, draw=none, from=3-6, to=3-7]
	\arrow["{\textstyle\Delta^1 \times \Delta^1}"{description}, draw=none, from=3-8, to=3-9]
	\arrow[shorten <=2pt, Rightarrow, no head, from=0, to=2-7]
\end{tikzcd}\]
	\caption{Some shapes}
	\label{fig:sample-shapes}
\end{figure}
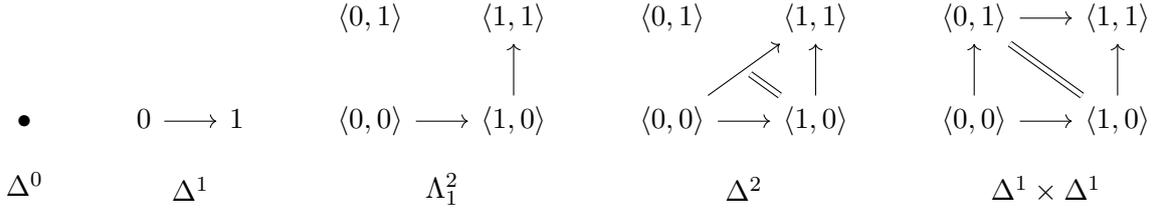

Via the inequality tope of the interval $\I$ we can define simplices and subshapes familiar from simplicial homotopy theory. The first few low-dimensional simplices are given by
\begin{align*}
	\Delta^0 \defeq \{ t:\unit \; | \; \top\}, ~~ 	\Delta^1 \defeq \{ t:\I  \; | \; \top\}, ~~ \Delta^2 \defeq \{ \! \pair{t}{s} : \I \times \I \; | \;  s \le t \}.
\end{align*}
The logical connectives of the tope layer enable us to carve out subshapes, such as boundaries and horns, \eg
\begin{align*}
	& \partial \Delta^1 \defeq \{ t:\I \; | \;  t \jdeq 0 \lor t \jdeq 1\}, \quad \Lambda_1^2 \defeq \{ \pair{t}{s} : \I \times \I  \; | \;  s \jdeq 0 \lor t \jdeq 1\}, \\
	& \partial \Delta^2 \defeq \{ \pair{t}{s} : \I \times \I  \; | \;  s \jdeq t \lor s \jdeq 0 \lor t \jdeq 1\}.
\end{align*}
Cf.~\cref{fig:sample-shapes} for an illustration and~\cite[Section 3.2]{RS17} for a detailed discussion.

\begin{figure}
	\centering
	\includegraphics{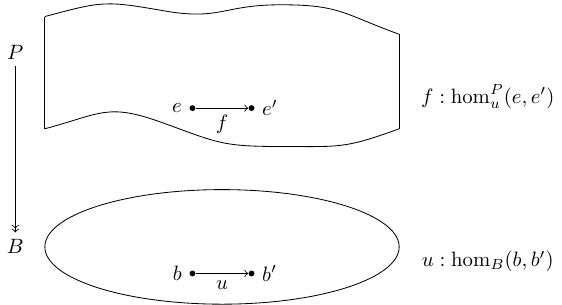}
	\caption{Dependent arrows}
	\label{fig:dhom}
\end{figure}

We then can define, for any type $B$ and fixed elements $b,b':B$ the type of arrows from $b$ to $b'$ as
\[
  \hom_B(b,b') \defeq \ndexten{\Delta^1}{B}{\partial \Delta^1}{[b,b']}.
\]
Given a type family $P:B \to \UU$ and an arrow $u:\hom_B(b,b')$ in the base, the type of arrows lying over $u$, from $e:P\,b$ to $e':P\,b'$, is given by
\[ \dhom^P_u(e,e') \defeq \exten{t:\Delta^1}{P(u(t))}{\partial \Delta^1}{[e,e']}.\]
Such an arrow is also called a \emph{dependent arrow} or \emph{dependent homomomorphism}, \cf~\cref{fig:dhom}.

We will also be considering types of $2$-cells, defined by\footnote{The boundary here is given by
	\[ [u,v,w]: \partial \Delta^2 \to B, \quad [u,v,w](t,s) \defeq
	\begin{cases}
		u & s \jdeq 0 \\
		v & t \jdeq 1 \\
		w & s \jdeq t
	\end{cases}
	\]
	and similarly for the dependent case.}
\[ \hom_B^2(u,v;w) \jdeq \ndexten{\Delta^2}{B}{\partial \Delta^2}{[u,v,w]}, \quad \hom^{2,P}_\sigma(f,g;h) \jdeq \exten{\langle t,s \rangle : \Delta^2}{P(\sigma(t,s))}{\partial \Delta^2}{[f,g,h]},\]
see~\cref{fig:dhom2} for an illustration:
\begin{figure}[H]
	\centering
	\includegraphics{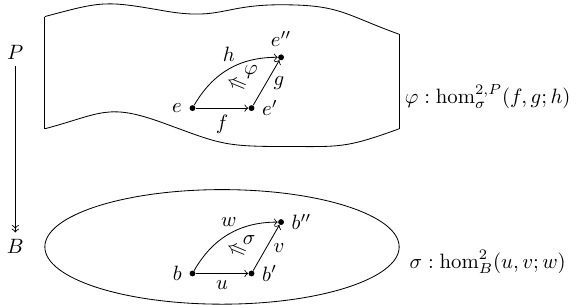}
	\caption{Dependent $2$-cells}
	\label{fig:dhom2}
\end{figure}
We abbreviate
\[ \hom_B(b,b') \jdeq  (b \to_B b') \jdeq  (b \to b')\]
when the intent is clear from the context. For families $P : B \to \UU$, we write
\[ \hom_u^P(e,e') \jdeq (e \to^P_u e').\]
For a type $B$, we have two projections from the arrow type, given by evaluation
\[ \partial_k: B^{\Delta^1} \to B, \quad \partial_k \defeq \lambda u.u(k),\]
for $k=0,1$.

Similarly to the notation introduced above, for the type of natural transformations between a fixed pair of functors we abbreviate
\[ \nat{A}{B}(f,g) \jdeq  (f \Rightarrow g).\]
Sometimes, we also denote the type of $2$-simplices by
\[ \hom_B^2(u,v;w) \jdeq (u,v \Rightarrow w),\]
and likewise for the dependent version.
With these prerequisites, Riehl--Shulman~\cite{RS17} define a type $B$ to be a \emph{Segal type} such that the proposition
\[ \isSegal(B) \defeq \prod_{b,b',b'':B} \prod_{\substack{u:b \to b' \\ v: b' \to b''}} \isContr\Big( \sum_{w:b \to b''} \hom_B^2(u,v;w)\Big) \]
is true. This means that $B$ has \emph{weak composition of directed arrows}. After Joyal, the Segal condition can be stated as
\[
  \isSegal(B) \simeq \isEquiv (B^\iota), \]
  with the inclusion where $\iota: \Lambda_1^2 \hookrightarrow \Delta^2$, \ie, $B^\iota \defeq (- \circ \iota): B^{\Delta^2} \to B^{\Lambda_1^2}$ restricts filled $2$-simplices in $B$ onto their $(2,1)$-horn.
  
Segal types can be thought of as synthetic \emph{pre}-$\inftyone$-categories, which here in simplicial homotopy type theory is expressed as a \emph{property} rather than structure, echoing the familiar situation from the semantics in simplicial spaces. As discussed in~\cite[Section 5]{RS17}, Segal types can be endowed with a weak composition \emph{operation} which is weakly associative. For arrows $f:a \to b$, $g:b \to c$ in some Segal type $B$, one writes $g \circ f$ for the chosen composite arrow.

The \emph{identity arrow} of an element $b:B$ is given by the constant map
\[\id_b \defeq \lambda t.b:\Delta^1 \to B. \]

Often, naturality \wrt~directed arrows comes ``for free''. In particular, any function~$f:A \to B$ between Segal types is a \emph{functor} in the sense that it preserves composition and identities up to propositional equality, as proven in~\cite[Section 6.1]{RS17}. The action of a functor on points already determines its actions on arrows as discussed in \emph{loc.~cit}.

Although the semantics is given by a structure presenting an $\inftyone$-topos---a certain kind of $\inftyone$-category---we, in fact, have access to portions of the $2$-dimensional structure present in the model as well. Since Segal types form an exponential ideal, the type $A \to B$ is Segal if $B$ is, and this allows us to study natural transformations between functors and lax diagrams of types, cf.\ Appendix~\ref{app:lax-sq} and the groundwork in~\cite[Section~6]{RS17}. This enables us to adapt several developments from Riehl--Verity's model-independent higher category theory from \emph{$\infty$-cosmoses} to type theory. Roughly, an $\infty$-cosmos\footnote{The name ``$\infty$-cosmos'' has been suggested by Peter May~\cite[Acknowledgments]{RV}, after Street's (fibrational) cosmoses \cite{StreetEltCosI,StreetCosIntl}.} is a model for an $\inftytwo$-categorical universe of a given notion of ``$\infty$-category'' which could mean $(\infty,n)$-categories (for $0 \le n \le \infty$) but also \eg,~fibrations thereof. This semantical theory will turn out to be an important guiding stone for our internal developments. 

Segal types come with two possible notions of isomorphism (analogous to the semantic situation for Segal spaces), the ``spatial'' one given by propositional equality, and the ``categorical'' one derived from the directed arrows. Namely, an arrow $f:a \to b$ in a Segal type $B$ is a (categorical) \emph{isomorphism} if the type
\[ \isIso(f) \defeq \sum_{g,h:b \to a} (gf = \id_a) \times (fh = \id_b) \]
is inhabited. As discussed in~\cite[Section 10]{RS17}, the type $\isIso(f)$ turns out to be a proposition, so we can define the subtypes
\[ \iso_B(a,b) \defeq \sum_{f:\hom_A(a,b)} \isIso(a,b).\]
In particular, for all $b:B$ one has\footnote{leaving the witnessing proposition implicit in our notation} $\id_b : \iso_B(a,b)$.
By path induction we define the comparison map
\[ \idtoiso_B: \prod_{a,b:B} (a=_Bb) \to \iso_B(a,b), \quad \idtoiso_{B,a,a}(\refl_a) \defeq \id_a , \]
and demanding that this be an equivalence leads to the notion of \emph{Rezk completeness:}
\[ \isRezkCompl(B) \defeq \isEquiv(\idtoiso_B).\]
A \emph{complete Segal type} \aka~\emph{Rezk type} is a Segal type that is also Rezk complete. A type $B:\UU$ being Rezk is witnessed by the proposition
\begin{align*}
	{\isRezk}(B) & \defeq \isSegal(B) \times \isRezkCompl(B) \\
			   & \simeq \isEquiv(B^\iota) \times \isEquiv(\idtoiso_B).
\end{align*}
Then Rezk completeness condition can be understood as a local version of the Univalence Axiom. In the simplicial space model, Rezk types are interpreted as Rezk spaces, which justifies viewing Rezk types as \emph{synthetic $\inftyone$-categories}. Even though a lot of the development in~\cite{RS17} actually already works well on the level of (not necessarily complete) Segal types, our study of cocartesian families mostly restricts to complete Segal types, which is in line with preexisting studies of 
% TODO: manual linebreak
\linebreak co-/cartesian fibrations in the higher categorical context \cite{JoyNotesQcat,LurHTT,RV,Ras17Cart,dB16segal,AFfib,BarwickShahFib}.

Among the synthetic $\inftyone$-categories, we can also consider types that are synthetic $\inftyzero$-categories, \ie, $\infty$-groupoids. These are called \emph{discrete types},\footnote{semantically known as \emph{Bousfield--Segal spaces} due to Bergner~\cite[Section~6]{Bergner-AddInvII}, \cf~also earlier work by Bousfield~\cite{BousItLoop} and more recent work by Stenzel~\cite{Stenzel-CBS}} which refers to the condition that all directed arrows be invertible, namely the comparison map defined inductively by
\[ \idtoarr_B: \prod_{a,b:B} (a=_B b) \to \hom_B(a,b), \quad \idtoarr_{B,a,a}(\refl_a) \defeq \id_a,\]
be an equivalence:
\[ \isDisc(B) \defeq \prod_{a,b:B} \isEquiv(\idtoarr_{B,a,b})\]
In fact, this discreteness condition entails the Rezk condition as shown by Riehl--Shulman. Furthermore, if $B$ is a Segal type, for any $a,b:B$, the hom-type $\hom_B(a,b)$ is discrete by~\cite[Proposition 8.13]{RS17}.

\subsection{Covariant families}

Riehl--Shulman have introduced the notion of \emph{covariant family}, \ie, families of discrete types varying functorially \wrt~directed arrows in the base. A type family $P:B \to \UU$ over a (Segal or Rezk) type $B$ is \emph{covariant} if
\[ \isCovFam(P) \defeq \prod_{\substack{a,b:B \\ u:a \to b}} \prod_{d:P\,a} \isContr\Big( \sum_{e:P\,b} (d \to^P_u e) \Big), \]
\ie, arrows in the base can be uniquely lifted w.r.t.~a given source vertex.

Semantically, these correspond to \emph{left fibrations}, which encode $\inftyone$-copresheaves. Hence, as expected, an example is given, for any $a:B$ by the family
\[ \hom_B(a,-) \defeq \lambda b.\hom_B(a,b): B \to \UU.\]

The central topic of our work is to generalize this study to synthetic \emph{cocartesian fibrations}, \ie, the case where the fibers are Rezk rather than discrete.\footnote{Everything dualizes to the case of cartesian fibrations, of course, but we don't spell this out.}

\subsection{Fibrant shapes}\label{ssec:fib-shapes}

In the models we are interested in, the shapes will be fibrant objects,
so we may reflect the shape layer into the type layer.
In a generalized algebraic presentation of the syntax,
this will be done with a constructor corresponding to the inference rule:\footnote{The introduction, elimination and computation rules state elements of type $\shapety I\varphi$ are precisely those of the corresponding shape $I\mid\varphi$. We elide these forms from our notation to ease readability.}
\[
\frac{I\,\cube \qquad t : I \vdash \varphi\,\tope}%
{\shapety{t : I}{\varphi}\,\type}
\]
In fact, we take the interval $\Delta^1$ as a type, and the inequality relation ${\le} : \Delta^1 \to \Delta^1 \to \Prop$ as a type family, and then all shapes are types using the ordinary type formers.
It seems it should be possible to develop everything we do using standard type theory extended with $(\Delta^1,{\le})$ along with the postulate that this is an \emph{interval object}, \ie, a totally ordered set with bottom element $0$ and top element $1$,
where $0 \ne 1$.
This nicely complements the fact that the $1$-topos of simplicial sets is the classifying topos for the geometric theory of a strict interval~\cite{JohnstoneTopTop}.

Recall that if $A$ is a type and $B$ is a type family over $A$,
then $\prod_{x:A}B(x)$ is equivalent to the type of sections of
the first projection $\sum_{x:A}B(x) \to A$.
Similarly, the extension type $\exten{x:\shapety I\psi}{A(x)}{\varphi}{a}$
is equivalent to
\[
\sum_{f : \prod_{x:\shapety I\psi}A(x)}\prod_{x:\shapety I\varphi}(a\,x = f\,x),
\]
\ie, the type of functions over the large shape together with a homotopy
to the given function over the small shape.
Indeed, for any $a$, we have a map from the extension type to the latter type,
so it suffices to show that the induced map on total types,
\[
  \left(\sum_{a:\prod_{x:\shapety I\varphi}A(x)}
    \exten{x:\shapety I\psi}{A(x)}{\varphi}{a}\right)
\to
\left(\sum_{a:\prod_{x:\shapety I\varphi}A(x)}
\sum_{f : \prod_{x:\shapety I\psi}A(x)}\prod_{x:\shapety I\varphi}(a\,x = f\,x)\right)
\]
is an equivalence. The codomain here is equivalent to
just $\prod_{x:\shapety I\psi}A(x)$ (by relative function extensionality),
and the composite maps $\pair{a}{f}$ to $f$. This map has an obvious inverse
that maps $f$ to the pair of $f$ restricted to $\shapety I\varphi$ and $f$.

In particular, this gives the following formulation of (diagrammatic, weak) lifting problems in terms of (formulaic, strict) contractibility statements.

\begin{obs}
Consider a family $P:B \to \UU$ and a shape inclusion $\Phi \subseteq \Psi$. Then, given a total diagram $\sigma:\Psi \to B$ with a partial diagram $\kappa:\prod_\Phi \sigma^*P$ lying over,
the diagram
% https://q.uiver.app/?q=WzAsNixbMCwwLCJcXFBoaSJdLFswLDEsIlxcUHNpIl0sWzIsMSwiQiJdLFsyLDAsIlxcd2lkZXRpbGRle1B9Il0sWzMsMF0sWzMsMV0sWzAsMSwiIiwwLHsic3R5bGUiOnsidGFpbCI6eyJuYW1lIjoiaG9vayIsInNpZGUiOiJ0b3AifX19XSxbMSwyLCJcXHNpZ21hIiwxXSxbMywyLCJcXHBpIiwwLHsic3R5bGUiOnsiaGVhZCI6eyJuYW1lIjoiZXBpIn19fV0sWzAsMywiXFxrYXBwYSIsMV0sWzEsMywiIiwyLHsic3R5bGUiOnsiYm9keSI6eyJuYW1lIjoiZGFzaGVkIn19fV1d
\[\begin{tikzcd}
	\Phi && {\widetilde{P}} & {} \\
	\Psi && B & {}
	\arrow[hook, from=1-1, to=2-1]
	\arrow["\sigma"{description}, from=2-1, to=2-3]
	\arrow["\pi", two heads, from=1-3, to=2-3]
	\arrow["\kappa"{description}, from=1-1, to=1-3]
	\arrow[dashed, from=2-1, to=1-3]
\end{tikzcd}\]
possesses a diagonal filler uniquely up to homotopy if and only if the proposition
\[ %\prod_{\sigma:\Psi \to B} \prod_{\kappa: \prod_{t:\Phi} P(\sigma(t))}
  \mathrm{isContr} \Big( \Big\langle\prod_{t:\Psi} P(\sigma(s))\Big|^\Phi_\kappa \Big \rangle\Big)
\]
is inhabited.
\end{obs}

\begin{expl}
Recall from~\cite{RS17} that a type $A$ is Segal precisely if $A \to \unit$ is right orthogonal to $\Lambda_1^2 \hookrightarrow \Delta^2$. Another example is given by the class of covariant families, namely $P:B \to \UU$ is covariant if and only if $\totalty{P} \to B$ is right orthogonal to the initial vertex inclusion $i_0: \unit \hookrightarrow \Delta^1$.
\end{expl}

We can also type-theoretically express the \emph{Leibniz construction} familiar from categorical homotopy theory~\cite[Definitions C.2.8, C.2.10, C.3.8]{RV} as follows. We remark that Leibniz cotensor maps will be ubiquitous in our treatise since the fibrations of interest are defined by conditions on them.

\begin{defn}[Leibniz cotensor]
	Let $j:Y \to X$ be a type map or shape inclusion, and $\pi :E \to B$ a map between types. The \emph{Leibniz cotensor of $j$ and $\pi$} (\aka~\emph{Leibniz exponential of $\pi$ by $j$} or \emph{pullback hom}) is defined as the following gap map:
	% https://q.uiver.app/?q=WzAsMTEsWzAsMSwiWSJdLFswLDIsIlgiXSxbMSwxLCJFIl0sWzEsMiwiQiJdLFsyLDEsIkVeWCJdLFsyLDIsIkJeWSBcXHRpbWVzX3tCXlh9IEVeWSJdLFs0LDEsIlxcY2RvdCJdLFs0LDIsIkJeWSJdLFs2LDIsIkJeWSJdLFs2LDEsIkVeWSJdLFszLDAsIkVeWCJdLFswLDEsImoiLDJdLFs0LDUsImogXFx3aWRlaGF0e1xccGl0Y2hmb3JrfSBcXHBpIl0sWzYsN10sWzcsOF0sWzYsOV0sWzksOF0sWzEwLDYsImpcXHdpZGVoYXR7XFxwaXRjaGZvcmt9IFxccGkiLDEseyJzdHlsZSI6eyJib2R5Ijp7Im5hbWUiOiJkYXNoZWQifX19XSxbMTAsOSwiIiwyLHsiY3VydmUiOi0yfV0sWzEwLDcsIiIsMix7ImN1cnZlIjoyfV0sWzYsOCwiIiwxLHsic3R5bGUiOnsibmFtZSI6ImNvcm5lciJ9fV0sWzIsM10sWzIxLDEyLCJcXGVxdWl2IiwxLHsibGVuZ3RoIjo3MCwic3R5bGUiOnsiYm9keSI6eyJuYW1lIjoibm9uZSJ9LCJoZWFkIjp7Im5hbWUiOiJub25lIn19fV0sWzExLDIxLCJcXHdpZGVoYXR7XFxwaXRjaGZvcmt9IiwxLHsibGVuZ3RoIjo3MCwic3R5bGUiOnsiYm9keSI6eyJuYW1lIjoibm9uZSJ9LCJoZWFkIjp7Im5hbWUiOiJub25lIn19fV1d
	\[\begin{tikzcd}
		&&& {E^X} \\
		{Y} & {E} & {E^X} && {\cdot} && {E^Y} \\
		{X} & {B} & {B^X \times_{B^Y} E^Y} && {B^X} && {B^Y}
		\arrow["{j}"{name=0, swap}, from=2-1, to=3-1]
		\arrow["{j \widehat{\pitchfork} \pi}"{name=1}, from=2-3, to=3-3]
		\arrow[from=2-5, to=3-5]
		\arrow[from=3-5, to=3-7]
		\arrow[from=2-5, to=2-7]
		\arrow[from=2-7, to=3-7]
		\arrow["{j\widehat{\pitchfork} \pi}" description, from=1-4, to=2-5, dashed]
		\arrow[from=1-4, to=2-7, curve={height=-12pt}]
		\arrow[from=1-4, to=3-5, curve={height=12pt}]
		\arrow["\lrcorner"{very near start, rotate=0}, from=2-5, to=3-7, phantom]
		\arrow["{\pi}"{name=2, swap}, from=2-2, to=3-2, swap]
		\arrow[Rightarrow, "{\defeq}"{description,pos=0.25}, from=2, to=1, shorten <=7pt, shorten >=7pt, phantom, no head]
		\arrow[Rightarrow, "{\widehat{\pitchfork}}" description, from=0, to=2, shorten <=5pt, shorten >=5pt, phantom, no head]
	\end{tikzcd}\]
\end{defn}

The map $\pi$ is right orthogonal to $j$, meaning that for any square as below there exists a filler uniquely up to homotopy
% https://q.uiver.app/?q=WzAsNCxbMCwwLCJZIl0sWzAsMSwiWCJdLFsxLDEsIkIiXSxbMSwwLCJFIl0sWzAsMSwiaiIsMl0sWzEsMl0sWzAsM10sWzMsMiwiXFxwaSJdLFsxLDMsIiIsMix7InN0eWxlIjp7ImJvZHkiOnsibmFtZSI6ImRvdHRlZCJ9fX1dXQ==
\[\begin{tikzcd}
	{Y} & {E} \\
	{X} & {B}
	\arrow["{j}"', from=1-1, to=2-1]
	\arrow[from=2-1, to=2-2]
	\arrow[from=1-1, to=1-2]
	\arrow["{\pi}", from=1-2, to=2-2]
	\arrow[from=2-1, to=1-2, dotted]
\end{tikzcd}\]
if and only if the Leibniz cotensor map is an equivalence
\[ j \widehat{\pitchfork} \pi: E^X \stackrel{\equiv}{\longrightarrow} B^X \times_{B^Y} E^Y, \]
\cf~\cref{prop:orth-maps}.

Though sparsely explicitly present in the text, we will also mention the dual operation.

\begin{defn}[Pushout product]
	Let $j: Y \to X$ and $k: T \to S$ each be type maps or shape inclusions. The \emph{Leibniz tensor of $j$ and $k$} (or \emph{pushout product}) is defined as the following cogap map:
	% https://q.uiver.app/?q=WzAsMTEsWzAsMCwiWSJdLFswLDEsIlgiXSxbMSwwLCJUIl0sWzEsMSwiUyJdLFsyLDAsIlkgXFx0aW1lcyBTXFxiaWdzcWN1cF97WSBcXHRpbWVzIFR9IFggXFx0aW1lcyBUIl0sWzIsMSwiWCBcXHRpbWVzIFMiXSxbNCwwLCJZIFxcdGltZXMgVCJdLFs0LDEsIlggXFx0aW1lcyBUIl0sWzYsMSwiXFxjZG90Il0sWzYsMCwiWSBcXHRpbWVzIFMiXSxbNywyLCJYIFxcdGltZXMgUyJdLFswLDEsImoiLDJdLFsyLDMsImsiXSxbNCw1LCJqXFx3aWRlaGF0e1xcb3RpbWVzfSBrIl0sWzYsN10sWzcsOF0sWzYsOV0sWzksOF0sWzcsMTAsIiIsMix7ImN1cnZlIjoyfV0sWzksMTAsIiIsMix7ImN1cnZlIjotMn1dLFs4LDEwLCJqIFxcd2lkZWhhdHtcXG90aW1lc31rIiwxLHsic3R5bGUiOnsiYm9keSI6eyJuYW1lIjoiZGFzaGVkIn19fV0sWzgsNiwiIiwxLHsic3R5bGUiOnsibmFtZSI6ImNvcm5lciJ9fV0sWzExLDEyLCJcXHdpZGVoYXR7XFxvdGltZXN9IiwxLHsibGVuZ3RoIjo3MCwic3R5bGUiOnsiYm9keSI6eyJuYW1lIjoibm9uZSJ9LCJoZWFkIjp7Im5hbWUiOiJub25lIn19fV0sWzEyLDEzLCJcXGVxdWl2IiwxLHsibGVuZ3RoIjo3MCwic3R5bGUiOnsiYm9keSI6eyJuYW1lIjoibm9uZSJ9LCJoZWFkIjp7Im5hbWUiOiJub25lIn19fV1d
	\[\begin{tikzcd}
		{Y} & {T} & {Y \times S\bigsqcup_{Y \times T} X \times T} && {Y \times T} && {Y \times S} \\
		{X} & {S} & {X \times S} && {X \times T} && {\cdot} \\
		&&&&&&& {X \times S}
		\arrow["{j}"{name=0, swap}, from=1-1, to=2-1]
		\arrow["{k}"{name=1, swap}, from=1-2, to=2-2, swap]
		\arrow["{j\widehat{\otimes} k}"{name=2}, from=1-3, to=2-3]
		\arrow[from=1-5, to=2-5]
		\arrow[from=2-5, to=2-7]
		\arrow[from=1-5, to=1-7]
		\arrow[from=1-7, to=2-7]
		\arrow[from=2-5, to=3-8, curve={height=12pt}]
		\arrow[from=1-7, to=3-8, curve={height=-12pt}]
		\arrow["{j \widehat{\otimes}k}" description, from=2-7, to=3-8, dashed]
		\arrow["\lrcorner"{very near start, rotate=180}, from=2-7, to=1-5, phantom]
		\arrow[Rightarrow, "{\widehat{\otimes}}" description, from=0, to=1, shorten <=5pt, shorten >=5pt, phantom, no head]
		\arrow[Rightarrow, "{\defeq}" {description,pos=0.25}, from=1, to=2, shorten <=9pt, shorten >=9pt, phantom, no head]
	\end{tikzcd}\]
\end{defn}

In particular, recall from~\cite[Theorem 4.2]{RS17}, the explicit formula for the pushout product of two shape inclusions:
% https://q.uiver.app/?q=WzAsNyxbMCwwLCJcXHt0OklcXCx8XFwsXFx2YXJwaGlcXH0iXSxbMCwyXSxbMCwxLCJcXHt0OklcXCx8XFwsXFxwc2lcXH0iXSxbMSwwLCJcXHtzOkpcXCx8XFwsXFxjaGlcXH0iXSxbMSwxLCJcXHtzOkpcXCx8XFwsXFx6ZXRhXFx9Il0sWzIsMCwiXFx7IFxcbGFuZ2xlIHQsc1xccmFuZ2xlIDogSSBcXHRpbWVzIEogXFwsIHwgXFwsIChcXHZhcnBoaSBcXGxhbmQgXFx6ZXRhKSBcXGxvciAoXFxwc2kgXFxsYW5kIFxcY2hpKVxcfSJdLFsyLDEsIlxceyBcXGxhbmdsZSB0LHNcXHJhbmdsZSA6IEkgXFx0aW1lcyBKIFxcLCB8IFxcLCBcXHBzaSBcXGxhbmQgXFx6ZXRhXFx9Il0sWzAsMiwiIiwwLHsic3R5bGUiOnsidGFpbCI6eyJuYW1lIjoiaG9vayIsInNpZGUiOiJ0b3AifX19XSxbMyw0LCIiLDAseyJzdHlsZSI6eyJ0YWlsIjp7Im5hbWUiOiJob29rIiwic2lkZSI6InRvcCJ9fX1dLFs1LDYsIiIsMCx7InN0eWxlIjp7InRhaWwiOnsibmFtZSI6Imhvb2siLCJzaWRlIjoidG9wIn19fV0sWzcsOCwiXFx3aWRlaGF0e1xcb3RpbWVzfSIsMSx7Imxlbmd0aCI6NzAsInN0eWxlIjp7ImJvZHkiOnsibmFtZSI6Im5vbmUifSwiaGVhZCI6eyJuYW1lIjoibm9uZSJ9fX1dLFs4LDksIlxcZXF1aXYiLDEseyJsZW5ndGgiOjcwLCJzdHlsZSI6eyJib2R5Ijp7Im5hbWUiOiJub25lIn0sImhlYWQiOnsibmFtZSI6Im5vbmUifX19XV0=
\[\begin{tikzcd}
	{\{t:I\,|\,\varphi\}} & {\{s:J\,|\,\chi\}} & {\{ \langle t,s\rangle : I \times J \, | \, (\varphi \land \zeta) \lor (\psi \land \chi)\}} \\
	{\{t:I\,|\,\psi\}} & {\{s:J\,|\,\zeta\}} & {\{ \langle t,s\rangle : I \times J \, | \, \psi \land \zeta\}} \\
	{}
	\arrow[""{name=0, inner sep=0}, from=1-1, to=2-1, hook]
	\arrow[""{name=1, inner sep=0}, from=1-2, to=2-2, hook]
	\arrow[""{name=2, inner sep=0}, from=1-3, to=2-3, hook]
	\arrow[Rightarrow, "{\widehat{\otimes}}" description, from=0, to=1, shorten <=7pt, shorten >=7pt, phantom, no head]
	\arrow[Rightarrow, "{\defeq}"  {description,pos=0.25}, from=1, to=2, shorten <=13pt, shorten >=13pt, phantom, no head]
\end{tikzcd}\]

\subsection{Families vs.~fibrations}\label{ssec:fam-vs-fib}

Recall from~\cite{hottbook} that in presence of the univalence axiom, there is an equivalence between type families and fibrations.\footnote{Assuming universes with better structural properties---such as the Segal condition or directed univalence---would be fruitful for further considerations, but this is part of future work.}

Consider the types
\[ \Fib(\UU) \defeq \sum_{A,B:\UU} A \to B, \qquad \Fam(\UU) \defeq \sum_{B:\UU} (B \to \UU)  \]
of functions in $\UU$ (viewed as type-theoretic fibrations\footnote{The inhabitants of $\Fib(\UU)$ are just maps between arbitrary $\UU$-small types, but viewed as ``$\UU$-small type theoretic fibrations over a $\UU$-small base''.}), and families with $\UU$-small fibers, resp. Both these types naturally are fibered over $\UU$ via the following maps:
% https://q.uiver.app/?q=WzAsMyxbMCwwLCJcXEFycihcXFVVKSJdLFsyLDAsIlxcVVUiXSxbNCwwLCJcXEZhbShcXFVVKSJdLFswLDEsIlxccGFydGlhbF8xIl0sWzIsMSwiXFxwcl8xIiwyXV0=
\[\begin{tikzcd}
	{\Fib(\UU)} && \UU && {\Fam(\UU)}
	\arrow["{\partial_1 \defeq \lambda A,B,f.B}", from=1-1, to=1-3]
	\arrow["{\pr_1 \defeq \lambda B,P.B}"', from=1-5, to=1-3]
\end{tikzcd}\]
Over a type $B:\UU$, we obtain the type of \emph{maps into (or fibrations over) $B$} as the fiber:
% https://q.uiver.app/?q=WzAsNCxbMCwwLCJcXFVVL0IiXSxbMCwxLCIxIl0sWzIsMSwiXFxVVSJdLFsyLDAsIlxcbWF0aHJte0Fycn0oXFxVVSkiXSxbMCwxXSxbMSwyLCJCIiwyXSxbMCwzXSxbMywyLCJcXHBhcnRpYWxfMSJdLFswLDIsIiIsMix7InN0eWxlIjp7Im5hbWUiOiJjb3JuZXIifX1dXQ==
\[\begin{tikzcd}
	{\UU/B} && {\mathrm{Fib}(\UU)} \\
	\unit && \UU
	\arrow[from=1-1, to=2-1]
	\arrow["B"', from=2-1, to=2-3]
	\arrow[from=1-1, to=1-3]
	\arrow["{\partial_1}", from=1-3, to=2-3]
	\arrow["\lrcorner"{anchor=center, pos=0.125}, draw=none, from=1-1, to=2-3]
\end{tikzcd}\]

\begin{theorem}[Type-theoretic Grothendieck construction, cf.~{\protect\cite[Theorem 4.8.3]{hottbook}}]
	There is a fiberwise quasi-equivalence
	% https://q.uiver.app/?q=WzAsMyxbMCwwLCJcXEFycihcXFVVKSJdLFsxLDEsIlxcVVUiXSxbMiwwLCJcXEZhbShcXFVVKSJdLFswLDEsIlxccGFydGlhbF8xIiwyXSxbMiwxLCJcXHByXzEiXSxbMiwwLCJcXFVuIiwwLHsib2Zmc2V0IjotMX1dLFswLDIsIlxcU3QiLDAseyJvZmZzZXQiOi0yfV1d
	\[\begin{tikzcd}
		{\Fib(\UU)} && {\Fam(\UU)} \\
		& \UU
		\arrow["{\partial_1}"', from=1-1, to=2-2]
		\arrow["{\pr_1}", from=1-3, to=2-2]
		\arrow["\Un", shift left=1, from=1-3, to=1-1]
		\arrow["\St", shift left=2, from=1-1, to=1-3]
	\end{tikzcd}\]
	at stage $B:\UU$ given by a pair
	% https://q.uiver.app/?q=WzAsMixbMCwwLCJcXFVVL0IiXSxbMiwwLCIoQlxcdG8gXFxVVSkiXSxbMCwxLCJcXFN0X0IiLDAseyJvZmZzZXQiOi0yfV0sWzEsMCwiXFxVbl9CIiwwLHsib2Zmc2V0IjotMn1dXQ==
	\[\begin{tikzcd}
		{\UU/B} && {(B\to \UU)}
		\arrow["{\St_B}", shift left=2, from=1-1, to=1-3]
		\arrow["{\Un_B}", shift left=1, from=1-3, to=1-1]
	\end{tikzcd}\]
	with \emph{straightening}
	\[ \St_B(\pi) \defeq \lambda b.\fib_b(\pi)\]
	and \emph{unstraightening}
	\[ \Un_B(P) \defeq \pair{\totalty{P}}{\pi_P}\]
	($\pi_P: \totalty{P} \defeq \sum_{b:B} P\,b \to B$ the total space projection).
\end{theorem}

The spirit of dependent type theory somewhat favors type families over fibrations, but we will often resort to the fibrational viewpoint because it allows us to replay familiar categorical arguments. For instance, Riehl--Shulman's covariant type families are a type-theoretic version of left fibrations, and we want to be able to conveniently make use of both incarnations of the same concept which motivates the following:

\begin{defn}[Notions of families and fibrations]
	A \emph{notion of family (or notion of fibration)} is a family
	\[ \mathcal F : \Fam(\UU) \to \Prop \]
	of propositions on the type of $\UU$-small fibrations. For a notion of family $\mathcal F$, we say that a family $P:B\to \UU$ is an \emph{$\mathcal F$-family}\footnote{In practice, the name of $\mathcal F$ will often be a linguistic predicate such as ``covariant'', ``cocartesian'' etc.~in which case we drop the hyphen and treat it as part of the natural meta-language, \eg,~we will simply speak of ``cocartesian'' or ``covariant fibrations''.} if and only if the proposition
	\[ \isFam_{\mathcal F}(P) \defeq \mathcal F(P)\]
	holds. A map $\pi:E \to B$ is called an \emph{$\mathcal F$-fibration} if its family of fibers $\St_B(\pi)$ is an $\mathcal F$-family.
\end{defn}
By univalence and the Grothendieck construction, this definition is well-behaved, \ie, a~($\UU$-small) map is an~$\mathcal F$-fibration if and only if it is (equivalent to) a projection associated to an $\mathcal F$-family (valued in $\UU$).

In particular, we observe the following. Considering
\[ \Fib_\mathcal F(\UU) \defeq \sum_{\substack{E,B:\UU \\ \pi:E \to B}} \isFam_{\mathcal F}(\St(\pi)), \quad \Fam_{\mathcal F}(\UU) \defeq \sum_{P:\Fam(\UU)} \isFam_{\mathcal F}(P), \]
the Grothendieck construction descends to a fiberwise equivalence, for any notion of fibration/family $\mathcal F$:

% https://q.uiver.app/?q=WzAsNSxbMCwwLCJcXG1hdGhybXtGaWJ9X3tcXG1hdGhjYWwgRn0oXFxtYXRoY2FsIFUpIl0sWzEsMSwiXFxtYXRocm17RmlifShcXG1hdGhjYWwgVSkiXSxbMiwwLCJcXG1hdGhybXtGYW19X3tcXG1hdGhjYWwgRn0oXFxtYXRoY2FsIFUpIl0sWzMsMSwiXFxtYXRocm17RmFtfShcXG1hdGhjYWwgVSkiXSxbMSwyLCJcXG1hdGhjYWwgVSJdLFswLDEsIiIsMix7InN0eWxlIjp7InRhaWwiOnsibmFtZSI6Imhvb2siLCJzaWRlIjoidG9wIn19fV0sWzAsMiwiXFxtYXRocm17U3R9IiwwLHsib2Zmc2V0IjotMn1dLFsyLDMsIiIsMCx7InN0eWxlIjp7InRhaWwiOnsibmFtZSI6Imhvb2siLCJzaWRlIjoidG9wIn19fV0sWzEsMywiXFxtYXRocm17U3R9IiwwLHsib2Zmc2V0IjotMn1dLFswLDQsIiIsMix7ImN1cnZlIjoxfV0sWzIsNF0sWzMsNF0sWzEsNF0sWzIsMCwiXFxtYXRocm17VW59IiwwLHsib2Zmc2V0IjotMn1dLFszLDEsIlxcbWF0aHJte1VufSIsMCx7Im9mZnNldCI6LTJ9XSxbNiwxMywiXFxzaW1lcSIsMSx7InNob3J0ZW4iOnsic291cmNlIjoyMCwidGFyZ2V0IjoyMH0sInN0eWxlIjp7ImhlYWQiOnsibmFtZSI6Im5vbmUifX19XSxbOCwxNCwiXFxzaW1lcSIsMSx7InNob3J0ZW4iOnsic291cmNlIjoyMCwidGFyZ2V0IjoyMH19XV0=
\[\begin{tikzcd}
	{\mathrm{Fib}_{\mathcal F}(\mathcal U)} && {\mathrm{Fam}_{\mathcal F}(\mathcal U)} \\
	& {\mathrm{Fib}(\mathcal U)} && {\mathrm{Fam}(\mathcal U)} \\
	& {\mathcal U}
	\arrow[hook, from=1-1, to=2-2]
	\arrow[""{name=0, anchor=center, inner sep=0}, "{\mathrm{St}}", shift left=2, from=1-1, to=1-3]
	\arrow[hook, from=1-3, to=2-4]
	\arrow[curve={height=6pt}, from=1-1, to=3-2]
	\arrow[from=1-3, to=3-2]
	\arrow[from=2-4, to=3-2]
	\arrow[from=2-2, to=3-2]
	\arrow[""{name=2, anchor=center, inner sep=0}, "{\mathrm{Un}}", shift left=2, from=1-3, to=1-1]
	\arrow[""{name=3, anchor=center, inner sep=0}, "{\mathrm{Un}}", shift left=2, from=2-4, to=2-2, crossing over]
	\arrow[""{name=1, anchor=center, inner sep=0}, "{\mathrm{St}}", shift left=2, from=2-2, to=2-4, crossing over]
	\arrow["\simeq"{description}, shorten <=1pt, shorten >=1pt, Rightarrow, from=1, to=3]
	\arrow["\simeq"{description}, shorten <=1pt, shorten >=1pt, Rightarrow, no head, from=0, to=2]
\end{tikzcd}\]

\begin{rem}
As a convention, we will always state the definitions of the various notions of fibration in terms of \emph{families}, and the above definition schema immediately yields the respective corresponding notion in fibrational terms.

As a convention, we will also often denote a map which satisfies such a fibration condition (or possibly even just a usual map which is to be regarded as a type-theoretic fibration) by a double hooked arrow $\pi:E \fibarr B$, as is customary in homotopical algebra or categorical homotopy theory.
\end{rem}

Furthermore, as a consequence of univalence, any such propositionally defined notion of family/fibration is invariant under equivalence.

\begin{prop}[Homotopy invariance of notions of fibrations]
	Let $\mathcal F$ be a notion of fibration. When given a commutative square
	% https://q.uiver.app/?q=WzAsNCxbMCwwLCJGIl0sWzAsMSwiQSJdLFsyLDEsIkIiXSxbMiwwLCJFIl0sWzAsMSwiXFx4aSIsMl0sWzEsMiwiXFxzaW1lcSIsMl0sWzAsMywiXFxzaW1lcSJdLFszLDIsIlxccGkiXV0=
	\[\begin{tikzcd}
		F && E \\
		A && B
		\arrow["\xi"', from=1-1, to=2-1]
		\arrow["\simeq"', from=2-1, to=2-3]
		\arrow["\simeq", from=1-1, to=1-3]
		\arrow["\pi", from=1-3, to=2-3]
	\end{tikzcd}\]
	the map $\xi$ is an $\mathcal F$-fibration if and only if $\pi$ is.
\end{prop}

\subsection{Comma and co-/cone types}\label{ssec:commas-cones}
\begin{defn}[Comma types]\label{def:comma-obj}
Consider a cospan of types
% https://q.uiver.app/?q=WzAsMyxbMCwwLCJDIl0sWzIsMCwiQSJdLFs0LDAsIkIiXSxbMCwxLCJnIl0sWzIsMSwiZiIsMl1d
\[\begin{tikzcd}
	C && A && B
	\arrow["g", from=1-1, to=1-3]
	\arrow["f"', from=1-5, to=1-3]
\end{tikzcd}\]
The \emph{comma type} $f \downarrow g$ is given by the following pullback:
	% https://q.uiver.app/?q=WzAsNCxbMCwwLCJmIFxcZG93bmFycm93IGciXSxbMCwxLCJDIFxcdGltZXMgQiJdLFsyLDEsIkEgXFx0aW1lcyBBIl0sWzIsMCwiQV57XFxEZWx0YV4xfSJdLFswLDFdLFsxLDIsIlxcbGFuZ2xlIGcsIGYgXFxyYW5nbGUgIl0sWzAsM10sWzMsMiwiXFxsYW5nbGUgXFxwYXJ0aWFsXzAsIFxccGFydGlhbF8xIFxccmFuZ2xlIl0sWzAsMiwiIiwwLHsic3R5bGUiOnsibmFtZSI6ImNvcm5lciJ9fV1d
	\[\begin{tikzcd}
		{f \downarrow g} && {A^{\Delta^1}} \\
		{C \times B} && {A \times A}
		\arrow[from=1-1, to=2-1]
		\arrow["g \times f", from=2-1, to=2-3]
		\arrow[from=1-1, to=1-3]
		\arrow["{\langle \partial_1, \partial_0 \rangle}", from=1-3, to=2-3]
		\arrow["\lrcorner"{anchor=center, pos=0.125}, draw=none, from=1-1, to=2-3]
	\end{tikzcd}\]
In the case that $f$ is the identity $\id_A$, we write shorthand $\comma{A}{g}$ for $\comma{f}{g}$, and dually $\comma{f}{A}$ if $g$ is the identity.
\end{defn}

\begin{defn}[Co-/cone types]
Let $X$ be a type or a shape and $A$ a type. In a setting such as the present one, a map $u:X \to A$ is sometimes referred to as an \emph{$X$-shaped diagram in $A$}. The cospans
% https://q.uiver.app/?q=WzAsNixbMCwwLCJBIl0sWzIsMCwiQV5YIl0sWzQsMCwiXFx1bml0Il0sWzYsMCwiXFx1bml0Il0sWzgsMCwiQV5YIl0sWzEwLDAsIkEiXSxbMCwxLCJcXG1hdGhybXtjc3R9X0EiXSxbMiwxLCJ1IiwyXSxbMyw0LCJ1Il0sWzUsNCwiXFxtYXRocm17Y3N0fV9BIiwyXV0=
\[\begin{tikzcd}
	A && {A^X} && \unit && \unit && {A^X} && A
	\arrow["{\mathrm{cst}_A}", from=1-1, to=1-3]
	\arrow["u"', from=1-5, to=1-3]
	\arrow["u", from=1-7, to=1-9]
	\arrow["{\mathrm{cst}_A}"', from=1-11, to=1-9]
\end{tikzcd}\]
give rise to the type $u/A \defeq \comma{u}{\cst_A}$ of \emph{cocones in $A$ under $u$}, and, dually  $A/u \defeq \comma{\cst_A}{u}$ of \emph{cones in $A$ over $u$}, resp., defined as comma objects:
% https://q.uiver.app/?q=WzAsOCxbMCwwLCJ1L0EiXSxbMCwxLCJcXGNzdF9BIFxcdGltZXMgXFx1bml0Il0sWzIsMSwiXFxiaWcoQV5YXFxiaWcpXntcXERlbHRhXjF9IFxcdGltZXMgXFxiaWcoQV5YXFxiaWcpXntcXERlbHRhXjF9Il0sWzIsMCwiXFxiaWcoQV5YKV57XFxEZWx0YV4xfSJdLFs0LDAsIkEvdSJdLFs0LDEsIlxcdW5pdCBcXHRpbWVzIFxcY3N0X0EiXSxbNiwxLCJcXGJpZyhBXlhcXGJpZylee1xcRGVsdGFeMX0gXFx0aW1lcyBcXGJpZyhBXlhcXGJpZylee1xcRGVsdGFeMX0iXSxbNiwwLCJcXGJpZyhBXlgpXntcXERlbHRhXjF9Il0sWzEsMiwiXFxjc3RfQSBcXHRpbWVzIHUiXSxbMCwzXSxbMywyLCJcXGxhbmdsZSBcXHBhcnRpYWxfMSwgXFxwYXJ0aWFsXzAgXFxyYW5nbGUiXSxbMCwxXSxbMCwyLCIiLDAseyJzdHlsZSI6eyJuYW1lIjoiY29ybmVyIn19XSxbNCw1XSxbNSw2LCJ1IFxcdGltZXMgXFxjc3RfQSJdLFs0LDddLFs3LDYsIlxcbGFuZ2xlIFxccGFydGlhbF8xLCBcXHBhcnRpYWxfMCBcXHJhbmdsZSJdLFs0LDYsIiIsMSx7InN0eWxlIjp7Im5hbWUiOiJjb3JuZXIifX1dXQ==
\[\begin{tikzcd}
	{u/A} && {\big(A^X)^{\Delta^1}} && {A/u} && {\big(A^X)^{\Delta^1}} \\
	{A \times \unit} && {A^X \times A^X} && {\unit \times A} && {A^X \times A^X}
	\arrow["{\cst_A \times u}", from=2-1, to=2-3]
	\arrow[from=1-1, to=1-3]
	\arrow["{\langle \partial_1, \partial_0 \rangle}", from=1-3, to=2-3]
	\arrow[from=1-1, to=2-1]
	\arrow["\lrcorner"{anchor=center, pos=0.125}, draw=none, from=1-1, to=2-3]
	\arrow[from=1-5, to=2-5]
	\arrow["{u \times \cst_A}", from=2-5, to=2-7]
	\arrow[from=1-5, to=1-7]
	\arrow["{\langle \partial_1, \partial_0 \rangle}", from=1-7, to=2-7]
	\arrow["\lrcorner"{anchor=center, pos=0.125}, draw=none, from=1-5, to=2-7]
\end{tikzcd}\]
\end{defn}

\begin{expl}[Co-/slice types]
For a fixed point $b:B$, the \emph{co-/slice types} are defined as $b/B$ and $B/b$, resp. Note that $b/B \equiv \comma{b}{B}$, and similarly for the slice types.
\end{expl}

\section{Right orthogonal and LARI families}\label{sec:closure}
An important part in our study of synthetic fibered $(\infty,1)$-categories is to provide proofs of certain closure properties, which are chosen to parallel those of $\infty$-cosmoses~\cite[Definition~1.2.1]{RV}. Recall from~\cite{RV}, that any $\infty$-cosmos $\mathcal K$ provides an intrinsic notion of cocartesian fibrations, which themselves form an $\infty$-cosmos ${co\mc{C}art}({\mc{K}})$. The discrete cocartesian fibrations form an embedded $\infty$-cosmos $\mc{D}iscco\mc{C}art({\mc{K}}) \hookrightarrow {co\mc{C}art}({\mc{K}})$. We prove type-theoretic analogues of the $\infty$-cosmological closure properties, formulated internally to the type theory of the ``ambient'' $\inftyone$-topos of simplicial objects.

First, we consider maps which are, more generally, defined by a unique right lifting property against an arbitrary map. Next, we discuss $j$-LARI maps which are defined by a left adjoint right inverse condition on a Leibniz cotensor map.

Specifically, let $j:Y \to X$ be some type map.\footnote{This is often a shape inclusion; recall from Subsection~\ref{ssec:fib-shapes} that we have coercion of (strict) shapes into types.} A map $\pi:E \to B$ is called \emph{$j$-orthogonal} if any square as below has a contractible space of fillers:
% https://q.uiver.app/?q=WzAsNCxbMCwwLCJZIl0sWzAsMSwiWCJdLFsyLDEsIkIiXSxbMiwwLCJFIl0sWzAsMSwiaiIsMl0sWzEsMl0sWzAsM10sWzMsMiwiXFxwaSJdLFsxLDMsIiIsMSx7InN0eWxlIjp7ImJvZHkiOnsibmFtZSI6ImRhc2hlZCJ9fX1dXQ==
\[\begin{tikzcd}
	Y && E \\
	X && B
	\arrow["j"', from=1-1, to=2-1]
	\arrow[from=2-1, to=2-3]
	\arrow[from=1-1, to=1-3]
	\arrow["\pi", from=1-3, to=2-3]
	\arrow[dashed, from=2-1, to=1-3]
\end{tikzcd}\]
If a map $\pi:E \to B$ is right orthogonal to a map $j:Y \to X$, we write $j \bot \pi$. Similarly, for families $P:B \to \UU$, we write $j \bot P$ if $j \bot \Un_B(P)$.

Classes of maps defined by such lifting conditions play an important role in categorical homotopy theory and have been extensively studied in various contexts. In particular, classes defined by right orthogonal lifting conditions necessarily satisfy certain closure properties. We are giving type theoretic proofs which will later apply for the specific kinds of $j$-orthogonal maps that we are interested in, namely (iso)inner fibrations and left fibrations \aka~discrete covariant fibrations. For instance, a map $\pi:E \to B$ (over a Segal type $B$) is a covariant fibration if and only if it is right orthogonal to the initial vertex inclusion $i_0: \unit \hookrightarrow \Delta^1$.

In general, $\pi:E \to B$ being $j$-orthogonal can be rephrased as the condition that the gap map in the following diagram be an equivalence:
% https://q.uiver.app/?q=WzAsNSxbMSwxLCJCXlggXFx0aW1lc197Ql5ZfSBFXlkiXSxbMSwyLCJCXlgiXSxbMywyLCJCXlkiXSxbMywxLCJFXlkiXSxbMCwwLCJFXlgiXSxbMCwxXSxbMSwyXSxbMCwzXSxbMywyXSxbMCwyLCIiLDEseyJzdHlsZSI6eyJuYW1lIjoiY29ybmVyIn19XSxbNCwzLCIiLDEseyJjdXJ2ZSI6LTJ9XSxbNCwwLCJcXHNpbWVxIiwxLHsic3R5bGUiOnsiYm9keSI6eyJuYW1lIjoiZGFzaGVkIn19fV0sWzQsMSwiIiwxLHsiY3VydmUiOjJ9XV0=
\[\begin{tikzcd}
	{E^X} \\
	& {B^X \times_{B^Y} E^Y} && {E^Y} \\
	& {B^X} && {B^Y}
	\arrow[from=2-2, to=3-2]
	\arrow[from=3-2, to=3-4]
	\arrow[from=2-2, to=2-4]
	\arrow[from=2-4, to=3-4]
	\arrow["\lrcorner"{anchor=center, pos=0.125}, draw=none, from=2-2, to=3-4]
	\arrow[curve={height=-12pt}, from=1-1, to=2-4]
	\arrow["\equiv"{description}, dashed, from=1-1, to=2-2]
	\arrow[curve={height=12pt}, from=1-1, to=3-2]
\end{tikzcd}\]
Weakening this condition by requiring the gap map to only have a \emph{left adjoint right inverse (LARI)} leads to the notion of \emph{$j$-LARI} map, \ie, $\pi:E \to B$ is a $j$-LARI map if and only if the induced map $E^X \to B^X \times_{B^Y} E^Y$ has a LARI:
% https://q.uiver.app/?q=WzAsNSxbMiwxLCJCXlggXFx0aW1lc197Ql5ZfSBFXlkiXSxbMiwyLCJCXlgiXSxbNCwyLCJCXlkiXSxbNCwxLCJFXlkiXSxbMCwwLCJFXlgiXSxbMCwxXSxbMSwyXSxbMCwzXSxbMywyXSxbMCwyLCIiLDEseyJzdHlsZSI6eyJuYW1lIjoiY29ybmVyIn19XSxbNCwzLCIiLDEseyJjdXJ2ZSI6LTR9XSxbNCwxLCIiLDEseyJjdXJ2ZSI6Mn1dLFswLDQsIiIsMSx7ImN1cnZlIjoyLCJzdHlsZSI6eyJib2R5Ijp7Im5hbWUiOiJkb3R0ZWQifX19XSxbNCwwLCIiLDEseyJjdXJ2ZSI6MX1dLFsxMiwxMywiIiwxLHsibGV2ZWwiOjEsInN0eWxlIjp7Im5hbWUiOiJhZGp1bmN0aW9uIn19XV0=
\[\begin{tikzcd}
	{E^X} \\
	&& {B^X \times_{B^Y} E^Y} && {E^Y} \\
	&& {B^X} && {B^Y}
	\arrow[from=2-3, to=3-3]
	\arrow[from=3-3, to=3-5]
	\arrow[from=2-3, to=2-5]
	\arrow[from=2-5, to=3-5]
	\arrow["\lrcorner"{anchor=center, pos=0.125}, draw=none, from=2-3, to=3-5]
	\arrow[curve={height=-24pt}, from=1-1, to=2-5]
	\arrow[curve={height=12pt}, from=1-1, to=3-3]
	\arrow[""{name=0, anchor=center, inner sep=0}, curve={height=12pt}, dotted, from=2-3, to=1-1]
	\arrow[""{name=1, anchor=center, inner sep=0}, curve={height=6pt}, from=1-1, to=2-3]
	\arrow["\dashv"{anchor=center, rotate=-110}, draw=none, from=0, to=1]
\end{tikzcd}\]
Between Rezk types, a map $\pi:E \to B$ is a cocartesian fibration if and only if it is an $i_0$-LARI map, for the initial vertex inclusion $i_0: \unit \hookrightarrow \Delta^1$.

Both $j$-orthogonal and $j$-LARI maps are closed under dependent products, composition, and pullback. In addition, $j$-orthogonal maps are closed under sequential limits and Leibniz cotensoring. In addition, they satisfy left canceling.
\subsection{Right orthogonal fibrations}

\subsubsection{$j$-orthogonal families}

Recall from the discussion~\cite[Section 32.3]{RijIntro} the following formulation of right orthogonality of a map in type theoretic terms.
\begin{prop}\label{prop:orth-maps}
A map $\pi:E \to B$ is right orthogonal to a type map $j:Y \to X$ if and only if the \emph{Leibniz cotensor} map
	% https://q.uiver.app/?q=WzAsNSxbMCwwLCJFXlgiXSxbMSwyLCJCXlgiXSxbMywyLCJCXlkiXSxbMywxLCJFXlkiXSxbMSwxLCJcXGJ1bGxldCJdLFsxLDIsImpeKiIsMl0sWzMsMiwiXFxwaV5ZIl0sWzAsMywial4qIiwwLHsiY3VydmUiOi0yfV0sWzQsMV0sWzQsM10sWzQsMiwiIiwwLHsic3R5bGUiOnsibmFtZSI6ImNvcm5lciJ9fV0sWzAsMSwiXFxwaV5YIiwyLHsiY3VydmUiOjJ9XSxbMCw0LCJqIFxcd2lkZWhhdHtcXHBpdGNoZm9ya31cXHBpIiwxLHsic3R5bGUiOnsiYm9keSI6eyJuYW1lIjoiZGFzaGVkIn19fV1d
	\[\begin{tikzcd}
		{E^X} \\
		& \cdot && {E^Y} \\
		& {B^X} && {B^Y}
		\arrow["{j^*}"', from=3-2, to=3-4]
		\arrow["{\pi^Y}", from=2-4, to=3-4]
		\arrow["{j^*}", curve={height=-12pt}, from=1-1, to=2-4]
		\arrow[from=2-2, to=3-2]
		\arrow[from=2-2, to=2-4]
		\arrow["\lrcorner"{anchor=center, pos=0.125}, draw=none, from=2-2, to=3-4]
		\arrow["{\pi^X}"', curve={height=12pt}, from=1-1, to=3-2]
		\arrow["{j \widehat{\pitchfork}\pi}"{description}, dashed, from=1-1, to=2-2]
	\end{tikzcd}\]
	is an equivalence, \ie, the following proposition is true:
	\[ \prod_{v: X \to B} \prod_{f: \prod_Y (j^*v)^*P} \isContr\Big( \sum_{g:\prod_X v^*P} j^*g = f\Big)\]
\end{prop}
In the case of shape inclusions, by de-/strictifcation we obtain a ``strict'' version of this statement in terms of extension types (for a special case \cf~\cite[Theorem 8.5]{RS17}):
\begin{cor}\label{cor:orth-shapes}
	Let $P: B \to \UU$ be a type family. Then $P$ is orthogonal w.r.t.~a shape inclusion $t:I \, | \, \varphi \vdash \psi$ if and only if
	\[ \prod_{v:\Psi \to B} \prod_{f:\prod_{t:\Phi} P(v(t))} \isContr \left( \exten{t:\Psi}{P(v(t))}{\Phi}{f} \right),\]
	where $\Phi$ and $\Psi$ denote the topes, corresponding to $\varphi$ and $\psi$, resp.
\end{cor}

At first, we prove the analogous universal properties, for maps between types rather than families. Note that these are equivalent to the constructions given for families since as a general principle of homotopy type theory any map between types can be replaced by the projection of the total space of a family. Likewise, any (homotopy) commutative square can be seen as a map between families.

Stability under pullback follows from an instance of a Pullback Lemma for cube-shaped diagrams as considered by Rijke in~\cite[Section 24]{RijIntro}.

\subsubsection{Closure properties of $j$-orthogonal families}\label{ssec:clos-orth}

\begin{prop}[Equivalences are $j$-orthogonal]\label{prop:orth-maps-equivs}
For a type map or shape inclusion $j:Y \to X$, any equivalence $f:B' \equiv B$ is $j$-orthogonal.
\end{prop}	

\begin{proof}
	By univalence, we can assume $f$ to be the identity $\id_B$. Then the claim follows easily.
\end{proof}

\begin{prop}[Closedness under dependent products]
	Let $I$ be a type, and $B: I \to \UU$ be a family. Assume there is a function $P: \prod_{i:I} B_i \to \UU$. If there is a map or shape inclusion $f: Y \to X$ such that $f\bot P_i$ for all $i: I$, we also have $f \bot \prod_I P$.
\end{prop}

\begin{proof}
For $i:I$, denote by $\pi_i :E_i \to B_i$ the projection associated to the family $P_i : B_i \to \UU$.
By assumption, the exponential squares
% https://q.uiver.app/?q=WzAsNCxbMCwwLCJFX2leWCJdLFswLDEsIkJfaV5YIl0sWzIsMSwiQl9pXlkiXSxbMiwwLCJFX2leWSJdLFswLDFdLFsxLDJdLFswLDNdLFszLDJdLFswLDIsIiIsMSx7InN0eWxlIjp7Im5hbWUiOiJjb3JuZXIifX1dXQ==
\[\begin{tikzcd}
	{E_i^X} && {E_i^Y} \\
	{B_i^X} && {B_i^Y}
	\arrow[from=1-1, to=2-1]
	\arrow[from=2-1, to=2-3]
	\arrow[from=1-1, to=1-3]
	\arrow[from=1-3, to=2-3]
	\arrow["\lrcorner"{anchor=center, pos=0.125}, draw=none, from=1-1, to=2-3]
\end{tikzcd}\]
are pullbacks. Since dependent products preserve pullbacks and commute with exponents, we have
% https://q.uiver.app/?q=WzAsNCxbMCwwLCIoXFxwcm9kX3tpOkl9IEVfaSleWCBcXHNpbWVxIFxccHJvZF97aTpJfSBFX2leWCJdLFswLDEsIihcXHByb2Rfe2k6SX0gQl9pKV5YIFxcc2ltZXEgXFxwcm9kX3tpOkl9IEJfaV5YIl0sWzIsMSwiXFxwcm9kX3tpOkl9IEJfaV5ZIFxcc2ltZXEoXFxwcm9kX3tpOkl9IEJfaSleWSJdLFsyLDAsIlxccHJvZF97aTpJfSBFX2leWSBcXHNpbWVxKFxccHJvZF97aTpJfSBFX2kpXlkiXSxbMCwxXSxbMSwyXSxbMCwzXSxbMywyXSxbMCw1LCIiLDEseyJsZXZlbCI6MSwic3R5bGUiOnsibmFtZSI6ImNvcm5lciJ9fV1d
\[\begin{tikzcd}
	{(\prod_{i:I} E_i)^X \equiv \prod_{i:I} E_i^X} && {\prod_{i:I} E_i^Y \equiv(\prod_{i:I} E_i)^Y} \\
	{(\prod_{i:I} B_i)^X \equiv \prod_{i:I} B_i^X} && {\prod_{i:I} B_i^Y \equiv(\prod_{i:I} B_i)^Y}
	\arrow[from=1-1, to=2-1]
	\arrow[""{name=0, anchor=center, inner sep=0}, from=2-1, to=2-3]
	\arrow[from=1-1, to=1-3]
	\arrow[from=1-3, to=2-3]
	\arrow["\lrcorner"{anchor=center, pos=0.125}, draw=none, from=1-1, to=0]
\end{tikzcd}\]
as desired.
\end{proof}

\begin{cor}[Closedness under binary products]
	Let $C, B:\UU$ with families $P: B \to \UU$ and $Q: C \to \UU$. If both $P$ and $Q$ are orthogonal to a map $j:Y \to X$, then so is the binary product family $P \times Q : B \times C \to \UU$.
\end{cor}

\begin{cor}[Closedness under exponentiation]\label{prop:orth-exp}
	Let $P: B \to \UU$ be a family, orthogonal to a shape inclusion or type map $j:Y \to X$. Then, for any type or shape $Z$, also $P^Z$ is $j$-orthogonal.
\end{cor}

\begin{prop}[Closedness under composition, and left cancelation, cf.~{\protect\cite[Lemma~5.5.9]{RV}}]
	Let $j: Y \to X$ be a map. Let $P: B \to \UU$ and $Q: \totalty{P} \to \UU$ be families with composite $Q \compfam P: B \to \UU$. If $j \bot P$ then $j \bot Q \compfam P$ if and only if $j \bot Q$.
\end{prop}
\begin{proof}
	Consider $E \to B$ and $F \to B$, the associated total projections to the family $P$ and $Q$, resp. Observe that $\widetilde{Q \compfam P} \equiv \widetilde{Q}$. We assume $i \bot P$ which yields the following commutative diagram:
	\[
	\begin{tikzcd}
		F^X \ar[dd, bend right = 30] \ar[d] \ar[r] & F^Y\ar[d]  \ar[dd, bend left = 30] \\
		E^X \ar[dr, phantom, "\lrcorner", very near start] \ar[d] \ar[r] & E^Y \ar[d] \\
		B^X \ar[r] & B^Y \\
	\end{tikzcd}
	\]
	Then by the Pasting Lemma for pullbacks (e.g.~\cite[Theorem 22.5.8]{RijIntro}) the composite diagram is a pullback if and only if the upper square is a pullbacks.
\end{proof}

In presence of categorical universes of fibrations this statement would imply that the left adjoint to the change of base functor, \ie, postcomposition with the given fibration, preserves and reflects $j$-orthogonality.

\begin{theorem}[Closedness under pullback along families]
	Let $\pi:E \to B$ be right orthogonal to a type map $j:Y \to X$. Then, for any pullback
	% https://q.uiver.app/?q=WzAsNCxbMCwwLCJGIl0sWzAsMSwiQSJdLFsxLDAsIkUiXSxbMSwxLCJCIl0sWzAsMSwiXFx4aSIsMl0sWzAsMl0sWzEsMywiayIsMl0sWzIsMywiXFxwaSJdLFswLDMsIiIsMSx7InN0eWxlIjp7Im5hbWUiOiJjb3JuZXIifX1dXQ==
	\[\begin{tikzcd}
		F & E \\
		A & B
		\arrow["\xi"', from=1-1, to=2-1]
		\arrow[from=1-1, to=1-2]
		\arrow["k"', from=2-1, to=2-2]
		\arrow["\pi", from=1-2, to=2-2]
		\arrow["\lrcorner"{anchor=center, pos=0.125}, draw=none, from=1-1, to=2-2]
	\end{tikzcd}\]
the map $\xi \jdeq k^* \pi: F \to A$ is also right orthogonal to $j$.
\end{theorem}

\begin{proof}
By assumption, $E^X \equiv B^X \times_{B^Y} E^Y$. Now, since exponentiation preserves pullbacks, we obtain the following cube:
	% https://q.uiver.app/?q=WzAsOCxbMSwxLCJFXlkiXSxbMSwzLCJCXlkiXSxbMywzLCJBXlgiXSxbMywxLCJGXlgiXSxbMCwwLCJFXlgiXSxbMiwwLCJGXlkiXSxbMiwyLCJBXlkiXSxbMCwyLCJCXlgiXSxbMCwxXSxbMSwyXSxbMCwzXSxbMywyXSxbNCw1XSxbNSw2XSxbNCw3XSxbNyw2XSxbNywxXSxbNiwyXSxbNSwzXSxbNCwwXSxbMCwyLCIiLDEseyJzdHlsZSI6eyJuYW1lIjoiY29ybmVyIn19XSxbNSwyLCIiLDEseyJzdHlsZSI6eyJuYW1lIjoiY29ybmVyIn19XSxbNCwxLCIiLDEseyJzdHlsZSI6eyJuYW1lIjoiY29ybmVyIn19XV0=
	\[\begin{tikzcd}
		{E^X} && {F^Y} \\
		& {E^Y} && {F^X} \\
		{B^X} && {A^Y} \\
		& {B^Y} && {A^X}
		\arrow[from=2-4, to=4-4]
		\arrow[from=1-1, to=1-3]
		\arrow[from=1-3, to=3-3]
		\arrow[from=1-1, to=3-1]
		\arrow[from=3-1, to=4-2]
		\arrow[from=3-3, to=4-4]
		\arrow[from=1-3, to=2-4]
		\arrow[from=1-1, to=2-2]
		\arrow[from=3-1, to=3-3]
		\arrow["\lrcorner"{anchor=center, pos=0.125}, draw=none, from=2-2, to=4-4]
		\arrow["\lrcorner"{anchor=center, pos=0.125, rotate=-45}, draw=none, from=1-3, to=4-4]
		\arrow["\lrcorner"{anchor=center, pos=0.125, rotate=-45}, draw=none, from=1-1, to=4-2]
		\arrow[from=2-2, to=4-2, crossing over]
		\arrow[from=4-2, to=4-4]
		\arrow[from=2-2, to=2-4, crossing over]
	\end{tikzcd}\]
We have to show that the back square is a pullback as well, and this follows from the Cube Pullback Lemma~\cite[Lemma~24.1.4 and Remark~24.1.5]{RijIntro}.
\end{proof}

\begin{cor}[Closedness under pullback along maps]
	Let $B$ be a type, $P: B \to \UU$ a family, and $r:C \to B$ a map. Assume $j \bot P$. Then for the pullback $r^*P : C \to \UU$ we have $j \bot r^*P $.
\end{cor}

\begin{cor}[Fibers of $j$-orthogonal maps]\label{cor:fib-j-orth}
	Let $P:B \to \UU$ be a $j$-orthogonal family. Then, for any $b:B$ the fiber $P\,b$ is $j$-orthogonal.
\end{cor}

\begin{cor}[Closedness under fiber product]
	If $P,Q:B \to \UU$ are $j$-orthogonal families, then so is their fiber product $P \times_B Q : B \to \UU$.
\end{cor}

\begin{cor}[$j$-orthogonal maps from $j$-orthogonal types]\label{prop:orth-ext}
	Let $B$ be a type and $j:Y \to X$ a map between types, or a shape inclusion. If the type $B$ is $j$-orthogonal, then so is the map $B^k: B^C \to B^D$ for all $k: D \to C$.
\end{cor}

\begin{cor}[Closedness under Leibniz cotensors]
Let $\pi:E \to B$ be a $j$-orthogonal map. For any $k:D \to C$, the Leibniz cotensor map $k \cotens \pi: E^C \to E^D \times_{B^D} B^C$ is $j$-orthogonal as well.
\end{cor}

\begin{prop}[Orthogonal families are closed under sequential limits]
	Let $j:Y \to X$ be a type map. Consider a tower of maps given by $A: \N \to \UU$ together with $f:\prod_{n:\N} A_{n+1} \to A_n$ such that for any $n:\N$, the map $f_n: A_{n+1} \to A_n$ is $j$-orthogonal. Then, for the sequential limit $A_\infty\defeq  \lim_{n:\N} A_n \defeq \lim_{n:\N} \pair{A_n}{f_n}$, the map $\pi_0 : A_\infty \to A_0$ is $j$-orthogonal as well.\footnote{Then, by left cancelation so are all projections $\pi_n:A_\infty \to A_n$ for $n:\N$.}
\end{prop}

\begin{proof}
By precondition, for any $n:\N$ we have equivalences $A_{n+1}^X \equiv A_{n}^X \times_{A_n^Y} A_{n+1}^Y$ via the structure maps. By composition, we get equivalences $A_n^X \equiv A_0^X \times_{A_0^Y} A_n^Y$. In sum, we obtain
\[ A_\infty^X \equiv \lim_{n:\mathbb N} A_n^X \equiv  \lim_{n:\mathbb N} \big( A_0^X \times_{A_0^Y} A_n^Y \big) \equiv A_0^X \times_{A_0^Y} A_\infty^Y,\]
again via the structure maps.
\end{proof}

\begin{prop}[Maps between $j$-orthogonal types are $j$-orthogonal]\label{prop:maps-orth-types}
	Let $j:Y \to X$ be a type map or shape inclusion. If both $E$ and $B$ are $j$-orthogonal types, then any map $\pi:E \to B$ is $j$-orthogonal as well.
\end{prop}

\begin{proof}
Since $B$ is $j$-orthogonal, the restriction map $B^X \to B^Y$  is an equivalence, and so is its pullback along $E^Y \to B^Y$:
% https://q.uiver.app/?q=WzAsNSxbMCwwLCJFXlgiXSxbMCwyLCJCXlgiXSxbMywyLCJCXlkiXSxbMywwLCJFXlkiXSxbMSwxLCJcXGJ1bGxldCJdLFswLDFdLFsxLDIsIlxcc2ltZXEiXSxbMCwzLCJcXHNpbWVxIl0sWzMsMl0sWzAsNCwiIiwyLHsic3R5bGUiOnsiYm9keSI6eyJuYW1lIjoiZGFzaGVkIn19fV0sWzQsMywiXFxzaW1lcSJdLFs0LDFdLFs0LDIsIiIsMSx7InN0eWxlIjp7Im5hbWUiOiJjb3JuZXIifX1dXQ==
	\[\begin{tikzcd}
		{E^X} &&& {E^Y} \\
		& \cdot \\
		{B^X} &&& {B^Y}
		\arrow[from=1-1, to=3-1]
		\arrow["\simeq", from=3-1, to=3-4]
		\arrow["\simeq", from=1-1, to=1-4]
		\arrow[from=1-4, to=3-4]
		\arrow[dashed, from=1-1, to=2-2]
		\arrow["\simeq", from=2-2, to=1-4]
		\arrow[from=2-2, to=3-1]
		\arrow["\lrcorner"{anchor=center, pos=0.125}, draw=none, from=2-2, to=3-4]
	\end{tikzcd}\]
	Now, since by assumption $E^X \to E^Y$ is an equivalence is well, so is the gap map by $2$-for-$3$ for equivalences.
\end{proof}

\begin{prop}[Total types of $j$-orthogonal maps]\label{prop:totalty-orth-maps}
	Let $j:Y \to X$ be a type map or shape inclusion and $\pi:E \to B$ be a map with $j$-orthogonal codomain. Then $E$ is $j$-orthogonal if and only if $\pi$ is.
\end{prop}

\begin{proof}
	Similar as for \cref{prop:maps-orth-types}.
\end{proof}

\subsection{LARI families}\label{ssec:lari-fam}

For a type map $j:Y \to X$, we call a family $P:B \to \UU$ a \emph{$j$-LARI family} if the Leibniz cotensor map $j \cotens \pi: \totalty{P}^X \to \totalty{P}^Y \times_{B^Y} B^X$ has a left adjoint right inverse (LARI). Typically, $j:Y \to X$ will be a shape inclusion. In fact, we eventually consider the case of the initial vertex inclusion $i_0: \unit \hookrightarrow \Delta^1$ which precisely defines the notion of cocartesian fibration.

However, already for general maps $j:Y \to X$ we can prove that the $j$-LARI fibrations satisfy several closure properties considered before in the case of $j$-orthogonal fibrations.

A basic account of LARI adjunctions in the synthetic setting is given in Appendix~\ref{app:sec:adj}, based on Riehl--Shulman's theory of adjunctions~\cite[Section 11]{RS17}.

\begin{rem}\label{rem:lari-ff}
	We also note the following. If $U : E \to B : F$ is a LARI adjunction of Segal types,
	we may assume---by \emph{fibrant replacement}, \cf~\cite[Theorem~4.8.3]{hottbook} and~\Cref{ssec:fam-vs-fib}---that $E$ is the total type of a family $P:B\to\UU$,
	that $U$ is the projection, and that $F\,x\defeq \pair{x}{f\,x}$
	for some section $f : \prod_{x:B}P\,x$.

	In any LARI adjuction, the left adjoint is fully faithful,
	as $\hom_B(b,b') \equiv \hom_B(b,UF\,b') \equiv \hom_E(F\,b, F\,b')$.
\end{rem}

\subsubsection{LARI families}

\begin{defn}[LARI families]
	Let $P:B \to \UU$ be a family, and $j:Y \to X$ be any type map or shape inclusion. Writing $\pi: E \to B$ for the associated fibration, we call $P$ a \emph{$j$-LARI family} if and only we have a LARI adjunction (in one of the equivalent formulations of~\cref{thm:char-lari}) in the following diagram:
	%% TODO: convert to tikz
	\[
	\tikzset{%
		symbol/.style={%
			draw=none,
			every to/.append style={%
				edge node={node [sloped, allow upside down, auto=false]{$#1$}}}
		}
	}
	\begin{tikzcd}
		E^X \ar[rrd, bend left = 15, "j \cotens \pi" , ""{name=A, below}] \ar[bend left = 20, rrrrd] \ar[bend right = 30, ddrr] & &  & &   \\
		&& B^X \times_{B^Y} E^Y \ar[rr]   \ar[llu, bend left = 15, dashed, "\chi", ""{name=B, above}] \ar[d]
		\ar[drr, phantom, "\lrcorner", very near start]
		\ar["\dashv"{anchor=center, rotate=60}, draw=none, from=A, to=B] & & E^Y	\ar[d] \\
		& &	B^X \ar[rr]	& & B^Y \\
	\end{tikzcd}
	\]
\end{defn}

LARI fibrations generalize right orthogonal fibrations, as we will see.
\begin{prop}[Adjoint equivalences, {\protect\cite[Proposition 2.1.11/12]{RV}}]\label{prop:adj-equiv}
Let $u:B \to A$ be an equivalence between Rezk types. Then there exists a functor $f:A \to B$ s.t.~$f \dashv u$.

Here, $f$ can be taken to be a quasi-inverse of $u$.
\end{prop}

\begin{proof}
Since $u:B \to A$ is an equivalence it has a quasi-inverse, \ie, there exists a map $f:A \to B$ together with homotopies $\eta: \id_A = uf$, $\varepsilon: fu = \id_B$. Then, by $2$-for-$3$ for isomorphisms in a Rezk type,\footnote{The $2$-for-$3$ law for isomorphisms in a Segal type can easily be proven as in the case of $1$-categories, using~\cite[Proposition~10.1]{RS17}.} this gives homotopies $u \varepsilon \circ \eta u = \id_u$ and $\varepsilon f \circ f \eta = \id_f$. Together, these data make up a quasi-diagrammatic adjunction which can be promoted to a coherent, bi-diagrammatic adjunction by~\cite[Theorem~11.23]{RS17}.
\end{proof}

\begin{cor}\label{cpr:orth-fam-lari-fam}
Let $P:B \to \UU$ be a $j$-orthogonal family over a Rezk type $B$. Then $P$ is also a $j$-LARI family.
\end{cor}

\subsubsection{Closure properties of $j$-LARI families}\label{ssec:lari-closed}

\begin{prop}[Closedness under products]\label{prop:lari-maps-closed-pi}
	Let $I$ be a type, and $B: I \to \UU$ be a family. Assume there is a function $P: \prod_{i:I} B(i) \to \UU$. If there is a map or shape inclusion $j: Y \to X$ such that $P(i)$ is a $j$-LARI family for each $i: I$, then also $\prod_I P$ is a $j$-LARI family
\end{prop}

\begin{proof}
	Let $E$ the total type of $P$ and $E_i$ the total type of $P_i$, for any $i:I$. Observe that we have
	\begin{align*}
		\left( X \to \prod_I E\right) & \equiv \sum_{\sigma:X \to \prod_I B} \prod_{i:I} \left( \prod_{x:X} P_i((\sigma\,x)_i)\right) \\
		& \equiv \sum_{\sigma:X \to \prod_I B}  \left( X \to \prod_I P \right)(\sigma).
	\end{align*}
	For any $i:I$, we are given a LARI adjunction
% https://q.uiver.app/?q=WzAsMixbMCwwLCJFX2leWCJdLFsyLDAsIkJfaV5YIFxcdGltZXNfe0JfaV5ZfSBFX2leWSJdLFswLDEsInJfaSIsMl0sWzEsMCwiXFxlbGxfaSIsMix7ImN1cnZlIjozLCJzdHlsZSI6eyJib2R5Ijp7Im5hbWUiOiJkb3R0ZWQifX19XSxbMywyLCIiLDAseyJsZXZlbCI6MSwic3R5bGUiOnsibmFtZSI6ImFkanVuY3Rpb24ifX1dXQ==
\[\begin{tikzcd}
	{E_i^X} && {B_i^X \times_{B_i^Y} E_i^Y}
	\arrow[""{name=0, anchor=center, inner sep=0}, "{r_i}"', from=1-1, to=1-3]
	\arrow[""{name=1, anchor=center, inner sep=0}, "{\ell_i}"', curve={height=18pt}, dotted, from=1-3, to=1-1]
	\arrow["\dashv"{anchor=center, rotate=-90}, draw=none, from=1, to=0]
\end{tikzcd}\]
	which induces a LARI adjunction
	\[
\begin{tikzcd}
	{(\prod_I E)^X} && {(\prod_I B)^X \times_{(\prod_I B)^Y} (\prod_I E)^Y}
	\arrow[""{name=0, anchor=center, inner sep=0}, "{\prod_I r}"', from=1-1, to=1-3]
	\arrow[""{name=1, anchor=center, inner sep=0}, "{\prod_I \ell}"', curve={height=18pt}, dotted, from=1-3, to=1-1]
	\arrow["\dashv"{anchor=center, rotate=-90}, draw=none, from=1, to=0]
\end{tikzcd}
	\]
	by \cref{prop:lari-closed-under-pi}, since exponentiation commutes with dependent products.
	%Note that
	%\[ S \equiv \big(\prod_I B\big)^X \times_{\left(\prod_I B \right)^Y} \big( \prod_I E \big)^Y \equiv \sum_{\sigma: Y \to \prod_I B} \Ext(\prod_I B, j, \sigma) \times \big( Y \to \prod_I P\big)(\sigma). \]
	%For the component families, the left adjoint right inverse $\ell(i)$ can be written as a two component function
	% \[ \ell(i) \jdeq \lambda u \, v \, f.\pair{v}{P(i)_!(v,f)}\]
	% with ``base part'' $v$ and ``principle part'', denoted as $\widehat{\ell(i)}(v,f)$, which is simultaneously a ``lift'' of $v$ and an extension of $f$.
	% % \footnote{In the instance of cocartesian families, this will be the cocartesian lift of an arrow in the base \wrt~to a given source vertex.}
	%  Given $\tau: X \to \prod_I B$ and a section $\nu: X \to \prod_I P$ lying over, the projection to the pullback object $S$ is given by
	% \[  \overline{\rho}(\tau,\nu) :\jdeq \langle j^*\tau, \tau, j^* \rangle.  \]
	% Given sections
	% %% TODO: make into single diagram
	% \[
	% \begin{tikzcd}
	% 	Y \ar[r, "j"] \ar[d, "\sigma" swap]& X \ar[dl, "\tau"]& & Y \ar[r, "j"] \ar[d, "\mu" swap] & X \ar[dl, dashed] \\
	% 	\prod_I B &														& & 	\prod_I P
	% \end{tikzcd}
	% \]
	% with $\sigma$ extending $\tau$ and $\mu$ lying over $\sigma$, we define the desired lift by
	% \[ \overline{\ell}(\lambda, \sigma, \tau, \mu) :\jdeq \langle \tau, \lambda \, i.\widehat{\ell(i)}(\tau(-)(i), \mu(-)(i))\rangle. \]
	% It is clear that $\overline{\rho}$ is a section of $\ell$
\end{proof}

\begin{cor}[Closedness of $j$-LARI families under binary products]
	Let $C, B:\UU$ with families $P: B \to \UU$ and $Q: C \to \UU$. If both $P$ and $Q$ are $j$-LARI families \wrt~to a map $j:Y \to X$, then so is the binary product family $P \times Q : B \times C \to \UU$.
\end{cor}

\begin{prop}[Closedness of $j$-LARI families under composition]\label{prop:lari-maps-closed-comp}
	Let $j: Y \to X$ be a map. Let $P: B \to \UU$ and $Q: \totalty{P} \to \UU$ be families with composite $Q \compfam P: B \to \UU$. If both $P$ and $Q$ are $j$-LARI families, then so is $Q \compfam P$.
\end{prop}

\begin{proof}
	Let $\pi: E \to B$ and $\xi: F \to E$ be the projection maps associated to $P$ and $Q$, respectively. Abbreviating $S :\jdeq B^X \times_{B^Y} E^Y$, $T :\jdeq E^X \times_{E^Y} F^Y$ and $R\defeq B^X \times_{B^Y} F^Y$, from the LARI adjunctions
	\[
	\tikzset{%
		symbol/.style={%
			draw=none,
			every to/.append style={%
				edge node={node [sloped, allow upside down, auto=false]{$#1$}}}
		}
	}
	\begin{tikzcd}
		E^X \ar[rr, bend left = 15, "r" , ""{name=A, below}]  & & \ar[ll, bend left = 15, dashed, "\ell", ""{name=B, above}]  S
		\ar[from=A, to=B, symbol=\vdash]
	\end{tikzcd}
	\quad
	\begin{tikzcd}
		F^X \ar[rr, bend left = 15, "r'" , ""{name=A, below}]  & & \ar[ll, bend left = 15, dashed, "\ell'", ""{name=B, above}]   T
		\ar[from=A, to=B, symbol=\vdash]
	\end{tikzcd}
	\]
	we want to get a LARI adjunction:
	\[
	\tikzset{%
		symbol/.style={%
			draw=none,
			every to/.append style={%
				edge node={node [sloped, allow upside down, auto=false]{$#1$}}}
		}
	}
	\begin{tikzcd}
		F^X \ar[rr, bend left = 15, "r''" , ""{name=A, below}]  & & \ar[ll, bend left = 15, dashed, "\ell''", ""{name=B, above}]   R
		\ar[from=A, to=B, symbol=\vdash]
	\end{tikzcd}
	\]
	Factoring the pullback square for $T$ as
	\[
	\tikzset{%
		symbol/.style={%
			draw=none,
			every to/.append style={%
				edge node={node [sloped, allow upside down, auto=false]{$#1$}}}
		}
	}
	\begin{tikzcd}
		F^X \ar[rrddd, bend right = 15] \ar[rrd, bend left = 15, "r'" , ""{name=A, below}]  \ar[rrrrrrd, bend left = 15] & & & && &  \\
		& & \ar[llu, bend left = 15, dashed, "\ell'", ""{name=B, above}]  T \ar[dd] \ar[rr, bend left = 20, ""{name=C, above}]
		\ar[ddrr, phantom, "\lrcorner", very near start]
		& & R \ar[ll, bend left = 20, dashed, ""{name=D, above}] \ar[dd] \ar[rr] \ar[ddrr, phantom, "\lrcorner", very near start] & & F^Y \ar[dd]
		\ar[from=A, to=B, symbol=\vdash] \ar[from=C, to=D, symbol=\vdash] \\
		& & & & & & \\
		&& E^X \ar[rr, bend left = 20, "r", ""{name=E, above}] & & S  \ar[ll, bend left = 20, dashed, "\ell", ""{name=F, above}] \ar[rr] && E^Y
		\ar[from=E, to=F, symbol=\vdash]
	\end{tikzcd}
	\]
	yields the claim, by first pulling back the LARI adjunction between $E^X$ and $S$, and then composing with the one between $F^X$ and $T$. Indeed, by~\Cref{prop:lari-closed-under-comp} LARI adjunctions compose.
\end{proof}

\begin{theorem}[Closedness of $j$-LARI families under pullback]\label{prop:lari-maps-closed-pb}
	Let $B$ be a type, $P: B \to \UU$ be a family, and $k:A \to B$ be a map. Given a map $j:Y \to X$, if $P$ is a $j$-LARI family, then the pullback $k^*P : A\to \UU$ is a $j$-LARI family as well.
\end{theorem}

\begin{proof}
	By assumption, writing $\pi:E \to B$ and $\xi:F \to A$ for the associated projection maps, we have a pullback together with a LARI adjunction:
% https://q.uiver.app/?q=WzAsNixbMCwwLCJGIl0sWzAsMSwiQSJdLFsxLDEsIkIiXSxbMSwwLCJFIl0sWzIsMCwiRV5YIl0sWzIsMSwiRV5ZIFxcdGltZXNfe0JeWX0gQl5YIl0sWzAsMSwiXFx4aSIsMix7InN0eWxlIjp7ImhlYWQiOnsibmFtZSI6ImVwaSJ9fX1dLFsxLDIsImoiLDJdLFswLDNdLFszLDIsIlxccGkiLDAseyJzdHlsZSI6eyJoZWFkIjp7Im5hbWUiOiJlcGkifX19XSxbMCwyLCIiLDEseyJzdHlsZSI6eyJuYW1lIjoiY29ybmVyIn19XSxbNSw0LCIiLDIseyJjdXJ2ZSI6MSwic3R5bGUiOnsiYm9keSI6eyJuYW1lIjoiZG90dGVkIn19fV0sWzQsNSwiIiwyLHsiY3VydmUiOjF9XSxbMTEsMTIsIiIsMix7ImxldmVsIjoxLCJzdHlsZSI6eyJuYW1lIjoiYWRqdW5jdGlvbiJ9fV1d
\[\begin{tikzcd}
	{F} & {E} & {E^X} \\
	{A} & {B} & {E^Y \times_{B^Y} B^X}
	\arrow["{\xi}"', from=1-1, to=2-1, two heads]
	\arrow["{j}"', from=2-1, to=2-2]
	\arrow[from=1-1, to=1-2]
	\arrow["{\pi}", from=1-2, to=2-2, two heads]
	\arrow["\lrcorner"{very near start, rotate=0}, from=1-1, to=2-2, phantom]
	\arrow[""{name=0, inner sep=0}, from=2-3, to=1-3, curve={height=6pt}, dotted]
	\arrow[""{name=1, inner sep=0}, from=1-3, to=2-3, curve={height=6pt}]
	\arrow["\dashv"{rotate=-180}, from=0, to=1, phantom]
\end{tikzcd}\]
	We have to show that the map $j \cotens \xi:F^X \to F^Y \times_{A^Y} A^X$ has a LARI. Note first that we can factor it through equivalences as follows:

	\begin{align}\label{eq:cotens-larifam-pb}
		{F^X \equiv A^X \times_{B^X} E^X} \longrightarrow {A^X \times_{B^Y} E^Y \equiv A^X \times_{A^Y} (A^Y \times_{B^Y} E^Y) \equiv A^X \times_{A^Y} F^Y}
	\end{align}

Next, observe that by the Pullback Lemma, the inner left square is a pullback:
% https://q.uiver.app/?q=WzAsOCxbMCwwLCJFXlkgXFx0aW1lc197Ql5ZfSBBXlgiXSxbMiwwLCJFXlkgXFx0aW1lc197Ql5ZfSBFXlkiXSxbNCwwLCJFXlkiXSxbMCwxLCJBXlgiXSxbMiwxLCJCXlgiXSxbNCwxLCJCXlkiXSxbMSwwXSxbMywwXSxbMCwxXSxbMSwyXSxbMyw0XSxbMSw0XSxbNCw1XSxbMiw1XSxbMCwzXSxbMSw1LCIiLDIseyJzdHlsZSI6eyJuYW1lIjoiY29ybmVyIn19XV0=
\[\begin{tikzcd}
	{E^Y \times_{B^Y} A^X} & {} & {E^Y \times_{B^Y} B^X} & {} & {E^Y} \\
	{A^X} && {B^X} && {B^Y}
	\arrow[from=1-1, to=1-3]
	\arrow[from=1-3, to=1-5]
	\arrow[from=2-1, to=2-3]
	\arrow[from=1-3, to=2-3]
	\arrow[from=2-3, to=2-5]
	\arrow[from=1-5, to=2-5]
	\arrow[from=1-1, to=2-1]
	\arrow["\lrcorner"{very near start, rotate=0}, from=1-3, to=2-5, phantom]
\end{tikzcd}\]
Similarly, the top square becomes a pullback:
% https://q.uiver.app/?q=WzAsNixbMCwwLCJBXlggXFx0aW1lc197Ql5YfSBFXlgiXSxbMiwwLCJFXlgiXSxbMCwxLCJFXlkgXFx0aW1lc197Ql5ZfSBBXlkiXSxbMiwxLCJFXllcXHRpbWVzX3tCXll9Ql5ZIl0sWzAsMiwiQV5YIl0sWzIsMiwiQl5YIl0sWzAsMV0sWzAsMl0sWzIsNF0sWzQsNV0sWzMsNV0sWzIsM10sWzIsNSwiIiwxLHsic3R5bGUiOnsibmFtZSI6ImNvcm5lciJ9fV0sWzEsM10sWzMsMSwiIiwwLHsiY3VydmUiOjMsInN0eWxlIjp7ImJvZHkiOnsibmFtZSI6ImRvdHRlZCJ9fX1dLFsxNCwxMywiIiwwLHsibGV2ZWwiOjEsInN0eWxlIjp7Im5hbWUiOiJhZGp1bmN0aW9uIn19XV0=
\[\begin{tikzcd}
	{A^X \times_{B^X} E^X} && {E^X} \\
	{E^Y \times_{B^Y} A^Y} && {E^Y\times_{B^Y}B^X} \\
	{A^X} && {B^X}
	\arrow[from=1-1, to=1-3]
	\arrow[from=1-1, to=2-1]
	\arrow[from=2-1, to=3-1]
	\arrow[from=3-1, to=3-3]
	\arrow[from=2-3, to=3-3]
	\arrow[from=2-1, to=2-3]
	\arrow["\lrcorner"{very near start}, from=2-1, to=3-3, phantom]
	\arrow[""{name=0, inner sep=0}, from=1-3, to=2-3]
	\arrow[""{name=1, inner sep=0}, from=2-3, to=1-3, curve={height=18pt}, dotted]
	\arrow["\vdash", from=1, to=0, phantom]
\end{tikzcd}\]
LARI adjunctions are closed under pullback by~\Cref{prop:lari-closed-under-pullback}. Hence, this serves to induce the LARI of the map
\[
  A^X \times_{B^X} E^X \to E^Y \times_{B^Y} A^X
\]
which corresponds to $j \cotens \xi: F^X \to F^Y \times_{A^Y} A^X$ by~\ref{eq:cotens-larifam-pb}.
\end{proof}

\begin{cor}[Closedness of $j$-LARI families under pullback along families]
	Let $B$ be a types, and $P,Q: B \to \UU$ be families. If $P$ is a $j$-LARI family, then the pullback family $Q^*P : \totalty{Q}\to \UU$ is as well.
\end{cor}

\begin{cor}[Closedness of $j$-LARI families under fiber product]
	If $P,Q:B \to \UU$ are $j$-LARI families, then so is the fiber product $P \times_B Q : B \to \UU$.
\end{cor}

\section{(Iso)inner families}\label{sec:isoinner-fams}

Our main objects of study are cocartesian families, \ie, functorial families of synthetic $\inftyone$-categories.

Since at the most general level types are not Segal, as an intermediate step to defining cocartesian families, we have to deal with families of (complete) Segal types that are not necessarily functorial. \emph{Inner families} are those type families for which the associated projection is right orthogonal to the horn inclusion $\Lambda_1^2 \hookrightarrow \Delta^2$. Hence, between Segal types, inner families correspond to fibrations in the Segal model structure. Bringing in Rezk completeness motivates our definition of \emph{isoinner family}, which in addition to being inner requires all fibers to be Rezk-complete. Over Rezk types, this can be expressed by requiring the associated projection to be right orthogonal to the terminal projection from the free bi-invertible arrow $\walkBinv$.

In this section, we discuss the behavior of these non-functorial families of Segal and Rezk types, respectively.
Several closure properties will follow from the results of \cref{sec:closure}. We also introduce the free bi-invertible arrow $\walkBinv$, constructed as a colimit after~\cite[Appendix A.3]{RS17}.\footnote{The walking bi-invertible arrow hence arises as a \emph{type} in our case, although it might be possible to obtain it as a shape by a suitable extension of the shape theory.} Along the way, we characterize Rezk types via right orthogonality \wrt~either inclusion $\unit \to \walkBinv$, or alternatively, the terminal projection $\walkBinv \to \unit$.

\subsection{Inner families}

We introduce the notion of an inner family, resembling---in some sense---\emph{inner} or \emph{mid fibrations} from quasi-category theory, used by Joyal and Lurie as an auxiliary tool to investigate co-/cartesian fibrations.

In the same vein, an inner family can be seen as a \emph{relativized} version of the notion of Segal type: in an inner family any $(2,1)$-horn sitting over a $2$-simplex in the base can be extended, uniquely up to homotopy, to a $2$-simplex lying over. In particular, an inner family over the terminal type $\unit$ is the same as a Segal type.

\subsubsection{Definition and characterization}

\begin{defn}[Inner Family]
	Let $B$ be a type. A family $P: B \to \UU$ is called \emph{inner family} if
	\[ 	\isInnerFam(P) :\jdeq \prod_{\sigma:\Delta^2 \to B} \prod_{ \eta:\prod_{t:\Lambda_1^2} P(\sigma(t))} \isContr \left( \exten{t:\Delta^2}{P\,\sigma(t)}{\Lambda_1^2}{\eta} \right). \]
	Unpacking this gives the following logically equivalent proposition:
	\begin{eqnarray*}
		\isInnerFam(P) & \equiv & \prod_{b,b',b'':B}
		\prod_{\substack{u: b \to b'
				\\ v:b' \to b'' \\ w: b \to b''}} \prod_{\sigma:u,v \Rightarrow_B w} \prod_{\substack{e:Pb \\ e':Pb' \\ e'':Pb''}} \prod_{\substack{f:e \to_u^P e' \\ g: e' \to^P_v e''}} \\
		& & \isContr\left( \sum_{h:e \to^P_w e''}
		f,g \Rightarrow_w^P h \right).
	\end{eqnarray*}
\end{defn}

\begin{figure}
	\centering

  \includegraphics{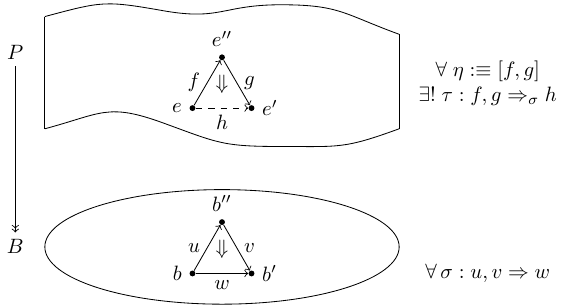}

	\caption{In an inner family, inner $2$-horns can be extended uniquely up to homotopy relative to a given $2$-simplex in the base.}
	\label{fig:innfam}
\end{figure}

This can be expressed in categorical terms as the condition that the unstraightening $\widetilde{P} \fibarr B$ be right orthogonal to $\Lambda_1^2 \hookrightarrow \Delta^2$, or equivalently, the Leibniz cotensor of these maps be an equivalence.

\begin{prop}[Orthogonality characterization of inner families]\label{prop:innfam-pb}
	A type family $P: B \to \UU$ is inner if and only if the square
	% https://q.uiver.app/?q=WzAsNCxbMCwwLCJcXHdpZGV0aWxkZXtQfV57XFxEZWx0YV4yfSJdLFswLDEsIkJee1xcRGVsdGFeMn0iXSxbMiwxLCJCXntcXExhbWJkYV4yXzF9Il0sWzIsMCwiXFx3aWRldGlsZGV7UH1ee1xcTGFtYmRhXjJfMX0iXSxbMSwyXSxbMCwzXSxbMywyXSxbMCwxXV0=
	\[\begin{tikzcd}
		{\widetilde{P}^{\Delta^2}} && {\widetilde{P}^{\Lambda^2_1}} \\
		{B^{\Delta^2}} && {B^{\Lambda^2_1}}
		\arrow[from=2-1, to=2-3]
		\arrow[from=1-1, to=1-3]
		\arrow[from=1-3, to=2-3]
		\arrow[from=1-1, to=2-1]
	\end{tikzcd}\]
	is a homotopy pullback.
\end{prop}

\begin{proof}
	This is an instance of \cref{cor:orth-shapes}.
\end{proof}

\subsubsection{Closure and structural properties}

As an instance of \cref{ssec:clos-orth} it follows that inner families enjoy several closure properties.
\begin{prop}
	Inner families over arbitrary types are closed under fibered equivalences, dependent products, composition, pullback, left cancelation, and sequential limits. Families corresponding to equivalences are always isoinner.
\end{prop}

\begin{prop}\label{prop:innerfib-has-segal-fibers}
	Let $P: B \to \UU$ be an inner family. Then for every $b:B$ the fiber $P\,b$ is a Segal type.
\end{prop}

\begin{proof}
	This is an instance of \cref{cor:fib-j-orth}.
\end{proof}

\begin{prop}\label{prop:innerfib-total-sp-over-segal-base-is-segal}
	Let $P: B \to \UU$ be a family over a Segal type $B$. Then the total type $\totalty{P} :\jdeq \sum_{b:B} P\,b$ is Segal if and only if $P$ is an inner family.
\end{prop}

\begin{proof}
This is an instance of \cref{prop:totalty-orth-maps}.
\end{proof}

The following observation reflects the semantic fact that between Segal spaces Reedy fibrations coincide with the fibrations of the Segal space model structure. It is also a formal consequence of \cref{prop:maps-orth-types}.
\begin{prop}[Maps between Segal types are inner]\label{prop:maps-between-segal-types}
	Let $\pi:E \to B$ be a map between Segal types $B, E$. Then $\pi:E \to B$ is inner.
\end{prop}

\subsection{Isoinner families}

As a next step towards the synthetic version of cocartesian fibrations, we consider non-functorial families of \emph{complete} Segal, \aka~Rezk types. To be able to formulate an appropriate right orthogonality condition, we start by defining the \emph{free} (or \emph{walking}) \emph{bi-invertible arrow}. This gives alternative formulations of the Rezk-completeness condition.

\subsubsection{The free bi-invertible arrow}

In order to exhibit Rezk completeness as an orthogonality condition, we want to define a type $\walkBinv$ which for any type corepresents its type of isomorphisms, \ie,
\[ \iso(X) \equiv  (\walkBinv \to X)\]
for any type $X:\UU$. Moreover, this will also allow us to define a notion of (non-functorial) family of Rezk types, which we call \emph{isoinner families}. Semantically, these correspond---at least between fibrant objects---to the fibrations of the Rezk model structure.

We internalize the construction of $\walkBinv$ due to~\cite[Definition A.~24,25]{RS17}, by forming the colimit $\walkBinv$ of the following diagram $\mathcal D$ of types
% https://q.uiver.app/?q=WzAsNyxbMSwwLCJcXERlbHRhXjEiXSxbMCwxLCJcXHVuaXQiXSxbMiwxLCJcXERlbHRhXjIiXSxbMywwLCJcXERlbHRhXjEiXSxbNCwxLCJcXERlbHRhXjIiXSxbNSwwLCJcXERlbHRhXjEiXSxbNiwxLCJcXHVuaXQiXSxbMCwxXSxbMCwyLCJkXjEiXSxbMywyLCJkXjAiLDJdLFszLDQsImReMiJdLFs1LDQsImReMSIsMl0sWzUsNl1d
	\[\begin{tikzcd}
		& {\Delta^1} && {\Delta^1} && {\Delta^1} \\
		{\unit} && {\Delta^2} && {\Delta^2} && {\unit}
		\arrow[from=1-2, to=2-1]
		\arrow["{d^1}", from=1-2, to=2-3]
		\arrow["{d^0}"', from=1-4, to=2-3]
		\arrow["{d^2}", from=1-4, to=2-5]
		\arrow["{d^1}"', from=1-6, to=2-5]
		\arrow[from=1-6, to=2-7]
	\end{tikzcd}\]
where $d^k: \Delta^1 \to \Delta^2$ is the $k$-th face inclusion, \ie, the ``injection whose image is missing $k$''. Concretely, this can be done using the theory of colimits over reflexive graphs as described by Rijke~\cite[Section~3]{RijPhd}.

The free bi-invertible arrow is constructed as a colimit over a reflexive graph.
Let $\mathcal G$ denote the graph underlying the diagram $\mathcal D$ of types as shown above. By~\cite[Remark~3.5.6]{RijPhd}, it follows that the colimit of the $\mathcal G$-diagram $\mathcal D$ is given by the pushout of the cospan $\mathcal S(\mathcal D) :\jdeq \pair{\pr_1}{\pr_2}$ associated to $\mathcal D$ via the total graph:
% https://q.uiver.app/?q=WzAsNSxbMCwwLCJcXHN1bV97XFxzdWJzdGFja3tcXGxhbmdsZSBpLHggXFxyYW5nbGUgXFxcXCBcXGxhbmdsZSBqLHkgXFxyYW5nbGUgOiBcXHdpZGV0aWxkZXtEXzB9fX0gR18xKGksaikgXFx0aW1lcyBEXzEoeCk9eSJdLFsyLDAsIlxcd2lkZXRpbGRle0RfMH0iXSxbMCwxLCJcXHdpZGV0aWxkZXtEXzB9Il0sWzIsMSwiXFxtYXRoYmIgRSJdLFsxLDBdLFswLDEsIlxcbWF0aHJte3ByfV8yIl0sWzAsMiwiXFxtYXRocm17cHJ9XzEiLDJdLFsxLDMsImkiXSxbMiwzLCJpIiwyXSxbMyw0LCIiLDIseyJzdHlsZSI6eyJuYW1lIjoiY29ybmVyIn19XV0=
\[\begin{tikzcd}
	{\sum_{\substack{\langle i,x \rangle, \\ \langle j,y \rangle : \widetilde{D_0}}} G_1(i,j) \times D_1(x)=y} & {} & {\widetilde{D_0}} \\
	{\widetilde{D_0}} && {\mathbb E}
	\arrow["{\mathrm{pr}_2}", from=1-1, to=1-3]
	\arrow["{\mathrm{pr}_1}"', from=1-1, to=2-1]
	\arrow[from=1-3, to=2-3]
	\arrow[from=2-1, to=2-3]
	\arrow["\lrcorner"{very near start, rotate=180}, from=2-3, to=1-2, phantom]
\end{tikzcd}\]

\begin{prop}
The colimit $\walkBinv$ of the diagram $\mathcal D$ as above (covariantly) represents the types of isomorphisms in a Segal type, \ie, for any Segal type $A$ there is an equivalence of types
\[ (\walkBinv \to A) \equiv \iso(A).\]
\end{prop}

\begin{proof}
We leave out the explicit descriptions of $\mathcal G$ and $\mathcal D$, but they are straightforward. By the universal property of the pushout, for an arbitrary type $A:\UU$, there is an equivalence
\[ (\walkBinv \to A) \equiv \cocone_{\mathcal S(\mathcal D)}(A).\]
Then it follows by unwinding the data of the diagram $\mathcal D$ that there is a chain of equivalences
\begin{align*}
	(\walkBinv \to A)
  &\equiv \sum_{a,b:A} \sum_{\sigma,\tau:\Delta^2 \to A} \left(\sigma \circ d^1 = \id_b\right) \times \left(\sigma \circ d^2 = \tau \circ d^0\right) \times \left(\tau \circ d^1 = \id_a\right) \\
	 &\equiv \sum_{a,b:A} \sum_{\sigma,\tau:\Delta^2 \to A}  \prod_{t:\Delta^1} \sigma(t,t)=\id_b \times \sigma(1,t) = \tau(t,0) \times \tau(t,t) = \id_a \\
	&\equiv \sum_{a,b:A} \sum_{f:a \to b} \ndexten{\Delta^2}{A}{\Lambda_2^2}{[f,\id_b]} \times  \ndexten{\Delta^2}{A}{\Lambda_0^2}{[f,\id_a]} \\
	&\equiv \sum_{a,b:A} \sum_{\substack{f:a \to b \\ g,h:b \to a}} fg=\id_b \times fh=\id_a \tag{\text{$A$ is Segal, cf.~\cite[Proposition 5.10]{RS17}}}\\
	&\equiv \iso(A),
\end{align*}
where one invokes de-/strictifcation as necessary.
\end{proof}

\subsubsection{Rezk types}
For the walking bi-invertible arrow $\walkBinv$, consider the terminal projection $!_\walkBinv : \walkBinv \to\unit$ and the inclusion maps $j_k : \unit \hookrightarrow \walkBinv$ for $k=0,1$. For a fixed type $B$, we write $\partial_k :\jdeq B^{j_k} : B^\walkBinv \to B$. We prove a characterization of Rezk types \`{a}~la~\cite[Proposition 6.4]{rez01}.
In particular, the Rezk types are exactly those among Segal types that are $\walkBinv$-null. This is in line with~\cite[Theorem 6.2]{rez01}.

In general, given a map or shape inclusion $j: Y \to X$ we say a type $B$ is \emph{$j$-local} (\cf~\Cite[Definition~2.1]{RSSmod}) if the terminal projection $!_B: B \to \unit$ is $j$-orthogonal which is equivalent to the map $B^j: B^X \to B^Y$ being an equivalence.  \emph{E.g.}, a type is Segal iff it is $\iota$-local, where $\iota: \Lambda_1^2 \hookrightarrow \Delta^2$.

In case $j \jdeq !_Y: Y \to \unit$ is a terminal projection and $B$ is $!_Y$-local we say that $B$ is \emph{$Y$-null}.

\begin{prop}[Characterization of Rezk types]\label{prop:char-rezk}
	Let $B$ be a Segal type. Then the following are equivalent:
	\begin{enumerate}
		\item\label{thm:rezk-idtoiso-equiv} The type $B$ is Rezk, \ie,
		\[ \isEquiv(\idtoiso_B) \equiv\prod_{\substack{x,y:B \\ f:\iso_B(x,y)}} \isContr\Big( \sum_{\substack{u,v:B \\ p:(u =_B v)}} \angled{u,v,\idtoiso_{B,u,v}(p)}  = \angled{x,y,f} \Big) .\]
		\item\label{thm:rezk-j-local} The type $B$ is $\walkBinv$-null, meaning $B$ is $!_\walkBinv$-local, \ie,
		\[  \isEquiv(B^{!_\walkBinv}) \equiv  \prod_{f:B^\walkBinv} \isContr\Big( \sum_{x:B} f=\id_x\Big). \]
		\item\label{thm:rezk-i-local} The type $B$ is $j_k$-local for $k=0$ or $k=1$, \ie,
		\[ \isEquiv(B^{j_k}) \equiv \prod_{x:B} \isContr\Big( \sum_{f:B^\walkBinv} \partial_k f = x \Big).  \]
	\end{enumerate}
\end{prop}

\begin{proof}
	\leavevmode
	\begin{enumerate}
		\item[$\ref{thm:rezk-idtoiso-equiv} \implies \ref{thm:rezk-j-local}$:]
		By assumption, for $\angled{x,y,f}:\sum_{x,y:B} \iso_B(x,y)$ there exists, uniquely up to homotopy, a datum $\angled{u,v,p}: \sum_{u,v:B} (u=_Bv)$ such that $\angled{u,v,p} = \angled{x,y,\idtoiso_{B,u,v}(p)}$. By path induction, we can take $\angled{u,v,p} :\jdeq \angled{x,x,\refl_x}$, thus yielding $x:B$ uniquely with $f=\id_x$.

		\item[$\ref{thm:rezk-j-local} \implies \ref{thm:rezk-idtoiso-equiv}$:]
		By assumption, it suffices to find for any $x:B$, uniquely up to homotopy a path $p:(u=v)$, $u,v:B$ with $\idtoiso_{B,u,v}(p)=\id_x$. By path induction, $p$ can be taken to be $\refl_x$.

		\item[$\ref{thm:rezk-j-local} \implies \ref{thm:rezk-i-local}$:]
		For any $x:B$ clearly $f\defeq \id_x$ is an isomorphism with $\partial_k(f) =  x$. By assumption however, for \emph{any} isomorphism $f:B^\walkBinv$, there is a path from $f$ to some identity.

		\item[$\ref{thm:rezk-i-local} \implies \ref{thm:rezk-j-local}$:]
		Let $f:\iso_B(x,y)$, $x,y:B$. By assumption, any $g:B^\walkBinv$ with $\partial_k(g)=x$ is uniquely determined up to homotopy. Thus, $f=\id_x$.\qedhere
	\end{enumerate}
\end{proof}

\paragraph{NB} Even though the $!_\walkBinv$-local and the $j_k$-local Segal types coincide, the relative version of this statement is not true: not every $!_\walkBinv$-orthogonal family or map is $j_k$-orthogonal. As a counterexample, consider~$\unit+\unit\to \walkBinv$.
%Being $j_k$-local is stronger in general.

\subsubsection{Isoinner families}

We introduce the notion of isoinner family over general bases, even though we mostly will require the base to be (complete) Segal. Over Segal bases, being isoinner can be expressed through relative Segal and Rezk conditions. Then it follows again by \cref{ssec:clos-orth} that these families are closed under several operations.

\begin{defn}[Isoinner family]
	A type family $P: B \to \UU$ is an \emph{isoinner family} if it is an inner family and every fiber is $!_\walkBinv$-local, \ie,~the proposition
	\[ \isIsoInnerFam(P) :\jdeq \isInnerFam(P) \times \prod_{b:B} \prod_{f:\iso(P\,b)} \isContr \Big( \sum_{e:P\,b} f = \id_e \Big)\]
	is inhabited.
\end{defn}

In case $B$ is a Segal type, $P:B \to \UU$ being an isoinner family is equivalent to
\[ \isIsoInnerFam(P) \, \equiv \isInnerFam(P) \times \prod_{b:B} \isRezk(P\,b) \]
by \cref{prop:char-rezk}.

Clearly, an isoinner family over $\unit$ is the same as a Rezk type.

Recall in particular, that total spaces of inner families over Segal types are Segal. It turns out that over Segal types, being isoinner can be expressed as a right orthogonality property.

\begin{prop}[Orthogonality characterization of isoinner families]
Let $P:B \to \UU$ be a family s.t.~every fiber $P\,b$ is Segal. Then every fiber is Rezk if and only if $\pi:E \to B$ is right orthogonal to $\walkBinv \to \unit$, \ie,~the diagram
% https://q.uiver.app/?q=WzAsNCxbMCwwLCJcXHRvdGFsdHl7UH0iXSxbMCwxLCJCIl0sWzIsMSwiQl5cXG1hdGhiYiBFIl0sWzIsMCwiXFx0b3RhbHR5e1B9XlxcbWF0aGJiIEUiXSxbMCwxXSxbMSwyLCJcXHB0dG9pZF9CIiwyXSxbMCwzLCJcXHB0dG9pZF97XFx0b3RhbHR5IFB9Il0sWzMsMl1d
\[\begin{tikzcd}
	{\totalty{P}} && {\totalty{P}^\mathbb E} \\
	B && {B^\mathbb E}
	\arrow[from=1-1, to=2-1]
	\arrow["{\pttoid_B}"', from=2-1, to=2-3]
	\arrow["{\total(\pttoid_{P})}", from=1-1, to=1-3]
	\arrow[from=1-3, to=2-3]
\end{tikzcd}\]
	is a homotopy pullback, where
	\[ \pttoid_B \defeq \lambda b.\id_b, \quad \pttoid_P  \defeq \lambda b, e.\pair{\id_b}{\id_e}. \]
\end{prop}

\begin{proof}
The diagram is a pullback if and only if
\[ \prod_{b:B} \prod_{f:\walkBinv \to P\,b} \isContr\big( \sum_{e:P\,b} \id_e =_{\walkBinv \to P\,b} f\big).  \]
By \cref{prop:char-rezk} this is equivalent to every fiber being Rezk.
\end{proof}

\begin{cor}
For any type $B$, an inner family $P:B \to \UU$ is isoinner if and only if it is right orthogonal to $\walkBinv \to \unit$.
\end{cor}

\begin{prop}
	Let $B$ be a Rezk type and $P: B \to \UU$ an isoinner family. Then the total type $\totalty{P} :\jdeq \sum_{b:B}P\,b$ is a Rezk type.
\end{prop}

\begin{proof}
	This is again an instance of \cref{prop:totalty-orth-maps}.
\end{proof}

Formally from Subsection~\ref{ssec:clos-orth} it follows that, over Segal bases, isoinner families enjoy several closure properties.
\begin{prop}
	Isoinner families over Segal types are closed under dependent products, composition, pullback, left cancelation, and sequential limits. Families corresponding to equivalences are always isoinner.
\end{prop}

\section{Cocartesian families}\label{sec:cocart-fams}

Cocartesian families $P:B \to \UU$ encode copresheaves of $\inftyone$-categories. All fibers $P\,b$ are Rezk types, and $P$ is (covariantly) \emph{functorial} in the sense that an arrow $u:a \to b$ in $B$ induces a functor $u_!: P\,a \to P\,b$, and this transport operation is natural w.r.t.~directed arrows in $B$, \ie,~it respects composition and identities. In fact, we will often reason about cocartesian families $P:B \to \UU$ in terms of their associated projection $\pi \defeq \pi_P: E \to B$. Our study is informed by~\cite[Chapter 5]{RV} and~\cite{RVyoneda} in an essential way, where Riehl--Verity develop a model-independent theory of cocartesian fibrations intrinsic to an arbitrary $\infty$-cosmos. While this constitutes more generally a fibrational theory of $\inftyn$-categories, for $0 \le n \le \infty$, our present study restricts to $\inftyone$-categories, and at the same time extends Riehl--Shulman's treatment of synthetic $\inftyone$-categories fibered in $\infty$-groupoids.

Reminiscent to the classical (1-categorical) definition, we introduce cocartesian families in terms of the existence of enough cocartesian liftings. However, we also give alternative characterizations, such as the \emph{Chevalley criterion} which allows us to develop the theory in the style of formal category theory, as done by Riehl and Verity~\cite{RV} for $\infty$-cosmoses.

Specifically, over Rezk types cocartesian families are exactly the isoinner families that are $i_0$-LARI families in the sense of Subsection~\ref{ssec:lari-fam}, for $i_0: \unit \to \Delta^1$. Spelled out, this means that the gap map in the pullback
% https://q.uiver.app/?q=WzAsNixbMiwxLCJcXHBpIFxcZG93bmFycm93IEIiXSxbMiwyLCJCXntcXERlbHRhXjF9Il0sWzQsMiwiQiJdLFs0LDEsIkUiXSxbMSwwXSxbMCwwLCJFXntcXERlbHRhXjF9Il0sWzAsMSwiIiwwLHsic3R5bGUiOnsiaGVhZCI6eyJuYW1lIjoiZXBpIn19fV0sWzEsMiwiXFxwYXJ0aWFsXzAiLDJdLFswLDNdLFszLDIsIlxccGkiLDAseyJzdHlsZSI6eyJoZWFkIjp7Im5hbWUiOiJlcGkifX19XSxbMCwyLCIiLDAseyJzdHlsZSI6eyJuYW1lIjoiY29ybmVyIn19XSxbNSwzLCJcXHBhcnRpYWxfMCIsMCx7Im9mZnNldCI6LTIsImN1cnZlIjotM31dLFswLDUsIlxcY2hpIiwyLHsiY3VydmUiOjIsInN0eWxlIjp7ImJvZHkiOnsibmFtZSI6ImRhc2hlZCJ9fX1dLFs1LDAsImlcXHdpZGVoYXR7XFxwaXRjaGZvcmt9XFxwaSIsMix7ImN1cnZlIjoyfV0sWzUsMSwiXFxwaV57XFxEZWx0YV4xfSIsMix7Im9mZnNldCI6MiwiY3VydmUiOjMsInN0eWxlIjp7ImhlYWQiOnsibmFtZSI6ImVwaSJ9fX1dLFsxMiwxMywiIiwyLHsibGV2ZWwiOjEsInN0eWxlIjp7Im5hbWUiOiJhZGp1bmN0aW9uIn19XV0=
\[\begin{tikzcd}
	{E^{\Delta^1}} & {} \\
	&& {\pi \downarrow B} && E \\
	&& {B^{\Delta^1}} && B
	\arrow[two heads, from=2-3, to=3-3]
	\arrow["{\partial_0}"', from=3-3, to=3-5]
	\arrow[from=2-3, to=2-5]
	\arrow["\pi", two heads, from=2-5, to=3-5]
	\arrow["\lrcorner"{anchor=center, pos=0.125}, draw=none, from=2-3, to=3-5]
	\arrow["{\partial_0}", shift left=2, curve={height=-18pt}, from=1-1, to=2-5]
	\arrow[""{name=0, anchor=center, inner sep=0}, "\chi"', curve={height=12pt}, dashed, from=2-3, to=1-1]
	\arrow[""{name=1, anchor=center, inner sep=0}, "{i_0\widehat{\pitchfork}\pi}"', curve={height=12pt}, from=1-1, to=2-3]
	\arrow["{\pi^{\Delta^1}}"', shift right=2, curve={height=18pt}, two heads, from=1-1, to=3-3]
	\arrow["\dashv"{anchor=center, rotate=-118}, draw=none, from=0, to=1]
\end{tikzcd}\]
has a left adjoint right inverse $\chi: \comma{\pi}{B} \to E^{\Delta^1}$ which yields the up to homotopy uniquely determined cocartesian lifts.

We proceed by studying \emph{cocartesian functors} between cocartesian families (incarnated as fibrations). Then, using Subsection~\ref{ssec:lari-closed}, we are able to prove type-theoretic versions of the $\infty$-cosmological closure properties of cocartesian fibrations, which in our case means that the $\inftyone$-category of cocartesian fibrations is complete \wrt~to certain $\inftyone$-limits.\footnote{In more technical terms, our results can be interpreted as type-theoretic proofs of the completeness of the $\inftyone$-categorical core of the $\infty$-cosmos $\cosCocart(\cosRezk)$, itself presenting an $\inftytwo$-category (cf.~\cite[Definition 12.1.10, Remark~12.1.11]{RV}).} Note that, ideally, these would be statements involving universe types which themselves are Rezk. These are beyond the scope of the current discussion, but nevertheless we can ``externalize'' these completeness statements to our univalent universe $\UU$ of arbitrary simplicial types, yielding \cref{prop:cocart-cosm-closure}.

Finally, we prove characterizations of cocartesian functors complementing the ones for cocartesian fibrations, \eg,~the Chevalley criterion for cocartesian functors says that a fibered functor
% https://q.uiver.app/?q=WzAsNCxbMCwwLCJGIl0sWzAsMSwiQSJdLFsyLDEsIkIiXSxbMiwwLCJFIl0sWzAsMSwiXFx4aSIsMix7InN0eWxlIjp7ImhlYWQiOnsibmFtZSI6ImVwaSJ9fX1dLFsxLDJdLFswLDNdLFszLDIsIlxccGkiLDAseyJzdHlsZSI6eyJoZWFkIjp7Im5hbWUiOiJlcGkifX19XV0=
\[\begin{tikzcd}
	F && E \\
	A && B
	\arrow["\xi"', two heads, from=1-1, to=2-1]
	\arrow[from=2-1, to=2-3]
	\arrow[from=1-1, to=1-3]
	\arrow["\pi", two heads, from=1-3, to=2-3]
\end{tikzcd}\]
between cocartesian fibrations is a cocartesian functor if and only if the mate of the induced square
% https://q.uiver.app/?q=WzAsNCxbMCwwLCJGXntcXERlbHRhXjF9Il0sWzAsMSwiXFx4aSBcXGRvd25hcnJvdyBBIl0sWzIsMSwiXFxwaSBcXGRvd25hcnJvdyBCIl0sWzIsMCwiRV57XFxEZWx0YV4xfSJdLFswLDEsIlxceGkiLDIseyJzdHlsZSI6eyJoZWFkIjp7Im5hbWUiOiJlcGkifX19XSxbMSwyXSxbMCwzXSxbMywyLCJcXHBpIiwwLHsic3R5bGUiOnsiaGVhZCI6eyJuYW1lIjoiZXBpIn19fV0sWzEsMywiPSIsMV1d
\[\begin{tikzcd}
	{F^{\Delta^1}} && {E^{\Delta^1}} \\
	{\xi \downarrow A} && {\pi \downarrow B}
	\arrow["i_0 \cotens \xi", swap, two heads, from=1-1, to=2-1]
	\arrow[from=2-1, to=2-3]
	\arrow[from=1-1, to=1-3]
	\arrow["i_0 \cotens \pi", two heads, from=1-3, to=2-3]
	\arrow["{=}"{description}, Rightarrow, from=2-1, to=1-3]
\end{tikzcd}\]
is invertible.\footnote{A development of the required results about adjunctions in simplicial type theory is given in \cref{app:ssec:pasting-lax,app:ssec:mates}.} As an application, we prove that fibered adjunctions between cocartesian fibrations are cocartesian functors, adapting Riehl--Verity's proof in $\infty$-cosmoses to the type-theoretic setting.

\subsection{Cocartesian arrows}\label{ssec:cocart-arr}

Just as in $1$-category theory, there is also a notion of cocartesian arrow in higher dimensional category theory. These are dependent arrows\footnote{In fact, in general $\infty$-cosmoses Riehl--Verity consider cocartesian $2$-cells between generalized elements~\cite[Definition~5.1.1]{RV}. In simplicial type theory one could imagine a similar development as well, using exponential transposition and closedness of cocartesian families under dependent products, cf.~also~\cite[Lemma~5.6.5]{RV}.} satisfying a certain initiality property.

In this section, we provide a few characterizations, analogous to Joyal and Lurie's for quasi-categories, and a few properties that will be useful later on when investigating cocartesian fibrations.

\subsubsection{Definition and basic properties}

\begin{defn}[Cocartesian arrow]
	Let $B$ be a type and $P: B \to \UU$ be an inner family.
	Let $b,b': B$, $u:\hom_B(b,b')$, and $e:P\,b$, $e':P\,b'$. An arrow $f : \hom^P_u(e,e')$ is a \emph{($P$-)cocartesian morphism} or \emph{($P$-)cocartesian arrow} if and only if
	\[
	\isCocartArr^P_u f :\jdeq \prod_{\sigma:\ndexten{\Delta^2}{B}{\Delta_0^1}{u}}
	\prod_{h:\prod_{t:\Delta^1} P\,\sigma(t,t)} \isContr \left( \exten{\pair{t}{s}:\Delta^2}{P \sigma(t,s)}{\Lambda_0^2}{[f,h]}  \right).
	\]
\end{defn}
This is illustrated in~\cref{fig:cocart-arr}. Notice that being a cocartesian arrow is a proposition.
\begin{figure}
	\centering
	\includegraphics{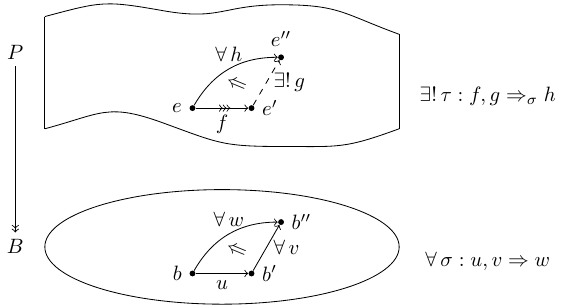}
	\caption{The universal property of cocartesian arrows}
	\label{fig:cocart-arr}
\end{figure}
By expressing the functions on simplices in terms of
objects, morphisms and composition, we obtain an equivalent type:
\begin{equation}\label{eq:isCocartArrEquiv}
  \begin{split}
    \isCocartArr^P_u f \equiv
  		&\prod_{\substack{b'':B \\ v:b' \to b'' \\ w:b \to b''}}
		\prod_{\sigma:u,v \Rightarrow_B w} \prod_{\substack{e'':P\,b'' \\ h:e \to^P_w e''}} \isContr\Big( \sum_{g:e' \to_v^P e''} (f,g \Rightarrow_\sigma^P h)\Big)
  \end{split}
\end{equation}

Diagrammatically, this is expressed as the existence of a filler, unique up to homotopy, as in a diagram of the following form, cf.~\cref{ssec:fib-shapes}:
% https://q.uiver.app/?q=WzAsNSxbMCwwLCJcXERlbHRhXjEiXSxbMSwwLCJcXExhbWJkYV8yXjIiXSxbMSwxLCJcXERlbHRhXjIiXSxbMiwwLCJcXHdpZGV0aWxkZXtQfSJdLFsyLDEsIkIiXSxbMCwxLCJcXHswLDFcXH0iLDJdLFsxLDJdLFsxLDNdLFsyLDRdLFszLDQsIlxccGkiXSxbMCwzLCJmIiwwLHsiY3VydmUiOi0zfV0sWzIsMywiIiwyLHsic3R5bGUiOnsiYm9keSI6eyJuYW1lIjoiZGFzaGVkIn19fV1d
\[\begin{tikzcd}
	{\Delta^1} & {\Lambda_0^2} & {\widetilde{P}} \\
	& {\Delta^2} & {B}
	\arrow["{\{0,1\}}"', from=1-1, to=1-2]
	\arrow[from=1-2, to=2-2]
	\arrow[from=1-2, to=1-3]
	\arrow[from=2-2, to=2-3]
	\arrow["{\pi}", from=1-3, to=2-3]
	\arrow["{f}", from=1-1, to=1-3, curve={height=-18pt}]
	\arrow[from=2-2, to=1-3, dashed]
\end{tikzcd}\]

\paragraph{NB} We will refrain from reasoning about cocartesian arrows at this level of generality, \ie,~for the case of arbitrary type families with non-trivial composition $2$-cells. Hence, from now on we will mostly consider families of (complete) Segal types.

\emph{Notation:} Let $f$ be a cocartesian arrow. Then, given an arrow $v$ in the base with $\partial_1(\pi\,f) = \partial_0(v)$ and $h$ a dependent arrow over $v \circ \pi(f)$ with $\partial_0\,f = \partial_0\,h$, we denote by $\tyfill_{v,f}(h)$ the arrow $g$ determined uniquely up to homotopy such that $gf = h$. In the notation, we will usually leave the arrow $v$ implicit, writing $\tyfill_{f}(h)$. This is justified syntactically since syntactically the datum $v$ can be inferred from $f$. But please do keep in my mind that the universal property of $f$ quantifies explicitly first over $v$, then $h$.

\begin{defn}[Cocartesian lift]
	Let $B$ be a type and $P: B \to \UU$ be an inner family.
  For $b,b': B$, $u:\hom_B(b,b')$, and $e:P\,b$,
  we define the type of \emph{($P$-)cocartesian lifts of $u$
    starting at $e$} to be
  \[
    \CocartLift_P(u,e) \defeq
    \sum_{e':P\,b'} \sum_{f : \dhom^P_u(e,e')}
    \isCocartArr^P_u f.
  \]
\end{defn}
For Segal types, where composites are uniquely determined,
we can further rewrite \eqref{eq:isCocartArrEquiv}:%
\footnote{Here, composition in the family is to be understood as
  \emph{dependent composition} in the sense of~\cite[Remark~8.11]{RS17}, though we will mostly leave this implicit in our notation.}
\[
	\isCocartArr^P_u f \equiv \prod_{b'':B}~\prod_{v:b'\to_B b''}~
	\prod_{e'':P\,b''}~ \prod_{h:e\to^P_{v \circ u} e''}
  \isContr\Bigl( \sum_{g:e'\to^P_v e''} g \circ^P f = h \Bigr).
\]
In fact, we recognize the right-hand side as expressing the initiality
of the object $\angled{b',\id_{b'},e',f}$ in the type
\begin{equation}\label{eq:cocart-initiality-type}
  A(u,e) \defeq \sum_{b'':B}~ \sum_{v:b' \to_B b''}~
	\sum_{e'':P\,b''} (e \to^P_{v\circ u} e'').
\end{equation}

Indeed, $A(u,e)$ is a Segal type as can be seen as follows.

First, note that we can use the types
\[ b^*B^{\Delta^2} \equiv \ndexten{\Delta^2}{B}{\Delta_0^0}{b}, \quad b^*E^{\Delta^1}  \equiv \sum_{u:b \downarrow B} \sum_{e:P\,b} e \downarrow E\]
as strict models for the pullbacks:
\[
\begin{tikzcd}
 b^*B^{\Delta^2} \ar[d] \arrow[rd, phantom, "\lrcorner", very near start] \ar[r]& B^{\Delta^2} \ar[d, "\ev_{00}"] & & 	b^*E^{\Delta^1} \ar[d] \arrow[rd, phantom, "\lrcorner", very near start] \ar[r]& E^{\Delta^1} \ar[d]
	\ar[r,  "\partial_0"] & E \ar[dl, "\pi"] \\
	 \unit \ar[r, "b"] & B & &	\unit \ar[r, "b"] & B
\end{tikzcd}
\]
These serve to define the pullback
\[
\begin{tikzcd}
 S \ar[d] \arrow[rd, phantom, "\lrcorner", very near start] \ar[r]& b^*E^{\Delta^1} \ar[d, "q"] \\
 b^*B^{\Delta^2} \times P\,b \ar[r, "r \times \id_{P\,b}"] & b \downarrow B \times P\,b
\end{tikzcd}
\]
where
\[ q\defeq \lambda w,d,d',f.\pair{w}{d}: b^*E^{\Delta^1} \to b \downarrow B \times P\,b \]
and
\[ r\defeq \lambda u,v.v \circ u:b^*B^{\Delta^2} \to b \downarrow B.\]
All the occuring types are Segal, hence $S \to b \downarrow B \times P\,b$ is an inner fibration, equivalent to the projection $\totalty{A} \equiv \sum_{\pair{u}{e}:b \downarrow B \times P\,b} A(u,e) \to b \downarrow B \times P\,b$.

In particular $\totalty{A} \equiv S$ is Segal. Initiality in the type $A(u,e)$ implies furthermore that cocartesian lifts (w.r.t.~the given data) are uniquely determined up to homotopy.

\begin{prop}[Uniqueness of cocartesian lifts (in isoinner families), {\protect\cite[Lemma~5.1.11]{RV}}]\label{prop:cocart-lifts-unique-in-isoinner-fams}
	Let $B$ be a Rezk type and $P: B \to \UU$ be an isoinner family. Then $P$-cocartesian lifts of arrows of $B$ are unique up to homotopy.
\end{prop}

\begin{cor}\label{cor:cocart-trivfill}
	Let $P:B \to \UU$ be an isoinner family over a Rezk type $B$.
	\begin{enumerate}
		\item\label{it:cocart-fill-comp} If $f:e \to e'$ is a $P$-cocartesian arrow, and $h:e' \to e''$ is an arbitrary arrow, then $\tyfill_f(h) \circ f = h$.
		\item\label{it:cocart-comp-fill} If $f:e \to e'$ is a $P$-cocartesian arrow, and $g:e' \to e''$ is an arbitrary arrow, then $\tyfill_f(gf) = g$.
	\end{enumerate}
\end{cor}

As an indication that our definition of cocartesian arrow is correct,
we note the following expected result.
\begin{prop}\label{prop:pushout-domainfib}
  Let $B$ be a Rezk type, and let $b' \xleftarrow u{} b \xrightarrow v{} c$
  be a span in $B$.
  This span is equivalently determined by an arrow $u$ in $B$
  and an element $v$ in the domain fibration $\partial_0 : B^{\Delta^1} \to B$
  over $b$, the domain of $u$.
  Then a cocone $\angled{v',f,\alpha}$ of the span, giving a square in $B$,
  % https://q.uiver.app/?q=WzAsNCxbMCwwLCJiIl0sWzEsMCwiYiciXSxbMCwxLCJjIl0sWzEsMSwiYyciXSxbMCwxLCJ1Il0sWzAsMiwidiIsMl0sWzIsMywiZiIsMix7InN0eWxlIjp7ImJvZHkiOnsibmFtZSI6ImRhc2hlZCJ9fX1dLFsxLDMsInYnIiwwLHsic3R5bGUiOnsiYm9keSI6eyJuYW1lIjoiZGFzaGVkIn19fV0sWzEsMiwiXFxhbHBoYSIsMCx7Imxlbmd0aCI6NjAsImxldmVsIjoyLCJzdHlsZSI6eyJib2R5Ijp7Im5hbWUiOiJkYXNoZWQifX19XV0=
  \[\begin{tikzcd}
      {b} & {b'} \\
      {c} & {c',}
      \arrow["{u}", from=1-1, to=1-2]
      \arrow["{v}"', from=1-1, to=2-1]
      \arrow["{f}"', from=2-1, to=2-2, dashed]
      \arrow["{v'}", from=1-2, to=2-2, dashed]
      \arrow[Rightarrow, "{\alpha}", from=1-2, to=2-1,
      shorten <=4pt, shorten >=4pt, dashed]
    \end{tikzcd}\]
  makes a \emph{pushout square}, if and only if $\pair{f}{\alpha}$ is a cocartesian
  lift of $u$ starting at $v$.
\end{prop}

Let us first clarify what we mean by a pushout (and, for the sake of completeness, a pullback) in a Rezk type. Recall from \cref{ssec:commas-cones} the definition of cocone types.

\begin{defn}[Co-/span shape]
The \emph{span} and \emph{cospan shape}, resp., are defined as the shapes
\[ \shpspan \defeq \{ \pair{t}{s}:\I^2 \; | \; t \jdeq 0 \lor s \jdeq 0 \}, \quad \shpcspan \defeq \{ \pair{t}{s}:\I^2 \; | \; t \jdeq 1 \lor s \jdeq 1 \}.\]
\end{defn}
Of course there are weak equivalences
\[ \shpspan \equiv \Lambda_0^2, \quad \shpcspan \equiv \Lambda_2^2,\]
but these pairs of shapes do not coincide on the strict level of topes.

\begin{defn}[Pullbacks and pushouts in a Rezk type]
A \emph{span} $\sigma$ in a Rezk type is a diagram $\sigma: B^\shpspan $. Dually, a cospan $\tau$ is a diagram $\tau: B^\shpcspan$. A \emph{pullback over $\sigma$} in $B$ is a terminal element in the cone type $B/\sigma$. Dually, a \emph{pushout over $\tau$} in $B$ is an initial element in the cocone type $\sigma/B$.
\end{defn}
The type of cocones can be rewritten as
\[ \sigma/B \equiv \ndexten{\Delta^1 \times \Delta^1}{B}{\shpspan}{\sigma} .\]
For a span $\sigma:B^\shpspan$ with $\sigma \jdeq (c \stackrel{v}{\leftarrow} a \stackrel{u}{\to} b)$, a cocone $\vartheta:\sigma/B$ is given by a span $\vartheta \jdeq (c \stackrel{g}{\to} d \stackrel{f}{\leftarrow} b)$.

One checks that a cocone $\vartheta \defeq \angled{x,\iota_0,\iota_1}:\sigma/B$ is a pushout square if and only if it satisfies the familiar universal property\footnote{It is possible to provide a completely strict description replacing the identity types by strict extensions. This could be done elegantly using a join operation on shapes, as is common in simplicial homotopy theory, but we will not need this here.} internally to the given Rezk type:
\[
	\isPushout_{B,\sigma}(\vartheta) \equiv \prod_{\kappa \defeq (c \stackrel{g}{\to} d \stackrel{f}{\leftarrow} b):\sigma/B} \isContr \Big( \sum_{\gamma:x \to d} \gamma \circ \iota_0 = f \times \gamma \circ \iota_1 = g  \Big)
\]
One can make similar considerations for pullbacks internal to a Rezk type.

\begin{proof}[Proof of \cref{prop:pushout-domainfib}]
  We observed above that $\pair{f}{\alpha}$ is a cocartesian lift of $u$
  starting at $v$, if and only if it is an initial object of the
  type~\eqref{eq:cocart-initiality-type}, which in the case of the
  domain family $P(b) \defeq \sum_{c:B}(b \to_B c)$ becomes
  \[
    A(u,v) \defeq \sum_{b'':B}~\sum_{u':b'\to_Bb''}
    ~\sum_{\pair{c'}{v'}:P\,b'}~\bigl(
    (c,v) \to^P_{u'\circ u}(c',v')\bigr),
  \]
  as illustrated below:
  % https://q.uiver.app/?q=WzAsNixbMCwxLCJiIl0sWzIsMSwiYicnIl0sWzAsMiwiYyJdLFsyLDIsImMnIl0sWzEsMCwiYiciXSxbMSwxXSxbMCwxLCJ1J1xcY2lyYyB1Il0sWzAsMiwidiIsMl0sWzIsMywiZiIsMix7InN0eWxlIjp7ImJvZHkiOnsibmFtZSI6ImRhc2hlZCJ9fX1dLFsxLDMsInYnIiwwLHsic3R5bGUiOnsiYm9keSI6eyJuYW1lIjoiZGFzaGVkIn19fV0sWzEsMiwiXFxhbHBoYSIsMCx7Imxlbmd0aCI6NjAsImxldmVsIjoyLCJzdHlsZSI6eyJib2R5Ijp7Im5hbWUiOiJkYXNoZWQifX19XSxbMCw0LCJ1Il0sWzQsMSwidSciXV0=
  \[\begin{tikzcd}
      & {b'} \\
      {b} & {} & {b''} \\
      {c} && {c'}
      \arrow["{u'\circ u}", from=2-1, to=2-3]
      \arrow["{v}"', from=2-1, to=3-1]
      \arrow["{f}"', from=3-1, to=3-3, dashed]
      \arrow["{v'}", from=2-3, to=3-3, dashed]
      \arrow[Rightarrow, "{\alpha}", from=2-3, to=3-1,
      shorten <=10pt, shorten >=10pt, dashed]
      \arrow["{u}", from=2-1, to=1-2]
      \arrow["{u'}", from=1-2, to=2-3]
    \end{tikzcd}\]
  There is a forgetful map $U$ from this type to the type of cocones,
  and $U$ has a LARI given by taking $u' \defeq \id_{b'}$.
  Now we conclude by~\cref{lem:LARI-initial}.
\end{proof}

\begin{prop}[Closedness under composition, and right cancelation, {\protect\cite[Lemma 5.1.5]{RV}}]\label{prop:cocart-arr-closure} Let $P: B \to \UU$ be a cocartesian family over a Rezk type $B$. For arrows $u:\hom_B(b,b')$, $v:\hom_B(b',b'')$, with $b,b',b'':B$, consider dependent arrows $f:\dhom^P_u(e,e')$, $g:\dhom^P_v(e',e'')$, for $e:P\,b$, $e':P\,b'$, $e'':P\,b''$.
	\begin{enumerate}
		\item\label{it:cocart-arr-comp} If both $f$ and $g$ are are cocartesian arrows, then so is their composite $g \circ f$.
		\item\label{it:cocart-arr-cancel} If $f$ and $g \circ f$ are cocartesian arrows, then so is $g$.
		%\item\label{it:cocart-arr-compfill} Let $f$ and $g$ be cocartesian. For any $b''':B$, $w:\hom_B(b'',b''')$, let  $e''':P\,b'''$, and $h:\dhom^P_{w\circ u \circ v}(e,e''')$ some dependent arrow. Then $\tyfill_f(h) = \tyfill_{g \circ f}(h) \circ g$.
	\end{enumerate}
\end{prop}

\begin{proof}
	\begin{enumerate}
		\item Let $w:\hom_B(b'',b''')$, $b''':B$. Writing $r\defeq w \circ v \circ u:\hom_B(b,b''')$, let $h:\dhom^P_r(e,e''')$, $e''':P\,b'''$. Define $m\defeq \tyfill_{g}(\tyfill_f(h)):\dhom^P_w(e'',e''')$, then $m \circ (g \circ f) = h$.

		For any $m':\dhom^P_w(e'',e''')$ with $h = m' \circ (g \circ f) = (m' \circ g) \circ f$, necessarily $m' \circ g = \tyfill_f(h)$. But then $m' = \tyfill_g(\tyfill_f(h))$. Thus, $g \circ f$ is cocartesian.
		\item Let $w\defeq v \circ u:\hom_B(b,b'')$ and $h\defeq g \circ h:\dhom^P_w(e,e'')$. For $b''':B$, consider $r:\hom_B(b'',b''')$ and write $t\defeq r \circ v : \hom_B(b',b''')$. Let $k:\dhom^P_t(e',e''')$, $e''':P\,b'''$. Since $h=gf$ is cocartesian, there is a filler $m:\dhom^P_r(e'',e''')$, $m=\tyfill_h(kf)$, so $mh=kf$.

		Then in turn $k=\tyfill_f(mh)$. But since $kf = (mg)f$, we have $k=mg$.

		Given any $m'$ s.t.~$m'h = kf$, this means $m'=\tyfill_h(kf) = m$.\qedhere
	\end{enumerate}
\end{proof}

\begin{lem}[\protect{\cite[Lemma~5.1.6]{RV}}]\label{lem:cocart-arrows-isos}
	Let $P: B \to \UU$ be an inner family over a Segal type $B$.
	\begin{enumerate}
		\item\label{it:depisos-are-over-isos} If $f$ is a dependent isomorphism in $P$ over some morphism $u$ in $B$, then $u$ is itself an isomorphism.
		\item\label{it:depisos-are-cocart} Any dependent isomorphism in $P$ is cocartesian.
		\item\label{it:cocart-arrows-over-ids-are-isos} If $f$ is a cocartesian arrow in $P$ over an identity in $B$, then $f$ is an isomorphism.
	\end{enumerate}
\end{lem}

\begin{proof}
	\begin{enumerate}
		\item Let $u:\hom_B(b,b')$ with $b,b':B$, and $f:\dhom^P_u(e,e')$ for $e:P\,b$ and $e':P\,b'$. If $f$ is an isomorphism, since $\totalty{P}$ is Segal, there exists a morphism $f^{-1}:\dhom^P_{u'}(e',e)$ over some morphism $u':\hom_B(b',b)$ such that $f\inv \circ f = \id_e$ and $f \circ f\inv = \id_{e'}$. But for the projection $\pi : \totalty{P} \to B$, this implies
		\[ u' \circ u = \pi(f\inv) \circ \pi(f) = \pi(f\inv \circ f) = \pi(\id_e) = \id_b,\]
		and likewise $u \circ u' = \id_{b'}$. Thus, $u' \jdeq: u^{-1}$ is a two-sided inverse for $u$, so $u$ is an isomorphism, because $B$ is Segal (cf.~\cite[Proposition 10.1]{RS17}). In particular, $\pi(f\inv) = u\inv$ implies $f\inv:\dhom^P_{u\inv}(e',e)$.
		\item Let $b,b':B$, $u:\hom_B(b,b')$ and $f:\dhom^P_u(e,e')$ such that $f$ is an isomorphism. For $b''':B$, $v:\hom_B(b,b''')$ consider a dependent morphism $g:\dhom^P_{v\circ u}(e,e''')$. Since by \cref{prop:innerfib-total-sp-over-segal-base-is-segal} the type $\totalty{P}$ is Segal, both $u$ and $f$ are invertible , cf.~\cref{lem:cocart-arrows-isos}(\ref{it:depisos-are-over-isos}) and~\cite[Proposition 1.10]{RS17}.

		Define $h\defeq g \circ f^{-1} : \dhom^P_v(e',e'')$. Then $h \circ f = (g \circ f\inv) \circ f = g \circ \id_e = g$. This yields the desired filler.

		For uniqueness, observe that for any $k:\dhom^P_v(e',e''')$ s.t.~$g=kf$ we have $hf=kf$, \ie, $h=k$ applying $f^{-1}$.

		\item In general, arrows over an identity morphism are just arrows in a fiber (even definitionally). Thus, for $b:B$ consider a cocartesian arrow $f:\hom_{P\,b}(e,e') \jdeq \dhom^P_{\id_b}(e,e')$, where $e,e':P\,b$ .

		Since $f$ is cocartesian it has a retraction: there exists an arrow $v \colon b\to b$ with $f':\dhom^P_v(e',e)$, uniquely up to homotopy, s.t.~$\id_e= f' \circ f$. But this implies $\id_b = \pi(\id_e) = \pi(f') \circ \pi(f) = \pi(f')$, so we can assume $f'$ to be over $\id_b$.

		Since $f$ is cocartesian, for any morphism $g:\dhom^P_w(e',e')$, $w:\hom_B(b',b')$ with $g \circ f = f$ we get that $g=\id_{e'}$. Now, take $g:=ff'$, $w:=\id_{b'}$. Then $g \circ f = (f \circ f') \circ f = f \circ (f' \circ f) = f$, thus $f \circ f' = \id_{e'}$. Hence, $f$ is an isomorphism.\qedhere
	\end{enumerate}
\end{proof}

\subsubsection{Characterizations of cocartesian arrows}

\begin{prop}[\protect{\cite[Definition~2.4.1.10]{LurHTT}}]
	Let $B$ be a Rezk type, and $P:B \to \UU$ be an isoinner family with total type $E\defeq \totalty{P}$. Fix $b,b':B$ and let $u:\hom_B(b,b')$ be a morphism in $B$ and $f:\dhom^P_u(e,e')$, for $e:P\,b$, $e':P\,b'$, a dependent morphism.

	The morphism $f$ is cocartesian if and only if, for any $b'':B$, $e'':P\,b''$, the diagram
	\[
	\begin{tikzcd}
		\hom_E(\pair{b'}{e'}, \pair{b''}{e''}) \ar[r, "{-\circ \pair{u}{f}}"] \ar[d] & \hom_E(\pair{b}{e}, \pair{b'}{e'})  \ar[d] \\
		\hom_B(b',b'') \ar[r, "{-\circ u}"] & \hom_B(b,b'')
	\end{tikzcd}
	\]
	is a pullback.
\end{prop}

\begin{proof}
	By fibrant replacement, de-/strictification, and choice for extension types, we can replace the square in question by:
	% https://q.uiver.app/?q=WzAsOCxbMiwzXSxbMiwxLCJcXG1hdGhsbGFwe1xcc3VtX3t3OmIgXFx0byBiJyd9IFxcc3VtX3tcXHN1YnN0YWNre1xcc2lnbWE6XFxuZGV4dGVue1xcRGVsdGFeMn17Qn17XFxMYW1iZGFfMF4yfXtbdSx3XX1cXFxcaDplIFxcdG9eUF93IGUnJ319fSBcXGV4dGVue3Q6XFxEZWx0YV4yfXtQKFxcc2lnbWEodCkpfXtcXExhbWJkYV8wXjJ9e1tmLGhdfSJdLFs1LDNdLFs1LDFdLFswLDBdLFsyLDIsIlxcc3VtX3t3OmJcXHRvIGInJ30gXFxuZGV4dGVue1xcRGVsdGFeMn17Qn17XFxMYW1iZGFfMF4yfXtbdSx3XX0iXSxbMywxLCJcXHN1bV97dzpiIFxcdG8gYicnfSAoZSBcXHRvX3deUCBlJycpIl0sWzMsMiwiXFxob21fQihiLGInJykiXSxbMSw1XSxbMSw2XSxbNiw3XSxbNSw3XV0=
	\[
	\hspace{55mm}
	\begin{tikzcd}
		{\mathllap{\sum_{w:b \to b''} \sum_{\sigma:\ndexten{\Delta^2}{B}{\Lambda_0^2}{[u,w]}} \sum_{h:e \to^P_w e''}} \exten{\pair{t}{s}:\Delta^2}{P(\sigma(t,s))}{\Lambda_0^2}{[f,h]}}\ar[d]\ar[r] &
    {\sum_{w:b \to b''} (e \to_w^P e'')}\ar[d] \\
		{\sum_{w:b\to b''} \ndexten{\Delta^2}{B}{\Lambda_0^2}{[u,w]}}\ar[r] &
    {\hom_B(b,b'')}
	\end{tikzcd}\]
	Undwinding what it means for this square to be a pullback precisely recovers the condition that $f$ be cocartesian.
\end{proof}

We prove a characterization for cocartesian edges, recovering the definition established by Joyal and Lurie, and transferred to complete Segal spaces by Rasekh.

Let $B$ be a
Segal type and $P: B \to \UU$ be an isoinner family. Consider its total space $E :\jdeq \sum_{b:B} P\,b$. For $b,b':B$, let $u:\hom_B(b,b')$ an arrow with a dependent arrow $f:\dhom^P_f(e,e')$ above it, where $e:P\,b, e':P\,b'$.

There is an induced commutative square involving comma objects, each of which can be described using extension types:
% https://q.uiver.app/?q=WzAsOSxbMSwwLCJcXG1hdGhsbGFwe1xcc3VtX3tcXHNpZ21hOnUgXFxkb3duYXJyb3cgQn0gXFxleHRlbnt0OlxcRGVsdGFeMX17UChcXHNpZ21hKHQpKX17XFxEZWx0YV8xXjF9e2Z9fSBcXHNpbWVxIGYvRSJdLFszLDFdLFszLDAsInUvQiBcXHNpbWVxIFxcbWF0aHJsYXB7XFxuZGV4dGVue1xcRGVsdGFeMn17Qn17XFxEZWx0YV8xXjF9e3V9fSJdLFs0LDBdLFs0LDFdLFswLDBdLFswLDFdLFsxLDIsIlxcbWF0aGxsYXB7XFxzdW1fe3U6YlxcZG93bmFycm93IEJ9IFxcZXh0ZW57dDpcXERlbHRhXjF9e1AodSh0KSl9e1xccGFydGlhbF4xfXtlfX0gXFxzaW1lcSBlL0UiXSxbMywyLCJiIFxcZG93bmFycm93IEIiXSxbMCwyXSxbMCw3LCJcXHBhcnRpYWxfMCIsMl0sWzcsOF0sWzIsOCwiXFxwYXJ0aWFsXzAiXV0=
\[\hspace{25mm}\begin{tikzcd}[sep=large]
    {\mathllap{\sum_{\sigma:u \downarrow B} \exten{\pair{t}{s}:\Delta^1}{P(\sigma(t))}{\Delta_1^1}{f}} \simeq f/E}
    \ar[r]\ar[d,"{\partial_0}"'] &
    {u/B \simeq \mathrlap{\ndexten{\Delta^2}{B}{\Delta_1^1}{u}}}
    \ar[d,"{\partial_0}"] \\
    {\mathllap{\sum_{u:b\downarrow B} \exten{t:\Delta^1}{P(u(t))}{0}{e}} \simeq e/E}
    \ar[r] & {b/B}
\end{tikzcd}\]

\begin{prop}[Joyal's Criterion, \protect{\cite[Paragraph~20.4]{JoyNotesQcat}, \cite[Definition~2.4.1.1]{LurHTT}}]
	Let $B$ be Segal and $P: B \to \UU$ be an inner family. Write $E\defeq \totalty{P}$ and consider the canonical projection $\pi : E \to B$.

	A dependent arrow $f:\dhom^P_u(e,e')$, $e:P\,b, e':P\,b'$, $u:\hom_B(b,b')$, $b,b':B$, is cocartesian if and only if the mediating map $\varphi$ in the pullback
% https://q.uiver.app/?q=WzAsNSxbMSwxLCJlL0UgXFx0aW1lc197Yi9CfSB1L0IiXSxbMSwyLCJlL0UiXSxbMywyLCJiL0IiXSxbMywxLCJ1L0IiXSxbMCwwLCJmL0UiXSxbMCwxXSxbMSwyXSxbMCwzXSxbMywyLCJcXHBhcnRpYWxfMCJdLFs0LDMsIiIsMix7ImN1cnZlIjotMn1dLFs0LDEsIlxccGFydGlhbF8wIiwyLHsiY3VydmUiOjJ9XSxbNCwwLCJcXHZhcnBoaSIsMSx7InN0eWxlIjp7ImJvZHkiOnsibmFtZSI6ImRhc2hlZCJ9fX1dLFswLDIsIiIsMCx7InN0eWxlIjp7Im5hbWUiOiJjb3JuZXIifX1dXQ==
\[\begin{tikzcd}
	{f/E} \\
	& {e/E \times_{b/B} u/B} && {u/B} \\
	& {e/E} && {b/B}
	\arrow[from=2-2, to=3-2]
	\arrow[from=3-2, to=3-4]
	\arrow[from=2-2, to=2-4]
	\arrow["{\partial_0}", from=2-4, to=3-4]
	\arrow[curve={height=-12pt}, from=1-1, to=2-4]
	\arrow["{\partial_0}"', curve={height=12pt}, from=1-1, to=3-2]
	\arrow["\varphi"{description}, dashed, from=1-1, to=2-2]
	\arrow["\lrcorner"{anchor=center, pos=0.125}, draw=none, from=2-2, to=3-4]
\end{tikzcd}\]
	is an equivalence.
\end{prop}

\begin{proof}
	Note that by precondition both $B$ and $E$ are Segal types. The map $\varphi$ is an equivalence if and only if
	\[
	\prod_{b'':B} \, \prod_{v:\hom_B(b',b'')} \, \prod_{e'':P \, b''} \, \prod_{h:\dhom^P_{v \circ u}(e,e'')} \, \isContr \Big( \sum_{g:\dhom^P_v(e',e'')} g \circ f = h \Big)
	\]
	which is equivalent to $f$ being cocartesian.
\end{proof}

\begin{cor}\label{cor:cocart-arr-idfam}
	For a Segal type $B$, any arrow in $B$ is cocartesian in the unit family~$\lambda b.\unit : B \to \UU$.
\end{cor}

Another characterization familiar from the semantics is given in terms of the ``cubical horn''. The pushout product of the two shape inclusions $b_1: \partial \Delta^1 \hookrightarrow \Delta^1$, $i_0:\unit \to \Delta^1$, can be depicted as:
% https://q.uiver.app/?q=WzAsMTAsWzAsMSwiXFxjZG90Il0sWzEsMSwiXFxjZG90Il0sWzAsMCwiXFxjZG90Il0sWzEsMCwiXFxjZG90Il0sWzAsMiwiXFxjZG90Il0sWzEsMiwiXFxjZG90Il0sWzAsMywiXFxjZG90Il0sWzEsMywiXFxjZG90Il0sWzIsMSwiKFxccGFydGlhbCBcXERlbHRhXjEgXFx0aW1lcyBcXERlbHRhXjEpIFxcYmlnc3FjdXBfe1xccGFydGlhbCBcXERlbHRhXjEgXFx0aW1lcyBcXHswXFx9fSAoXFxEZWx0YV4xIFxcdGltZXMgXFx7MFxcfSkgIl0sWzIsMiwiXFxEZWx0YV4xIFxcdGltZXMgXFxEZWx0YV4xIl0sWzAsMV0sWzAsMl0sWzEsM10sWzQsNV0sWzYsNF0sWzYsN10sWzcsNV0sWzgsOSwiIiwyLHsic3R5bGUiOnsidGFpbCI6eyJuYW1lIjoiaG9vayIsInNpZGUiOiJ0b3AifX19XSxbNiw1LCIiLDIseyJsZW5ndGgiOjcwLCJsZXZlbCI6Miwic3R5bGUiOnsiaGVhZCI6eyJuYW1lIjoibm9uZSJ9fX1dLFsxMCwxMywiIiwyLHsibGVuZ3RoIjo3MCwic3R5bGUiOnsidGFpbCI6eyJuYW1lIjoiaG9vayIsInNpZGUiOiJ0b3AifX19XV0=
\[\begin{tikzcd}
	{\cdot} & {\cdot} \\
	{\cdot} & {\cdot} & {(\partial \Delta^1 \times \Delta^1) \bigsqcup_{\partial \Delta^1 \times \{0\}} (\Delta^1 \times \{0\}) } \\
	{\cdot} & {\cdot} & {\Delta^1 \times \Delta^1} \\
	{\cdot} & {\cdot}
	\arrow[""{name=0, inner sep=0}, from=2-1, to=2-2]
	\arrow[from=2-1, to=1-1]
	\arrow[from=2-2, to=1-2]
	\arrow[""{name=1, inner sep=0}, from=3-1, to=3-2]
	\arrow[from=4-1, to=3-1]
	\arrow[from=4-1, to=4-2]
	\arrow[from=4-2, to=3-2]
	\arrow[from=2-3, to=3-3, hook]
	\arrow[Rightarrow, from=4-1, to=3-2, shorten <=3pt, shorten >=3pt, no head]
	\arrow[from=0, to=1, shorten <=3pt, shorten >=3pt, hook]
\end{tikzcd}\]
\begin{prop}[\protect{\cite[Proposition~2.4.1.8]{LurHTT}}]
	Let $f$ be a dependent arrow in an isoinner family $P:B \to \UU$ over a Rezk type $B$. Then $f$ is cocartesian if and only if every diagram of the following form has a filler uniquely up to homotopy:
	% https://q.uiver.app/?q=WzAsNSxbMCwwLCJcXERlbHRhXjEiXSxbMSwwLCIoXFxwYXJ0aWFsIFxcRGVsdGFeMSBcXHRpbWVzIFxcRGVsdGFeMSkgXFxiaWdzcWN1cF97XFxwYXJ0aWFsIFxcRGVsdGFeMSBcXHRpbWVzIFxcezBcXH19IFxcRGVsdGFeMSBcXHRpbWVzIFxcezBcXH0iXSxbMywwLCJcXHdpZGV0aWxkZXtQfSJdLFsxLDEsIlxcRGVsdGFeMSBcXHRpbWVzIFxcRGVsdGFeMSJdLFszLDEsIkIiXSxbMCwxLCJcXHswLDFcXH0iLDJdLFsxLDJdLFsxLDMsIiIsMCx7InN0eWxlIjp7InRhaWwiOnsibmFtZSI6Imhvb2siLCJzaWRlIjoidG9wIn19fV0sWzMsNF0sWzIsNCwiIiwxLHsic3R5bGUiOnsiaGVhZCI6eyJuYW1lIjoiZXBpIn19fV0sWzAsMiwiZiIsMSx7ImN1cnZlIjotM31dXQ==
	\[\begin{tikzcd}
		{\Delta^1} & {(\partial \Delta^1 \times \Delta^1) \bigsqcup_{\partial \Delta^1 \times \{0\}} \Delta^1 \times \{0\}} && {\widetilde{P}} \\
		& {\Delta^1 \times \Delta^1} && {B}
		\arrow["{\{0,1\}}"', from=1-1, to=1-2]
		\arrow[from=1-2, to=1-4]
		\arrow[from=1-2, to=2-2, hook]
		\arrow[from=2-2, to=2-4]
		\arrow[from=1-4, to=2-4, two heads]
		\arrow["{f}" description, from=1-1, to=1-4, curve={height=-18pt}]
	\end{tikzcd}\]
\end{prop}

We omit a proof, but it is straightforward to do, jumping between (degenerate) squares and $2$-simplices.

As has been pointed out to us by Emily Riehl, cocartesian arrows can also be characterized along similar lines like cocartesian families in \cref{thm:cocart-fam-intl-char-fib,thm:cocart-fams-intl-char} and cocartesian functors in \cref{thm:cocart-fun-intl-char}, using the formalism of relative adjunctions/absolute left lifting diagrams~\cite[Theorem 5.1.7]{RV}. Translating to the type-theoretic setting, this unwinds exactly to the above ``cubical'' formulation of the universal property. This viewpoint seems suitable for the type-theoretic setting and would also yield a more general account to the formal properties of cocartesian arrows, such as closure under composition or right cancelation. In this sense, cocartesian arrows can be seen as an instance of the notion of \emph{$j$-LARI cells} which is further and more explicitly discussed in~\Cite[Appendix~A.2.1]{jw-phd}.

\subsection{Cocartesian families}\label{ssec:cocart-fam}

Cocartesian families are families of Rezk types such that every map in the base has a cocartesian lift w.r.t.~a choice of the source vertex. After showing elementary properties such as functoriality we prove the Chevalley criterion which exhibits cocartesian families as LARI fibrations (w.r.t.~the initial vertex inclusion $i_0:\unit \to \Delta^1$). By the results from \cref{ssec:lari-closed} it then follows that cocartesian fibrations are closed under pullback, composition, and dependent products. We proceed by giving three kinds of examples of cocartesian families: the comma codomain, the domain projection in case the base category has all pushouts, and the cocartesian replacement of an arbitrary map between Rezk types. Relating to the latter, we show that cocartesian replacement really satisfies the desired universal property.

Here, we are only concerned with \emph{co-}cartesian fibrations, so we omit the dual case of right adjoint right inverse (RARI) adjunctions. Occasionally, reminiscent of the jargon of classical fibered $1$-category theory, we will refer to cocartesian fibrations simply as \emph{opfibrations} (\eg,~when talking about the \emph{codomain opfibration} of a Rezk type).

\subsubsection{Definition and basic properties}

\begin{defn}[Cocartesian lifting property]
	A family $P: B \to \UU$ is said to \emph{have (all) cocartesian lifts} if
	\[
	\hasCocartLifts \, P :\jdeq \prod_{b,b':B} \prod_{u:\hom_B(b,b')} \prod_{e:P\,b} \sum_{e':P\,b'} \sum_{f:\hom^{P}_{u}(e,e')} \isCocartArr^P_u \, f.
	\]
\end{defn}

\begin{defn}[Cocartesian family]
	For any type $B$, we call a family $P : B \to \UU$ a \emph{cocartesian family} if
	\[
	\isCocartFam \, P :\jdeq \isIsoInnerFam \, P \times \hasCocartLifts \, P.
	\]
\end{defn}

If $B$ is a Rezk type and $P:B \to \UU$ is a cocartesian family, then any arrow $u:\hom_B(a,b)$ induces a functor $\coliftptfammap{P}{u}: P\,a \to P\,b$ defined by
\[ \coliftptfammap{P}{u}: \jdeq \lambda d. \partial_1 \coliftarr{P}{u}{d}.\]
We will often omit the superscript if the family is clear from the context.

\emph{Notation:} We often denote (types of) cocartesian arrows by a decorated arrow $\cocartarr$. (Dually, cartesian arrows are denoted $\cartarr$.)

\begin{defn}[Vertical arrow]\label{def:vert-arr}
Let $\pi:E \to B$ be an inner fibration over a Segal type. A dependent arrow $f: \Delta^1 \to E$ is called \emph{vertical} if $\pi \, f$ is an isomorphism.
\end{defn}

Observe that, since being an isomorphism is a proposition in a Segal type by~\cite[Proposition~1.10]{RS17}, being a vertical arrow is a proposition when in an inner family over a Segal type. In a cocartesian family one recovers the classically well-known fact that any dependent arrow factors as a cocartesian arrow followed by a vertical arrow.

Cocartesian families implement the idea of a \emph{functorial} family of Rezk types, \ie,~in addition to transport along paths---which exists for arbitrary type families---there is also a notion of transport along \emph{directed} arrows, which turns out to be compatible with the familiar path transport.

\begin{prop}[Functoriality]\label{prop:cocart-functoriality}
	Let $B$ be a Rezk type and $P:B \to \UU$ a cocartesian family. For any $a:B$ and $x:P\,a$ there is an identity
	\[ \coliftarr{P}{\id_a}{x} = \id_x, \]
	and for any $u:\hom_B(a,b)$, $v:\hom_B(b,c)$, there is an identity
	\[ \coliftarr{P}{v \circ u}{x} = \coliftarr{P}{v}{\coliftpt{u}{x}} \circ \coliftarr{P}{u}{x}. \]
\end{prop}

\begin{proof}
	The first claim follows from \cref{lem:cocart-arrows-isos}(\ref{it:cocart-arrows-over-ids-are-isos}), in combination with \cref{prop:cocart-lifts-unique-in-isoinner-fams}.

	The second claim follows from \cref{prop:cocart-arr-closure}(\ref{it:cocart-arr-comp}).
\end{proof}

\begin{prop}
Let $B$ be a Rezk type and $P:B \to \UU$ be a cocartesian family. For any arrow $u:\hom_B(a,b)$ and elements $d:P\,a$, $e:P\,b$, we have equivalences between the types of (cocartesian) lifts of arrows from $d$ to $e$ and maps (isomorphisms, resp.) from $u_!d$ to $e$:
% https://q.uiver.app/?q=WzAsOCxbMSwwLCJlKSJdLFsxLDEsImUpIl0sWzMsMCwiKHVfKmQiXSxbMywxLCIodV8qZCJdLFswLDAsIihkIl0sWzAsMSwiKGQiXSxbNCwxLCJlKSJdLFs0LDAsImUpIl0sWzAsMiwiXFxzaW1lcSJdLFsxLDMsIlxcc2ltZXEiXSxbNCwwLCJ1IiwyXSxbMyw2LCJQXFwsYiJdLFs1LDEsInUiLDJdLFsyLDcsIlBcXCxiIl0sWzEwLDEyLCIiLDIseyJzaG9ydGVuIjp7InNvdXJjZSI6MzAsInRhcmdldCI6MzB9LCJzdHlsZSI6eyJ0YWlsIjp7Im5hbWUiOiJob29rIiwic2lkZSI6InRvcCJ9fX1dLFsxMywxMSwiIiwyLHsic2hvcnRlbiI6eyJzb3VyY2UiOjMwLCJ0YXJnZXQiOjMwfSwic3R5bGUiOnsidGFpbCI6eyJuYW1lIjoiaG9vayIsInNpZGUiOiJ0b3AifX19XV0=
	\[\begin{tikzcd}
		{(d} & {e)} && {(u_!d} & {e)} \\
		{(d} & {e)} && {(u_!d} & {e)}
		\arrow["\equiv", from=1-2, to=1-4]
		\arrow["\equiv", from=2-2, to=2-4]
		\arrow[""{name=0, anchor=center, inner sep=0}, "u"', from=1-1, to=1-2, cocart]
		\arrow[""{name=1, anchor=center, inner sep=0}, "{P\,b}" below, from=2-4, to=2-5]
		\arrow[""{name=2, anchor=center, inner sep=0}, "u"', from=2-1, to=2-2]
		\arrow[""{name=3, anchor=center, inner sep=0}, "{P\,b}" below,equals,  from=1-4, to=1-5]
		\arrow[shorten <=9pt, shorten >=6pt, hook, from=0, to=2]
		\arrow[shorten <=9pt, shorten >=6pt, hook, from=3, to=1]
	\end{tikzcd}\]
\end{prop}

\begin{proof}
Consider the maps
\begin{align*}
	\Phi & :  (d \to_u e) \longrightarrow (u_!d \to_{P\,b} e), \quad \Phi(f):\jdeq \tyfill_{P_!(u,d)}(f), \\
	 \Psi & : (u_!d \to_{P\,b} e) \longrightarrow (d \to_u e), \quad \Psi(p):\jdeq p \circ P_!(u,d).
\end{align*}
From the universal property of cocartesian fillers, as argued in \cref{cor:cocart-trivfill}, we find
\begin{align*}
	\Phi(\Psi(g)) & = \tyfill_{P_!(u,d)}(g \circ P_!(u,d)) = g, \\
	\Psi(\Phi(h)) & = \tyfill_{P_!(u,d)}(h) \circ P_!(u,d) = h.
\end{align*}

Now, if $u:d \cocartarr_u e$ is a cocartesian arrow by uniqueness of cocartesian lifts the filler $\Phi(u)$ must be a path. Thus $\Phi$ restricts to a map $\Phi':(d \cocartarr_u e) \longrightarrow (u_!d =_{P\,b} e)$.

If $p:u_!d = e$ is a path, it is in particular cocartesian, so $\Psi(p)$ is as well since cocartesian arrows are closed under composition. Hence, $\Psi$ restricts to a map
\[
  \Psi':(u_!d =_{P\,b} e) \longrightarrow (d \cocartarr_u e).\qedhere
\]
\end{proof}

Thus, just as in the classical case, our cocartesian families capture the notion of covariantly functorial families of categories.

Due to \cref{prop:cocart-lifts-unique-in-isoinner-fams}, over Rezk types being a cocartesian family is a proposition, and indeed this is the setting that we are interested in. In particular, cocartesian families over Rezk types are thus ``cloven up to homotopy''. Given an arrow $u:\hom_B(b,b')$ together with $e:P\,b$, we write $P_!(u,e)$ for the homotopically unique cocartesian lift of $u$. Even more, from the point of view of homotopy type theory, these cleavages are automatically ``\emph{split}''.\footnote{This is understood in the sense analogous to~\cite[Definition 3.1]{streicher2020fibered}.}

\subsubsection{Characterizations of cocartesian families}

We now turn to characterizations of cocartesian fibrations after~\cite{RV,RVyoneda}.

\paragraph{Cocartesian families via lifting}

\begin{theorem}[Chevalley criterion: Cocartesian families via lifting, {\protect\cite[Proposition 5.2.8(ii)]{RV}}]\label{thm:cocart-fams-intl-char}
	Let $B$ be a Rezk type, $P: B \to \UU$ be an isoinner family, and denote by $\pi: E \to B$ the associated projection map. The family $P$ is cocartesian if and only if the Leibniz cotensor map $i_0 \cotens \pi: E^{\Delta^1} \to \comma{\pi}{B}$ has a left adjoint right inverse:
	% https://q.uiver.app/?q=WzAsNixbMiwxLCJcXHBpIFxcZG93bmFycm93IEIiXSxbMiwyLCJCXntcXERlbHRhXjF9Il0sWzQsMiwiQiJdLFs0LDEsIkUiXSxbMSwwXSxbMCwwLCJFXntcXERlbHRhXjF9Il0sWzAsMSwiIiwwLHsic3R5bGUiOnsiaGVhZCI6eyJuYW1lIjoiZXBpIn19fV0sWzEsMiwiXFxwYXJ0aWFsXzAiLDJdLFswLDNdLFszLDIsIlxccGkiLDAseyJzdHlsZSI6eyJoZWFkIjp7Im5hbWUiOiJlcGkifX19XSxbMCwyLCIiLDAseyJzdHlsZSI6eyJuYW1lIjoiY29ybmVyIn19XSxbNSwzLCJcXHBhcnRpYWxfMCIsMCx7Im9mZnNldCI6LTIsImN1cnZlIjotM31dLFswLDUsIlxcY2hpIiwyLHsiY3VydmUiOjIsInN0eWxlIjp7ImJvZHkiOnsibmFtZSI6ImRhc2hlZCJ9fX1dLFs1LDAsImlcXHdpZGVoYXR7XFxwaXRjaGZvcmt9XFxwaSIsMix7ImN1cnZlIjoyfV0sWzUsMSwiXFxwaV57XFxEZWx0YV4xfSIsMix7Im9mZnNldCI6MiwiY3VydmUiOjMsInN0eWxlIjp7ImhlYWQiOnsibmFtZSI6ImVwaSJ9fX1dLFsxMiwxMywiIiwyLHsibGV2ZWwiOjEsInN0eWxlIjp7Im5hbWUiOiJhZGp1bmN0aW9uIn19XV0=
	\[\begin{tikzcd}
		{E^{\Delta^1}} & {} \\
		&& {\pi \downarrow B} && E \\
		&& {B^{\Delta^1}} && B
		\arrow[two heads, from=2-3, to=3-3]
		\arrow["{\partial_0}"', from=3-3, to=3-5]
		\arrow[from=2-3, to=2-5]
		\arrow["\pi", two heads, from=2-5, to=3-5]
		\arrow["\lrcorner"{anchor=center, pos=0.125}, draw=none, from=2-3, to=3-5]
		\arrow["{\partial_0}", shift left=2, curve={height=-18pt}, from=1-1, to=2-5]
		\arrow[""{name=0, anchor=center, inner sep=0}, "\chi"', curve={height=12pt}, dashed, from=2-3, to=1-1]
		\arrow[""{name=1, anchor=center, inner sep=0}, "{i_0\widehat{\pitchfork}\pi}"', curve={height=12pt}, from=1-1, to=2-3]
		\arrow["{\pi^{\Delta^1}}"', shift right=2, curve={height=18pt}, two heads, from=1-1, to=3-3]
		\arrow["\dashv"{anchor=center, rotate=-118}, draw=none, from=0, to=1]
	\end{tikzcd}\]
\end{theorem}

\begin{proof}
	Assume $P:B \to \UU$ is a cocartesian family. The gap map can be taken to be
	\[ i_0 \cotens \pi \defeq \lambda u,f.\pair{u}{\partial_0 f}: E^{\Delta^1} \to \comma{\pi}{B}.\]
	For the candidate LARI we take the map that produces the cocartesian lifts, \ie,
	\[ \chi \defeq \lambda u,e.\pair{u}{P_!(u,e)} : \comma{\pi}{B} \to E^{\Delta^1}.\]
	This is clearly a (strict) section of $i_0 \cotens \pi$.

	We show that for any $\pair{u:b \to b'}{e:P\,b}$ in $\comma{\pi}{B}$ and $\pair{v:c \to c'}{g:d \to d'}$ in $E^{\Delta^1}$ the maps
	% https://q.uiver.app/?q=WzAsMixbMCwwLCJcXGhvbV97XFxEZWx0YV4xIFxcdG8gRX0oXFxjaGkodSxlKSwgXFxsYW5nbGUgdixnIFxccmFuZ2xlKSJdLFsyLDAsIlxcaG9tX3tcXHBpIFxcZG93bmFycm93IEJ9KFxcbGFuZ2xlIHUsIGXCoFxccmFuZ2xlLFxcbGFuZ2xlIHYsXFxwYXJ0aWFsXzBcXCxnIFxccmFuZ2xlKSJdLFswLDEsIlxcUGhpIiwwLHsib2Zmc2V0IjotMn1dLFsxLDAsIlxcUHNpIiwwLHsib2Zmc2V0IjotMn1dXQ==
	\[\begin{tikzcd}
		{\hom_{\Delta^1 \to E}(\chi(u,e), \langle v,g \rangle)} && {\hom_{\pi \downarrow B}(\langle u, e \rangle,\langle v,\partial_0\,g \rangle)}.
		\arrow["\Phi_{\pair{u}{e},\pair{v}{g}}", shift left=2, from=1-1, to=1-3]
		\arrow["\Psi_{\pair{u}{e},\pair{v}{g}}", shift left=2, from=1-3, to=1-1]
	\end{tikzcd}\]
	defined by
	\[ \Phi_{\pair{u}{e},\pair{v}{g}}(k,m) \defeq \angled{\pi \, k,\pi\, m, k}, \quad \Psi_{\pair{u}{e},\pair{v}{g}}(w,r,k) \defeq \pair{k}{\tyfill_{P_!(u,e)}(g\circ k)} \]
	yield a quasi-transposing adjunction (cf.~\cref{fig:transp-chevalley} for illustration).\footnote{For brevity, we shall henceforth leave the fixed parameters $\pair{u}{e},\pair{v}{g}$ implicit.}

	On the one hand, for a morphism $\angled{w,r,k}$ we find
	\[ \Phi(\Psi(w,r,k)) = \Phi(k, \tyfill_{P_!(u,e)}(g \circ k)) = \angled{w,r,k} \]
	since $\tyfill_{P_!(u,e)}(g \circ k)$ lies over $r$ and $k$ lies over $w$ (even strictly so).

	For a morphism $\pair{k}{m}$ we obtain
	\[ \Psi(\Phi(k,m))=\Psi(\pi \,k, \pi\,m, k) = \pair{k}{m}\]
	from the cocartesian universal property.

	This shows $\chi \dashv i_0 \cotens \pi$ is a LARI adjunction, as claimed.

	Conversely, suppose $\chi$ is a given LARI of $i_0 \cotens \pi$, \wlogg~a strict section. Then, for any $u:b \to b'$ in $B$, $c:B$, $e:P\,b$, $d:P\,c$, the map
	% https://q.uiver.app/?q=WzAsMixbMCwwLCJcXGhvbV97XFxEZWx0YV4xIFxcdG8gRX0oXFxjaGkodSxlKSxcXGxhbmdsZSBcXGlkX2MsXFxpZF9kXFxyYW5nbGUpIl0sWzIsMCwiXFxob21fe1xccGkgXFxkb3duYXJyb3cgQn0oXFxsYW5nbGUgdSxlXFxyYW5nbGUsXFxsYW5nbGUgXFxpZF9jLGRcXHJhbmdsZSkiXSxbMCwxLCJcXFBoaSJdXQ==
	\[\begin{tikzcd}
		{\hom_{\Delta^1 \to E}(\chi(u,e),\langle \id_c,\id_d\rangle)} && {\hom_{\pi \downarrow B}(\langle u,e\rangle,\langle \id_c,d\rangle)}.
		\arrow["\Phi", from=1-1, to=1-3]
	\end{tikzcd}\]
	defined by
	\[ \Phi(w,v,h,g) \jdeq \angled{w,v,h}\]
	is an equivalence. Finally, contractibility of the fibers amounts to the cocartesian universal property of $\chi(u,e):e \to_u \partial_1(\chi(u,e))$.
\end{proof}
	\begin{figure}
		\centering
		% https://q.uiver.app/?q=WzAsMjAsWzAsMSwiZSJdLFswLDIsInVfKmUiXSxbMSwyLCJkIl0sWzEsMSwiZCJdLFsyLDEsImIiXSxbMiwyLCJiJyJdLFszLDIsImMnIl0sWzMsMSwiYyJdLFs3LDIsInVfKmUiXSxbMiwwLCJlIl0sWzMsMCwiZCJdLFs1LDEsImIiXSxbNSwyLCJiJyJdLFs2LDIsImMnIl0sWzYsMSwiYyJdLFs1LDAsImUiXSxbNiwwLCJkIl0sWzcsMSwiZSJdLFs4LDIsImQiXSxbOCwxLCJkIl0sWzAsMV0sWzEsMiwibSIsMl0sWzAsMywiayJdLFszLDIsImciLDJdLFs0LDUsInUiXSxbNSw2LCJcXHBpXFwsbSIsMl0sWzQsNywiXFxwaVxcLGsiXSxbNyw2LCJ2IiwyXSxbOSwxMCwiayJdLFsxMSwxMiwidSJdLFsxMiwxMywiciIsMl0sWzExLDE0LCJ3Il0sWzE0LDEzLCJ2IiwyXSxbMTUsMTYsImsiXSxbOCwxOCwiIiwwLHsic3R5bGUiOnsiYm9keSI6eyJuYW1lIjoiZGFzaGVkIn19fV0sWzE3LDE5LCJrIl0sWzE5LDE4LCJnIl0sWzE3LDhdLFs5LDQsIiIsMCx7ImxldmVsIjoyLCJzdHlsZSI6eyJib2R5Ijp7Im5hbWUiOiJkb3R0ZWQifSwiaGVhZCI6eyJuYW1lIjoibm9uZSJ9fX1dLFsxMCw3LCIiLDAseyJsZXZlbCI6Miwic3R5bGUiOnsiYm9keSI6eyJuYW1lIjoiZG90dGVkIn0sImhlYWQiOnsibmFtZSI6Im5vbmUifX19XSxbMTUsMTEsIiIsMSx7ImxldmVsIjoyLCJzdHlsZSI6eyJib2R5Ijp7Im5hbWUiOiJkb3R0ZWQifSwiaGVhZCI6eyJuYW1lIjoibm9uZSJ9fX1dLFsxNiwxNCwiIiwxLHsibGV2ZWwiOjIsInN0eWxlIjp7ImJvZHkiOnsibmFtZSI6ImRvdHRlZCJ9LCJoZWFkIjp7Im5hbWUiOiJub25lIn19fV0sWzIzLDI0LCJcXFBoaSIsMCx7InNob3J0ZW4iOnsic291cmNlIjoyMCwidGFyZ2V0IjoyMH0sImxldmVsIjoxLCJzdHlsZSI6eyJ0YWlsIjp7Im5hbWUiOiJtYXBzIHRvIn19fV0sWzMyLDM3LCJcXFBzaSIsMCx7InNob3J0ZW4iOnsic291cmNlIjoyMCwidGFyZ2V0IjoyMH0sImxldmVsIjoxLCJzdHlsZSI6eyJ0YWlsIjp7Im5hbWUiOiJtYXBzIHRvIn19fV1d
		\[\begin{tikzcd}
			&& e & d && e & d \\
			e & d & b & c && b & c & e & d \\
			{u_!e} & {d'\vphantom{b'}} & {b'} & {c'} && {b'} & {c'} & {u_!e\vphantom{c'}} & d'
			\arrow[from=2-1, to=3-1]
			\arrow["m"', from=3-1, to=3-2]
			\arrow["k", from=2-1, to=2-2]
			\arrow[""{name=0, anchor=center, inner sep=0}, "g"', from=2-2, to=3-2]
			\arrow[""{name=1, anchor=center, inner sep=0}, "u", from=2-3, to=3-3]
			\arrow["{\pi\,m}"', from=3-3, to=3-4]
			\arrow["{\pi\,k}", from=2-3, to=2-4]
			\arrow["v"', from=2-4, to=3-4]
			\arrow["k", from=1-3, to=1-4]
			\arrow["u", from=2-6, to=3-6]
			\arrow["r"', from=3-6, to=3-7]
			\arrow["w", from=2-6, to=2-7]
			\arrow[""{name=2, anchor=center, inner sep=0}, "v"', from=2-7, to=3-7]
			\arrow["k", from=1-6, to=1-7]
			\arrow[dashed, from=3-8, to=3-9]
			\arrow["k", from=2-8, to=2-9]
			\arrow["g", from=2-9, to=3-9]
			\arrow[""{name=3, anchor=center, inner sep=0}, from=2-8, to=3-8]
			\arrow[Rightarrow, dotted, no head, from=1-3, to=2-3]
			\arrow[Rightarrow, dotted, no head, from=1-4, to=2-4]
			\arrow[Rightarrow, dotted, no head, from=1-6, to=2-6]
			\arrow[Rightarrow, dotted, no head, from=1-7, to=2-7]
			\arrow["\Phi", shorten <=6pt, shorten >=6pt, maps to, from=0, to=1]
			\arrow["\Psi", shorten <=6pt, shorten >=6pt, maps to, from=2, to=3]
		\end{tikzcd}\]
		\caption{Transposing maps of the adjunction $\chi \dashv i_0 \cotens \pi$}\label{fig:transp-chevalley}
	\end{figure}

\paragraph{Cocartesian families via transport}

There is another characterization of cocartesian families in terms of an adjointness condition. Any map $\pi:E \to B$ between Rezk types is exhibited as a retract of the pullback map $\partial_0^*\pi:\comma{\pi}{B} \to B^{\Delta^1}$ in the following way:
% https://q.uiver.app/?q=WzAsNyxbNSwzLCJCIl0sWzUsMSwiRSJdLFswLDAsIkUiXSxbMiwyXSxbMCwzLCJCIl0sWzIsMywiQl57XFxEZWx0YV4xfSJdLFsyLDEsIlxccGlcXGRvd25hcnJvdyBCIl0sWzEsMCwiXFxwaSIsMCx7InN0eWxlIjp7ImhlYWQiOnsibmFtZSI6ImVwaSJ9fX1dLFsyLDEsIiIsMix7ImN1cnZlIjotNCwic3R5bGUiOnsiaGVhZCI6eyJuYW1lIjoibm9uZSJ9fX1dLFs0LDUsIlxcbWF0aHJte2NzdH0iLDJdLFs2LDUsIlxccGFydGlhbF8wXipcXHBpIiwwLHsic3R5bGUiOnsiaGVhZCI6eyJuYW1lIjoiZXBpIn19fV0sWzYsMV0sWzUsMCwiXFxwYXJ0aWFsXzAiLDJdLFsyLDQsIlxccGkiLDEseyJzdHlsZSI6eyJoZWFkIjp7Im5hbWUiOiJlcGkifX19XSxbMiw2LCJcXGlvdGFfXFxwaSIsMSx7InN0eWxlIjp7ImJvZHkiOnsibmFtZSI6ImRhc2hlZCJ9fX1dLFs2LDAsIiIsMix7InN0eWxlIjp7Im5hbWUiOiJjb3JuZXIifX1dLFsyLDVdLFs0LDAsIiIsMix7ImN1cnZlIjo0fV1d
\[\begin{tikzcd}
	E \\
	&& {\pi\downarrow B} &&& E \\
	&& {} \\
	B && {B^{\Delta^1}} &&& B
	\arrow["\pi", from=2-6, to=4-6]
	\arrow[curve={height=-24pt}, no head, from=1-1, to=2-6, equals]
	\arrow["{\mathrm{cst}}"', from=4-1, to=4-3]
	\arrow["{\partial_0^*\pi}", from=2-3, to=4-3]
	\arrow[from=2-3, to=2-6]
	\arrow["{\partial_0}"', from=4-3, to=4-6]
	\arrow["\pi"{description},from=1-1, to=4-1]
	\arrow["{\iota_\pi}"{description}, dashed, from=1-1, to=2-3]
	\arrow["\lrcorner"{anchor=center, pos=0.125}, draw=none, from=2-3, to=4-6]
	\arrow[from=1-1, to=4-3]
	\arrow[curve={height=26pt}, from=4-1, to=4-6, equals]
\end{tikzcd}\]
The mediating map
\[ \iota \defeq \iota_\pi \defeq \lambda \pair{b}{e}.\pair{\id_b}{e}: E \to \comma{\pi}{B} \]
is a fibered functor from $\pi:E \to B$ to the cocartesian replacement\footnote{Explicitly, $\partial'_1 \defeq \lambda u,e.u(1):\comma{\pi}{B} \fibarr B$, cf.~\cref{def:free-cocart}.} of $\partial_0^*\pi: \comma{\pi}{B} \to B^{\Delta^1}$, \ie,~there is a commutative triangle:
% https://q.uiver.app/?q=WzAsMyxbMCwwLCJFIl0sWzIsMCwiXFxjb21tYXtcXHBpfXtCfSJdLFsxLDEsIkIiXSxbMCwxLCJcXGlvdGFfXFxwaSJdLFswLDIsIlxccGkiLDIseyJzdHlsZSI6eyJoZWFkIjp7Im5hbWUiOiJlcGkifX19XSxbMSwyLCJcXHBhcnRpYWxfMSIsMCx7InN0eWxlIjp7ImhlYWQiOnsibmFtZSI6ImVwaSJ9fX1dXQ==
\[\begin{tikzcd}
	E && {\comma{\pi}{B}} \\
	& B
	\arrow["{\iota_\pi}", from=1-1, to=1-3]
	\arrow["\pi"', two heads, from=1-1, to=2-2]
	\arrow["{\partial'_1}", two heads, from=1-3, to=2-2]
\end{tikzcd}\]
Denote the family of fibers of $\pi$ by $P:B \to \UU$. If $P$ is cocartesian, it has ``directed transport''
\[ \tau \defeq \iota_\pi \defeq \tau_P \defeq \transp_P\defeq \lambda u,e.u^P_!e:\comma{\pi}{B} \to E.\]
This transport map $\tau$ is easily checked to be a fibered functor from $\partial'_1$ to $\pi$. We will show that it is a fibered left adjoint of $\iota_\pi$, and conversely, the existence of a fibered left adjoint $\tau_\pi$ to $\iota_\pi$ will imply that $\pi$ is cocartesian.

\begin{theorem}[Cocartesian families via transport, {\protect\cite[Proposition 5.2.8(iii)]{RV}}]\label{thm:cocart-fam-intl-char-fib}
	Let $B$ be a Rezk type, and $P:B \to \UU$ an isoinner family with associated total type projection $\pi:E \to B$.

	Then, $P$ is cocartesian if and only if the map
	\[ \iota \defeq \iota_P : E \to \comma{\pi}{B}, \quad \iota \, \pair{b}{e} :\jdeq \pair{\id_b}{e}  \]
	has a fibered left adjoint\footnote{\cf~\cref{def:fib-adj}} $\tau \defeq \tau_P: \comma{\pi}{B} \to E$ as indicated in the diagram:
	% https://q.uiver.app/?q=WzAsNCxbMCwwXSxbMiwwLCJFIl0sWzQsMCwiXFxwaSBcXGRvd25hcnJvdyBCIl0sWzMsMiwiQiJdLFsxLDIsIlxcaW90YSIsMix7ImN1cnZlIjoyfV0sWzEsMywiXFxwaSIsMix7InN0eWxlIjp7ImhlYWQiOnsibmFtZSI6ImVwaSJ9fX1dLFsyLDMsIlxccGFydGlhbF8xJyIsMCx7InN0eWxlIjp7ImhlYWQiOnsibmFtZSI6ImVwaSJ9fX1dLFsyLDEsIlxcdGF1IiwyLHsiY3VydmUiOjJ9XSxbNyw0LCIiLDIseyJsZXZlbCI6MSwic3R5bGUiOnsibmFtZSI6ImFkanVuY3Rpb24ifX1dXQ==
	\[\begin{tikzcd}
		{} && E && {\pi \downarrow B} \\
		\\
		&&& B
		\arrow[""{name=0, anchor=center, inner sep=0}, "\iota"', curve={height=12pt}, from=1-3, to=1-5]
		\arrow["\pi"', two heads, from=1-3, to=3-4]
		\arrow["{\partial_1'}", two heads, from=1-5, to=3-4]
		\arrow[""{name=1, anchor=center, inner sep=0}, "\tau"', curve={height=12pt}, from=1-5, to=1-3]
		\arrow["\dashv"{anchor=center, rotate=-90}, draw=none, from=1, to=0]
	\end{tikzcd}\]
\end{theorem}

\begin{proof}
	Assume $P$ is cocartesian. For the candidate left adjoint we take the map given by cocartesian transport
	\[ \tau:\comma{\pi}{B} \to E, \quad \tau(u,e)\defeq \pair{\partial_1 \, u}{u_!e}.\]
	Then $\pi(\tau(u,e)) \jdeq \partial_1 \,u \jdeq \partial_1'(u,e)$, so $\tau$ is a fibered functor from $\partial_1'$ to $\pi$. We show that for any $\pair{u:b \to b'}{e:P\,b}$ and $\pair{c:B}{d:P\,c}$ in $E$ the maps
	% https://q.uiver.app/?q=WzAsMixbMCwwLCJcXGhvbV97RX0oXFxsYW5nbGUgYicsdV8qZSBcXHJhbmdsZSwgXFxsYW5nbGUgdixnIFxccmFuZ2xlKSJdLFsyLDAsIlxcaG9tX3tcXHBpIFxcZG93bmFycm93IEJ9KFxcbGFuZ2xlIHUsIGXCoFxccmFuZ2xlLFxcbGFuZ2xlIFxcbWF0aHJte2lkfV9jLCBkIFxccmFuZ2xlKSJdLFswLDEsIlxcUGhpX3tcXGxhbmdsZSB1LGUgXFxyYW5nbGUsIFxcbGFuZ2xlIHYsZ1xccmFuZ2xlfSIsMCx7Im9mZnNldCI6LTJ9XSxbMSwwLCJcXFBzaV97XFxsYW5nbGUgdSxlIFxccmFuZ2xlLCBcXGxhbmdsZSB2LGdcXHJhbmdsZX0iLDAseyJvZmZzZXQiOi0yfV1d
	\[\begin{tikzcd}
		{\hom_{E}(\langle b',u_!e \rangle, \langle c,d \rangle)} && {\hom_{\pi \downarrow B}(\langle u, e \rangle,\langle \mathrm{id}_c, d \rangle)}
		\arrow["{\Phi_{\langle u,e \rangle, \langle c,d\rangle}}", shift left=2, from=1-1, to=1-3]
		\arrow["{\Psi_{\langle u,e \rangle, \langle c,d\rangle}}", shift left=2, from=1-3, to=1-1]
	\end{tikzcd}\]
	defined by
	\[ \Phi_{\langle u,e \rangle, \langle c,d\rangle}(v,g) \defeq \angled{vu,v,g \circ P_!(u,e)}, \quad \Psi_{\langle u,e \rangle, \langle c,d\rangle}(w,v,h) \defeq \pair{v}{\tyfill_{P_!(u,e)}(h)}\]
	form a quasi-equivalence (cf.~\cref{fig:transp-fibadj} for illustration).
	We have
	\[ \Phi(\Psi(w,v,h)) = \Phi(v, \tyfill_{P_!(u,e)}(h)) = \angled{vu,v,\tyfill_{P_!(u,e)}(h) \circ P_!(u,e)} = \angled{w,v,h}\]
	by \cref{cor:cocart-trivfill}(\ref{it:cocart-fill-comp}), and noting that for any square $\pair{w}{v} : u \to \id_{c}$ there is an identification $w=vu$.
	Next, we find
	\[ \Psi(\Phi(v,g)) = \Psi(vu,v,g \circ P_!(u,e)) = \pair{v}{\tyfill_{P_!(u,e)}(g \circ P_!(u,e))} = \pair{v}{g}.\]
	using again the properties of the fillers defined by the cocartesian lifts, cf.~\cref{cor:cocart-trivfill}, \ref{it:cocart-comp-fill}.

	So indeed $\tau$ is left adjoint to $\iota$. Moreover, it is a fibered left adjoint as can be seen as follows. The unit is defined by
	\[ \eta_{\pair{u}{e}} \defeq \Phi(\id_{b'},\id_{u_!e}) = \angled{u, \id_{b'}, P_!(u,e)}:\hom_{\comma{\pi}{B}}(\pair{u}{e},\pair{\id_{b'}}{u_!e}).\]
	Since the second component is an identity this is a vertical arrow in $\partial_1': \comma{\pi}{B} \fibarr B$ which proves the adjunction is fibered.

	Suppose on the converse that $\tau \dashv_B \iota$ is some fibered left adjunction. Since $\tau$ is a fibered functor, for $\pair{u:b\to b'}{e:P\,b}$ we can assume
	\[ \tau(u,e) \jdeq \pair{b'}{\widehat{\tau}(u,e)}. \]
	Next, $\eta$ being a \emph{fibered} natural transformation fixes its part in $B$, \ie,~since in the square the lower horizontal edge has to be an identity the upper horizontal edge must be $u$ (up to identification), so the only degree of freedom is the dependent arrow $f_{u,e}$ as indicated:
	% https://q.uiver.app/?q=WzAsNyxbMSwwLCJlIl0sWzIsMCwidV8qZSJdLFsxLDEsImIiXSxbMSwyLCJiJyJdLFsyLDIsImInIl0sWzIsMSwiYiciXSxbMCwxLCJcXGV0YV97dSxlfToiXSxbMCwxLCJmX3t1LGV9Il0sWzIsMywidSIsMl0sWzMsNCwiIiwwLHsibGV2ZWwiOjIsInN0eWxlIjp7ImhlYWQiOnsibmFtZSI6Im5vbmUifX19XSxbMiw1LCJ1Il0sWzUsNCwiIiwyLHsibGV2ZWwiOjIsInN0eWxlIjp7ImhlYWQiOnsibmFtZSI6Im5vbmUifX19XV0=
	\[\begin{tikzcd}
		& e & {\widehat{\tau}(u,e)} \\
		{\eta_{u,e}:} & b & {b'} \\
		& {b'} & {b'}
		\arrow["{f_{u,e}}", from=1-2, to=1-3]
		\arrow["u"', from=2-2, to=3-2]
		\arrow[Rightarrow, no head, from=3-2, to=3-3]
		\arrow["u", from=2-2, to=2-3]
		\arrow[Rightarrow, no head, from=2-3, to=3-3]
	\end{tikzcd}\]
	Hence, we can assume
	\[ \eta_{u,e} \jdeq \angled{u,\id_{b'}, f_{u,e}}.\]
	By assumption, the transposing map induced by the unit
	\[ \Phi \defeq \Phi_\eta \defeq \lambda v,g.\iota(v,g) \circ \eta_{u,e} \jdeq \angled{vu,v,g \circ f_{u,e}}:\hom_E(\iota(u,e),\pair{c}{d}) \longrightarrow \hom_{\comma{\pi}{B}}(\pair{u}{e},\pair{\id_c}{d})\]
	is an equivalence.	Spelled out, this means for any $v:b' \to c$, $h:e \to^P_{vu} d$ there exists an arrow $g_h: \widehat{\tau}(u,e) \to_v d$, uniquely up to homotopy, s.t.~$h = g_h \circ f_{u,e}$. This says exactly that $f_{u,e}:e \to \widehat{\tau}(u,e)$ is a cocartesian lift of $u$ \wrt~$e$.
\end{proof}

\begin{figure}
	\centering
	\[\begin{tikzcd}
		&& e & d && e & d \\
		{u_!e} & d & b & c && b & c & {u_!e} & d \\
		{b'} & {c\vphantom{b'}} & {b'} & c && {b'} & {c\vphantom{b'}} & {b'} & c
		\arrow["v", from=3-1, to=3-2]
		\arrow["g", from=2-1, to=2-2]
		\arrow[""{name=0, anchor=center, inner sep=0}, "u", from=2-3, to=3-3]
		\arrow["v"', from=3-3, to=3-4]
		\arrow["vu", from=2-3, to=2-4]
		\arrow["{\mathrm{id}_c}", Rightarrow, no head, from=2-4, to=3-4]
		\arrow["u"', from=2-6, to=3-6]
		\arrow["v"', from=3-6, to=3-7]
		\arrow["w", from=2-6, to=2-7]
		\arrow[""{name=1, anchor=center, inner sep=0}, "{\mathrm{id}_c}"', Rightarrow, no head, from=2-7, to=3-7]
		\arrow["{g\circ P_!(u,e)}", from=1-3, to=1-4]
		\arrow["h", from=1-6, to=1-7]
		\arrow[Rightarrow, dotted, no head, from=2-1, to=3-1]
		\arrow[""{name=2, anchor=center, inner sep=0}, Rightarrow, dotted, no head, from=2-2, to=3-2]
		\arrow[Rightarrow, dotted, no head, from=1-3, to=2-3]
		\arrow[Rightarrow, dotted, no head, from=1-4, to=2-4]
		\arrow[Rightarrow, dotted, no head, from=1-6, to=2-6]
		\arrow[Rightarrow, dotted, no head, from=1-7, to=2-7]
		\arrow["{\mathrm{fill}_{P_!(u,e)}(h)}", dashed, from=2-8, to=2-9]
		\arrow["v", from=3-8, to=3-9]
		\arrow[Rightarrow, dotted, no head, from=2-9, to=3-9]
		\arrow[""{name=3, anchor=center, inner sep=0}, Rightarrow, dotted, no head, from=2-8, to=3-8]
		\arrow["\Phi", shorten <=6pt, shorten >=6pt, maps to, from=2, to=0]
		\arrow["\Psi", shorten <=6pt, shorten >=6pt, maps to, from=1, to=3]
	\end{tikzcd}\]
	\caption{Transposing maps of the fibered adjunction $\tau_P \dashv_B \iota_P$}\label{fig:transp-fibadj}
\end{figure}
\subsubsection{Closure properties of cocartesian families}\label{ssec:cocart-clos}

In this section, when considering type families, we again assume the base types to be Rezk.

Stability of cocartesian families under composition follows as an instance of \cref{prop:lari-maps-closed-comp}.
In principle, the computation of the cocartesian lifts follows by instantiation and unpacking of the constructions given in the proofs of Proposition~\ref{prop:lari-maps-closed-comp} and~\ref{prop:lari-closed-under-comp}, respectively. But actually, we are showing a stronger statement, in a direct way, where the condition on the second factor of the composition is being slightly weakened. The reason is that this will simplify a later proof about induced cocartesian fibrations between pullback types.

\begin{prop}[Identities are cocartesian fibrations]
	For a (complete) Segal type $B$, the unit family~$\lambda b.\unit: B \to \UU$ is cocartesian.	
\end{prop}

\begin{proof}
	This is an easy consequence of \cref{cor:cocart-arr-idfam}.
\end{proof}

\begin{prop}[Cocartesian families are closed under the dependent product]\label{prop:cocart-fam-prod}
	Let $B:I \to \UU$ be a family of Rezk types, and for each $i:I$, assume $P_i: B_i \to \UU$ is a cocartesian family. We denote $E_i :\jdeq \totalty{P_i}$. Then, the family associated to the map of sections
	\[ \prod_{i:I} E_i \xfib{}  \prod_{i:I} B_i \]
	is a cocartesian family.

	In particular, a cocartesian lift for $u:\alpha \to \beta$ in $\prod_{i:I} B_i$ \wrt~$\sigma:\prod_{i:I} P_i(\alpha_i)$ is given by
	\[ \left(\prod_{i:I} P_i\right)_!(u,\sigma) :\jdeq \lambda i.\big(P(i)\big)_!(u(i), \sigma(i)).\]
\end{prop}

\begin{proof}
	Stability of cocartesian families under products follows as an instance of \cref{prop:lari-maps-closed-pi}. In principle, the computation of the cocartesian lifts follows by instantiating and unpacking the constructions given in the proofs of Propositions~\ref{prop:lari-maps-closed-pi} and~\ref{prop:lari-closed-under-pi}, respectively.

	More directly, one can alternatively take the suggested candidate lift in the product fibration and show that it is indeed cocartesian. But this follows from employing the universal property pointwisely and then invoking function extensionality.
\end{proof}

\begin{cor}[Cocartesian families are closed under exponentiation]\label{cor:cocart-fam-closed-exp}
	Let $B$ be a Rezk type and $P: B \to \UU$ be a cocartesian family. For any shape or type $x:X$, the family $P^X$ is cocartesian.
\end{cor}

\Cref{cor:cocart-fam-closed-exp} can be seen as a shadow of the fact that cocartesian fibrations are \emph{representably defined}, \cf~\cite[Section~5.6]{RV}.
To fully express this, however, we would need to employ modal type theory.

\begin{defn}[Partial cocartesian families]
	Let $P:B \to \UU$ be a cocartesian family over a Rezk type $B$. An isoinner family $Q: \totalty{P} \to \UU$ is called \emph{partial cocartesian (w.r.t.~$P$)} if every $P$-cocartesian arrow in $\totalty{P}$ has a $Q$-cocartesian lift (w.r.t.~a given point in the fiber over the source).
\end{defn}

\begin{prop}[Composite cocartesian fibrations, {\protect\cite[Lemma~5.2.3]{RV}}]\label{prop:cocart-fam-comp}
Let
% https://q.uiver.app/?q=WzAsMyxbMCwwLCJGIl0sWzEsMCwiRSJdLFsyLDAsIkIiXSxbMCwxLCJcXHhpIl0sWzEsMiwiXFxwaSIsMCx7InN0eWxlIjp7ImhlYWQiOnsibmFtZSI6ImVwaSJ9fX1dLFswLDIsIlxcY2hpIiwxLHsiY3VydmUiOjJ9XV0=
\[\begin{tikzcd}
	F & E & B
	\arrow["\xi", from=1-1, to=1-2]
	\arrow["\pi", two heads, from=1-2, to=1-3]
	\arrow["\rho"{description}, curve={height=12pt}, from=1-1, to=1-3]
\end{tikzcd}\]
	where $\pi$ is a cocartesian fibration. Then the composite $\rho \jdeq \pi \circ \xi$ is a cocartesian fibration, whenever $\xi$ is a partial cocartesian fibration over $\pi$.

		Call $P$ the straightening of $\pi$, $Q$ the straightening of $\xi$, and $R$ the straightening of the composite $\rho$. In particular, up to homotopy the cocartesian lift for $u:a \to b$ in $B$ \wrt~to the pair $e:P\,a$, $x:Q\,a\,e$ is given by
	\[ (P \compfam Q)_!(u,\pair{e}{x}) :\jdeq Q_!(P_!(u,e),x). \]
\end{prop}

\begin{proof}
	 Consider an arrow $u:b \to b'$ in $B$ with lifts $f : e \cocartarr^P_u e'$ and $k: x \cocartarr^Q_f x'$.

	 Let $v:b' \to b''$ in $B$ be any arrow, and $h:e \cocartarr^P_{vu} e''$, $r:x \cocartarr^Q_{h} x''$. Since $f$ is $P$-cocartesian, there is a unique filler $g:e' \cocartarr^P_v e''$ s.t.~$gf = h$. Since $r$ is $Q$-cocartesian, there is a unique filler $m:x' \cocartarr^Q_g x''$ s.t.~$mk=r$. Taken together, this implies that the pair $\pair{f}{k}$ is $R$-cocartesian.
\end{proof}

We directly\footnote{Of course, this also already follows from \cref{prop:lari-closed-under-comp}.} conclude:
\begin{cor}
The composite of two cocartesian fibration is itself a cocartesian fibration.
\end{cor}

\begin{prop}[Cocartesian families are closed under pullback]\label{prop:cocart-fam-pb}
	Let $P:B \to \UU$ be a cocartesian family, and $k:A \to B$ be a map. Then the pullback family $k^*P:A \to \UU$ with associated projection $k^*\pi$ in
	% https://q.uiver.app/?q=WzAsNCxbMCwwLCJBIFxcdGltZXNfQiBFIl0sWzAsMSwiQSJdLFsxLDEsIkIiXSxbMSwwLCJFIl0sWzAsMSwia14qXFxwaSIsMix7InN0eWxlIjp7ImhlYWQiOnsibmFtZSI6ImVwaSJ9fX1dLFsxLDIsImsiLDJdLFswLDNdLFszLDIsIlxccGkiLDAseyJzdHlsZSI6eyJoZWFkIjp7Im5hbWUiOiJlcGkifX19XSxbMCwyLCIiLDIseyJzdHlsZSI6eyJuYW1lIjoiY29ybmVyIn19XV0=
	\[\begin{tikzcd}
		{A \times_B E} & E \\
		A & B
		\arrow["{k^*\pi}"', two heads, from=1-1, to=2-1]
		\arrow["k"', from=2-1, to=2-2]
		\arrow[from=1-1, to=1-2]
		\arrow["\pi", two heads, from=1-2, to=2-2]
		\arrow["\lrcorner"{anchor=center, pos=0.125}, draw=none, from=1-1, to=2-2]
	\end{tikzcd}\]
	is a cocartesian family.

	In particular, a cocartesian lift of $u:a \to a'$ \wrt~to $e:P\,ka$ is given by
	\[ (k^*P)_!(u,e) :\jdeq P_!(ku,e).\]
\end{prop}

\begin{proof}
Stability of cocartesian families under products follows as an instance of \cref{prop:lari-maps-closed-pb}. The computation of the cocartesian lifts follows by instantiation and unpacking of the constructions given in the proofs of Propositions~\ref{prop:lari-maps-closed-pb} and~\ref{prop:lari-closed-under-pullback}, respectively.
\end{proof}

\subsubsection{Examples of cocartesian families}

We give three important kinds of examples of cocartesian fibrations: comma opfibrations (including codomain opfibrations), domain opfibrations, and free cocartesian fibrations. In a setting including universe types which are Rezk, one would naturally be interested in \emph{universal op-/fibrations}, arising as co-/domain projections associated to the categorical universes.

\begin{prop}[Comma opfibration]
Let $g:C \to A \leftarrow B:f$ be a cospan of Rezk types. Then the codomain projection from the comma object
% https://q.uiver.app/?q=WzAsNSxbMCwwLCJmIFxcZG93bmFycm93IGciXSxbMCwxLCJDIFxcdGltZXMgQiJdLFsyLDEsIkEgXFx0aW1lcyBBIl0sWzIsMCwiQV57XFxEZWx0YV4xfSJdLFswLDIsIkMiXSxbMCwxXSxbMSwyLCJnIFxcdGltZXMgZiIsMl0sWzAsM10sWzMsMiwiXFxsYW5nbGUgXFxwYXJ0aWFsXzEsIFxccGFydGlhbF8wIFxccmFuZ2xlIl0sWzAsMiwiIiwxLHsic3R5bGUiOnsibmFtZSI6ImNvcm5lciJ9fV0sWzEsNF0sWzAsNCwiXFxwYXJ0aWFsXzEiLDIseyJjdXJ2ZSI6M31dXQ==
\[\begin{tikzcd}
	{f \downarrow g} && {A^{\Delta^1}} \\
	{C \times B} && {A \times A} \\
	C
	\arrow[from=1-1, to=2-1]
	\arrow["{g \times f}"', from=2-1, to=2-3]
	\arrow[from=1-1, to=1-3]
	\arrow["{\langle \partial_1, \partial_0 \rangle}", from=1-3, to=2-3]
	\arrow["\lrcorner"{anchor=center, pos=0.125}, draw=none, from=1-1, to=2-3]
	\arrow[from=2-1, to=3-1]
	\arrow["{\partial_1}"', curve={height=22pt}, from=1-1, to=3-1]
\end{tikzcd}\]
is a cocartesian fibration.
\end{prop}

\begin{proof}
We show that, for $v:c \to c'$ in $C$, and $\alpha:fb \to gc$ in $f \downarrow g$ the candidate for the cocartesian lift is defined as the following square:
% https://q.uiver.app/?q=WzAsOCxbMCwwLCJmIFxcZG93bmFycm93IGciXSxbMCwyLCJDIl0sWzEsMiwiYyJdLFsyLDIsImMnIl0sWzEsMCwiZmIiXSxbMSwxLCJnYyJdLFsyLDEsImdjJyJdLFsyLDAsImZiIl0sWzAsMSwiXFxwYXJ0aWFsXzEiLDJdLFsyLDMsInYiLDJdLFs0LDUsIlxcYWxwaGEiLDJdLFs1LDYsImd2IiwyXSxbNCw3LCIiLDAseyJzdHlsZSI6eyJoZWFkIjp7Im5hbWUiOiJub25lIn19fV0sWzcsNiwiXFxhbHBoYSciXV0=
\[\begin{tikzcd}
	{f \downarrow g} & fb & fb \\
	& gc & {gc'} \\
	C & c & {c'}
	\arrow["{\partial_1}"', from=1-1, to=3-1]
	\arrow["v"', from=3-2, to=3-3]
	\arrow["\alpha"', from=1-2, to=2-2]
	\arrow["gv"', from=2-2, to=2-3]
	\arrow[no head, from=1-2, to=1-3, equals]
	\arrow["{\alpha' \defeq gv \circ \alpha}", from=1-3, to=2-3]
\end{tikzcd}\]
For any $v':c' \to c''$ we are to consider a dependent arrow over $v'v$ in the comma object with domain $\alpha$. This amounts to giving $u:b \to b'$ in $B$ and $\beta:fb' \to gc''$ s.t.~$\beta \circ fu = g(v'\circ v) \circ \alpha$. As indicated in the figure below, this represents the outer rectangle, and we have to uniquely exhibit a square, on the inner right, as a filler:
% https://q.uiver.app/?q=WzAsMTMsWzAsMCwiZiBcXGRvd25hcnJvdyBnIl0sWzAsMiwiQyJdLFsxLDIsImMiXSxbMiwyLCJjJyJdLFsxLDAsImZiIl0sWzEsMSwiZ2MiXSxbMiwxLCJnYyciXSxbMiwwLCJmYiJdLFszLDIsImMnJyJdLFszLDAsImZiJyJdLFszLDEsImdjJyciXSxbNCwwLCJiIl0sWzUsMCwiYiciXSxbMCwxLCJcXHBhcnRpYWxfMSIsMl0sWzIsMywidiIsMl0sWzQsNSwiXFxhbHBoYSIsMl0sWzUsNiwiZ3YiXSxbNCw3LCIiLDAseyJzdHlsZSI6eyJoZWFkIjp7Im5hbWUiOiJub25lIn19fV0sWzcsNiwiXFxhbHBoYSciXSxbMyw4LCJ2JyIsMl0sWzcsOSwiZnUiLDIseyJzdHlsZSI6eyJib2R5Ijp7Im5hbWUiOiJkYXNoZWQifX19XSxbOSwxMCwiXFxiZXRhIiwyXSxbNiwxMCwiZ3YnIiwwLHsic3R5bGUiOnsiYm9keSI6eyJuYW1lIjoiZGFzaGVkIn19fV0sWzUsMTAsImcodidcXGNpcmMgdikiLDEseyJjdXJ2ZSI6M31dLFs0LDksImZ1IiwxLHsiY3VydmUiOi0zfV0sWzExLDEyLCJ1Il1d
\[\begin{tikzcd}
	{f \downarrow g} & fb & fb & {fb'} & b & {b'} \\
	& gc & {gc'} & {gc''} \\
	C & c & {c'} & {c''}
	\arrow["{\partial_1}"', from=1-1, to=3-1]
	\arrow["v"', from=3-2, to=3-3]
	\arrow["\alpha"', from=1-2, to=2-2]
	\arrow["gv", from=2-2, to=2-3]
	\arrow[no head, from=1-2, to=1-3, equals]
	\arrow["{\alpha'}", from=1-3, to=2-3]
	\arrow["{v'}"', from=3-3, to=3-4]
	\arrow["fu"', dashed, from=1-3, to=1-4]
	\arrow["\beta"', from=1-4, to=2-4]
	\arrow["{gv'}", dashed, from=2-3, to=2-4]
	\arrow["{g(v'\circ v)}"{description}, curve={height=18pt}, from=2-2, to=2-4]
	\arrow["fu"{description}, curve={height=-18pt}, from=1-2, to=1-4]
	\arrow["u", from=1-5, to=1-6]
\end{tikzcd}\]
Clearly the horizontal arrows can be chosen as $fu$ and $gv'$, resp., by the conditions read off from the adjacent triangles. Since by assumption, there is an identity $\beta \circ fu = g(v'v) \circ \alpha$, hence $gv' \circ \alpha' = \beta \circ fu$, so the square on the right is indeed commutative, and moreover the uniquely determined filler.
\end{proof}

\begin{cor}[Codomain opfibration]\label{prop:cod-cocartfam}
	For any Rezk type $B$, the projection
	\[ \partial_1 : B^{\Delta^1} \to B, ~  \partial_1(f) :\jdeq f(1)\]
	is a cocartesian fibration, called the \emph{codomain opfibration}.
\end{cor}

Directly by \cref{prop:pushout-domainfib} we obtain that the domain projection of a Rezk type is a cocartesian fibration given that the base has pushouts.
\begin{prop}[Domain opfibration]
	If $B$ is a Rezk type that has all pushouts, then the domain projection
	\[ \partial_0 : B^{\Delta^1} \to B, \quad \partial_0(u) :\jdeq u(0)\]
	is a cocartesian fibration.
\end{prop}

\subsubsection{Towards monadicity: the free cocartesian family}

As discussed in~\cite{AFfib,GHT17,RVexp} cocartesian fibrations are monadic (and comonadic) over general functors (over a fixed base). This means that for any functor $\pi:E \fibarr B$ there is a \emph{free cocartesian fibration} $L(\pi) : L(E) \fibarr B$. Due to the current absence of categorical universes in our type theory we postpone a discussion with emphasis on a global perspective similar to the cited works. However, we can still state and prove the universal property for this construction, so that in future work, in presence of the desired universes the actual monadicity statement will easily follow.
Namely, we define a ``unit map'' $\iota_\pi \defeq \iota: \pi \to_B L(\pi)$, and prove that precomposition constitutes an equivalence of types\footnote{In general, $\CocartFun_B(\pi,\xi)$ is the $\Sigma$-type of fiberwise maps from $\pi$ to $\xi$ which preserve cocartesian lifts. Cf.~Subsection~\ref{ssec:cocart-fun} for a more thorough treatment.}
\[ -\circ \iota_P : \CocartFun_{B}(L(\pi),\xi) \stackrel{\simeq}{\longrightarrow} \Fun_B(\pi,\xi),\]
for any \emph{cocartesian} fibration $\xi:F \fibarr B$.

\begin{defn}[Free cocartesian family, {\protect{\cite[Definition 4.1]{GHT17}, \cite[Terminology 3.3.6]{AFfib}}}]\label{def:free-cocart}
	Let $B$ be a Rezk type and $P: B \to \UU$ be an isoinner family. Then the family
	\[ L(P) \defeq \lambda b. \sum_{u:\comma{B}{b}} P(\partial_0u)  : B \to \UU \]
	is the \emph{free cocartesian family associated to $P$}, or the \emph{cocartesian replacement of $P$}.
\end{defn}

In more categorical terms, the free cocartesian family---in its incarnation as a fibration---is constructed by first pulling back the map $\pi:E \fibarr B$ along the domain projection, and then postcomposing with the \emph{codomain} projection:
% https://q.uiver.app/?q=WzAsNSxbMCwwLCJMKFxccGkpIl0sWzIsMCwiRSJdLFswLDEsIkJee1xcRGVsdGFeMX0iXSxbMiwxLCJCIl0sWzAsMiwiQiJdLFswLDFdLFswLDIsIiIsMix7InN0eWxlIjp7ImhlYWQiOnsibmFtZSI6ImVwaSJ9fX1dLFsyLDMsIlxccGFydGlhbF8wIiwyXSxbMSwzLCJcXHBpIiwwLHsic3R5bGUiOnsiaGVhZCI6eyJuYW1lIjoiZXBpIn19fV0sWzIsNCwiXFxwYXJ0aWFsXzEiLDAseyJzdHlsZSI6eyJoZWFkIjp7Im5hbWUiOiJlcGkifX19XSxbMCwzLCIiLDEseyJzdHlsZSI6eyJuYW1lIjoiY29ybmVyIn19XSxbMCw0LCJcXHBhcnRpYWxfMSciLDIseyJjdXJ2ZSI6NCwic3R5bGUiOnsiaGVhZCI6eyJuYW1lIjoiZXBpIn19fV1d
\[\begin{tikzcd}
	{L(\pi)} && E \\
	{B^{\Delta^1}} && B \\
	B
	\arrow[from=1-1, to=1-3]
	\arrow[two heads, from=1-1, to=2-1]
	\arrow["{\partial_0}"', from=2-1, to=2-3]
	\arrow["\pi", two heads, from=1-3, to=2-3]
	\arrow["{\partial_1}", two heads, from=2-1, to=3-1]
	\arrow["\lrcorner"{anchor=center, pos=0.125}, draw=none, from=1-1, to=2-3]
	\arrow["{\partial_1'}"', curve={height=24pt}, two heads, from=1-1, to=3-1]
\end{tikzcd}\]
Morphisms in the cocartesian replacement can be depicted as follows:
% https://q.uiver.app/?q=WzAsOCxbMCwyLCJiIl0sWzAsMSwiYSJdLFsyLDEsImEnIl0sWzIsMiwiYiciXSxbMCwwLCJlIl0sWzIsMCwiZSciXSxbMCwzLCJiIl0sWzIsMywiYiciXSxbMSwyLCJ3Il0sWzAsMywidSJdLFsyLDMsInYnIl0sWzEsMCwidiIsMl0sWzQsNSwiaCJdLFs2LDcsInUiXSxbNCwxLCIiLDIseyJsZXZlbCI6Miwic3R5bGUiOnsiYm9keSI6eyJuYW1lIjoiZG90dGVkIn0sImhlYWQiOnsibmFtZSI6Im5vbmUifX19XSxbNSwyLCIiLDEseyJsZXZlbCI6Miwic3R5bGUiOnsiYm9keSI6eyJuYW1lIjoiZG90dGVkIn0sImhlYWQiOnsibmFtZSI6Im5vbmUifX19XV0=
\[\begin{tikzcd}
	e && {e'} \\
	a && {a'} \\
	b && {b'} \\
	b && {b'}
	\arrow["w", from=2-1, to=2-3]
	\arrow["u", from=3-1, to=3-3]
	\arrow["{v'}", from=2-3, to=3-3]
	\arrow["v"', from=2-1, to=3-1]
	\arrow["h", from=1-1, to=1-3]
	\arrow["u", from=4-1, to=4-3]
	\arrow[Rightarrow, dotted, no head, from=1-1, to=2-1]
	\arrow[Rightarrow, dotted, no head, from=1-3, to=2-3]
\end{tikzcd}\]

We will see that, indeed the free cocartesian family is a cocartesian family.
\begin{theorem}[Cocartesian replacement is cocartesian, {\protect\cite[Theorem~4.3]{GHT17}, \cite[Lemma~3.3.1]{AFfib}}]
	If $B$ is a Rezk type and $P:B \to \UU$ an isoinner family, then the family $L(P) :B \to \UU$  is cocartesian.
\end{theorem}
\begin{proof}
	By the closure properties of isoinner families, since $P$ is an isoinner family, so is $L(P)$.

	Let $u:\hom_B(b,b')$ be an arrow in $B$, and $\pair{v}{e}:LP(b)$ a point over $b$, where $v:\hom_B(a,b)$ and $e:P\,a$.

	We define the candidate lift to be $\pair{\id_a, u}{\id_e}:\dhom^{LP}_u(\pair{v}{e},\pair{vu}{e})$, \ie:
	\[
	\begin{tikzcd}
		e \ar[r, equal] & e \\
		%e \ar[r, "\id_e"] & e \\
		a \ar[r, equal]
		%a \ar[r, "\id_a"]
		\ar[d, "v" swap] & a \ar[d, "vu"] \\
		b \ar[r, "u" swap] & b'
	\end{tikzcd}
	\]
	One can readily verify that this arrow is $LP$-cocartesian.\footnote{Cf.~\cref{prop:cod-cocartfam}.} Namely, for $u':\hom_B(b',b'')$, let $t:\hom_B(a',b'')$, $e':P\,a'$, together with $w:\hom_B(a,a')$ and $f:\dhom^P_{u'u}(e,e')$ s.t.~$t \circ w = (u'u) \circ v$. We find the ensuing filler over $v$ as indicated:
	\[
	\begin{tikzcd}
		L(\pi) \ar[dddddd, two heads] &
		e \ar[r, equal]
		%e \ar[r, "\id_e"]
		\ar[rr, bend left = 40, "f"] & e \ar[r, dashed, "f" swap] & e' \\
		& & & \\
		& a \ar[r, equal]
		%& a \ar[r, "\id_a"]
		\ar[rr, bend left = 40, "w"] \ar[d, "v" swap] & a \ar[d, "vu"] \ar[r, dashed, "w"] & a' \ar[d, "t"]\\
		& b \ar[r, "u" swap]  \ar[rr, bend right = 40, "u'u" swap]  & b' \ar[r, dashed, "u'" swap] & b''
		& & & \\
		& & & \\
		& & & \\
		B & b \ar[r, "u" swap] \ar[rr, bend left = 40, "u'u"] & b' \ar[r, "u'" swap] & b'' \\
	\end{tikzcd}
	\]
 By construction, the dashed arrows are unique up to homotopy.
\end{proof}

For example, the free fibration of the identity $\id_B: B \fibarr B$ is the codomain opfibration $\partial_1: B^{\Delta^1} \fibarr B$ (cf.~\cite[Example 4.2]{GHT17}).

We now demonstrate how the cocartesian replacement is in fact a ``free'' construction.\footnote{With categorical universes at hand, one would obtain statements more closely resembling~\cite[Theorem~4.5]{GHT17}, \cite[Theorem~3.3.5]{AFfib}, and~\cite[Theorem~7.2.6]{RVexp}.} We take as the ``unit map''
% https://q.uiver.app/?q=WzAsMyxbMCwwLCJFIl0sWzEsMSwiQiJdLFsyLDAsIlxccGkgXFxkb3duYXJyb3cgQiJdLFswLDEsIlxccGkiLDIseyJzdHlsZSI6eyJoZWFkIjp7Im5hbWUiOiJlcGkifX19XSxbMCwyLCJcXGlvdGEiXSxbMiwxLCJcXHBhcnRpYWxfMSciLDAseyJzdHlsZSI6eyJoZWFkIjp7Im5hbWUiOiJlcGkifX19XV0=
\[\begin{tikzcd}
	E && {\pi \downarrow B} \\
	& B
	\arrow["\pi"', two heads, from=1-1, to=2-2]
	\arrow["\iota", from=1-1, to=1-3]
	\arrow["{\partial_1'}", two heads, from=1-3, to=2-2]
\end{tikzcd}\]
the ``inclusion''
\[ \iota\defeq \lambda b,e.\pair{\id_b}{e}:\prod_{b:B} P\,b \to (LP)\,b,\]
known from \cref{thm:cocart-fam-intl-char-fib}.

\begin{prop}[Universal property of cocartesian replacement, cf.~{\protect\cite[Theorem~4.5]{GHT17}, \cite[Corollary~3.3.4]{AFfib}}]\label{prop:univ-prop-cocart-repl}
For a Rezk type $B$, consider an isoinner fibration $\pi:E \to B$, and a cocartesian fibration $\xi:F \to B$. Then the map
\[
  \begin{tikzcd}[column sep=large]
  \CocartFun_{B}(L(\pi),\xi) \ar[r,"{\Phi \;{\defeq}\; {{-}\circ \iota_P}}"] & \Fun_B(\pi,\xi)
  \end{tikzcd}
\]
is an equivalence of types.
\end{prop}

\begin{proof}
	Denote by $P, Q:B \to \UU$ the straightenings of $\pi: E \to B$ and $\xi:F \to B$, resp.
	
	We aim to give a quasi-inverse of the precomposition map. Let
	\[ \Psi \defeq \lambda \varphi.\varphi': \Fun_B(\pi, \xi) \to \CocartFun_{B}(L(\pi),\xi)\]
	where
	\[ \varphi'_b(v,e) \defeq v_!^Q(\varphi_a \,e),\]
	for $v:a \to b$, $e:P\,a$.
	First, we are to show that this operation is really valued in cocartesian functors.
	For this, we have to show that, for any $v:a \to b$, $e:P\,a$, the arrow
	\[ \lambda t.(u(t) \circ v)_!^Q(\varphi_b \,e) : v_!Q(\varphi_b\,e) \longrightarrow^P_u (uv)_!^Q(\varphi_b\,a)\]
	is $Q$-cocartesian. To that end, we observe the following. Let $a:B$, $d:Q\,a$ be fixed. Consider the maps $\cst(d), \tau_Q(-,d): \comma{a}{B} \to E$ defined by
	\[ \cst(d)(v:a \to b) \defeq  d, \quad \tau_Q(-,d)(v:a \to b) \defeq v^Q_!(d):d . \]
	We define the natural transformation
	\[ Q_!(-,d):\nat{\comma{\pi}{B}}{E}(\cst(d),\tau_Q(-,d)), \quad Q_!(-,d)(v:a \to b) \defeq Q_!(v,d) : d \cocartarr_v v^Q_!(d).\]
	Morphisms in $\comma{a}{B}$ are given by commutative triangles $u:v \to w$, so for fixed $v$ the type of morphisms in $\comma{a}{B}$ starting at $v$ is equivalent to the type $\comma{(\partial_1 \, v)}{B}$. Hence, any morphism in $\comma{a}{B}$ can be taken to be of the form $u:v \to uv$, for $v:a \to b$, $u:b \to b'$. The naturality squares of $Q_!(-,d)$ thus are of the following form:
	% https://q.uiver.app/?q=WzAsOCxbMCwyLCJhIl0sWzIsMiwiXFxidWxsZXQiXSxbMCwzLCJiIl0sWzIsMywiYiciXSxbMCwwLCJkIl0sWzAsMSwiUV8qKHYsZCkiXSxbMiwxLCJRXyoodXYsZCkiXSxbMiwwLCJkIl0sWzAsMSwiIiwwLHsibGV2ZWwiOjIsInN0eWxlIjp7ImhlYWQiOnsibmFtZSI6Im5vbmUifX19XSxbMCwyLCJ1IiwyXSxbMiwzLCJ2IiwyXSxbMSwzLCJ2dSJdLFs0LDVdLFs1LDYsIlFfKih1OnYgXFx0byB1dixkKSIsMl0sWzQsNywiIiwyLHsibGV2ZWwiOjIsInN0eWxlIjp7ImhlYWQiOnsibmFtZSI6Im5vbmUifX19XSxbNyw2XV0=
	\[\begin{tikzcd}
		d && d \\
		{v_!^Q(d)} && {(uv)_!^Q(d)} \\
		a && a \\
		b && {b'}
		\arrow[Rightarrow, no head, from=3-1, to=3-3]
		\arrow["u"', from=3-1, to=4-1]
		\arrow["v"', from=4-1, to=4-3]
		\arrow["vu", from=3-3, to=4-3]
		\arrow[from=1-1, to=2-1, cocart]
		\arrow["{(u:v \to uv)^Q_!(d)}"', from=2-1, to=2-3]
		\arrow[Rightarrow, no head, from=1-1, to=1-3]
		\arrow[from=1-3, to=2-3, cocart]
	\end{tikzcd}\]
Note that the lower vertical arrow is given by
\[ \lambda t.(u(t) \circ v)^Q_!(d) = (u:v \to uv)^Q_!(d).\]
By right cancelation, $(u:v \to uv)^Q_!(d)$ is cocartesian, and hence we have an identity of arrows:
% https://q.uiver.app/?q=WzAsMyxbMCwxLCJ2XlFfKihkKSJdLFsyLDEsInVeUV8qKHZfKl5RKGQpKSJdLFsyLDAsIih1dilfKl5RKGQpIl0sWzAsMV0sWzAsMiwiKHU6diBcXHRvIHV2KV5RXyooZCkiXSxbMiwxLCIiLDIseyJsZXZlbCI6Miwic3R5bGUiOnsiaGVhZCI6eyJuYW1lIjoibm9uZSJ9fX1dXQ==
\[\begin{tikzcd}
	&& {(uv)_!^Q(d)} \\
	{v^Q_!(d)} && {u^Q_!(v_!^Q(d))}
	\arrow["{Q_!(u,v_!^Q(d))}", from=2-1, to=2-3, cocart, swap]
	\arrow["{(u:v \to uv)^Q_!(d)}", from=2-1, to=1-3, cocart]
	\arrow[Rightarrow, no head, from=1-3, to=2-3]
\end{tikzcd}\]
In the cocartesian replacement $\partial_1': \comma{\pi}{B} \fibarr B$, the cocartesian lift of $u:b \to b'$ w.r.t.~$\angled{v:a \to b,e:P\,a}$ is given by $\angled{\id_a,u,\id_e}$. Now, by the previous discussion we have
\begin{align*}
	\varphi'_u(\id_a,u,\id_e)  & = \lambda t.(u(t) \circ v)_!^Q(\varphi_a\,e) \\
	& =  (u:v \to uv)^Q_!(\varphi_a\,e) \\
	& = Q_!(u,v^Q_!(\varphi_a\,e)) \\
	& = Q_!(u, \varphi'_a(v,e))
\end{align*}
which shows that $\varphi': L(\pi) \to_B \xi$ is a cocartesian functor, as desired.

We now turn to showing that precomposing with $\iota$ gives an equivalence
\[ \Fun_B(\pi, \xi) \simeq \CocartFun_{B}(L(\pi),\xi).\]
We define
\[ \Phi \defeq \lambda \psi. \iota^*\psi \defeq \lambda \psi.\psi \circ \iota : \Fun_B(\pi, \xi) \to \CocartFun_{B}(L(\pi),\xi),\]
and recall that in the converse direction
\[ \Psi \defeq \lambda \varphi. \varphi' : \CocartFun_{B}(L(\pi),\xi) \to  \Fun_B(\pi, \xi)\]
with $\varphi'_b(v,e) \defeq v_!^Q(\varphi_a\,e)$.
Let $\psi: L\,\pi \to_B F$ be a cocartesian functor.
We compute
\[ (\iota^*\psi)_b'(v,e) = v_!^Q(\iota^*\psi_a(e)) = v_!^Q(\psi_a(\id_a,e)).\]
Since $\psi$ is cocartesian, we have $v_!^Q(\psi_a(\id_a,e)) = \psi_b(v_!^{LP}(\id_a,e))$, cf.~\cref{cor:nat-cocartlift-pt}. Now, the $LP$-cocartesian lift of $v:a \to b$ w.r.t.~$\pair{\id_a}{e}$ is given by $\angled{\id_a, v,\id_e}: \pair{\id_a}{e} \to \pair{v}{e}$:
% https://q.uiver.app/?q=WzAsOCxbMCwwLCJlIl0sWzIsMCwiZSJdLFswLDEsImEiXSxbMCwyLCJhIl0sWzIsMiwiYiJdLFsyLDEsImEiXSxbMCwzLCJhIl0sWzIsMywiYiJdLFswLDEsIiIsMCx7ImxldmVsIjoyLCJzdHlsZSI6eyJoZWFkIjp7Im5hbWUiOiJub25lIn19fV0sWzIsMywiIiwwLHsibGV2ZWwiOjIsInN0eWxlIjp7ImhlYWQiOnsibmFtZSI6Im5vbmUifX19XSxbMyw0LCJ2IiwyXSxbMiw1LCIiLDIseyJsZXZlbCI6Miwic3R5bGUiOnsiaGVhZCI6eyJuYW1lIjoibm9uZSJ9fX1dLFs1LDQsInYiXSxbNiw3LCJ2Il1d
\[\begin{tikzcd}
	e && e \\
	a && a \\
	a && b \\
	a && b
	\arrow[Rightarrow, no head, from=1-1, to=1-3]
	\arrow[Rightarrow, no head, from=2-1, to=3-1]
	\arrow["v"', from=3-1, to=3-3]
	\arrow[Rightarrow, no head, from=2-1, to=2-3]
	\arrow["v", from=2-3, to=3-3]
	\arrow["v", from=4-1, to=4-3]
\end{tikzcd}\]
As a dependent arrow in $LP$, the codomain of this morphism is the pair $\pair{v}{e}$. In sum, this means
\[ v_!^Q(\psi_a(\id_a,e)) = \psi_b(v_!^{LP}(\id_a,e)) = \psi_b(v,e) = (\iota^*\psi)_b'(v,e),\]
\ie,~$\Psi(\Phi(\psi))$.
On the other hand, for an arbitrary fiberwise map $\varphi$ from $P$ to $Q$, we find that
\[ (\iota^* \varphi)'(b,e)= \varphi_b'(\id_b,e) = (\id_b)_!^Q(\varphi_b\,e) = \varphi_b(e)\]
since cocartesian lifts of identities are themselves identities. This gives $\Phi(\Psi(\varphi))= \varphi$.
\end{proof}

\begin{rem}
	In fact, there also exists a discrete version of co-/cartesian replacement, cf.~e.g.~\cite[Section 4.3]{AFfib}. In type theory, this construction demands some more involvement, since it will include fiberwise localization.
\end{rem} 

\subsection{Cocartesian functors}\label{ssec:cocart-fun}

Cocartesian functors are an approproiate notion of morphism between cocartesian fibrations (not necessarily over the same base). They are defined as fiberwise maps that preserve cocartesian morphisms. We discuss several closure properties of cocartesian functors coming from $\infty$-cosmos theory. We also transfer the alternative characterizations of cocartesian functors from~\cite[Theorem 5.3.4]{RV}, reworking the necessary tools from formal category theory in \cref{app:ssec:pasting-lax,app:ssec:mates}.

\subsubsection{Morphisms of sections}\label{ssec:mor-of-sections}

We commence this section with a brief discussion of morphisms between sections. This willl turn out to be particularly important in the context of the Yoneda Lemma in \cref{sec:yoneda}.

First, let $B$ be any type and $P: B \to \UU$ any family. For any morphism $u:\hom_B(a,b)$, $a,b:B$, there is an induced dependent morphism
\[ \sigma\, u: \dhom^P_u(\sigma \,a, \sigma \, b), \quad \sigma \,u: \jdeq \lambda t.\sigma(u(t)). \]

Consider sections $\sigma,\tau :\prod_B P$ and a morphism $\kappa : \hom_{\prod_B P}(\sigma, \tau)$. Again, we imagine $\kappa$ as a kind of $2$-cell, and abbreviate $\hom_{\prod_B P}(\sigma, \tau) \defeq (\sigma \Rightarrow^P \tau) = (\sigma \Rightarrow \tau)$.

For any $x:B$, $\kappa$ induces a vertical morphism
\[ \kappa \, x :\hom_{P\,x}(\sigma \, x, \tau \, x), \quad \kappa \, x : \jdeq \lambda t.\kappa(t,x). \]

By the axiom of choice, we have
\[ (\Delta^1 \to \prod_B P ) \equiv \sum_{u:\Delta^1 \to B} (\Delta^1 \to P(u)). \]
Let $B$ is a Rezk type and $P: B \to \UU$ an isoinner family. Then $\kappa$ acts on arrows in the base $B$ in the sense that in $\prod_B P$ there are canonical squares
\[
\begin{tikzcd}
	\sigma a \ar[d, "\kappa a" swap] \ar[r, "\sigma u"] & \sigma b \ar[d, "\kappa b"] \\
	\tau a \ar[r, "\tau u" swap] & \tau b \\
	a \ar[r, "u"] & b
\end{tikzcd}
\]
and these squares compose over composable morphisms.\footnote{But we will not discuss the validity of the Segal condition for $\Pi$-types here.}

\subsubsection{Definition and basic properties}\label{ssec:fib-fun-defn}

\begin{defn}[Fiberwise maps]\label{def:fib-maps}
	Let $P: B \to \UU$ and $Q: C \to \UU$ be families. A \emph{fiberwise map} or \emph{fibered functor from $P$ to $Q$}  is a pair of functions $\Phi \jdeq \pair{j}{\varphi}$, where
	\begin{itemize}
		\item $j: B \to C$,
		\item $\varphi : \prod_{b:B} (P\,b \to Q\,j(b))$.
	\end{itemize}
	We call $\Phi$ a \emph{fibered equivalence} if both $j$ and $\varphi$ are equivalences.

	 We write $\FibFun_{B,C}(P,Q) \jdeq (P \xbigtoto[]{}  Q) \jdeq (\pi_P \xbigtoto[]{} \pi_Q)$ for the ensuing type of fiberwise maps. In the case of $B \jdeq C$, we denote this type as $\FibFun_{B}(P,Q) = (P \to_B Q) = (\pi_P \to_B \pi_Q)$.
\end{defn}
Note that (by fibrant replacement) the type of commutative squares is equivalent to the type of maps between families.

Observe that given a map between families $P$ and $Q$ as above we get a strictly commutative square:
% https://q.uiver.app/?q=WzAsNCxbMCwwLCJcXHdpZGV0aWxkZXtQfSJdLFswLDEsIkIiXSxbMiwxLCJDIl0sWzIsMCwiXFx3aWRldGlsZGV7UX0iXSxbMCwxLCJcXHBpX1AiLDJdLFsxLDIsImoiLDJdLFswLDMsIlxcbWF0aHJte3RvdGFsfShcXHZhcnBoaSkiXSxbMywyLCJcXHBpX1EiXV0=
\[\begin{tikzcd}
	{\widetilde{P}} && {\widetilde{Q}} \\
	B && C
	\arrow["{\pi_P}"', from=1-1, to=2-1]
	\arrow["j"', from=2-1, to=2-3]
	\arrow["{\mathrm{total}(\varphi)}", from=1-1, to=1-3]
	\arrow["{\pi_Q}", from=1-3, to=2-3]
\end{tikzcd}\]

For any $u:\hom_B(a,b)$, the fiberwise map $\Phi$ acts on arrows over $u$ in the following way. For $f:\dhom^P_u(d,e)$, $d:Pa$, $e:Pb$, we have
\[ \varphi_u(f):\jdeq \lambda t.\varphi_{u(t)}(f(t)) : \dhom^Q_{ju}(\varphi_a(d), \varphi_b(e)).\]

\begin{defn}[Cocartesian functors]\label{def:cocart-fun}
	Let $P:B \to \UU$ and $Q:B \to \UU$ be cocartesian families over Rezk types. Given a fibered functor $\Phi \defeq \pair{j}{\varphi}$ from $P$ to $Q$, if
	\[ \Phi = \total(\varphi): \totalty{P} \to \totalty{Q}, \quad \Phi_b(e) :\jdeq \pair{j(b)}{\varphi_{j\,b}(e) } \]
	preserves cocartesian arrows, then we call $\Phi$ a \emph{cocartesian functor}:\footnote{This is a proposition because being a cocartesian arrow is a proposition in our usual setting.}
	\[ \isCocartFun_{P,Q}(\Phi) :\jdeq \prod_{\pair{u}{f}:\Delta^1\to \widetilde{P}} \isCocartArr_u^P(f) \to \isCocartArr_{ju}^Q(\varphi f).\]
	We define
	\[ \CocartFun_{B,C}(P,Q) :\jdeq \sum_{\Phi:\FibFun_{B,C}(P,Q)} \isCocartFun_{P,Q}(\Phi).\]

	For cocartesian families $P: B \to \UU$ to $Q: B \to \UU$ over a common base, we have
	\[ \isCocartFun_{P,Q}(\varphi) \equiv \prod_{\substack{u:\Delta^1 \to B \\ f: (t:\Delta^1) \to P(u\,t)}} \isCocartArr^P_u(f) \to \isCocartArr^Q_u(\varphi f)\]
	and write
	\[  \CocartFun_{B}(P,Q) \simeq  \sum_{\varphi:P \to_B Q} \isCocartFun_{P,Q}(\varphi).\]
\end{defn}

\begin{defn}[Cocartesian sections]\label{def:cocart-sec}
	If the unstraightening of $P$ is an identity, $\Phi$ can be identified with its second component, which in turn is the same as a section of $Q$. In this case, if $\Phi$ is a cocartesian functor it is called a \emph{cocartesian section}. For a section $\sigma:\prod_B P$, this condition amounts to
	\[  \isCocartSect_P(\sigma) :\jdeq \prod_{\substack{u:\Delta^1 \to B}} \isCocartArr_{u}^P(\sigma \circ u). \]
	We denote by
	\[ \prod_B^{\mathrm{cocart}} P :\jdeq \sum_{\sigma:\prod_B P} \isCocartSect_P(\sigma)  \]
	the subtype of cocartesian sections.
\end{defn}

	\begin{figure}
		\centering
		% https://q.uiver.app/?q=WzAsMTQsWzAsMCwiKFxcRGVsdGFeMSlee1xcRGVsdGFeMX0iXSxbMCwyLCJcXERlbHRhXjEiXSxbMSwyLCIwIl0sWzMsMiwiMSJdLFs1LDAsIlxcRGVsdGFeMiJdLFs1LDIsIlxcRGVsdGFeMSJdLFs2LDIsIjAiXSxbOCwyLCIxIl0sWzEsMSwiZl8wIl0sWzEsMCwiZyJdLFszLDAsImZfMSJdLFs2LDEsIlxcbGFuZ2xlIDAsMCBcXHJhbmdsZSJdLFs2LDAsIlxcbGFuZ2xlIDAsMSBcXHJhbmdsZSJdLFs4LDAsIlxcbGFuZ2xlIDEsMSBcXHJhbmdsZSJdLFswLDEsIlxccGFydGlhbF8wIiwyLHsic3R5bGUiOnsiaGVhZCI6eyJuYW1lIjoiZXBpIn19fV0sWzIsM10sWzQsNSwiXFxwcl8xIiwyLHsic3R5bGUiOnsiaGVhZCI6eyJuYW1lIjoiZXBpIn19fV0sWzYsN10sWzgsOSwiXFxzaWdtYSJdLFs5LDEwLCJcXHRhdSJdLFs4LDEwLCJcXHRhdVxcY2lyYyBcXHNpZ21hIiwyXSxbMTEsMTJdLFsxMSwxM10sWzEyLDEzXV0=
		\[\begin{tikzcd}
			{(\Delta^1)^{\Delta^1}} & g && {f_1} && {\Delta^2} & {\langle 0,1 \rangle} && {\langle 1,1 \rangle} \\
			& {f_0} &&&&& {\langle 0,0 \rangle} \\
			{\Delta^1} & 0 && 1 && {\Delta^1} & 0 && 1
			\arrow["{\partial_0}"', two heads, from=1-1, to=3-1]
			\arrow[from=3-2, to=3-4]
			\arrow["{\pr_1}"', two heads, from=1-6, to=3-6]
			\arrow[from=3-7, to=3-9]
			\arrow["\sigma", from=2-2, to=1-2]
			\arrow["\tau", from=1-2, to=1-4]
			\arrow["{\tau\circ \sigma}"', from=2-2, to=1-4]
			\arrow[from=2-7, to=1-7]
			\arrow[from=2-7, to=1-9]
			\arrow[from=1-7, to=1-9]
		\end{tikzcd}\]
	\caption{The two equivalent fibrations $\partial_0:(\Delta^1)^{\Delta^1} \to \Delta^1$ and $\pr_1: \Delta^2 \to \Delta^1$}\label{fig:counterex-cocartsec}
	\end{figure}
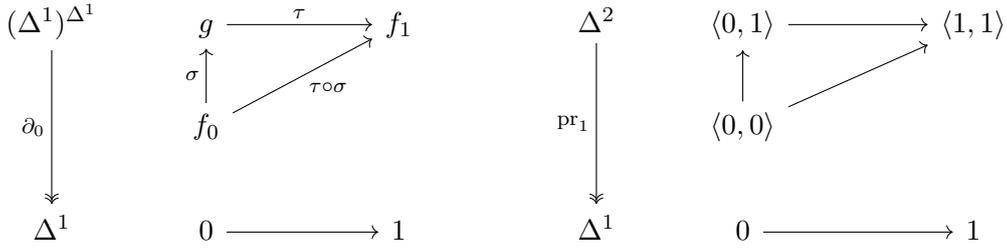

In general, cocartesian families can have non-cocartesian sections. This implies that fiberwise maps between cocartesian families are not automatically cocartesian. As an example, consider the cocartesian fibration
\[ \partial_0 : (\Delta^1)^{\Delta^1} \to \Delta^1, \quad \partial_0 \defeq \lambda f.f(0)\]
which is fibered equivalent to:
\[ \pr_1: \Delta^2 \to \Delta^1, \quad \pr_1 \defeq \lambda \pair{t}{s}.t \]
 The elements of $(\Delta^1)^{\Delta^1}$ are given by
\[ f_k \defeq \lambda t.k, \quad g \defeq \lambda t.t,\]
for $k=0,1$, and the non-identity morphisms are the squares $\sigma$, $\tau$ and $\tau \circ \sigma$, defined as follows:
% https://q.uiver.app/?q=WzAsNixbMCwwLCIwIl0sWzAsMiwiMCJdLFsyLDIsIjAiXSxbMiwwLCIxIl0sWzQsMiwiMSJdLFs0LDAsIjEiXSxbMCwxLCJmXzAiLDEseyJsZXZlbCI6Miwic3R5bGUiOnsiaGVhZCI6eyJuYW1lIjoibm9uZSJ9fX1dLFsxLDIsIiIsMSx7ImxldmVsIjoyLCJzdHlsZSI6eyJoZWFkIjp7Im5hbWUiOiJub25lIn19fV0sWzAsM10sWzIsMywiZyIsMV0sWzIsNF0sWzMsNSwiIiwxLHsibGV2ZWwiOjIsInN0eWxlIjp7ImhlYWQiOnsibmFtZSI6Im5vbmUifX19XSxbNSw0LCJmXzEiLDEseyJsZXZlbCI6Miwic3R5bGUiOnsiaGVhZCI6eyJuYW1lIjoibm9uZSJ9fX1dLFs2LDksIlxcc2lnbWEiLDEseyJzaG9ydGVuIjp7InNvdXJjZSI6MjAsInRhcmdldCI6MjB9fV0sWzksMTIsIlxcdGF1IiwxLHsic2hvcnRlbiI6eyJzb3VyY2UiOjIwLCJ0YXJnZXQiOjIwfX1dXQ==
\[\begin{tikzcd}
	0 && 1 && 1 \\
	\\
	0 && 0 && 1
	\arrow[""{name=0, anchor=center, inner sep=0}, "{f_0}"{description}, Rightarrow, no head, from=1-1, to=3-1]
	\arrow[Rightarrow, no head, from=3-1, to=3-3]
	\arrow[from=1-1, to=1-3]
	\arrow[""{name=1, anchor=center, inner sep=0}, "g"{description}, from=3-3, to=1-3]
	\arrow[from=3-3, to=3-5]
	\arrow[Rightarrow, no head, from=1-3, to=1-5]
	\arrow[""{name=2, anchor=center, inner sep=0}, "{f_1}"{description}, Rightarrow, no head, from=1-5, to=3-5]
	\arrow["\sigma"{description}, shorten <=13pt, shorten >=13pt, Rightarrow, from=0, to=1]
	\arrow["\tau"{description}, shorten <=13pt, shorten >=13pt, Rightarrow, from=1, to=2]
\end{tikzcd}\]
Cf.~\Cref{fig:counterex-cocartsec} for a visualization of the two equivalent fibrations $\partial_0:(\Delta^1)^{\Delta^1} \fibarr \Delta^1$ and $\pr_1:\Delta^2 \fibarr \Delta^1$ (for the latter, using the choice of coordinates and embeddings from~\cite[Subsection~2.3]{RS17}, \cf~\Cref{ssec:syn-higher-cats}).

The inclusion of $\Delta^1$ into $\Delta^2$ as the ``long edge'' $\pair{0}{0} \to \pair{1}{1}$ (corresponding to $\tau \circ \sigma$ in $(\Delta^1)^{\Delta^1}$) is a non-cocartesian section of $\pi$ since there is no map $\pair{1}{1} \to \pair{0}{1}$.

\begin{prop}[Naturality of cocartesian liftings]\label{prop:nat-cocartlift-arr}
Let $B$ be a Rezk type, $P:B \to \UU$, $Q:C \to \UU$ cocartesian families, and $\Phi \jdeq \pair{j}{\varphi} : \CocartFun_{B,C}(P,Q)$ a cocartesian functor. Then $\Phi$ commutes with cocartesian lifts, \ie,~for any $u:\hom_B(a,b)$ there is an identification of arrows
\[ \varphi \big(P_!(u,d)\big) =_{\Delta^1 \to (ju)^*Q} Q_!(ju,\varphi_ad) \]
and hence of endpoints
\[ \varphi_b(u_!^Pd) =_{Q(jb)} (ju)_!^Q(\varphi_ad). \]
In particular there is a homotopy commutative square:
	\[
\begin{tikzcd}
	Pa \ar[rr, "\varphi_a"] \ar[d, "\coliftptfammap{P}{u}" swap] & & Qa \ar[d, "\coliftptfammap{Q}{(ju)}"] \\
	Pb \ar[rr, "\varphi_b" swap] && Qb
\end{tikzcd}
\]
\end{prop}

\begin{proof}
For $u:\hom_B(a,b)$ and $d:P\,a$, consider the $P$-cocartesian lift $\coliftarr{P}{u}{d}:\dhom^P_u(d,\coliftptfam{P}{u}{d})$. Since $\varphi$ is a cocartesian functor the arrow $\varphi_u(\coliftarr{P}{u}{d}):\dhom^Q_{ju}(\varphi_ad,\varphi_b(\coliftptfam{P}{u}{d}))$ is $Q$-cocartesian.

On the other hand, $\coliftarr{Q}{ju}{\varphi_ad}:\dhom^Q_{ju}(\varphi_ad, \coliftptfam{Q}{ju}{\varphi_ad})$ is as well a $Q$-cocartesian lift of $ju$ with domain $\varphi_ad$, thus coincides with $\varphi_u(\coliftarr{P}{u}{d})$ up to a path, in particular this gives an identification $\varphi_b(\coliftptfam{P}{u}{d}) = \coliftptfam{Q}{ju}{\varphi_ad}$.
\end{proof}

\begin{cor}[Naturality over a common base ({\protect\cite[Exercise 5.3.iii]{RV}}; discrete case in sHoTT: {\protect\cite[Proposition 8.17]{RS17}})]\label{cor:nat-cocartlift-pt}
Consider a Rezk type $B$, cocartesian families $P,Q: B \to \UU$, and a cocartesian functor $\varphi:\CocartFun_{B}(P,Q)$. Then $\varphi$ commutes with the actions of arrows, \ie,~for any $a,b:B$, $u:\hom_B(a,b)$, $d:P(a)$, we get an
identification
	\[ \varphi_b(\coliftptfam{P}{u}{d}) =_{Qb} \coliftptfam{Q}{u}{\varphi_a(d)},\]
	thus a homotopy commutative square:
	\[
	\begin{tikzcd}
		Pa \ar[rr, "\varphi_a"] \ar[d, "\coliftptfammap{P}{u}" swap] & & Qa \ar[d, "\coliftptfammap{Q}{u}"] \\
		Pb \ar[rr, "\varphi_b" swap] && Qb
	\end{tikzcd}
	\]
\end{cor}

\subsubsection{Closure properties of cocartesian functors}

The closure properties stated in this section are to be understood w.r.t.~the (non-full) sub-$\inftyone$-category of the $\inftyone$-category of arrows which has as objects cocartesian fibrations $E\tofib B$ and as arrows cocartesian functors between those. In the absence of categorical universes, we capture the proclaimed limits and cotensors by spelling out type-theoretically their universal properties. 

We start by explaining horizontal and vertical composition of cocartesian functors, alluding to the fact that in the model cocartesian fibrations form an $\inftyone$-double category.

In the following, for a fibered functor $\varphi$, we often write $\widetilde{\varphi}$ for its totalization.
\begin{prop}[Horizontal and vertical composition of cocartesian functors, {\protect\cite[Exercise~5.3.ii]{RV}}]\label{prop:comp-cocart-fun}
~\\
	\begin{enumerate}
		\item Cocartesian functors compose horizontally: Suppose, we are given cocartesian families $P: A \to \UU$, $Q: B \to \UU$, $R: C \to \UU$ over Rezk types $A,B,C$. If $\Phi \jdeq \pair{j}{\varphi}:\CocartFun_{A,B}(P,Q)$, $\Psi \jdeq \pair{k}{\psi}:\CocartFun_{B,C}(Q,R)$ are cocartesian functors, then the \emph{horizontal composite} $\Psi \hcomp \Phi :\jdeq \pair{k \circ j}{\psi \circ \varphi}$ defines a cocartesian functor from $P$ to $R$:
	% https://q.uiver.app/?q=WzAsNixbMCwwLCJcXHdpZGV0aWxkZXtQfSJdLFsxLDAsIlxcd2lkZXRpbGRle0Z9Il0sWzEsMSwiQiJdLFsyLDAsIlxcd2lkZXRpbGRle0d9Il0sWzIsMSwiQyJdLFswLDEsIkEiXSxbMCwxLCJcXHdpZGV0aWxkZXtcXHZhcnBoaX0iXSxbMSwyLCJcXHhpIiwwLHsic3R5bGUiOnsiaGVhZCI6eyJuYW1lIjoiZXBpIn19fV0sWzEsMywiXFx3aWRldGlsZGV7XFxwc2l9Il0sWzMsNCwiXFxjaGkiLDAseyJzdHlsZSI6eyJoZWFkIjp7Im5hbWUiOiJlcGkifX19XSxbMCw1LCJcXHBpIiwyLHsic3R5bGUiOnsiaGVhZCI6eyJuYW1lIjoiZXBpIn19fV0sWzUsMiwiaiIsMl0sWzIsNCwiayIsMl0sWzAsMywiXFx3aWRldGlsZGV7XFxwc2l9IFxcY2lyYyBcXHdpZGV0aWxkZXtcXHZhcnBoaX0iLDAseyJjdXJ2ZSI6LTR9XSxbNSw0LCJrIFxcY2lyYyBqIiwyLHsiY3VydmUiOjR9XV0=
	\[\begin{tikzcd}
		{\widetilde{P}} & {\widetilde{F}} & {\widetilde{G}} \\
		{A} & {B} & {C}
		\arrow["{\widetilde{\varphi}}", from=1-1, to=1-2]
		\arrow["{\xi}", from=1-2, to=2-2, two heads]
		\arrow["{\widetilde{\psi}}", from=1-2, to=1-3]
		\arrow["{\chi}", from=1-3, to=2-3, two heads]
		\arrow["{\pi}"', from=1-1, to=2-1, two heads]
		\arrow["{j}"', from=2-1, to=2-2]
		\arrow["{k}"', from=2-2, to=2-3]
		\arrow["{\widetilde{\psi} \circ \widetilde{\varphi}}", from=1-1, to=1-3, curve={height=-24pt}]
		\arrow["{k \circ j}"', from=2-1, to=2-3, curve={height=24pt}]
	\end{tikzcd}\]
		\item Cocartesian functors compose vertically: Suppose, we are given cocartesian families $P: B \to \UU$, $P': \widetilde{P} \to \UU$, $Q: A \to \UU$, $Q': \widetilde{Q} \to \UU$, over Rezk types $A$ and $B$. If $\Phi \jdeq \pair{j}{\varphi}:\CocartFun_{B,A}(P,Q)$, $\Psi \jdeq \pair{\widetilde{\varphi}}{\psi}:\CocartFun_{B,A}(P',Q')$ are cocartesian functors, then the \emph{vertical composite} $\Phi \vcomp \Psi :\jdeq \pair{j}{\psi}$ defines a cocartesian functor from $R\defeq \Sigma_{P}Q$ to $R':\jdeq \Sigma_{P'}Q'$:\footnote{We have $\Sigma_P Q \jdeq Q \compfam P$.}
	% https://q.uiver.app/?q=WzAsNixbMCwwLCJcXHdpZGV0aWxkZXtRfSJdLFsxLDAsIlxcd2lkZXRpbGRle1EnfSJdLFswLDEsIlxcd2lkZXRpbGRle1B9Il0sWzEsMSwiXFx3aWRldGlsZGV7UCd9Il0sWzAsMiwiQiJdLFsxLDIsIkEiXSxbMCwxLCJcXHdpZGV0aWxkZXtcXHBzaX0iXSxbMCwyLCJcXHhpIiwyLHsic3R5bGUiOnsiaGVhZCI6eyJuYW1lIjoiZXBpIn19fV0sWzEsMywiXFx4aSciLDAseyJzdHlsZSI6eyJoZWFkIjp7Im5hbWUiOiJlcGkifX19XSxbMiwzLCJcXHdpZGV0aWxkZXtcXHZhcnBoaX0iLDJdLFsyLDQsIlxccGkiLDIseyJzdHlsZSI6eyJoZWFkIjp7Im5hbWUiOiJlcGkifX19XSxbNCw1LCJqIiwyXSxbMyw1LCJcXHBpJyIsMCx7InN0eWxlIjp7ImhlYWQiOnsibmFtZSI6ImVwaSJ9fX1dLFswLDQsIlxccGkgXFxjaXJjIFxceGkiLDEseyJjdXJ2ZSI6NCwic3R5bGUiOnsiaGVhZCI6eyJuYW1lIjoiZXBpIn19fV0sWzEsNSwiXFxwaScgXFxjaXJjIFxceGknIiwxLHsiY3VydmUiOi00LCJzdHlsZSI6eyJoZWFkIjp7Im5hbWUiOiJlcGkifX19XV0=
	\[\begin{tikzcd}
		{\widetilde{Q}} & {\widetilde{Q'}} \\
		{\widetilde{P}} & {\widetilde{P'}} \\
		{B} & {A}
		\arrow["{\widetilde{\psi}}", from=1-1, to=1-2]
		\arrow["{\xi}"', from=1-1, to=2-1, two heads]
		\arrow["{\xi'}", from=1-2, to=2-2, two heads]
		\arrow["{\widetilde{\varphi}}"', from=2-1, to=2-2]
		\arrow["{\pi}"', from=2-1, to=3-1, two heads]
		\arrow["{j}"', from=3-1, to=3-2]
		\arrow["{\pi'}", from=2-2, to=3-2, two heads]
		\arrow["{\pi \circ \xi}" description, from=1-1, to=3-1, curve={height=24pt}, two heads]
		\arrow["{\pi' \circ \xi'}" description, from=1-2, to=3-2, curve={height=-24pt}, two heads]
	\end{tikzcd}\]
	\end{enumerate}
\end{prop}

\begin{proof}
	Recall \cref{ssec:cocart-clos} for the closure properties of cocartesian fibrations, and how to compute cocartesian lifts in the respective constructions.
	\begin{enumerate}
		\item Fiberwise composition of the fiberwise maps $\varphi:\prod_{a:A} P\,a \to Q\,ja$, $\psi:\prod_{b:B} Q\,b \to R\,kb$ is given by
		\begin{align}
			\psi \circ \varphi :\jdeq \lambda a.\psi_{ja} \circ \varphi_a:\prod_{a:A} P\,a \to R\,kja \label{eq:fibw-comp}.
		\end{align}
	Since $\varphi$ and $\psi$ both are cocartesian functors, for arrows $u:a \to a'$ in $A$, $v:b \to b'$ in $B$, and points $e:P\,a$, $d:Q\,b$, there are paths:
	\begin{align}
		\varphi_u(P_!(u,e)) & = Q_!(ju, \varphi_a e) \label{eq:phi-hor-cocart} \\
		\psi_v(Q_!(v,d)) & = R_!(kv,\psi_b d) \label{eq:psi-hor-cocart}
	\end{align}
Using these identifications, we find
\begin{align*}
 	(\psi \circ \varphi)_u(P_!(u,e)) & \stackrel{\ref{eq:fibw-comp}}{=}  \psi_{ju}(\varphi_u(P_!(u,e))) \stackrel{\ref{eq:phi-hor-cocart}}{=} \psi_{ju}(Q_!(ju,\varphi_a e)) \\
 	\stackrel{\ref{eq:psi-hor-cocart}}{=}  R_!(kju,\psi_{kja}(\varphi_a e)) &  \stackrel{\ref{eq:fibw-comp}}{=} R_!(kju, (\psi \circ \varphi)(a)).
\end{align*}
	\item Let $u:b \to b'$ in $B$, $u':a \to a'$ in $A$, be arrows with points $e:P\,b$, $\pair{e}{d}:Q(b,e)$, and $e':P\,b'$. First, recall that lifts in the composite families are given by:
	\begin{align}
		R_!(u,e,d) & = Q_!(P_!(u,e),d) \label{eq:comp-lifts} \\
		R'_!(u',e',d') & = Q'_!(P'_!(u',e'),d') \label{eq:comp'-lifts}
	\end{align}
Since $\varphi$ and $\psi$ are cocartesian functors, there are identifications:
\begin{align}
	\varphi(P_!(u,e)) & = P_!'(ju, \varphi_b(e)) \label{eq:phi-vert-cocart} \\
	\psi(Q_!(u,f,e,d)) & =  Q_!'(ju, \varphi_u(f), \varphi_b(e), \psi_{\varphi_b(e)}(d)) \label{eq:psi-vert-cocart}
\end{align}
This gives a path
\begin{align*}
\psi(R_!(u,e,d)) & \stackrel{\ref{eq:comp-lifts}}{=} \psi(Q_!(P_!(u,e),e,d)) \stackrel{\ref{eq:psi-vert-cocart}}{=}Q'_!(ju, \varphi_uP_!(u,e),\varphi_b e, \psi_{\varphi_b e} d) \\
  \stackrel{\ref{eq:phi-vert-cocart}}{=}  Q'_!(ju,P_!'(ju,\varphi_b e), \varphi_b e, \psi_{\varphi_b e} d) & \stackrel{\ref{eq:comp'-lifts}}{=} R_!'(ju, \varphi_b e, \psi_{\varphi_b e}d).\qedhere
\end{align*} 
	\end{enumerate}
\end{proof}

\begin{prop}[Product cones are cocartesian functors]\label{prop:cocart-fun-prod}
Let $B:I \to \UU$ be a type family over an arbitrary type $I$, and $P_i: B_i \to \UU$ cocartesian families with total types $E_i\defeq \totalty{P_i}$. Then, for any $k:I$ the projections from the product fibration
% https://q.uiver.app/?q=WzAsNCxbMCwwLCJcXHByb2RfaSBFX2kiXSxbMCwxLCJcXHByb2RfaSBCX2kiXSxbMSwwLCJFX2siXSxbMSwxLCJCX2siXSxbMCwxLCIiLDAseyJzdHlsZSI6eyJoZWFkIjp7Im5hbWUiOiJlcGkifX19XSxbMCwyXSxbMSwzXSxbMiwzLCIiLDIseyJzdHlsZSI6eyJoZWFkIjp7Im5hbWUiOiJlcGkifX19XV0=
\[\begin{tikzcd}
	{\prod_i E_i} & {E_k} \\
	{\prod_i B_i} & {B_k}
	\arrow[two heads, from=1-1, to=2-1]
	\arrow[from=1-1, to=1-2]
	\arrow[from=2-1, to=2-2]
	\arrow[two heads, from=1-2, to=2-2]
\end{tikzcd}\]
are cocartesian.
\end{prop}

\begin{proof}
The projection maps are given by evaluation at $i$. From the description of the cocartesian arrows in \cref{prop:cocart-fam-prod} it follows that this defines a cocartesian functor.
\end{proof}

\begin{prop}[Cocartesian universality of products]\label{prop:cocart-fun-prod-ump}
Let $I$ be a small type, $B:I \to \UU$ be a family, and $P: \prod_{i:I} B_i \to \UU$ a family such that for all $i:I$ the family $P_i: B_i \to \UU$ is a cocartesian family. We denote the associated projections by $\pi_i : E_i \fibarr B_i$.

Then the product fibration $\prod_i E_i \tofib  \prod_i B_i$ satisfies the following universal property: For any cocartesian family $Q: A \to \UU$, with associated projection $\xi:F \fibarr A$, given a family of cocartesian functors $\pair{\alpha_k}{\psi_k}: \xi \xbigtoto[]{} \pi_k$ for $k:I$, there exists a unique cocartesian functor, the \emph{target tupling} $\pair{(\alpha_i)_{i:I}}{(\psi_i)_{I:I}} : \xi \xbigtoto[]{} \prod_{i:I} \pi_i$ s.t.~that every diagram of the form
% https://q.uiver.app/?q=WzAsNixbMCwwLCJcXGJpZyhGIl0sWzIsMCwiRV9rXFxiaWcpX3trOkl9Il0sWzEsMSwiXFxwcm9kX2kgRV9pIl0sWzEsMiwiXFxwcm9kX2kgQl9pIl0sWzAsMSwiXFxiaWcoQSJdLFsyLDEsIkJfa1xcYmlnKV97azpJfSJdLFswLDEsIlxccHNpX2siXSxbMiwzLCIiLDAseyJzdHlsZSI6eyJoZWFkIjp7Im5hbWUiOiJlcGkifX19XSxbMCw0LCJcXHhpIiwyLHsic3R5bGUiOnsiaGVhZCI6eyJuYW1lIjoiZXBpIn19fV0sWzQsMywiKFxcYWxwaGFfaSlfe2k6SX0iLDEseyJzdHlsZSI6eyJib2R5Ijp7Im5hbWUiOiJkYXNoZWQifX19XSxbMiwxLCJcXG1hdGhybXtldn1fayIsMV0sWzMsNSwiXFxtYXRocm17ZXZ9X2siLDFdLFswLDIsIihcXHBzaV9pKV97aTpJfSIsMSx7InN0eWxlIjp7ImJvZHkiOnsibmFtZSI6ImRhc2hlZCJ9fX1dLFsxLDUsIlxccGlfayIsMCx7InN0eWxlIjp7ImhlYWQiOnsibmFtZSI6ImVwaSJ9fX1dLFs0LDIsIiIsMix7InN0eWxlIjp7ImhlYWQiOnsibmFtZSI6Im5vbmUifX19XSxbMiw1LCJcXGFscGhhX2siXV0=
\[\begin{tikzcd}
	{F} && {E_k} \\
	{A} & {\prod_i E_i} & {B_k} \\
	& {\prod_i B_i}
	\arrow["{\psi_k}", from=1-1, to=1-3]
	\arrow[two heads, from=2-2, to=3-2]
	\arrow["\xi"', two heads, from=1-1, to=2-1]
	\arrow["{(\alpha_i)_{i:I}}"{description}, dashed, from=2-1, to=3-2]
	\arrow["{\mathrm{ev}_k}"{description}, from=2-2, to=1-3]
	\arrow["{\mathrm{ev}_k}"{description}, from=3-2, to=2-3]
	\arrow["{(\psi_i)_{i:I}}"{description}, dashed, from=1-1, to=2-2]
	\arrow["{\pi_k}", two heads, from=1-3, to=2-3]
	\arrow[no head, from=2-1, to=2-2]
	\arrow["{\alpha_k}", from=2-2, to=2-3]
\end{tikzcd}\]
commutes.
\end{prop}

\begin{proof}
	We only point out that the square given by $\pair{(\alpha_i)_{i:I}}{(\psi_i)_{i:I}}$ is cocartesian. But this is again clear, because cocartesian lifts in the product fibration are defined pointwisely.
\end{proof}

\begin{prop}[Pullback squares are cocartesian functors, cf~{\protect{\cite[Example 5.3.3]{RV}}}]\label{prop:pb-sq-cocart}
	Let $P: B \to \UU$ be a cocartesian family with projection $\pi:E \tofib B$. For any map $k:A \to B$, the pullback square
	% https://q.uiver.app/?q=WzAsNCxbMCwwLCJrXipFIl0sWzAsMSwiQSJdLFsyLDAsIkUiXSxbMiwxLCJCIl0sWzAsMSwia14qXFxwaSIsMix7InN0eWxlIjp7ImhlYWQiOnsibmFtZSI6ImVwaSJ9fX1dLFswLDIsIlxcdmFycGhpIl0sWzEsMywiayIsMl0sWzIsMywiXFxwaSIsMCx7InN0eWxlIjp7ImhlYWQiOnsibmFtZSI6ImVwaSJ9fX1dLFswLDMsIiIsMSx7InN0eWxlIjp7Im5hbWUiOiJjb3JuZXIifX1dXQ==
	\[\begin{tikzcd}
		{k^*E} && E \\
		A && B
		\arrow["{k^*\pi}"', two heads, from=1-1, to=2-1]
		\arrow["\varphi", from=1-1, to=1-3]
		\arrow["k"', from=2-1, to=2-3]
		\arrow["\pi", two heads, from=1-3, to=2-3]
		\arrow["\lrcorner"{anchor=center, pos=0.125}, draw=none, from=1-1, to=2-3]
	\end{tikzcd}\]
	is a cocartesian functor.
\end{prop}

\begin{proof}
	The square acts as
	\[ \pair{k}{\varphi}(v,d) = \pair{kv}{d}.\]
	By \cref{prop:cocart-fam-pb}, the pullback map $k^*\pi: k^*E \to A$ is a cocartesian fibration.
	
	We abbreviate $Q\defeq k^*P$. Let $v:a \to a'$ in $A$ be an arrow and $d:k^*P(a) \jdeq P(ka)$ a point in the fiber. Again by \cref{prop:cocart-fam-pb}, the $Q$-cocartesian lift is given by
	\[ Q_!(v,d) = P_!(kv,d).\]
	This leads to
	\[
	\varphi(v,Q_!(v,d)) = \pair{kv}{Q_!(v,d)} = \pair{kv}{P_!(kv,d)} = \pair{kv}{P_!(kv,\varphi d)}.\qedhere
	\]
\end{proof}

\begin{prop}[Pullback cones are cocartesian functors]\label{prop:cocart-fun-pb}
Let $P:B \to \UU$, $P':B' \to \UU$, and $P'':B'' \to \UU$ be cocartesian families with associated projection maps $\pi :E \to B$, $\pi':E' \to B'$, and $\pi'':E'' \to B''$, resp.

Consider a commutative cubical diagram as below, where the given vertical maps are cocartesian fibrations and the front and right square are cocartesian functors:
% https://q.uiver.app/?q=WzAsOCxbMCwwLCJFJyBcXHRpbWVzX0UgRScnIl0sWzAsMiwiQicgXFx0aW1lc19CIEInJyJdLFszLDIsIkInJyJdLFszLDAsIkUnJyJdLFsxLDEsIkUnIl0sWzEsMywiQiciXSxbNCwzLCJCIl0sWzQsMSwiRSJdLFswLDEsIlxccGknXFx0aW1lc19cXHBpIFxccGknJyIsMSx7InN0eWxlIjp7ImJvZHkiOnsibmFtZSI6ImRhc2hlZCJ9fX1dLFsxLDJdLFswLDNdLFszLDIsIlxccGknJyIsMSx7ImxhYmVsX3Bvc2l0aW9uIjoyMH1dLFswLDRdLFs0LDUsIlxccGknIiwxLHsibGFiZWxfcG9zaXRpb24iOjMwfV0sWzQsNywiXFx2YXJwaGknIiwxXSxbNyw2LCJcXHBpIiwxLHsibGFiZWxfcG9zaXRpb24iOjMwfV0sWzEsNV0sWzIsNiwiXFxiZXRhJyciLDFdLFszLDcsIlxcdmFycGhpJyciLDFdLFs1LDYsIlxcYmV0YSciLDFdLFsxLDYsIiIsMSx7InN0eWxlIjp7Im5hbWUiOiJjb3JuZXIifX1dLFswLDcsIiIsMSx7InN0eWxlIjp7Im5hbWUiOiJjb3JuZXIifX1dXQ==
\[\begin{tikzcd}
	{E' \times_E E''} &&& {E''} \\
	& {E'} &&& E \\
	{B' \times_B B''} &&& {B''} \\
	& {B'} &&& B
	\arrow["{\pi'\times_\pi \pi''}"{description}, dashed, two heads, from=1-1, to=3-1]
	\arrow[from=1-1, to=1-4]
	\arrow[from=1-1, to=2-2]
	\arrow["\pi"{description, pos=0.3}, from=2-5, to=4-5, two heads]
	\arrow["{\beta''}"{description}, from=3-4, to=4-5]
	\arrow["{\varphi''}"{description}, from=1-4, to=2-5]
	\arrow["{\beta'}"{description}, from=4-2, to=4-5]
	\arrow["\slpbcorner"{anchor=center, pos=0.125}, draw=none, from=3-1, to=4-5]
	\arrow["\slpbcorner"{anchor=center, pos=0.125}, draw=none, from=1-1, to=2-5]
	\arrow["{\pi''}"{description, pos=0.2}, from=1-4, to=3-4, two heads]
	\arrow[from=3-1, to=4-2]
	\arrow[from=3-1, to=3-4]
	\arrow["{\varphi'}"{description}, from=2-2, to=2-5, crossing over]
	\arrow["{\pi'}"{description, pos=0.3}, from=2-2, to=4-2, crossing over, two heads]
\end{tikzcd}\]
Then the mediating map
\[ \pi''':\jdeq \pi' \times_\pi \pi'' : E''': \jdeq E' \times_E E'' \to B' \times_{B} B'' \jdeq : B'''\]
is a cocartesian fibration. Furthermore, the back and the left square are cocartesian functors.

Specifically, the cocartesian lift of an arrow $u''':\jdeq \langle u,u',u''\rangle$ w.r.t.~$\langle e,e',e'' \rangle$ is given by
\[ \langle  P_!(u,e), P_!(u',e'), P_!(u'',e'') \rangle. \]
\end{prop}

\begin{proof}
Denote by $Q',Q'':B \to \UU$ the families associated to the maps $\beta'$ and $\beta''$, respectively.
The map in question is equivalent to (\cf~\cite[Theorem 24.2.3(ii)]{RijIntro}) the following projection:
\[\begin{tikzcd}
	E''':\jdeq \sum_{b''':B'''} P'\,b' \times_{P\,b} P''\,b'' \ar[rr,two heads, "\pi'''"] && B'''\equiv \sum_{b:B} Q' \,b \times_{Q\,b} Q''\,b
\end{tikzcd}\]
The candidate $\pi'''$-cocartesian lift of $\angled{u,u',u''}$ w.r.t.~$\angled{e:P(u0),e':P'(u'0),e'':P''(u''0)}$ is given by
\[ \angled{P_!(u,e), P_!'(u',e'), P_!''(u'',e'')},\]
and since the universal property is satisfied fiberwisely it is satisfied in the pullback fibration.
\end{proof}

\begin{prop}[Cocartesian universality of pullback cones]\label{prop:cocart-fun-pb-ump}
	Given a commutative cube of cocartesian fibrations
% https://q.uiver.app/?q=WzAsOCxbMCwwLCJGIl0sWzAsMiwiQSJdLFszLDIsIkInJyJdLFszLDAsIkUnJyJdLFsxLDEsIkUnIl0sWzEsMywiQiciXSxbNCwzLCJCIl0sWzQsMSwiRSJdLFswLDEsIlxceGkiLDFdLFsxLDIsIlxcYWxwaGEnJyJdLFswLDMsIlxccHNpJyciXSxbMywyLCJcXHBpJyciLDEseyJsYWJlbF9wb3NpdGlvbiI6MjB9XSxbMCw0LCJcXHBzaSciXSxbNCw1LCJcXHBpJyIsMSx7ImxhYmVsX3Bvc2l0aW9uIjozMH1dLFs0LDcsIlxcdmFycGhpJyIsMV0sWzcsNiwiXFxwaSIsMSx7ImxhYmVsX3Bvc2l0aW9uIjozMH1dLFsxLDUsIlxcYWxwaGEnIl0sWzIsNiwiXFxiZXRhJyciLDFdLFszLDcsIlxcdmFycGhpJyciLDFdLFs1LDYsIlxcYmV0YSciLDFdXQ==
\[\begin{tikzcd}
	F &&& {E''} \\
	& {E'} &&& E \\
	A &&& {B''} \\
	& {B'} &&& B
	\arrow["\xi"{description}, two heads, from=1-1, to=3-1]
	\arrow["{\alpha''}"{description,pos=0.7}, from=3-1, to=3-4]
	\arrow["{\psi''}"{description,pos=0.7}, from=1-1, to=1-4]
	\arrow["{\pi''}"{description, pos=0.2}, two heads, from=1-4, to=3-4]
	\arrow["{\psi'}"{description}, from=1-1, to=2-2]
	\arrow["{\pi'}"{description, pos=0.3}, two heads, from=2-2, to=4-2]
	\arrow["\pi"{description, pos=0.3}, two heads, from=2-5, to=4-5]
	\arrow["{\alpha'}"{description}, from=3-1, to=4-2]
	\arrow["{\beta''}"{description}, from=3-4, to=4-5]
	\arrow["{\varphi''}"{description}, from=1-4, to=2-5]
	\arrow["{\beta'}"{description}, from=4-2, to=4-5]
	\arrow["{\varphi'}"{description}, from=2-2, to=2-5, crossing over]
	\arrow["{\pi'}"{description, pos=0.3}, two heads, from=2-2, to=4-2, crossing over]
\end{tikzcd}\]
 where the vertical faces are cocartesian functors, there exists up to homotopy a unique cocartesian functor $\pair{\alpha'''}{\psi'''}:\xi \xbigtoto[]{} \pi'''$ s.t.~the diagram
% https://q.uiver.app/?q=WzAsMTAsWzIsMCwiRScnJyJdLFsyLDIsIkInJyciXSxbNSwwLCJFJyciXSxbNSwyLCJCJyciXSxbMywxLCJFJyJdLFs2LDEsIkUiXSxbMywzLCJCJyJdLFs2LDMsIkIiXSxbMCwwLCJGIl0sWzAsMiwiQSJdLFswLDEsIlxccGknJyciLDAseyJzdHlsZSI6eyJoZWFkIjp7Im5hbWUiOiJlcGkifX19XSxbMCwyXSxbMiwzLCJcXHBpJyIsMSx7InN0eWxlIjp7ImhlYWQiOnsibmFtZSI6ImVwaSJ9fX1dLFsxLDNdLFswLDRdLFs0LDUsIlxcdmFycGhpJyJdLFsyLDUsIlxcdmFycGhpJyciXSxbNCw2LCJcXHBpJyIsMSx7InN0eWxlIjp7ImhlYWQiOnsibmFtZSI6ImVwaSJ9fX1dLFs2LDcsIlxcYmV0YSciLDFdLFs1LDcsIlxccGkiLDEseyJzdHlsZSI6eyJoZWFkIjp7Im5hbWUiOiJlcGkifX19XSxbMSw2XSxbMyw3LCJcXGJldGEnJyIsMV0sWzEsNywiIiwwLHsic3R5bGUiOnsibmFtZSI6ImNvcm5lciJ9fV0sWzAsNSwiIiwwLHsic3R5bGUiOnsibmFtZSI6ImNvcm5lciJ9fV0sWzgsOSwiXFx4aSIsMix7InN0eWxlIjp7ImhlYWQiOnsibmFtZSI6ImVwaSJ9fX1dLFs5LDEsIlxcYWxwaGEnJyciLDEseyJzdHlsZSI6eyJib2R5Ijp7Im5hbWUiOiJkYXNoZWQifX19XSxbOCwwLCJcXHBzaScnJyIsMSx7InN0eWxlIjp7ImJvZHkiOnsibmFtZSI6ImRhc2hlZCJ9fX1dLFs4LDIsIlxccHNpJyciLDEseyJjdXJ2ZSI6LTN9XSxbOSw2LCJcXGFscGhhJyIsMV0sWzgsNCwiXFxwc2knIiwxXSxbOSwzLCJcXGFscGhhJyciLDEseyJjdXJ2ZSI6LTN9XV0=
\[\begin{tikzcd}
	F && {E'''} &&& {E''} \\
	&&& {E'} &&& E \\
	A && {B'''} &&& {B''} \\
	&&& {B'} &&& B
	\arrow[from=1-3, to=1-6]
	\arrow["{\pi''}"{description}, two heads, from=1-6, to=3-6, near start]
	\arrow[from=3-3, to=3-6]
	\arrow[from=1-3, to=2-4]
	\arrow["{\varphi''}", from=1-6, to=2-7, description]
	\arrow["{\beta'}"{description}, from=4-4, to=4-7]
	\arrow["\pi"{description}, two heads, from=2-7, to=4-7, near start]
	\arrow[from=3-3, to=4-4]
	\arrow["{\beta''}"{description}, from=3-6, to=4-7]
	\arrow["\slpbcorner"{anchor=center, pos=0.125}, draw=none, from=3-3, to=4-7]
	\arrow["\slpbcorner"{anchor=center, pos=0.125}, draw=none, from=1-3, to=2-7]
	\arrow["\xi"', two heads, from=1-1, to=3-1, description]
	\arrow["{\alpha'''}"{description}, dashed, from=3-1, to=3-3, description]
	\arrow["{\psi'''}"{description}, dashed, from=1-1, to=1-3]
	\arrow["{\psi''}"{description}, curve={height=-22pt}, from=1-1, to=1-6]
	\arrow["{\alpha'}"{description}, from=3-1, to=4-4]
	\arrow["{\alpha''}"{description}, curve={height=-22pt}, from=3-1, to=3-6]
	\arrow["{\varphi'}"{description}, from=2-4, to=2-7,crossing over, description]
	\arrow["{\pi'}"{description}, two heads, from=2-4, to=4-4, very near start,crossing over]
	\arrow["{\pi'''}"{description}, two heads, from=1-3, to=3-3, crossing over]
	\arrow["{\psi'}"{description}, from=1-1, to=2-4, crossing over]
\end{tikzcd}\]
commutes.
\end{prop}

\begin{proof}
By fibrant replacement, we can take the whole diagram to commute strictly. This gives a cocartesian functor comprised of the dependent pair
\[ \alpha\defeq \beta' \circ \alpha' \jdeq \beta'' \circ \alpha'': A \to B, \quad \psi\defeq \varphi' \circ \psi' \jdeq \varphi'' \circ \psi'':F \to E. \]
Then $\alpha'''$ acts as the fiber product tupling $\angled{\alpha, \alpha', \alpha''}$, and likewise $\psi'''$ is given by $\angled{\psi, \psi', \psi''}$. Denote by $R:A \to \UU$ the family associated to the fibration $\xi:F \to A$. For an arrow $v:a \to a'$ in $A$ and a point $d:R\,a$, the $R$-cocartesian lift gets mapped to
\begin{align*}
	\psi'''(R_!(v,d)) & = \angled{\psi(R_!(v,d)), \psi'(R'_!(v,d)), \psi''(R_!(v,d))} \\
	& = \angled{P_!(\alpha v,\psi d), P_!'(\alpha'v,\psi'd), P_!''(\alpha''v, \psi''d)}  \tag{\text{$\alpha, \alpha', \alpha''$ cocart.~functors}} \\
	& = P'''_!(\alpha \, v, \psi \,d ) \tag{\text{construction of lifts by \cref{prop:cocart-fun-pb}}},
\end{align*}
which establishes the claim.
\end{proof}

\begin{prop}[Sequential limit cones are cocartesian functors]\label{prop:cocart-fun-seqlim}
Consider an inverse diagram of cocartesian fibrations as below where all of the connecting squares are cocartesian functors:
% https://q.uiver.app/?q=WzAsMTIsWzUsMCwiRV9cXGluZnR5Il0sWzUsMiwiQl9cXGluZnR5Il0sWzYsMSwiRV8wIl0sWzYsMywiQl8wIl0sWzQsMywiQl8xIl0sWzQsMSwiRV8xIl0sWzIsMywiQl8yIl0sWzIsMSwiRV8yIl0sWzAsMSwiXFxjZG90cyJdLFswLDMsIlxcY2RvdHMiXSxbMSwwLCJcXGxkb3RzIl0sWzEsMiwiXFxsZG90cyJdLFswLDEsIlxccGlfXFxpbmZ0eSIsMSx7ImxhYmVsX3Bvc2l0aW9uIjozMCwic3R5bGUiOnsiYm9keSI6eyJuYW1lIjoiZGFzaGVkIn19fV0sWzIsMywiXFxwaV8wIiwxLHsibGFiZWxfcG9zaXRpb24iOjMwLCJzdHlsZSI6eyJoZWFkIjp7Im5hbWUiOiJlcGkifX19XSxbNCwzLCJmXzAiLDEseyJsYWJlbF9wb3NpdGlvbiI6NzB9XSxbNSwyLCJnXzAiLDEseyJsYWJlbF9wb3NpdGlvbiI6NzB9XSxbNSw0LCJcXHBpXzEiLDEseyJsYWJlbF9wb3NpdGlvbiI6MzAsInN0eWxlIjp7ImhlYWQiOnsibmFtZSI6ImVwaSJ9fX1dLFswLDIsIlxcdmFycGhpXzAiLDEseyJzdHlsZSI6eyJib2R5Ijp7Im5hbWUiOiJkYXNoZWQifX19XSxbMSwzLCJcXGJldGFfMCIsMSx7InN0eWxlIjp7ImJvZHkiOnsibmFtZSI6ImRhc2hlZCJ9fX1dLFs2LDQsImZfMSIsMSx7ImxhYmVsX3Bvc2l0aW9uIjo3MH1dLFs3LDUsImdfMSIsMSx7ImxhYmVsX3Bvc2l0aW9uIjo3MH1dLFs3LDYsIlxccGlfbiIsMSx7ImxhYmVsX3Bvc2l0aW9uIjozMCwic3R5bGUiOnsiaGVhZCI6eyJuYW1lIjoiZXBpIn19fV0sWzgsNywiZ18yIiwxLHsibGFiZWxfcG9zaXRpb24iOjcwfV0sWzksNiwiZl8yIiwxLHsibGFiZWxfcG9zaXRpb24iOjcwfV0sWzAsNSwiXFx2YXJwaGlfMSIsMSx7InN0eWxlIjp7ImJvZHkiOnsibmFtZSI6ImRhc2hlZCJ9fX1dLFswLDcsIlxcdmFycGhpXzIiLDEseyJzdHlsZSI6eyJib2R5Ijp7Im5hbWUiOiJkYXNoZWQifX19XSxbMSw2LCJcXGJldGFfMiIsMSx7InN0eWxlIjp7ImJvZHkiOnsibmFtZSI6ImRhc2hlZCJ9fX1dLFsxLDQsIlxcYmV0YV8xIiwxLHsic3R5bGUiOnsiYm9keSI6eyJuYW1lIjoiZGFzaGVkIn19fV0sWzAsOCwiXFx2YXJwaGlfMyIsMSx7ImN1cnZlIjoyLCJzdHlsZSI6eyJib2R5Ijp7Im5hbWUiOiJkYXNoZWQifX19XSxbMSw5LCJcXGJldGFfMyIsMSx7ImN1cnZlIjoyLCJzdHlsZSI6eyJib2R5Ijp7Im5hbWUiOiJkYXNoZWQifX19XV0=
\[\begin{tikzcd}
	& \ldots &&&& {E_\infty} \\
	\cdots && {E_2} && {E_1} && {E_0} \\
	& \ldots &&&& {B_\infty} \\
	\cdots && {B_2} && {B_1} && {B_0}
	\arrow["{\pi_\infty}"{description, pos=0.3}, dashed, from=1-6, to=3-6]
	\arrow["{\pi_0}"{description, pos=0.3}, two heads, from=2-7, to=4-7]
	\arrow["{f_0}"{description, pos=0.7}, from=4-5, to=4-7]
	\arrow["{\varphi_0}"{description}, dashed, from=1-6, to=2-7]
	\arrow["{\psi_0}"{description}, dashed, from=3-6, to=4-7]
	\arrow["{f_1}"{description, pos=0.7}, from=4-3, to=4-5]
	\arrow["{g_1}"{description, pos=0.7}, from=2-3, to=2-5]
	\arrow["{g_2}"{description, pos=0.7}, from=2-1, to=2-3]
	\arrow["{f_2}"{description, pos=0.7}, from=4-1, to=4-3]
	\arrow["{\varphi_1}"{description}, dashed, from=1-6, to=2-5]
	\arrow["{\varphi_2}"{description}, dashed, from=1-6, to=2-3]
	\arrow["{\psi_2}"{description}, dashed, from=3-6, to=4-3]
	\arrow["{\psi_1}"{description}, dashed, from=3-6, to=4-5]
	\arrow["{\varphi_3}"{description}, curve={height=12pt}, dashed, from=1-6, to=2-1]
	\arrow["{\psi_3}"{description}, curve={height=12pt}, dashed, from=3-6, to=4-1]
	\arrow["{g_0}"{description, pos=0.7}, from=2-5, to=2-7, crossing over]
	\arrow["{\pi_1}"{description, pos=0.3}, two heads, from=2-5, to=4-5, crossing over]
	\arrow["{\pi_2}"{description, pos=0.3}, two heads, from=2-3, to=4-3, crossing over]
\end{tikzcd}\]
Then the induced map $\pi_\infty: E_\infty \to B_\infty$ between the limit types is a cocartesian fibration, and the projections constitute cocartesian functors.
\end{prop}

\begin{proof}
Note that we can present the limiting types as pullbacks in the following way: First, for the base types, we have
% https://q.uiver.app/?q=WzAsNixbMSwwLCJcXHN1bV97XFxiZXRhOlxccHJvZF97bjpcXG1hdGhiYiBOfSBCX259IFxccHJvZF97azpcXG1hdGhiYiBOfSBcXGJldGFfe259ID0gZl9uKFxcYmV0YV97bisxfSkgXFxzaW1lcSBCX1xcaW5mdHkiXSxbMSwyLCJcXHByb2Rfe246XFxtYXRoYmIgTn0gQl97Mm4rMX0iXSxbMywyLCJcXHByb2Rfe246XFxtYXRoYmIgTn0gQl9uIl0sWzMsMCwiXFxwcm9kX3tuOlxcbWF0aGJiIE59IEJfezJufSJdLFswLDBdLFswLDJdLFswLDFdLFsxLDJdLFswLDNdLFszLDJdLFswLDIsIiIsMSx7InN0eWxlIjp7Im5hbWUiOiJjb3JuZXIifX1dXQ==
\[\begin{tikzcd}
	{} & {\sum_{\beta:\prod_{n:\mathbb N} B_n} \prod_{k:\mathbb N} \beta_{n} = f_n(\beta_{n+1}) \equiv B_\infty} && {\prod_{n:\mathbb N} B_{2n}} \\
	\\
	{} & {\prod_{n:\mathbb N} B_{2n+1}} && {\prod_{n:\mathbb N} B_n}
	\arrow[from=1-2, to=3-2]
	\arrow[from=3-2, to=3-4]
	\arrow[from=1-2, to=1-4]
	\arrow[from=1-4, to=3-4]
	\arrow["\lrcorner"{anchor=center, pos=0.125}, draw=none, from=1-2, to=3-4]
\end{tikzcd}\]
since
\[
\prod_{n:\mathbb N} B_{2n}  \equiv \sum_{\beta:\prod_{n:\mathbb N} B_n} \prod_{k:\mathbb N}\beta_{2k+1} = f_{2k+1}(\beta_{2k+2}), \quad \text{and} \quad
\prod_{n:\mathbb N} B_{2n+1} \equiv \prod_{\beta:\prod_{n:\mathbb N} B_n} \prod_{k:\mathbb N}\beta_{2k} = f_{2k}(\beta_{2k+1}).
\]
Lying over, the limiting total type arises as
% https://q.uiver.app/?q=WzAsNixbMSwwLCJcXHN1bV97XFxsYW5nbGUgXFxiZXRhLCBcXHNpZ21hIFxccmFuZ2xlOkJfXFxpbmZ0eX0gXFxzdW1fe1xcdmFydGhldGE6XFxwcm9kX3tuOlxcbWF0aGJiIE59IFBfbihcXGJldGFfbil9XFxwcm9kX3trOlxcbWF0aGJiIE59IFxcdmFydGhldGFfe259ID1fe1xcc2lnbWFfbn0gZ19uKFxcdmFydGhldGFfe24rMX0pIFxcc2ltZXEgRV9cXGluZnR5Il0sWzEsMiwiXFxwcm9kX3tuOlxcbWF0aGJiIE59IEVfezJuKzF9Il0sWzMsMiwiXFxwcm9kX3tuOlxcbWF0aGJiIE59IEVfbiJdLFszLDAsIlxccHJvZF97bjpcXG1hdGhiYiBOfSBFX3sybn0iXSxbMCwwXSxbMCwyXSxbMCwxXSxbMSwyXSxbMCwzXSxbMywyXSxbMCwyLCIiLDEseyJzdHlsZSI6eyJuYW1lIjoiY29ybmVyIn19XV0=
\[\begin{tikzcd}
	{} & {\sum_{\langle \beta, \sigma \rangle:B_\infty} \sum_{\vartheta:\prod_{n:\mathbb N} P_n(\beta_n)}\prod_{k:\mathbb N} \vartheta_{n} =_{\sigma_n} g_n(\vartheta_{n+1}) \equiv E_\infty} && {\prod_{n:\mathbb N} E_{2n}} \\
	\\
	{} & {\prod_{n:\mathbb N} E_{2n+1}} && {\prod_{n:\mathbb N} E_n}
	\arrow[from=1-2, to=3-2]
	\arrow[from=3-2, to=3-4]
	\arrow[from=1-2, to=1-4]
	\arrow[from=1-4, to=3-4]
	\arrow["\lrcorner"{anchor=center, pos=0.125}, draw=none, from=1-2, to=3-4]
\end{tikzcd}\]
because
\begin{align*}
	\prod_{n:\mathbb N} E_{2n}  & \equiv \sum_{\pair{\beta}{\sigma}:B_{2n}} \sum_{\vartheta:\prod_{n:\mathbb N} P_n(\beta_n)} \prod_{k:\mathbb N}\vartheta_{2k+1} =_{\sigma_{2k+1}} g_{2k+1}(\vartheta_{2k+2}) \\
	\prod_{n:\mathbb N} E_{2n+1}  & \equiv \sum_{\pair{\beta}{\sigma}:B_{2n+1}} \sum_{\vartheta:\prod_{n:\mathbb N} P_n(\beta_n)} \prod_{k:\mathbb N}\vartheta_{2k} =_{\sigma_{2k}} g_{2k}(\vartheta_{2k+1})
\end{align*}
Since cocartesian fibrations are closed under composition by \cref{prop:cocart-fam-comp} and dependent products by \cref{prop:cocart-fam-prod}, and the induced maps are cocartesian functors by \cref{prop:cocart-fun-prod} the diagram
% https://q.uiver.app/?q=WzAsMTEsWzAsMCwiRV9cXGluZnR5Il0sWzMsMCwiXFxwcm9kX3tuOlxcbWF0aGJiIE59IEVfezJufSJdLFsxLDEsIlxccHJvZF97bjpcXG1hdGhiYiBOfSBFX3sybisxfSJdLFs0LDEsIlxccHJvZF97bjpcXG1hdGhiYiBOfSBCX3sybn0iXSxbMCwyLCJCX1xcaW5mdHkiXSxbMSwzLCJcXHByb2Rfe246XFxtYXRoYmIgTn0gQl97Mm4rMX0iXSxbMywyLCJcXHByb2Rfe246XFxtYXRoYmIgTn0gRV9uIl0sWzQsMywiXFxwcm9kX3tuOlxcbWF0aGJiIE59IEJfbiJdLFsyLDFdLFszLDFdLFsxLDBdLFsyLDNdLFsxLDNdLFswLDQsIiIsMCx7InN0eWxlIjp7ImJvZHkiOnsibmFtZSI6ImRhc2hlZCJ9LCJoZWFkIjp7Im5hbWUiOiJlcGkifX19XSxbMiw1LCIiLDAseyJzdHlsZSI6eyJoZWFkIjp7Im5hbWUiOiJlcGkifX19XSxbMSw2XSxbMyw3LCIiLDAseyJzdHlsZSI6eyJoZWFkIjp7Im5hbWUiOiJlcGkifX19XSxbNiw3XSxbNCw1XSxbNSw3XSxbNCw2XSxbMiw3LCIiLDEseyJzdHlsZSI6eyJuYW1lIjoiY29ybmVyIn19XSxbMCwxXSxbMCwyLCIiLDEseyJvZmZzZXQiOjIsInN0eWxlIjp7ImJvZHkiOnsibmFtZSI6ImRhc2hlZCJ9fX1dLFsxMCw2LCIiLDAseyJzdHlsZSI6eyJuYW1lIjoiY29ybmVyIn19XV0=
\begin{equation}\label{cd:seqlim-pb}
\begin{tikzcd}
	{E_\infty} & {} && {\prod_{n:\mathbb N} E_{2n}} \\
	& {\prod_{n:\mathbb N} E_{2n+1}} & {} & {} & {\prod_{n:\mathbb N} E_n} \\
	{B_\infty} &&& {\prod_{n:\mathbb N} B_{2n}} \\
	& {\prod_{n:\mathbb N} B_{2n+1}} &&& {\prod_{n:\mathbb N} B_n}
	\arrow[from=1-4, to=2-5]
	\arrow["{\pi_\infty}"{description}, dashed, two heads, from=1-1, to=3-1]
	\arrow[two heads, from=2-2, to=4-2]
	\arrow[two heads, from=2-5, to=4-5]
	\arrow[from=3-4, to=4-5]
	\arrow[from=3-1, to=4-2]
	\arrow[from=4-2, to=4-5]
	\arrow[from=3-1, to=3-4]
	\arrow[from=1-1, to=1-4]
	\arrow[dashed, from=1-1, to=2-2]
	\arrow["\slpbcorner"{anchor=center, pos=0.125}, draw=none, from=3-1, to=4-5]
	\arrow["\slpbcorner"{anchor=center, pos=0.125}, draw=none, from=1-1, to=2-5]
	\arrow[two heads, from=2-2, to=4-2, crossing over]
	\arrow[from=1-4, to=3-4, crossing over]	
	\arrow[from=2-2, to=2-5, crossing over]
\end{tikzcd}
\end{equation}
shows that the induced map $\pi_\infty : E_\infty \to B_\infty$ is a cocartesian fibration, too, by \cref{prop:cocart-fun-pb}.

Finally, coming back to the sequential limit of squares, it is clear that the mediating squares
% https://q.uiver.app/?q=WzAsNCxbMCwwLCJFX1xcaW5mdHkiXSxbMCwxLCJCX1xcaW5mdHkiXSxbMiwxLCJCX2siXSxbMiwwLCJFX2siXSxbMCwxLCJcXHBpX1xcaW5mdHkiLDIseyJzdHlsZSI6eyJoZWFkIjp7Im5hbWUiOiJlcGkifX19XSxbMSwyLCJcXHBzaV9rIiwyXSxbMCwzLCJcXHZhcnBoaV9rIl0sWzMsMiwiXFxwaV9rIiwwLHsic3R5bGUiOnsiaGVhZCI6eyJuYW1lIjoiZXBpIn19fV1d
\[\begin{tikzcd}
	{E_\infty} && {E_k} \\
	{B_\infty} && {B_k}
	\arrow["{\pi_\infty}"', two heads, from=1-1, to=2-1]
	\arrow["{\psi_k}"', from=2-1, to=2-3]
	\arrow["{\varphi_k}", from=1-1, to=1-3]
	\arrow["{\pi_k}", two heads, from=1-3, to=2-3]
\end{tikzcd}\]
consisting of evaluations are cocartesian functors.
\end{proof}

\begin{prop}[Cocartesian universality of sequential limits]\label{prop:cocart-fun-seqlim-ump}
Given an inverse diagram of cocartesian fibrations
% https://q.uiver.app/?q=WzAsNCxbMCwwLCJFX3tuKzF9Il0sWzAsMSwiQl97bisxfSJdLFsyLDAsIkVfbiJdLFsyLDEsIkJfbiJdLFswLDEsIlxccGlfe24rMX0iLDJdLFswLDIsImdfbiJdLFsxLDMsImZfbiIsMl0sWzIsMywiXFxwaV9uIl1d
\[\begin{tikzcd}
	{E_{n+1}} && {E_n} \\
	{B_{n+1}} && {B_n}
	\arrow["{\pi_{n+1}}"', from=1-1, to=2-1, two heads]
	\arrow["{g_n}", from=1-1, to=1-3]
	\arrow["{f_n}"', from=2-1, to=2-3]
	\arrow["{\pi_n}", from=1-3, to=2-3, two heads]
\end{tikzcd}\]
and a cone
% https://q.uiver.app/?q=WzAsNixbMiwwLCJFX3tuKzF9Il0sWzIsMSwiQl97bisxfSJdLFs0LDAsIkVfbiJdLFs0LDEsIkJfbiJdLFswLDAsIkYiXSxbMCwxLCJBIl0sWzAsMSwiXFxwaV97bisxfSIsMix7InN0eWxlIjp7ImhlYWQiOnsibmFtZSI6ImVwaSJ9fX1dLFswLDIsImdfbiJdLFsxLDMsImZfbiIsMl0sWzIsMywiXFxwaV9uIiwwLHsic3R5bGUiOnsiaGVhZCI6eyJuYW1lIjoiZXBpIn19fV0sWzQsNSwiXFx4aSIsMix7InN0eWxlIjp7ImhlYWQiOnsibmFtZSI6ImVwaSJ9fX1dLFs0LDIsIlxcdmFycGhpX24iLDEseyJjdXJ2ZSI6LTR9XSxbNSwzLCJrX24iLDEseyJjdXJ2ZSI6NH1dLFs0LDAsIlxcdmFycGhpX3tuKzF9Il0sWzUsMSwiZl9uIiwyXV0=
\[\begin{tikzcd}
	F && {E_{n+1}} && {E_n} \\
	A && {B_{n+1}} && {B_n}
	\arrow["{\pi_{n+1}}"', two heads, from=1-3, to=2-3]
	\arrow["{g_n}", from=1-3, to=1-5]
	\arrow["{f_n}"', from=2-3, to=2-5]
	\arrow["{\pi_n}", two heads, from=1-5, to=2-5]
	\arrow["\xi"', two heads, from=1-1, to=2-1]
	\arrow["{\varphi_n}"{description}, curve={height=-24pt}, from=1-1, to=1-5]
	\arrow["{k_n}"{description}, curve={height=24pt}, from=2-1, to=2-5]
	\arrow["{\varphi_{n+1}}", from=1-1, to=1-3]
	\arrow["{k_{n+1}}"', from=2-1, to=2-3]
\end{tikzcd}\]
lying above, there exists uniquely up to homotopy a cocartesian functor
% https://q.uiver.app/?q=WzAsNCxbMCwwLCJGIl0sWzAsMSwiQSJdLFsyLDAsIkVfXFxpbmZ0eSJdLFsyLDEsIkJfXFxpbmZ0eSJdLFswLDEsIlxceGkiLDIseyJzdHlsZSI6eyJoZWFkIjp7Im5hbWUiOiJlcGkifX19XSxbMCwyLCJcXGxpbV97aTpJfSBcXHZhcnBoaV9pIl0sWzEsMywiXFxsaW1fe2k6SX0ga19pIiwyXSxbMiwzLCJcXHBpX1xcaW5mdHkiLDAseyJzdHlsZSI6eyJoZWFkIjp7Im5hbWUiOiJlcGkifX19XV0=
\[\begin{tikzcd}
	F && {E_\infty} \\
	A && {B_\infty}
	\arrow["\xi"', two heads, from=1-1, to=2-1]
	\arrow["{\lim_{i:I} \varphi_i}", from=1-1, to=1-3]
	\arrow["{\lim_{i:I} k_i}"', from=2-1, to=2-3]
	\arrow["{\pi_\infty}", two heads, from=1-3, to=2-3]
\end{tikzcd}\]
together with homotopies as indicated in the following diagram:
% https://q.uiver.app/?q=WzAsOCxbMCwwLCJGIl0sWzIsMCwiRV97bisxfSJdLFs0LDAsIkVfbiJdLFswLDIsIkEiXSxbMiwyLCJCX3tuKzF9Il0sWzQsMiwiQl9uIl0sWzEsMSwiRV9cXGluZnR5Il0sWzEsMywiQl9cXGluZnR5Il0sWzAsMSwiXFx2YXJwaGlfe24rMX0iLDEseyJsYWJlbF9wb3NpdGlvbiI6NjB9XSxbMSwyLCJnX3tufSIsMV0sWzAsMywiXFx4aSIsMSx7InN0eWxlIjp7ImhlYWQiOnsibmFtZSI6ImVwaSJ9fX1dLFszLDQsImtfe24rMX0iLDEseyJsYWJlbF9wb3NpdGlvbiI6NzB9XSxbMSw0LCJcXHBpX3tuKzF9IiwxLHsic3R5bGUiOnsiaGVhZCI6eyJuYW1lIjoiZXBpIn19fV0sWzIsNSwiXFxwaV9uIiwxLHsic3R5bGUiOnsiaGVhZCI6eyJuYW1lIjoiZXBpIn19fV0sWzQsNSwiZl9uIiwxXSxbMCw2LCJcXGxpbV97aTpJfSBcXHZhcnBoaV9pIiwxLHsic3R5bGUiOnsiYm9keSI6eyJuYW1lIjoiZGFzaGVkIn19fV0sWzYsNywiXFxwaV9cXGluZnR5IiwxLHsibGFiZWxfcG9zaXRpb24iOjcwLCJvZmZzZXQiOjEsInN0eWxlIjp7ImhlYWQiOnsibmFtZSI6ImVwaSJ9fX1dLFs2LDFdLFs2LDJdLFs3LDVdLFszLDcsIlxcbGltX3tpOkl9a19pIiwxLHsic3R5bGUiOnsiYm9keSI6eyJuYW1lIjoiZGFzaGVkIn19fV0sWzcsNF0sWzAsMiwiXFx2YXJwaGlfe259IiwxLHsibGFiZWxfcG9zaXRpb24iOjcwLCJjdXJ2ZSI6LTR9XSxbMyw1LCJrX24iLDEseyJsYWJlbF9wb3NpdGlvbiI6NzAsImN1cnZlIjotNH1dXQ==
	\[\begin{tikzcd}
		F && {E_{n+1}} && {E_n} \\
		& {E_\infty} \\
		A && {B_{n+1}} && {B_n} \\
		& {B_\infty}
		\arrow["{\varphi_{n+1}}"{description, pos=0.7}, from=1-1, to=1-3]
		\arrow["{g_{n}}"{description}, from=1-3, to=1-5]
		\arrow["\xi"{description}, two heads, from=1-1, to=3-1]
		\arrow["{\pi_n}"{description}, two heads, from=1-5, to=3-5]
		\arrow["{f_n}"{description}, from=3-3, to=3-5]
		\arrow["{\lim_{i:I} \varphi_i}"{description}, dashed, from=1-1, to=2-2]
		\arrow[from=2-2, to=1-3]
		\arrow["{\lim_{i:I}k_i}"{description}, dashed, from=3-1, to=4-2]
		\arrow[from=4-2, to=3-3]
		\arrow[from=4-2, to=3-5]
		\arrow["{\varphi_{n}}"{description, pos=0.7}, curve={height=-24pt}, from=1-1, to=1-5]
		\arrow["{k_n}"{description, pos=0.7}, curve={height=-24pt}, from=3-1, to=3-5]
		\arrow["{k_{n+1}}"{description, pos=0.7}, from=3-1, to=3-3]
		\arrow["{\pi_{n+1}}"{description}, two heads, from=1-3, to=3-3, crossing over]
		\arrow["{\pi_\infty}"{description, pos=0.7}, shift right=1, two heads, from=2-2, to=4-2,crossing over]
		\arrow[from=2-2, to=1-5, crossing over]
	\end{tikzcd}\]
\end{prop}

\begin{proof}
The type of cones over a fixed inverse diagram of squares is equivalent to the type of cones over the ensuing cube~\eqref{cd:seqlim-pb}. Then the claim follows by \cref{prop:cocart-fun-pb-ump}.
\end{proof}

\begin{prop}[Cocartesian functors are closed under dependent products]\label{prop:cocart-fun-dep-prod}
For a type $I:\UU$, suppose there are families $A,B: I \to \UU$ with families of families $P,Q: \prod_{i:I} B_i \to \UU$ such that $P_i:B_i\to \UU$ and $Q_i:A_i \to \UU$ are cocartesian for each $i:I$. Then, writing $\xi_i: F_i \to A_i$ and $\pi_i: E_i \to B_i$, resp., for the associated projections, given a family of cocartesian functors
% https://q.uiver.app/?q=WzAsNCxbMCwwLCJGX2kiXSxbMiwwLCJFX2kiXSxbMCwxLCJBX2kiXSxbMiwxLCJCX2kiXSxbMCwxLCJcXHZhcnBoaV9pIl0sWzAsMiwiXFx4aV9pIiwyLHsic3R5bGUiOnsiaGVhZCI6eyJuYW1lIjoiZXBpIn19fV0sWzEsMywiXFxwaV9pIiwwLHsic3R5bGUiOnsiaGVhZCI6eyJuYW1lIjoiZXBpIn19fV0sWzIsMywia19pIiwyXV0=
\[\begin{tikzcd}
	{F_i} && {E_i} \\
	{A_i} && {B_i}
	\arrow["{\varphi_i}", from=1-1, to=1-3]
	\arrow["{\xi_i}"', two heads, from=1-1, to=2-1]
	\arrow["{\pi_i}", two heads, from=1-3, to=2-3]
	\arrow["{k_i}"', from=2-1, to=2-3]
\end{tikzcd}\]
the square induced by taking the dependent product
% https://q.uiver.app/?q=WzAsNCxbMCwwLCJcXHByb2RfaSBGX2kiXSxbMiwwLCJcXHByb2RfaSBFX2kiXSxbMCwxLCJcXHByb2RfaSBBX2kiXSxbMiwxLCJcXHByb2RfaSBCX2kiXSxbMCwxLCJcXHByb2Rfe2k6SX0gXFx2YXJwaGlfaSJdLFswLDIsIlxccHJvZF97aTpJfSBcXHhpX2kiLDJdLFsyLDMsIlxccHJvZF97aTpJfSBrX2kiLDJdLFsxLDMsIlxccHJvZF97aTpJfSBcXHBpX2kiXV0=
\[\begin{tikzcd}
	{\prod_i F_i} && {\prod_i E_i} \\
	{\prod_i A_i} && {\prod_i B_i}
	\arrow["{\prod_{i:I} \varphi_i}", from=1-1, to=1-3]
	\arrow["{\prod_{i:I} \xi_i}"', from=1-1, to=2-1]
	\arrow["{\prod_{i:I} k_i}"', from=2-1, to=2-3]
	\arrow["{\prod_{i:I} \pi_i}", from=1-3, to=2-3]
\end{tikzcd}\]
is a cocartesian functor.
\end{prop}

\begin{proof}
	Recall that cocartesian lifts in the product families are computed pointwise. For an arrow $\sigma: \alpha \Rightarrow \alpha'$ of sections in $\prod_{i:I} A_i$ and a section $\vartheta:\prod_{i:I} \alpha^* P_i$ lying over we obtain
	\begin{align*}
		\big( \prod_{i:I} \varphi_i\big) \big( (\prod_{i:I} \xi_i)_!(\sigma,\vartheta) \big)
    &= \lambda i. \varphi_i(Q_i)_!(\sigma_i,\vartheta_i)
		 = \lambda i.(P_i)_!(k_i \sigma_i, \varphi_i \vartheta_i))\\
    &= \big( \prod_{i:I} \pi_i\big)_!(\kappa \sigma, \varphi \vartheta ).
    \qquad\qquad\text{($\varphi$ cocart.)}\qedhere
	\end{align*}
\end{proof}

\begin{prop}[Cocartesian fibrations are cotensored over maps/shape inclusions]\label{prop:cocart-fun-cotensor-maps}
	Let $P:B \to \UU$ be a cocartesian family with associated projection $\pi:E \fibarr B$. For any type map or shape inclusion $j:Y \to X$, the maps $\pi^X$ and $\pi^Y$ are cocartesian fibrations, and moreover the square
	% https://q.uiver.app/?q=WzAsNCxbMCwwLCJFXlgiXSxbMCwxLCJCXlgiXSxbMiwxLCJCXlkiXSxbMiwwLCJFXlkiXSxbMCwxLCIiLDAseyJzdHlsZSI6eyJoZWFkIjp7Im5hbWUiOiJlcGkifX19XSxbMSwyXSxbMCwzXSxbMywyLCIiLDIseyJzdHlsZSI6eyJoZWFkIjp7Im5hbWUiOiJlcGkifX19XV0=
	\[\begin{tikzcd}
		{E^X} && {E^Y} \\
		{B^X} && {B^Y}
		\arrow["\pi^X", swap, two heads, from=1-1, to=2-1]
		\arrow["B^j", from=2-1, to=2-3, swap]
		\arrow["E^j", from=1-1, to=1-3]
		\arrow["\pi^Y",two heads, from=1-3, to=2-3]
	\end{tikzcd}\]
is a cocartesian functor.
\end{prop}

\begin{proof}
By \cref{prop:cocart-fam-prod}, $P^X$ and $P^Y$ are cocartesian families with associated projections $\pi^X$ and $\pi^Y$, resp. Furthermore, said proposition tells us that cocartesian lifts in exponentials are computed pointwise. Invoking fibrant replacement and the axiom of choice, this gives the following.

Lifts in the cocartesian fibration
\[  E^Y \equiv \sum_{u:Y \to B} \prod_{y:Y} P(u(y)) \xfib{\pi^Y} B^Y \]
are given by
\[ (P^Y)_! \Big(\alpha:u\to u', f:\prod_Y u^*P \Big) :\jdeq \lambda y.P_!(\alpha_y:u_y \to u_y', f_y).\]
Lifts in the cocartesian fibration
\[ E^X \equiv \sum_{\substack{u:Y \to B \\ v: X \to B}} \sum_{p:j^*v=u} \sum_{f:\prod_Y u^*P} \sum_{g:\prod_X v^*P} j^*g=_p f \xfib{\pi^X} \sum_{\substack{u:Y \to B \\ v:X \to B}} j^*v=u \equiv B^X \]
``sitting above'' are given by
\begin{align*} & (P_!)^X\Big(\alpha:u\to u', f:\prod_Y u^*P, \beta:v \to v', g:\prod_X v^*P, \varphi:j^*\beta=\alpha, r:j^*g =_{\varphi_0} f\Big) \\
:\jdeq & \left \langle P_!^Y(\alpha,f), \lambda x.P_!(\beta_x,g_x), \lambda y.P_!^2(\varphi_y, r_y) \right \rangle
\end{align*}
where the $2$-dimensional lift of $\varphi$ w.r.t.~$r$ is constructed as follows:
% https://q.uiver.app/?q=WzAsMTAsWzMsMiwial4qdiJdLFs1LDIsInUiXSxbMywzLCJqXip2JyJdLFs1LDMsInUnIl0sWzMsMCwial4qZyJdLFszLDEsIlxcY2RvdCJdLFs1LDEsIlxcY2RvdCJdLFs1LDAsImYiXSxbMCwzLCJCXlkiXSxbMCwxLCJFXlkiXSxbMCwxLCIiLDAseyJzdHlsZSI6eyJoZWFkIjp7Im5hbWUiOiJub25lIn19fV0sWzAsMiwial4qXFxiZXRhIiwyXSxbMSwzLCJcXGFscGhhIl0sWzIsM10sWzQsNSwiKFBeWSlfKihqXipcXGJldGEsal4qZykiLDJdLFs1LDYsIiIsMix7InN0eWxlIjp7ImJvZHkiOnsibmFtZSI6ImRhc2hlZCJ9LCJoZWFkIjp7Im5hbWUiOiJub25lIn19fV0sWzQsNywiciIsMCx7InN0eWxlIjp7ImhlYWQiOnsibmFtZSI6Im5vbmUifX19XSxbNyw2LCIoUF5ZKV8qKFxcYWxwaGEsZikiXSxbOSw4LCIiLDAseyJzdHlsZSI6eyJoZWFkIjp7Im5hbWUiOiJlcGkifX19XSxbMTAsMTMsIlxcdmFycGhpIiwxLHsic2hvcnRlbiI6eyJzb3VyY2UiOjIwLCJ0YXJnZXQiOjIwfX1dLFsxNiwxNSwiKFBeWSlfKl4yKFxcdmFycGhpLHIpIiwxLHsic2hvcnRlbiI6eyJzb3VyY2UiOjIwLCJ0YXJnZXQiOjIwfSwic3R5bGUiOnsiYm9keSI6eyJuYW1lIjoiZGFzaGVkIn19fV1d
\[\begin{tikzcd}
	&&& {j^*g} && f \\
	{E^Y} &&& \cdot && \cdot \\
	&&& {j^*v} && u \\
	{B^Y} &&& {j^*v'} && {u'}
	\arrow[""{name=0, anchor=center, inner sep=0}, from=3-4, to=3-6,equals]
	\arrow["{j^*\beta}"', from=3-4, to=4-4]
	\arrow["\alpha", from=3-6, to=4-6]
	\arrow[""{name=1, anchor=center, inner sep=0}, from=4-4, to=4-6]
	\arrow["{(P^Y)_!(j^*\beta,j^*g)}"', from=1-4, to=2-4, cocart]
	\arrow[""{name=2, anchor=center, inner sep=0}, dashed, from=2-4, to=2-6, equals]
	\arrow[""{name=3, anchor=center, inner sep=0}, "r", from=1-4, to=1-6, equals]
	\arrow["{(P^Y)_!(\alpha,f)}", from=1-6, to=2-6, cocart]
	\arrow[two heads, from=2-1, to=4-1]
	\arrow["\varphi"{description}, shorten <=4pt, shorten >=4pt, Rightarrow, from=0, to=1]
	\arrow["{(P^Y)_!^2(\varphi,r)}"{description}, shorten <=4pt, shorten >=4pt, Rightarrow, dashed, from=3, to=2]
\end{tikzcd}\]
Since the functor $E^j$ maps tuples $\langle u,v,p,f,g, r \rangle$ in $E^X$ to pairs $\langle u,f \rangle$ in $E^Y$ one can readily verify that it is cocartesian, using the computations above.
\end{proof}

\begin{prop}[Cocartesian functors are closed under Leibniz cotensors]\label{prop:cocart-fun-leibniz}
	Let $j: Y \to X$ be a type map or shape inclusion. Then, given cocartesian fibrations $\pi: E \fibarr B$, $\xi: F \fibarr A$, and a cocartesian functor
	% https://q.uiver.app/?q=WzAsNCxbMCwwLCJGIl0sWzAsMSwiQSJdLFsyLDEsIkIiXSxbMiwwLCJFIl0sWzAsMSwiXFx4aSIsMix7InN0eWxlIjp7ImhlYWQiOnsibmFtZSI6ImVwaSJ9fX1dLFsxLDIsImsiLDJdLFswLDMsIlxcdmFycGhpIl0sWzMsMiwiXFxwaSIsMCx7InN0eWxlIjp7ImhlYWQiOnsibmFtZSI6ImVwaSJ9fX1dXQ==
	\[\begin{tikzcd}
		F && E \\
		A && B
		\arrow["\xi"', two heads, from=1-1, to=2-1]
		\arrow["k"', from=2-1, to=2-3]
		\arrow["\psi", from=1-1, to=1-3]
		\arrow["\pi", two heads, from=1-3, to=2-3]
	\end{tikzcd}\]
the square induced between the Leibniz cotensors
	% https://q.uiver.app/?q=WzAsNCxbMCwwLCJGXlgiXSxbMCwxLCJBXlgiXSxbMSwwLCJGXllcXHRpbWVzX3tFXll9IEVeWCJdLFsxLDEsIkFeWVxcdGltZXNfe0JeWX0gQl5YIl0sWzAsMSwiIiwwLHsic3R5bGUiOnsiaGVhZCI6eyJuYW1lIjoiZXBpIn19fV0sWzAsMl0sWzEsM10sWzIsMywiIiwyLHsic3R5bGUiOnsiaGVhZCI6eyJuYW1lIjoiZXBpIn19fV1d
	\[\begin{tikzcd}
		{F^X} & {F^Y\times_{E^Y} E^X} \\
		{A^X} & {A^Y\times_{B^Y} B^X}
		\arrow["\xi^X", two heads, from=1-1, to=2-1, swap]
		\arrow["j \cotens \psi", from=1-1, to=1-2]
		\arrow["j \cotens k", from=2-1, to=2-2, swap]
		\arrow["\xi^Y \times_{\pi^Y} \pi^X", two heads, from=1-2, to=2-2]
	\end{tikzcd}\]
is a cocartesian functor.
\end{prop}

\begin{proof}
This follows from using \cref{prop:cocart-fun-cotensor-maps} and then applying the universal property \cref{prop:cocart-fun-pb-ump}.
\end{proof}

To sum up the results from \cref{ssec:cocart-clos} and the ones above, our synthetic cocartesian families satisfy stability properties analogous to those of an $\infty$-cosmos, \cf~\cite[Definition~1.2.1]{RV}:

\begin{prop}[Cosmological closure properties of cocartesian families]\label{prop:cocart-cosm-closure}
	\leavevmode
	\begin{enumerate}
		\item Over Rezk bases, it holds that:
		\begin{itemize}
			\item[] Cocartesian families are closed under composition, dependent products, pullback along arbitrary maps, and cotensoring with maps/shape inclusions. Families corresponding to equivalences or terminal projections are always cocartesian.
		\end{itemize}
		\item Between cocartesian families over Rezk bases, it holds that:
		\begin{itemize}
			\item[] Cocartesian functors are closed under (both horizontal and vertical) composition, dependent products, pullback, sequential limits,\footnote{Here, all three objectwise limit notions are meant to satisfy the expected universal properties \emph{\wrt~to cocartesian functors}.} and Leibniz cotensors.
			\item[] Fibered equivalences and fibered functors into the identity of $\unit$ are always cocartesian.
		\end{itemize}
	\end{enumerate}
\end{prop}

\begin{proof}
	Cf.~the previous statements, as well as \cref{ssec:cocart-clos}, and \cref{prop:cocart-fun-fib-equiv}.
\end{proof}

We end the section with yet another useful stability property.

\begin{prop}[Base change of cocartesian functors, {\protect\cite[Exercise~5.3.i]{RV}}]\label{prop:cocart-fun-basechange}
	Let $A,B$ be a Rezk type and $j:A \to B$ be a map.

	Consider cocartesian families $P,P':B \to \UU$ and a cocartesian functor $\varphi:\CocartFun_B(P,P')$. Then the mediating map $\psi \jdeq \lambda a,e.\pair{ja}{e}:\CocartFun_A(j^*P,j^*P')$ is also a cocartesian functor.
We write $\pi:E \fibarr B$, $\pi':E' \fibarr B$ for the fibrations corresponding to $P$, $P'$, and similarly for the pullback fibrations:
% https://q.uiver.app/?q=WzAsNixbMSwxLCJGJyJdLFszLDEsIkUnIl0sWzEsMiwiQSJdLFszLDIsIkIiXSxbMCwwLCJGIl0sWzIsMCwiRSJdLFswLDFdLFswLDIsIlxceGknIiwwLHsic3R5bGUiOnsiaGVhZCI6eyJuYW1lIjoiZXBpIn19fV0sWzEsMywiXFxwaSciLDAseyJzdHlsZSI6eyJoZWFkIjp7Im5hbWUiOiJlcGkifX19XSxbNCwwLCJcXHdpZGV0aWxkZXtcXHBzaX0iLDIseyJzdHlsZSI6eyJib2R5Ijp7Im5hbWUiOiJkYXNoZWQifX19XSxbNCwyLCJcXHhpIiwyLHsiY3VydmUiOjEsInN0eWxlIjp7ImhlYWQiOnsibmFtZSI6ImVwaSJ9fX1dLFs1LDEsIlxcd2lkZXRpbGRle1xcdmFycGhpfSJdLFs1LDMsIlxccGkiLDIseyJjdXJ2ZSI6MSwic3R5bGUiOnsiaGVhZCI6eyJuYW1lIjoiZXBpIn19fV0sWzQsNV0sWzAsMywiIiwxLHsib2Zmc2V0IjoyLCJzdHlsZSI6eyJuYW1lIjoiY29ybmVyIn19XSxbNCwzLCIiLDEseyJvZmZzZXQiOi0xLCJzdHlsZSI6eyJuYW1lIjoiY29ybmVyIn19XSxbMiwzLCJqIiwyXV0=
\[\begin{tikzcd}
	{F} && {E} \\
	& {F'} && {E'} \\
	& {A} && {B}
	\arrow["{\xi'}", from=2-2, to=3-2, two heads]
\arrow["{\pi'}", from=2-4, to=3-4, two heads]
\arrow["{\widetilde{\psi}}"', swap, pos=0.8, from=1-1, to=2-2, dashed]
\arrow["{\xi}"', from=1-1, to=3-2, curve={height=6pt}, two heads]
\arrow["{\widetilde{\varphi}}", from=1-3, to=2-4]
\arrow["{\pi}"', from=1-3, to=3-4, curve={height=6pt}, two heads]
\arrow["\lrcorner"{very near start, rotate=0}, pos=0.125, from=2-2, to=3-4, shift right=2, phantom]
\arrow["\lrcorner"{very near start, rotate=0},pos=0.125, from=1-1, to=3-4, shift left=1, phantom]
\arrow["{j}"', from=3-2, to=3-4]
\arrow[from=2-2, to=2-4, crossing over]
\arrow[from=1-1, to=1-3, crossing over]
\end{tikzcd}\]
\end{prop}

\begin{proof}
	Let $v:a \to a'$ be an arrow in $A$, and $f:e \to_{jv} e'$ a dependent arrow in $P$. By definition, we have identifications:
	\begin{align}
		\psi_v(f) & = \varphi_{jv}(f) \label{eq:pb-fibw-f} \\
		\psi_a(e) & = \varphi_{ja}(e) \label{eq:pb-fibw-e}
	\end{align}
Since $\varphi$ is a cocartesian functor, there is a path
	\begin{align}
		\varphi_v(P_!(v,e)) = P_!'(v, \varphi_b e). \label{eq:pb-varphi-cocart} 
	\end{align}
The cocartesian lifts in the pullback family are given by
	\begin{align}
		Q_!'(v,e) = P_!'(jv,e). \label{eq:pb-lifts}
	\end{align}
Taken together, this gives rise to an identification
\[
	\psi_v(Q_!(v,e)) \stackrel{\ref{eq:pb-fibw-f}}{=} \varphi_{jv}(P_!(jv,e))   \stackrel{\ref{eq:pb-varphi-cocart}}{=} P_!'(jv,\varphi_{ja}(e))  \stackrel{\ref{eq:pb-fibw-e}, \ref{eq:pb-lifts}}{=} Q_!'(v,\psi_a e).\qedhere
\]
\end{proof}

\subsubsection{Characterizations of cocartesian functors}

\begin{theorem}[{\protect\cite[Theorem 5.3.4]{RV}}]\label{thm:cocart-fun-intl-char}
Let $A$ and $B$ be Rezk types, and consider cocartesian families $P:B \to \UU$ and $Q:A \to \UU$ with associated fibrations $\xi: F \fibarr  A$ and $\pi: E \fibarr B$, resp .

For a fibered functor $\Phi\defeq \pair{j}{\varphi}$ giving rise to a square
\[
	\begin{tikzcd}
		F \ar[r, "\varphi"] \ar[d, "\xi" swap] & E \ar[d, "\pi"] \\
		A \ar[r, "j" swap] & B
	\end{tikzcd}
\]
the following are equivalent:
\begin{enumerate}
	\item\label{it:cocart-fun-cocart} The fiberwise map $\Phi$ is a cocartesian functor.
	\item\label{it:cocart-fun-mate-fibered} The mate of the induced natural isomorphism, fibered over $j:A \to B$, is invertible, too:
	% https://q.uiver.app/?q=WzAsMTEsWzAsMCwiRiJdLFswLDEsIlxceGkgXFxkb3duYXJyb3cgQSJdLFsyLDAsIkUiXSxbMiwxLCJcXHBpIFxcZG93bmFycm93IEIiXSxbOCw1LCJcXGJ1bGxldCJdLFszLDBdLFszLDFdLFs0LDEsIlxceGkgXFxkb3duYXJyb3cgQSJdLFs0LDAsIkYiXSxbNiwwLCJFIl0sWzYsMSwiXFxwaSBcXGRvd25hcnJvdyBCIl0sWzAsMSwiaSIsMl0sWzAsMiwiXFx2YXJwaGkiXSxbMiwzLCJpJyJdLFsxLDIsIj0iLDAseyJsZW5ndGgiOjcwLCJsZXZlbCI6Mn1dLFs1LDYsIlxccmlnaHRzcXVpZ2Fycm93IiwxLHsic3R5bGUiOnsiYm9keSI6eyJuYW1lIjoibm9uZSJ9LCJoZWFkIjp7Im5hbWUiOiJub25lIn19fV0sWzcsOCwiXFxrYXBwYSJdLFs4LDksIlxcdmFycGhpIl0sWzcsMTAsIlxcdmFycGhpIFxcZG93bmFycm93IGoiLDJdLFsxMCw5LCJcXGthcHBhJyIsMl0sWzEwLDgsIj0iLDIseyJsZW5ndGgiOjcwLCJsZXZlbCI6Mn1dLFsxLDMsIlxcdmFycGhpIFxcZG93bmFycm93IGoiLDJdXQ==
	\[\begin{tikzcd}
		{F} && {E} & {} & {F} && {E} \\
		{\xi \downarrow A} && {\pi \downarrow B} & {} & {\xi \downarrow A} && {\pi \downarrow B}
		\arrow["{i}"', from=1-1, to=2-1]
		\arrow["{\varphi}", from=1-1, to=1-3]
		\arrow["{i'}", from=1-3, to=2-3]
		\arrow[Rightarrow, "{=}", from=2-1, to=1-3, shorten <=7pt, shorten >=7pt]
		\arrow["{\rightsquigarrow}" description, from=1-4, to=2-4, phantom, no head]
		\arrow["{\kappa}", from=2-5, to=1-5]
		\arrow["{\varphi}", from=1-5, to=1-7]
		\arrow["{\varphi \downarrow j}"', from=2-5, to=2-7]
		\arrow["{\kappa'}"', from=2-7, to=1-7]
		\arrow[Rightarrow, "{=}"', from=2-7, to=1-5, shorten <=7pt, shorten >=7pt]
		\arrow["{\varphi \downarrow j}"', from=2-1, to=2-3]
	\end{tikzcd}\]

	\item\label{it:cocart-fun-mate} The mate of the induced natural isomorphism is invertible, too:
% https://q.uiver.app/?q=WzAsMTAsWzAsMCwiRl57XFxEZWx0YV4xfSJdLFswLDEsIlxceGkgXFxkb3duYXJyb3cgQSJdLFsyLDEsIlxccGkgXFxkb3duYXJyb3cgQiJdLFsyLDAsIkVee1xcRGVsdGFeMX0iXSxbMywwXSxbMywxXSxbNCwxLCJcXHhpIFxcZG93bmFycm93IEEiXSxbNCwwLCJGXntcXERlbHRhXjF9Il0sWzYsMSwiXFxwaSBcXGRvd25hcnJvdyBCIl0sWzYsMCwiRV57XFxEZWx0YV4xfSJdLFswLDEsInIiLDJdLFsxLDIsIlxcdmFycGhpIFxcZG93bmFycm93IGoiLDJdLFswLDMsIlxcdmFycGhpXntcXERlbHRhXjF9Il0sWzMsMiwiciciXSxbNCw1LCJcXHJpZ2h0c3F1aWdhcnJvdyIsMSx7InN0eWxlIjp7ImJvZHkiOnsibmFtZSI6Im5vbmUifSwiaGVhZCI6eyJuYW1lIjoibm9uZSJ9fX1dLFsxLDMsIj0iLDAseyJsZW5ndGgiOjcwLCJsZXZlbCI6Mn1dLFs2LDcsIlxcZWxsIl0sWzYsOCwiXFx2YXJwaGkgXFxkb3duYXJyb3cgaiIsMl0sWzcsOSwiXFx2YXJwaGlee1xcRGVsdGFeMX0iXSxbOCw5LCJcXGVsbCciLDJdLFs4LDcsIj0iLDIseyJsZW5ndGgiOjcwLCJsZXZlbCI6Mn1dXQ==
\[\begin{tikzcd}
	{F^{\Delta^1}} && {E^{\Delta^1}} & {} & {F^{\Delta^1}} && {E^{\Delta^1}} \\
	{\xi \downarrow A} && {\pi \downarrow B} & {} & {\xi \downarrow A} && {\pi \downarrow B}
	\arrow["{r}"', from=1-1, to=2-1]
	\arrow["{\varphi \downarrow j}"', from=2-1, to=2-3]
	\arrow["{\varphi^{\Delta^1}}", from=1-1, to=1-3]
	\arrow["{r'}", from=1-3, to=2-3]
	\arrow["{\rightsquigarrow}" description, from=1-4, to=2-4, phantom, no head]
	\arrow[Rightarrow, "{=}", from=2-1, to=1-3, shorten <=7pt, shorten >=7pt]
	\arrow["{\ell}", from=2-5, to=1-5]
	\arrow["{\varphi \downarrow j}"', from=2-5, to=2-7]
	\arrow["{\varphi^{\Delta^1}}", from=1-5, to=1-7]
	\arrow["{\ell'}"', from=2-7, to=1-7]
	\arrow[Rightarrow, "{=}"', from=2-7, to=1-5, shorten <=7pt, shorten >=7pt]
\end{tikzcd}\]
\end{enumerate}

\end{theorem}

\begin{proof}
	Let $Q: A \to \UU$ be the straightening of $\xi$ and $P: B \to \UU$ the straightening of $\pi$.
\begin{description}
	\item[$\ref{it:cocart-fun-cocart} \iff \ref{it:cocart-fun-mate-fibered}$]: Consider the first (fibered) adjunction, where the mate of the canonical isomorphism cell is constructed through the following pasting diagram:
	% https://q.uiver.app/?q=WzAsNixbMCwwLCJcXHhpIFxcZG93bmFycm93IEEiXSxbMSwwLCJGIl0sWzEsMSwiXFx4aSBcXGRvd25hcnJvdyBBIl0sWzMsMCwiRSJdLFszLDEsIlxccGkgXFxkb3duYXJyb3cgQiJdLFs0LDEsIkUiXSxbMSwyLCJpIl0sWzAsMSwiXFxrYXBwYSJdLFsxLDMsIlxcdmFycGhpIl0sWzIsNCwiXFx2YXJwaGlcXGRvd25hcnJvdyBqIiwyXSxbMyw0LCJpJyIsMl0sWzMsNSwiIiwwLHsic3R5bGUiOnsiaGVhZCI6eyJuYW1lIjoibm9uZSJ9fX1dLFs0LDUsIlxca2FwcGEnIiwyXSxbMCwyLCIiLDAseyJzdHlsZSI6eyJoZWFkIjp7Im5hbWUiOiJub25lIn19fV0sWzIsMywiPSIsMCx7Imxlbmd0aCI6NzAsImxldmVsIjoyfV0sWzEzLDEsIlxcZXRhIiwwLHsibGVuZ3RoIjo3MH1dLFs0LDExLCI9IiwwLHsibGVuZ3RoIjo3MH1dXQ==
	\[\begin{tikzcd}
		{\xi \downarrow A} & {F} && {E} \\
		& {\xi \downarrow A} && {\pi \downarrow B} & {E}
		\arrow["{i}", from=1-2, to=2-2]
		\arrow["{\kappa}", from=1-1, to=1-2]
		\arrow["{\varphi}", from=1-2, to=1-4]
		\arrow["{\varphi\downarrow j}"', from=2-2, to=2-4]
		\arrow["{i'}"', from=1-4, to=2-4]
		\arrow[""{name=0, inner sep=0}, from=1-4, to=2-5, no head, equals]
		\arrow["{\kappa'}"', from=2-4, to=2-5]
		\arrow[""{name=1, inner sep=0}, from=1-1, to=2-2, no head, equals]
		\arrow[Rightarrow, "{=}", from=2-2, to=1-4, shorten <=7pt, shorten >=7pt]
		\arrow[Rightarrow, "{\eta}", from=1, to=1-2, shorten <=2pt, shorten >=2pt]
		\arrow[Rightarrow, "{=}", from=2-4, to=0, shorten <=1pt, shorten >=1pt]
	\end{tikzcd}\]
	The unit $\eta: \hom_{\xi \downarrow A}(\id_F, \kappa i)$ at $\pair{u:a \to a'}{d:P_a}$ is given as follows:
	% https://q.uiver.app/?q=WzAsNixbMCwwLCJkIl0sWzIsMCwidV8qZCJdLFswLDEsImEiXSxbMCwyLCJhJyJdLFsyLDIsImEnIl0sWzIsMSwiYSciXSxbMCwxLCJQXyoodSxkKSJdLFsyLDMsInUiLDJdLFszLDQsIiIsMCx7InN0eWxlIjp7ImhlYWQiOnsibmFtZSI6Im5vbmUifX19XSxbMiw1LCJ1Il0sWzUsNCwiIiwyLHsic3R5bGUiOnsiaGVhZCI6eyJuYW1lIjoibm9uZSJ9fX1dXQ==
	\[\begin{tikzcd}
		{d} && {u_!d} \\
		{a} && {a'} \\
		{a'} && {a'}
		\arrow["{P_!(u,d)}", from=1-1, to=1-3]
		\arrow["{u}"', from=2-1, to=3-1]
		\arrow[from=3-1, to=3-3, no head, equals]
		\arrow["{u}", from=2-1, to=2-3]
		\arrow[from=2-3, to=3-3, no head, equals]
	\end{tikzcd}\]
	The pasting $2$-cell can be identified with the natural transformation
	\[ \alpha: \hom_{F \to \pi \downarrow B}(\kappa' \circ \varphi \downarrow j, \varphi \circ \kappa) \]
	whose components at $\pair{u:a\to a'}{d:Pa}$ are given by the fillers:
	\[ \alpha_{\pair{u}{d}}:\jdeq (\kappa' \circ \varphi \downarrow j)\eta_{\pair{u}{d}}: \pair{jd'}{(ju)_!(\varphi d)} \to \pair{jd'}{\varphi(u_!d)} \]
	% https://q.uiver.app/?q=WzAsMTEsWzAsMSwiZCJdLFsyLDEsInVfKmQiXSxbMCwyLCJhIl0sWzIsMiwiYSciXSxbMywxXSxbMywyXSxbNCwyLCJqKGEpIl0sWzYsMiwiaihhJykiXSxbNCwxLCJcXHZhcnBoaShkKSJdLFs2LDEsIihqdSlfKihcXHZhcnBoaSBkKSJdLFs2LDAsIlxcdmFycGhpKHVfKmQpIl0sWzAsMSwiUV8qKHUsZCkiXSxbMiwzLCJ1Il0sWzQsNSwiXFxyaWdodHNxdWlnYXJyb3ciLDAseyJzdHlsZSI6eyJib2R5Ijp7Im5hbWUiOiJub25lIn0sImhlYWQiOnsibmFtZSI6Im5vbmUifX19XSxbNiw3LCJqKHUpIl0sWzgsOSwiUF8qKGp1LFxcdmFycGhpIGQpIiwyXSxbOSwxMCwiXFxhbHBoYV97XFxsYW5nbGUgdSxkIFxccmFuZ2xlfSIsMix7InN0eWxlIjp7ImJvZHkiOnsibmFtZSI6ImRhc2hlZCJ9fX1dLFs4LDEwLCJcXHZhcnBoaVxcYmlnKFBfKih1LGQpXFxiaWcpIl1d
	\[\begin{tikzcd}
		&&&&&& {\varphi(u_!d)} \\
		{d} && {u_!d} & {} & {\varphi(d)} && {(ju)_!(\varphi d)} \\
		{a} && {a'} & {} & {j(a)} && {j(a')}
		\arrow["{Q_!(u,d)}", from=2-1, to=2-3]
		\arrow["{u}", from=3-1, to=3-3]
		\arrow["{\rightsquigarrow}", from=2-4, to=3-4, phantom, no head]
		\arrow["{j(u)}", from=3-5, to=3-7]
		\arrow["{P_!(ju,\varphi d)}"', from=2-5, to=2-7, cocart]
		\arrow["{\alpha_{\langle u,d \rangle}}"', from=2-7, to=1-7, dashed]
		\arrow["{\varphi\big(P_!(u,d)\big)}", from=2-5, to=1-7]
	\end{tikzcd}\]
	If $\Phi \jdeq \pair{j}{\varphi}$ is a cocartesian functor there is an identification $\varphi(Q_!(u,d)) = P_!(ju, \varphi d)$ in $E^{\Delta^1}$, hence $\alpha_{\pair{u}{d}}$ is an identity.
	
	On the other hand, if the induced filler $\alpha_{\pair{u}{d}}$ happens to be an isomorphism, and thus an identity, we obtain an identification $\varphi(Q_!(u,d)) = P_!(ju, \varphi d)$ rendering $\Phi$ a cocartesian functor.
	
	\item[$\ref{it:cocart-fun-cocart} \iff \ref{it:cocart-fun-mate}$]: In case of the second adjunction, the mate is given by the cell constructed by pasting from the diagram:
	% https://q.uiver.app/?q=WzAsNixbMCwwLCJcXHhpXFxkb3duYXJyb3cgQSJdLFsxLDAsIkZee1xcRGVsdGFeMX0iXSxbMSwxLCJcXHhpXFxkb3duYXJyb3cgQSJdLFsyLDEsIlxcdmFycGhpIFxcZG93bmFycm93IEIiXSxbMiwwLCJFXntcXERlbHRhXjF9Il0sWzMsMSwiRV57XFxEZWx0YV4xfSJdLFswLDEsIlxcZWxsIl0sWzEsMiwiciIsMl0sWzIsMywiXFx2YXJwaGlcXGRvd25hcnJvdyBqIiwyXSxbMSw0LCJcXHZhcnBoaV57XFxEZWx0YV4xfSJdLFs0LDMsInInIiwyXSxbMyw1LCJcXGVsbCciLDJdLFswLDIsIiIsMSx7InN0eWxlIjp7ImhlYWQiOnsibmFtZSI6Im5vbmUifX19XSxbNCw1LCIiLDEseyJzdHlsZSI6eyJoZWFkIjp7Im5hbWUiOiJub25lIn19fV0sWzIsNCwiPSIsMCx7Imxlbmd0aCI6NzAsImxldmVsIjoyfV0sWzEyLDEsIj0iLDAseyJsZW5ndGgiOjcwfV0sWzMsMTMsIlxcdmFyZXBzaWxvbiciLDAseyJsZW5ndGgiOjcwfV1d
	\[\begin{tikzcd}
		{\xi\downarrow A} & {F^{\Delta^1}} & {E^{\Delta^1}} \\
		& {\xi\downarrow A} & {\varphi \downarrow B} & {E^{\Delta^1}}
		\arrow["{\ell}", from=1-1, to=1-2]
		\arrow["{r}"', from=1-2, to=2-2]
		\arrow["{\varphi\downarrow j}"', from=2-2, to=2-3]
		\arrow["{\varphi^{\Delta^1}}", from=1-2, to=1-3]
		\arrow["{r'}"', from=1-3, to=2-3]
		\arrow["{\ell'}"', from=2-3, to=2-4]
		\arrow[""{name=0, inner sep=0}, from=1-1, to=2-2, no head, equals]
		\arrow[""{name=1, inner sep=0}, from=1-3, to=2-4, no head, equals]
		\arrow[Rightarrow, "{=}", from=2-2, to=1-3, shorten <=3pt, shorten >=3pt]
		\arrow[Rightarrow, "{=}", from=0, to=1-2, shorten <=2pt, shorten >=2pt]
		\arrow[Rightarrow, "{\varepsilon'}", from=2-3, to=1, shorten <=1pt, shorten >=1pt]
	\end{tikzcd}\]
	The counit $\varepsilon':\hom_{E^{\Delta^1} \to E^{\Delta^1}}(\ell' \circ r', \id_{E^{\Delta^1}})$ at $\pair{v:b \to b'}{g:e \to e'}$ is given by the fillers:
	% https://q.uiver.app/?q=WzAsNSxbMCwyLCJiIl0sWzIsMiwiYiciXSxbMCwxLCJlIl0sWzIsMSwidl8qZSJdLFsyLDAsImUnIl0sWzAsMSwidiIsMl0sWzIsMywiUF8qKHYsZSkiLDJdLFsyLDQsImciXSxbMyw0LCJcXGJldGFfe1xcbGFuZ2xlIHt2LGd9IFxccmFuZ2xlfSIsMix7InN0eWxlIjp7ImJvZHkiOnsibmFtZSI6ImRhc2hlZCJ9fX1dXQ==
	\[\begin{tikzcd}
		&& {e'} \\
		{e} && {v_!e} \\
		{b} && {b'}
		\arrow["{v}"', from=3-1, to=3-3]
		\arrow["{P_!(v,e)}"', cocart, from=2-1, to=2-3]
		\arrow["{g}", from=2-1, to=1-3]
		\arrow["{\varepsilon'_{\langle {v,g} \rangle}}"', from=2-3, to=1-3, dashed]
	\end{tikzcd}\]
	The pasting cell can be identified with the natural transformation
	\[ \beta:\hom_{\xi \downarrow A \to F^{\Delta^1}}(\varphi \downarrow j \circ \ell', \varphi^{\Delta^1} \circ \ell)\]
	whose components at $\pair{u:a\to a'}{d:Pa}$ are given by:
	\[ \beta_{\pair{u}{d}} \jdeq \varepsilon'_{\varphi^{\Delta^1}(\ell \pair{u}{d})}: \pair{ju}{P_!(ju,\varphi(d))} \to \pair{ju}{\varphi Q_!(u,d)}  \]
	% https://q.uiver.app/?q=WzAsNSxbMCwyLCJqYSJdLFsyLDIsImphJyJdLFswLDEsIlxcdmFycGhpIGQiXSxbMiwxLCIoanUpXyooXFx2YXJwaGkgZCkiXSxbMiwwLCJcXHZhcnBoaSh1XypkKSJdLFswLDEsInZqdSIsMl0sWzIsMywiUF8qKGp1LFxcdmFycGhpIGQpIiwyXSxbMiw0LCJcXHZhcnBoaShRXyoodSxkKSkiXSxbMyw0LCJcXGJldGFfe1xcbGFuZ2xlIHt2LGd9IFxccmFuZ2xlfSIsMix7InN0eWxlIjp7ImJvZHkiOnsibmFtZSI6ImRhc2hlZCJ9fX1dXQ==
	\[\begin{tikzcd}
		&& {\varphi(u_!d)} \\
		{\varphi d} && {(ju)_!(\varphi d)} \\
		{ja} && {ja'}
		\arrow["{ju}"', from=3-1, to=3-3]
		\arrow["{P_!(ju,\varphi d)}"', cocart, from=2-1, to=2-3]
		\arrow["{\varphi(Q_!(u,d))}", from=2-1, to=1-3]
		\arrow["{\beta_{\langle {u,d} \rangle}}"', from=2-3, to=1-3, dashed]
	\end{tikzcd}\]
	If $\Phi \jdeq \pair{\varphi}{u}$ is a cocartesian functor, then $\varphi(Q_!(u,d))$ is a cocartesian arrow, so the filler $\beta_{\pair{u}{d}}$ must be too, thus an identity.
	
	Conversely, $\beta_{\pair{u}{d}}$ being an identity gives rise to an identification $\varphi(Q_!(u,d)) =  P_!(ju,\varphi d)$. This implies that $\Phi$ is a cocartesian functor.\qedhere
\end{description}	
\end{proof}

An application of the characterization theorem is a proof that any fibered left adjoint (over a common base) is a cocartesian functor.

\begin{lem}[$2$-cell conservativity for comma types, cf.~{\protect\cite[Proposition 3.4.6(iii)]{RV}}]\label{lem:comma-2cell-cons}
Let $g:C \rightarrow A \leftarrow B : f$ be a cospan of Rezk types. For generalized elements $\alpha, \alpha':X \to f \downarrow g$, consider a natural transformation $\tau: \alpha \Rightarrow \alpha'$. If $\partial_k \tau$  for both $k=0,1$ are natural isomorphisms, then $\tau$ is a natural isomorphism.
\end{lem}

\begin{proof}
We have an equivalence $X \to f \downarrow g \equiv \sum_{\substack{\alpha_0:X \to B \\ \alpha_1 : X \to C}} \prod_{x:X} f \alpha_0 x \to g \alpha_1 x$. The natural transformation $\tau$ at stage $x:X$ is given by a square:
% https://q.uiver.app/?q=WzAsNCxbMCwwLCJmIFxcYWxwaGFfMHgiXSxbMCwxLCJnIFxcYWxwaGFfMXgiXSxbMiwxLCJnIFxcYWxwaGFfMSd4Il0sWzIsMCwiZiBcXGFscGhhJ18weCJdLFswLDEsIlxcYWxwaGFfeCIsMl0sWzEsMiwiZyhcXHBhcnRpYWxfMVxcdGF1X3gpIiwyXSxbMCwzLCJmKFxccGFydGlhbF8wIFxcdGF1X3gpIl0sWzMsMiwiXFxhbHBoYSdfeCJdLFswLDIsIlxcdGF1X3giLDEseyJzdHlsZSI6eyJib2R5Ijp7Im5hbWUiOiJub25lIn0sImhlYWQiOnsibmFtZSI6Im5vbmUifX19XV0=
\[\begin{tikzcd}
	{f \alpha_0x} && {f \alpha'_0x} \\
	{g \alpha_1x} && {g \alpha_1'x}
	\arrow["{\alpha_x}"', from=1-1, to=2-1]
	\arrow["{g(\partial_1\tau_x)}"', from=2-1, to=2-3]
	\arrow["{f(\partial_0 \tau_x)}", from=1-1, to=1-3]
	\arrow["{\alpha'_x}", from=1-3, to=2-3]
	\arrow["{\tau_x}"{description}, draw=none, from=1-1, to=2-3]
\end{tikzcd}\]
Invoking transformation extensionality, if both whiskered $2$-cells $\partial_k \tau_x$ are isomorphisms, then by functoriality the squares $\tau_x$ degenerate to
% https://q.uiver.app/?q=WzAsNCxbMCwwLCJcXGNkb3QiXSxbMCwxLCJcXGNkb3QiXSxbMiwwLCJcXGNkb3QiXSxbMiwxLCJcXGNkb3QiXSxbMCwxLCJcXGFscGhhX3giLDJdLFswLDIsIiIsMix7InN0eWxlIjp7ImhlYWQiOnsibmFtZSI6Im5vbmUifX19XSxbMSwzLCIiLDAseyJzdHlsZSI6eyJoZWFkIjp7Im5hbWUiOiJub25lIn19fV0sWzIsMywiXFxhbHBoYSdfeCJdLFs1LDYsIlxcdGF1X3giLDEseyJzaG9ydGVuIjp7InNvdXJjZSI6MjAsInRhcmdldCI6MjB9LCJzdHlsZSI6eyJib2R5Ijp7Im5hbWUiOiJub25lIn0sImhlYWQiOnsibmFtZSI6Im5vbmUifX19XV0=
\[\begin{tikzcd}
	\cdot && \cdot \\
	\cdot && \cdot
	\arrow["{\alpha_x}"', from=1-1, to=2-1]
	\arrow[""{name=0, anchor=center, inner sep=0}, no head, from=1-1, to=1-3,equals]
	\arrow[""{name=1, anchor=center, inner sep=0}, no head, from=2-1, to=2-3,equals]
	\arrow["{\alpha'_x}", from=1-3, to=2-3]
	\arrow["{\tau_x}"{description}, Rightarrow, draw=none, from=0, to=1]
\end{tikzcd}\]
exhibiting $\tau$ as a natural isomorphism.
\end{proof}

\begin{prop}[Fibered left adjoints are cocartesian functors, {\protect\cite[Lemma~5.3.6]{RV}}]\label{prop:cocart-fun-fib-ladj}
Consider a fibered functor between cocartesian fibrations where the horizontal arrows make up a fibered adjunction:
% https://q.uiver.app/?q=WzAsNSxbMCwwLCJGIl0sWzIsMCwiRSJdLFszLDBdLFszLDJdLFsxLDIsIkIiXSxbMSwwLCJcXHBzaSIsMix7ImN1cnZlIjoyfV0sWzAsMSwiXFx2YXJwaGkiLDIseyJjdXJ2ZSI6Mn1dLFswLDQsIlxceGkiLDIseyJzdHlsZSI6eyJoZWFkIjp7Im5hbWUiOiJlcGkifX19XSxbMSw0LCJcXHBpIiwwLHsic3R5bGUiOnsiaGVhZCI6eyJuYW1lIjoiZXBpIn19fV0sWzYsNSwiIiwxLHsibGV2ZWwiOjEsInN0eWxlIjp7Im5hbWUiOiJhZGp1bmN0aW9uIn19XV0=
\[\begin{tikzcd}
	F && E & {} \\
	\\
	& B && {}
	\arrow[""{name=0, anchor=center, inner sep=0}, "\psi"', curve={height=12pt}, from=1-3, to=1-1]
	\arrow[""{name=1, anchor=center, inner sep=0}, "\varphi"', curve={height=12pt}, from=1-1, to=1-3]
	\arrow["\xi"', two heads, from=1-1, to=3-2]
	\arrow["\pi", two heads, from=1-3, to=3-2]
	\arrow["\dashv"{anchor=center, rotate=90}, draw=none, from=1, to=0]
\end{tikzcd}\]
Then $\varphi$ is a cocartesian functor.
\end{prop}

\begin{proof}
Since $\xi$ and $\pi$ are cocartesian, we obtain the following induced diagram:
% https://q.uiver.app/?q=WzAsNyxbMCwwXSxbMCwyXSxbMSwyLCJcXHhpIFxcZG93bmFycm93IEEiXSxbMywyLCJcXHBpIFxcZG93bmFycm93IEIiXSxbMywwLCJFXntcXERlbHRhXjF9Il0sWzEsMCwiRl57XFxEZWx0YV4xfSJdLFsyLDFdLFsyLDMsIlxcdmFycGhpXFxkb3duYXJyb3cgaiIsMix7ImN1cnZlIjoyfV0sWzMsMiwiXFxwc2kgXFxkb3duYXJyb3cgayIsMix7ImN1cnZlIjoyfV0sWzQsMywiciciLDJdLFs1LDQsIlxcdmFycGhpXntcXERlbHRhXjF9IiwyLHsiY3VydmUiOjJ9XSxbNCw1LCJcXHBzaV57XFxEZWx0YV4xfSIsMix7ImN1cnZlIjoyfV0sWzUsMiwiciJdLFsyLDUsIlxcZWxsIiwwLHsiY3VydmUiOi0zfV0sWzMsNCwiXFxlbGwnIiwyLHsiY3VydmUiOjN9XSxbMiw0LCI9IiwxLHsibGVuZ3RoIjo3MCwibGV2ZWwiOjJ9XSxbNyw4LCIiLDEseyJsZXZlbCI6MSwic3R5bGUiOnsibmFtZSI6ImFkanVuY3Rpb24ifX1dLFsxMCwxMSwiIiwyLHsibGV2ZWwiOjEsInN0eWxlIjp7Im5hbWUiOiJhZGp1bmN0aW9uIn19XSxbMTMsMTIsIiIsMCx7ImxldmVsIjoxLCJzdHlsZSI6eyJuYW1lIjoiYWRqdW5jdGlvbiJ9fV0sWzE0LDksIiIsMix7ImxldmVsIjoxLCJzdHlsZSI6eyJuYW1lIjoiYWRqdW5jdGlvbiJ9fV1d
\[\begin{tikzcd}
	{} & {F^{\Delta^1}} && {E^{\Delta^1}} \\
	&& {} \\
	{} & {\xi \downarrow B} && {\pi \downarrow B}
	\arrow["{\varphi\downarrow k}"{name=0, swap}, from=3-2, to=3-4, curve={height=12pt}]
	\arrow["{\psi \downarrow k}"{name=1, swap}, from=3-4, to=3-2, curve={height=12pt}]
	\arrow["{r'}"{name=2, swap}, from=1-4, to=3-4]
	\arrow["{\varphi^{\Delta^1}}"{name=3, swap}, from=1-2, to=1-4, curve={height=12pt}]
	\arrow["{\psi^{\Delta^1}}"{name=4, swap}, from=1-4, to=1-2, curve={height=12pt}]
	\arrow["{r}"{name=5}, from=1-2, to=3-2]
	\arrow["{\ell}"{name=6}, from=3-2, to=1-2, curve={height=-18pt}]
	\arrow["{\ell'}"{name=7, swap}, from=3-4, to=1-4, curve={height=18pt}]
	\arrow[Rightarrow, "{=}" description, from=3-2, to=1-4, shorten <=8pt, shorten >=8pt]
	\arrow["\dashv"{rotate=90}, from=0, to=1, phantom]
	\arrow["\dashv"{rotate=90}, from=3, to=4, phantom]
	\arrow["\dashv"{rotate=-1}, from=6, to=5, phantom]
	\arrow["\dashv"{rotate=-179}, from=7, to=2, phantom]
\end{tikzcd}\]
Specifically, the vertical adjunctions follow from the Chevalley criterion for cocartesian fibrations. The top horizontal adjunction exists because of \cref{prop:lari-closed-under-pi}. The bottom horizontal adjunction exists by \cref{prop:fib-adj-pb}.
We now have to show that the the mate w.r.t.~the vertical adjunctions is invertible, which by \cref{prop:inv-mates-composite-adjs} is equivalent to the mate w.r.t.~the horizontal adjunctions being invertible. But this is established by $2$-cell conservativity of comma objects, cf.~\cref{lem:comma-2cell-cons}.
\end{proof}

Any fibered equivalence between arbitrary maps is also cocartesian.

\begin{prop}[Fibered equivalences are cocartesian functors]\label{prop:cocart-fun-fib-equiv}
For (cocartesian) fibrations $\xi: F \fibarr A$ and $\pi:E \fibarr B$, any square of the form
% https://q.uiver.app/?q=WzAsNCxbMCwwLCJGIl0sWzAsMSwiQSJdLFsyLDEsIkIiXSxbMiwwLCJFIl0sWzAsMSwiXFx4aSIsMix7InN0eWxlIjp7ImhlYWQiOnsibmFtZSI6ImVwaSJ9fX1dLFsxLDIsIlxcc2ltZXEiLDJdLFswLDMsIlxcc2ltZXEiXSxbMywyLCJcXHBpIiwwLHsic3R5bGUiOnsiaGVhZCI6eyJuYW1lIjoiZXBpIn19fV1d
\[\begin{tikzcd}
	F && E \\
	A && B
	\arrow["\xi"', two heads, from=1-1, to=2-1]
	\arrow["\simeq"', from=2-1, to=2-3]
	\arrow["\simeq", from=1-1, to=1-3]
	\arrow["\pi", two heads, from=1-3, to=2-3]
\end{tikzcd}\]
is a cocartesian functor.
\end{prop}

\begin{proof}
	By univalence we can replace anonymous equivalences by identities (\emph{equivalence induction}), and these are easily seen to be cocartesian functors (since they preserve cocartesian lifts).
\end{proof}

Alternatively, in the case of cocartesian fibrations, one could argue similarly as in~\cref{prop:cocart-fun-fib-ladj}, cf.~\cite[Corollary~5.3.1]{RV}

\section{More on covariant families}\label{sec:cov-fam}

Covariant families have been introduced and thoroughly analyzed by Riehl--Shulman, cf.~\cite[Section~8]{RS17}. Since they can be characterized by a right orthogonality condition (namely \wrt~the map $0: \unit \hookrightarrow \Delta^1$), it follows formally that they satisfy all the closure properties from Subsection~\ref{ssec:clos-orth}, yielding the analog of the $\infty$-cosmological closure properties, \cf~\cref{prop:cov-cosm-clos}. As a perhaps rather non-obvious result, we prove that covariant families (over an arbitrary type) are automatically inner. Furthermore, we prove the expected result that, over a Rezk type, a family is covariant if and only if it is cocartesian and all its fibers are discrete. We also argue how covariant families can be seen as a type-theoretic analogue of discrete cocartesian fibrations in the $\infty$-cosmological sense.

The section concludes by establishing a directed version of the \emph{encode-decode method}~\cite[Section 8.9]{hottbook}, which we expect to be useful when analyzing localization or other higher inductive types in later work.

In the following section, the Yoneda Lemma for covariant families, originally established in~\cite[Section 9]{RS17}, is recovered in \cref{ssec:yoneda} as a special case of the Yoneda Lemma for cocartesian families.

\subsection{Properties and characterizations of covariant families}

\subsubsection{Covariant vs.~inner families}\label{ssec:cov-vs-inner-fam}

\begin{prop}\label{prop:covfib-is-innerfib}
	Every covariant family is an inner family.
\end{prop}
Since inner families are a relative version of Segal types the proposition at hand can be seen as a relative version of~\cite[Proposition 7.3]{RS17}.
\begin{proof}
	Let $P : B \to \UU$ be a covariant family.
	Suppose we are given a
	$2$-simplex $\sigma$ in $B$ with boundary given by
	$u:\hom_B(b,b')$, $v:\hom_B(b',b'')$, and $w:\hom_B(b,b'')$, for
	$b,b',b'':B$, \ie, $\sigma:\hom_B^2(u,v;w)$. Furthermore, assume
	there are dependent arrows $f:\hom^P_u(e,e')$,
	$g:\hom^{P}_v(e',e'')$ for $e:P\,b$, $e':P\,b'$, $e'':P\,b''$.
	%	\missingfigure{lift $2$-cell with prescribed inner horn}
	We prove that there is a contractible choice of elements of the
	extension type
	\begin{equation}\label{eq:cov-lift-type}
		\ext[\Big]{\prod_{t_2\le t_1}P(\sigma(t_1,t_2))}{\Lambda^2_1}{[f,g]}
	\end{equation}
	by giving a \emph{center} and by showing that any two elements are equal.\footnote{In general, a type $A$ is contractible iff the type $\isContr(A) \defeq \sum_{c:A} \prod_{x:A} (c=_Ax)$ is inhabited. For any pair $\pair{c}{H_c}$ we call $c:A$ a center of contraction of $A$ and $H_c$ a contraction of $A$. It is common to strengthen the articles to be \emph{definite}: contractible data is unique up to contractibility anyway, and this is the usual sense in which data in HoTT is deemed unique. We therefore allow ourselves to speak of \emph{the} center of contraction of $A$ \emph{etc.}, \cf~\cite[Definition~10.1.1]{RijIntro}.}

	The center is obtained using the homotopy extension property (HEP),
	cf.~\cite[Proposition 4.10]{RS17}, \cref{prop:hep}. Since $P$ is covariant, there is an element
	\[
	H : \prod_{b_1: B}\prod_{u_1:\hom_B(b,b_1)}
	\isContr\ext[\Big]{\prod_{t}P(u_1(t))}{0}{e}.
	\]
	To use the HEP, we first give an element of
	$\prod_{t_2\le t_1}P(\sigma(t_1,t_2))$. If $t_2$ is given, consider
	$K(t_2) :\jdeq H\,\sigma(1,t_2)\,(\lambda t.\,\sigma(t,t\land
	t_2))$. The corresponding center applied to $t_1$ is in
	$P(\sigma(t_1,t_1\land t_2)) \jdeq P(\sigma(t_1,t_2))$. We now need
	to provide a homotopy between this and $[f,g]$, but this is
	obtained from the contraction part of $H$.

	To show that any two elements $\tilde\sigma_1,\tilde\sigma_2$
	in~\eqref{eq:cov-lift-type} are equal, we use relative function
	extensionality~\cite[Proposition 4.8]{RS17}, \cref{prop:hep}, by which it suffices to give
	an element of the extension type
	\[
	\ext[\Big]{\prod_{t_2\le t_1}
		\tilde\sigma_1(t_1,t_2) = \tilde\sigma_2(t_1,t_2)}{\Lambda^2_1}{\refl}
	\]
	Again we appeal to HEP. Given $t_2\le t_1$, from $K(t_2)$ we get an
	identity between $\tilde\sigma_1$ and $\tilde\sigma_2$ restricted to
	the extension type $\ext{\prod_t P(t,t\land t_2)}{0}{e}$. From the
	easy direction of relative function extensionality we get an
	identity $\tilde\sigma_1(t_1,t_2) = \tilde\sigma_2(t_1,t_2)$.
	To show that the resulting homotopy is homotopic to $\refl$ on
	$\Lambda^2_1$, we use tope disjunction elimination.

	On the edge $t_2=0$ we have an identity $f=f$ in the contractible
	type $\ext{\prod_tP(\sigma(t,0))}{0}{e}$. But an identity
	type in a contractible type is itself contractible, so this identity
	is equal to $\refl$.

	On the edge $t_1=1$ we may consider an arbitrary $t_2$ and use the
	same argument.
\end{proof}

\subsubsection{Covariant vs.~cocartesian families}

\begin{prop}
	Any covariant family over a Rezk type $B$ is a cocartesian family.
\end{prop}

\begin{proof}
A family $P:B \to \UU$ is covariant if and only if the Leibniz cotensor~$i_0 \cotens \pi_P: \totalty{P}^{\Delta^1} \to B^{\Delta^1} \times_{B} \totalty{P}$ is an equivalence. By \cref{prop:adj-equiv} this also constitutes an adjoint equivalence, and by fibrant replacement the left adjoint can be chosen a (strict) section.
\end{proof}

\begin{prop}\label{prop:disc-fib-and-cocart-lifting-implies-cov}
	Let $B$ be a Segal type. Assume $P : B \to \UU$ is such that the following properties are satisfied:
	\begin{enumerate}
		\item $P$ is an inner family.
		\item $P$ has the cocartesian lifting property.
		\item All fibers $P\,b$ for $b:B$ are discrete.
	\end{enumerate}
	Then $P$ is covariant.
\end{prop}

\begin{proof}
	Let $b,b':B$ with $u:\hom_B(b,b')$, and $e:P\,b$. There is $e':P\,b'$ and a lift $f:\hom^P_u(e,e')$ which is cocartesian, meaning that for any $e'':P\,b'$ and $g:\hom^P_u(e,e'')$ there is a unique arrow~$h:\hom^{P}_{\id_{b'}}(e',e'')$ such that $h \circ f = g$. This dependent composition becomes an ordinary composition in the total space $\totalty{P}$ which is Segal, since $B$ is. Now, as $\hom^{P}_{\id_{b'}}(e',e'') \equiv \hom_{P\,b'}(e',e'')$, and $P\,b'$ is discrete, $h$ can be taken to be an identity. But then $f = g$.
\end{proof}

\begin{cor}\label{cor:cov-disc-fib}
A cocartesian family over a Segal type is covariant if and only if all its fibers are discrete.
\end{cor}

\begin{prop}\label{cor:cov-all-cocart}
A cocartesian family is covariant if and only if all dependent arrows are cocartesian.
\end{prop}

\begin{proof}
Let $P: B \to \UU$ be a cocartesian family.

If $P$ is covariant, then the unique lift for any $u:b \to b'$ in $B$ \wrt~$e:P\,b$ must be (identifiable with) a cocartesian arrow. Since any dependent arrow is given as a lift of some arrow in the base, it follows that any dependent arrow is cocartesian.

Conversely, if every dependent arrow in $P$ is cocartesian, any two lifts of $u:b \to b'$ beginning at some $e:P\,b$ become equal (since the filler induced by the universal property is exhibited as an isomorphism).
\end{proof}

\subsection{Closure properties of covariant families}

\begin{prop}[Cosmological closure properties of covariant families]\label{prop:cov-cosm-clos}
	\leavevmode
	\begin{enumerate}
	\item Over Rezk bases, it holds that:
		\begin{itemize}
			\item[]	Covariant families are closed under composition, dependent products, pullback along arbitrary maps, and cotensoring with maps/shape inclusions. Furthermore, they are closed under sequential limits, and satisfy left cancelation. Families corresponding to equivalences or terminal projections are always covariant.
		\end{itemize}
	\item Between covariant families over Rezk bases, it holds that:
		\begin{itemize}
			\item[] Fibered functors are closed under (both horizontal and vertical) composition, dependent products, pullback, sequential limits,\footnote{all three objectwise limit notions satisfying the expected universal properties \wrt~to fibered functors} and Leibniz cotensors.
			\item[] Fibered equivalences and fibered functors into the identity of $\unit$ are always covariant.
		\end{itemize}
	\end{enumerate}
\end{prop}

\begin{proof}
Most of the properties follow from the established results about cocartesian fibrations \cref{prop:cocart-cosm-closure}, or the additional closure properties that hold for right classes. Recall that fibered functors between covariant families are automatically cocartesian.
We are only left with stability under objectwise pullback (which will imply stability under objectwise sequential limits). We assume a diagram as follows whose vertically drawn maps are covariant fibrations:
% https://q.uiver.app/?q=WzAsMTEsWzEsMF0sWzEsMl0sWzIsMywiQiJdLFszLDBdLFsyLDEsIkU6XFxlcXVpdlxcd2lkZXRpbGRle1B9Il0sWzEsM10sWzEsMV0sWzAsMSwiRSc6XFxlcXVpdlxcd2lkZXRpbGRle1AnfSJdLFswLDMsIkInIl0sWzQsMSwiRScnOlxcZXF1aXZcXHdpZGV0aWxkZXtQJyd9Il0sWzQsMywiQicnIl0sWzQsMiwiIiwwLHsic3R5bGUiOnsiaGVhZCI6eyJuYW1lIjoiZXBpIn19fV0sWzcsOCwiIiwwLHsic3R5bGUiOnsiaGVhZCI6eyJuYW1lIjoiZXBpIn19fV0sWzcsNF0sWzgsMl0sWzksMTAsIiIsMCx7InN0eWxlIjp7ImhlYWQiOnsibmFtZSI6ImVwaSJ9fX1dLFsxMCwyXSxbOSw0XV0=
\[\begin{tikzcd}
	& {} && {} \\
	{E'\defeq \widetilde{P'}} & {} & {E \defeq \widetilde{P}} && {E'' \defeq \widetilde{P''}} \\
	& {} \\
	{B'} & {} & B && {B''}
	\arrow[two heads, from=2-3, to=4-3]
	\arrow[two heads, from=2-1, to=4-1]
	\arrow[from=2-1, to=2-3]
	\arrow[from=4-1, to=4-3]
	\arrow[two heads, from=2-5, to=4-5]
	\arrow[from=4-5, to=4-3]
	\arrow[from=2-5, to=2-3]
\end{tikzcd}\]
We are to prove that the induced map $E' \times_E E'' \to B' \times_B B''$ is a covariant fibration as well.
To this end, we observe the following. Given a triple of arrows $\angled{u,u', u''}$ in $B' \times_B B''$ with points $\angled{e_0,e'_0,e''_0}:E' \times_E E''$ lying over the respective domains the cocartesian lift is given fiberwise by
\[ \angled{P_!(u,e_0), P'_!(u',e_0'), P''_!(u'',e_0'')}.\]
Due to covariance all the components are uniquely determined (after projecting), and so the lift in $E' \times_E E'' \to B' \times_B B''$ is as well.
\end{proof}

\subsection{Covariant families as discrete objects}

In $\infty$-cosmos theory, covariant families arise as a kind of discrete objects. Their analogues in simplicial type theory satisfy a corresponding property as well.

Namely, over a fixed base, a discrete cocartesian fibration in an $\infty$-cosmos is defined as a cocartesian fibration (relative to the given $\infty$-cosmos) that is required to be discrete as an object of the slice $\infty$-cosmos~\cite[Proposition~1.2.22]{RV}. There are two equivalent ways of expressing the latter. In our current type-theoretic setting, in absence of categorical universes, we cannot give a systematic account of these definitions, but we can still give appropriate translations of both of these criteria, and prove they are satisfied by discrete covariant families.

To state these principles, we need to type-theoretically capture some of the structure of sliced $\infty$-cosmoses, namely sliced versions of function types as well as simplicial cotensors.

The enrichment over categories is simply given by the object of fiberwise maps (\cf~\cref{ssec:fib-fun-defn}), itself a Rezk type:

\begin{defn}[Sliced function type, \protect{\cite[Proposition~1.2.22(ii)]{RV}}]
For two maps $\pi:E\to B$, $\xi:F \to B$ over a common base type $B$, the \emph{sliced function type (over $B$)} is given by the pullback object:
	% https://q.uiver.app/?q=WzAsNCxbMCwwLCJcXHVuZGVybGluZXtcXG1hdGhybXtGdW59fV9CKFxceGksXFxwaSkiXSxbMCwxLCIxIl0sWzIsMSwiQl5GIl0sWzIsMCwiRV5GIl0sWzAsMV0sWzEsMiwiXFx4aSIsMl0sWzAsM10sWzMsMiwiXFxwaV5GIl0sWzAsMiwiIiwxLHsic3R5bGUiOnsibmFtZSI6ImNvcm5lciJ9fV1d
\[\begin{tikzcd}
	{\mathrm{Fun}_B(\xi,\pi)} && {E^F} \\
	{\unit} && {B^F}
	\arrow[from=1-1, to=2-1]
	\arrow["{\xi}"', from=2-1, to=2-3]
	\arrow[from=1-1, to=1-3]
	\arrow["{\pi^F}", from=1-3, to=2-3]
	\arrow["\lrcorner"{very near start, rotate=0}, from=1-1, to=2-3, phantom]
\end{tikzcd}\]
\end{defn}

Taking the sliced cotensor amounts to taking the usual exponential of each fiber:

\begin{defn}[Sliced cotensor, \protect{\cite[Proposition~1.2.22(vi)]{RV}}]
	Let $\pi:E\to B$ be a map, and $X$ be a type or shape. The \emph{sliced exponential (over $B$) of $\pi$ by $X$} is given by the map $X \iexp E \to B$ defined as:
% https://q.uiver.app/?q=WzAsNCxbMCwwLCJYIFxcYm94dGltZXMgRSJdLFswLDEsIkIiXSxbMiwxLCJCXlgiXSxbMiwwLCJFXlgiXSxbMCwxXSxbMSwyLCJcXG1hdGhybXtjc3R9IiwyXSxbMCwzXSxbMywyLCJcXHBpXlgiXSxbMCwyLCIiLDEseyJzdHlsZSI6eyJuYW1lIjoiY29ybmVyIn19XV0=
\[\begin{tikzcd}
	{X \boxtimes E} && {E^X} \\
	{B} && {B^X}
	\arrow[from=1-1, to=2-1]
	\arrow["{\mathrm{cst}}"', from=2-1, to=2-3]
	\arrow[from=1-1, to=1-3]
	\arrow["{\pi^X}", from=1-3, to=2-3]
	\arrow["\lrcorner"{very near start, rotate=0}, from=1-1, to=2-3, phantom]
\end{tikzcd}\]
\end{defn}

If $X$ is a shape $\Phi$, we can take the strict extension type ${\Phi\boxtimes E}$ which is fibered equivalent to the dependent pair type along the powers of the fibers, ~\ie, over $B$, we have an equivalence
% https://q.uiver.app/?q=WzAsMyxbMCwwLCJcXFBoaVxcYm94dGltZXMgRSJdLFsxLDEsIkIiXSxbMiwwLCJcXHN1bV97YjpCfSBcXFBoaSBcXHRvIFBcXCxiIl0sWzAsMSwiIiwwLHsic3R5bGUiOnsiaGVhZCI6eyJuYW1lIjoiZXBpIn19fV0sWzAsMiwiXFxzaW1lcSJdLFsyLDEsIiIsMix7InN0eWxlIjp7ImhlYWQiOnsibmFtZSI6ImVwaSJ9fX1dXQ==
\[\begin{tikzcd}
	{\Phi\boxtimes E} && {\sum_{b:B} \Phi \to P\,b} \\
	& {B}
	\arrow[from=1-1, to=2-2, two heads]
	\arrow["{\equiv}", from=1-1, to=1-3]
	\arrow[from=1-3, to=2-2, two heads]
\end{tikzcd}\]
where $P:B \to \UU$ denotes the family of fibers of $\pi$.

In particular, the sliced exponential \wrt~$\Phi \jdeq \Delta^1$ has as total type the type of vertical arrows.

\begin{prop}[Covariant families are cosmologically discrete, \cf~\protect{\cite[Definition~5.5.3]{RV}}]
	Over a Rezk type, for a cocartesian family $P:B \to \UU$ with associated projection $\pi:E \to B$ the following are equivalent:
	\begin{enumerate}
		\item\label{it:disc-orth} The map $\pi$ is $i_0$-orthogonal.
		\item\label{it:disc-shape} The map induced by the non-degenerate map $s:\Delta^1 \to \walkBinv$ is a fibered equivalence:
		% https://q.uiver.app/?q=WzAsMyxbMCwwLCJcXG1hdGhiYiBFIFxcYm94dGltZXMgRSJdLFsyLDAsIlxcRGVsdGFeMSBcXGJveHRpbWVzIEUiXSxbMSwxLCJCIl0sWzAsMSwiXFxzaW1lcSJdLFswLDIsIiIsMix7InN0eWxlIjp7ImhlYWQiOnsibmFtZSI6ImVwaSJ9fX1dLFsxLDIsIiIsMCx7InN0eWxlIjp7ImhlYWQiOnsibmFtZSI6ImVwaSJ9fX1dXQ==
		\[\begin{tikzcd}
			{\mathbb E \boxtimes E} && {\Delta^1 \boxtimes E} \\
			& {B}
			\arrow["{\equiv}", from=1-1, to=1-3]
			\arrow[from=1-1, to=2-2, two heads]
			\arrow[from=1-3, to=2-2, two heads]
		\end{tikzcd}\]
			\item\label{it:disc-fun} For all global elements $b:\unit \to B$, the sliced function type $\Fun_B(b,\pi)$ is discrete.
	\end{enumerate}
	Furthermore, if $P:B \to \UU$ is a covariant family, then, for any map $\xi: F \to B$, the sliced function type $\Fun_B(\xi,\pi)$ is discrete.
\end{prop}

\begin{proof} ~ \\
	\begin{enumerate}
		\item[$\ref{it:disc-orth} \iff \ref{it:disc-shape}$:] Criterion $\ref{it:disc-shape}$ is equivalent to $\prod_{b:B} \isEquiv(s \to P\,b)$, $s \to P\,b:(\walkBinv \to P\,b) \to (\Delta^1 \to P\,b)$, which precisely says that all fibers of $P$ are discrete. Given that $P$ is cocartesian, this is equivalent to $P$ being covariant by \cref{cor:cov-disc-fib}.
		\item[$\ref{it:disc-orth} \iff \ref{it:disc-fun}$:] Similarly, since $\Fun_B(b,\pi) \equiv P\,b$, Criterion $\ref{it:disc-fun}$ is also equivalent to $P$ being covariant.
	\end{enumerate}
Now, assume $P$ is covariant. Given a map $\xi:F \to B$, the exponential $\pi^F: E^F \to B^F$ is a covariant fibration, too, hence the fiber $\xi^* \pi^F \equiv \Fun_B(\xi,\pi)$ is discrete.
\end{proof}

\subsection{Directed encode-decode}

We point out a directed version of the
``encode-decode method'' in order to characterize
$\hom$-types in higher inductive types
(and localizations).

The encode-decode method is used for a type $A$ with $a:A$,
together with a family $P : A \to \UU$ with $d: P(a)$.
The elimination rule for identity types gives a fiberwise
map $\varphi : \prod_{x:A}(a=x \to P(x))$,
sending the reflexivity path to $d$.
Since $\sum_{x:A}a=x$ is contractible,
and a fiberwise map is a fiberwise equivalence if and only if
the map on total types is an equivalence,
we get that $\varphi$ is a fiberwise equivalence
if and only if $\sum_{x:A}P(x)$ is contractible.

This has the following directed analog:
\begin{theorem}
	Let $A$ be a Segal type with $a:A$, and let
	$P : A \to \UU$ be a covariant family with $d:P(a)$.
	The fiberwise map
	$\varphi : \prod_{x:A}(\hom_A(a,x) \to P(x))$
	given by $\varphi(x,f) := f_!\,x$,
	is a fiberwise equivalence if and only if
	$\pair ad$ is an initial object in $\sum_{x:A}P(x)$.\footnote{Recall that $\sum_{x:A}P(x)$ is a Segal type.}
\end{theorem}

\begin{proof}
	By~\cite[Lemma~9.8]{RS17}, $\sum_{x:A}\hom_A(a,x)$ has $\pair{a}{\id_a}$
	as initial object.

	Conversely, if $\pair ad$ is initial in $\sum_{x:A}P(x)$,
	then each type $\sum_{f:\hom_A(a,x)}\dhom^P_f(d,e)$
	is contractible, where $x:A$ and $e:P(x)$.
	By~\cite[Lemma~8.13]{RS17},
	we have equivalences $\dhom^P_f(d,e) \equiv (f_!\,d = e)$,
	and hence we get that the fiber of $\varphi_x$ at $e$,
	which is $\sum_{f:\hom_A(a,x)}(f_!\,d=e)$,
	is equivalent to $\sum_{f:\hom_A(a,x)}\dhom^P_f(d,e)$,
	and is hence contractible.
\end{proof}

\section{Yoneda Lemma for cocartesian families}\label{sec:yoneda}

In our synthetic setting, we prove a version of the Yoneda Lemma for cocartesian fibrations which reads as follows: For a Rezk type $B$ and a cocartesian family $P:B \to \UU$, evaluation at the identity arrow
is an equivalence:
\[ \prod_{b:B} \isEquiv \left( \Big( \prod^\cocart_{u:\comma{b}{B}} P(\partial_1\,u) \Big) \stackrel{\ev_{\id_b}}{\longrightarrow} P\,b\right)\]
Generalizing work by Street~\cite{StrBicat}, this Yoneda Lemma has been formulated and proved by Riehl--Verity w.r.t.~$\infty$-cosmoses~\cite[Theorem~5.7.3]{RV}, showing that it applies to cocartesian fibrations of $(\infty,n)$-categories for $0 \le n \le \infty$. For our proof, we adapt Riehl--Verity's methods to the type-theoretic setting, with significant simplifications because we are only considering $\inftyone$-categories. At the same time, this generalizes Riehl--Shulman's type-theoretic Yoneda Lemma for \emph{discrete} covariant fibrations~\cite[Theorems~9.1, 9.5]{RS17}.\footnote{In the discrete covariant case, the statement reads exactly the same, only without the restriction to cocartesian sections which has become vacuous.} A semantic version of the fibered Yoneda Lemma for discrete (co-)cartesian fibrations over Segal spaces has also been established by Rasekh~\cite[Theorem~3.49]{RasYon}.

As discussed by Riehl--Shulman, the synthetic Yoneda Lemma can be understood as a \emph{directed} path induction principle, now also applying to the case of \emph{categorical} fibers. In fact---following Riehl--Verity as well as Riehl--Shulman---the Yoneda Lemma will be implied by a general result about cocartesian fibrations $P:B \to \UU$ over a base $B$ with initial element $b:B$. Namely, evaluation at $b$ induces a LARI adjunction between the fiber and the type of sections which yields a (quasi-)equivalence when restricted to the cocartesian sections:
% https://q.uiver.app/?q=WzAsNCxbMCwwLCJcXHByb2RfQl57XFxtYXRocm17Y29jYXJ0fX0gUCJdLFsyLDAsIlBcXCxiIl0sWzEsMV0sWzEsMiwiXFxwcm9kX0IgUCJdLFswLDEsIlxcbWF0aHJte2V2fV9iIiwyLHsiY3VydmUiOjJ9XSxbMSwwLCJcXG1hdGhiZnt5fSIsMix7ImN1cnZlIjoyfV0sWzAsMywiIiwyLHsic3R5bGUiOnsidGFpbCI6eyJuYW1lIjoiaG9vayIsInNpZGUiOiJ0b3AifX19XSxbMywxLCJcXG1hdGhybXtldn1fYiIsMix7ImN1cnZlIjoyfV0sWzEsMywiXFxtYXRoYmZ7eX0iLDIseyJjdXJ2ZSI6Miwic3R5bGUiOnsiYm9keSI6eyJuYW1lIjoiZG90dGVkIn19fV0sWzUsNCwiXFxzaW1lcSIsMSx7ImxldmVsIjoxLCJzdHlsZSI6eyJib2R5Ijp7Im5hbWUiOiJub25lIn0sImhlYWQiOnsibmFtZSI6Im5vbmUifX19XSxbOCw3LCIiLDIseyJsZXZlbCI6MSwic3R5bGUiOnsibmFtZSI6ImFkanVuY3Rpb24ifX1dXQ==
\[\begin{tikzcd}
	{\prod_B^{\mathrm{cocart}} P} && {P\,b} \\
	& {} \\
	& {\prod_B P}
	\arrow[""{name=0, anchor=center, inner sep=0}, "{\mathrm{ev}_b}"', curve={height=12pt}, from=1-1, to=1-3]
	\arrow[""{name=1, anchor=center, inner sep=0}, "{\mathbf{y}_b}"', curve={height=12pt}, from=1-3, to=1-1]
	\arrow[hook, from=1-1, to=3-2]
	\arrow[""{name=2, anchor=center, inner sep=0}, "{\mathrm{ev}_b}"', curve={height=12pt}, from=3-2, to=1-3]
	\arrow[""{name=3, anchor=center, inner sep=0}, "{\mathbf{y}_b}"', curve={height=12pt}, dotted, from=1-3, to=3-2]
	\arrow["\simeq"{description}, draw=none, from=1, to=0]
	\arrow["\dashv"{anchor=center, rotate=-30}, draw=none, from=3, to=2]
\end{tikzcd}\]
This is discussed in \cref{ssec:cocart-sect-initial}. The Yoneda Lemma will follow as an instance of this result, which gets discussed in \cref{ssec:yoneda}.

Here, for a given $b:B$, the LARI $\yon \defeq \yon_b$ is defined at $d:P\,b$, $x:B$ via cocartesian transport as
\[ (\yon\,d)(x) \defeq \partial_1\,P_!(\emptyset_x,d).\]
A key step is showing that $\yon$ lands in \emph{cocartesian} sections. This can be shown by interpreting cocartesian lifts as $2$-cells $\chi_d: \cst \,d \rightarrow \yon\,d$, and then concluding by the formal properties of cocartesian arrows that the components of $\yon\,d$ are, in fact, cocartesian arrows.

\subsection{Cocartesian sections from initial elements}\label{ssec:cocart-sect-initial}

Recall the definition of an initial element in a type.

\begin{defn}[Initial element, {\protect\cite[Definition 9.6]{RS17}}]
Let $B$ be a type. An element $b:B$ is \emph{initial} if
\[ \prod_{a:B} \isContr(\hom_B(b,a)).\]
\end{defn}

We write $\emptyset_a:\hom_B(b,a)$ for the center of contraction if $b:B$ is initial. First, we establish the LARI adjunction result between the sections of the family and the fiber at the initial element.

\begin{prop}[{\protect{\cite[Proposition 5.7.13]{RV}}}]\label{prop:cocart-fam-evb-lari}
Let $P: B \to \UU$ be a cocartesian family over a Rezk type $B$. If $b:B$ is initial the evaluation functor $\ev_b: \prod_B P \to P\,b$ has a LARI:
	\[
\tikzset{%
	symbol/.style={%
		draw=none,
		every to/.append style={%
			edge node={node [sloped, allow upside down, auto=false]{$#1$}}}
	}
}
\begin{tikzcd}
	\prod_B P \ar[rr, bend right = 25, "\ev_b" swap, ""{name=B, below}] && P\,b \ar[ll, bend right = 25, dotted, "\yon" swap, ""{name=A, above}]
	\ar[from=B, to=A, symbol=\vdash]
\end{tikzcd}
\]
\end{prop}

\begin{proof}
The candidate for the left adjoint is
\[ \yon: P\,b \to \prod_BP, \quad \yon :\jdeq \lambda d.\lambda x.  \coliftpt{(\emptyset_x)}{d} \jdeq \lambda d.\lambda x. \partial_1 \, \coliftarr{P}{\emptyset_x}{d}. \]
We will show that the map
\[ \varphi: \hom_{\prod_B P}(\yon\,d,\sigma) \to \hom_{P\,b}(d,\sigma \,b), \quad \varphi : \jdeq \lambda \kappa.\kappa \,b \]
is an equivalence.

Since $b$ is initial, for any $x:B$, the hom-type $\hom_B(b,x)$ is contractible with center $\emptyset_x$. Since $P$ is cocartesian, we can define the map
\[ \psi:\hom_{P\,b}(d,\sigma \,b) \to \hom_{\prod_B P}(\yon\,d,\sigma), \psi\defeq \lambda f.\lambda x.\tyfill_{P_!(\emptyset_x,d)}(\sigma(\emptyset_x) \circ f). \]
To illustrate, $\psi$ yields the right-hand vertical arrow in squares of the form:
\[
\begin{tikzcd}
	d \ar[rr, "{P_!(\emptyset_x, d)}", cocart]  \ar[d,"f" swap]  \ar[drr] && (\yon  d)x \ar[d, "{\psi(f, x)}", dashed] \\
	\sigma b \ar[rr, "{\sigma \emptyset_x}" swap] && \sigma x \\
	b \ar[rr, "{\emptyset_x}"] && x
\end{tikzcd}
\]
where $f:\hom_{P\,b}(d,\sigma \,b)$. For the round-trip along through $\varphi$ and $\psi$, note that we obtain a path $\varphi(\psi f) = \id_b$ since cocartesian lifts of identities are again identities.

For the other direction, note that as described in \cref{ssec:mor-of-sections} a morphism of sections $\kappa: \hom_{\prod_B P}(\yon\,d,\sigma)$ yields a square
\[
\begin{tikzcd}
	d \ar[rr, "{P_!(\emptyset_x, d)}", cocart]  \ar[d,"\kappa b" swap]  \ar[drr] && (\yon  d)x \ar[d, "{\kappa x}", dashed] \\
	\sigma b \ar[rr, "{\sigma \emptyset_x}" swap] && \sigma x \\
	b \ar[rr, "{\emptyset_x}"] && x
\end{tikzcd}
\]
for every $x:B$.
But this means $\psi(\varphi \kappa)(x) = \kappa(x)$ as desired.

One checks that $(\yon f)(b) = f$, so $\yon$ is indeed a left adjoint right inverse to $\ev_b$.
\end{proof}

In fact, the transport map $\yon: P\,b \to \prod_BP$ is valued in cocartesian sections.	We prove this by a $2$-dimensional naturality property of the cocartesian lifts.

\begin{prop}[\protect{\cite[Proposition 5.7.18]{RV}}]\label{prop:yon-cocart-sec}
	Let $B$ be a Rezk type, $P: B \to \UU$ a cocartesian family and $b:B$ be an initial object. The map
	\[ \yon : P\, b \to \prod_B P, \quad \yon :\jdeq \lambda d.\lambda x. \coliftpt{(\emptyset_x)}{d} \]
	factors through the subtype of cocartesian sections, \ie,
	\[
	\begin{tikzcd}
		P\,b \ar[rr, "\yon" swap, bend right = 15] \ar[r, "\yon"] & \prod_B^{\mathrm{cocart}} P \ar[r] & \prod_B P.
	\end{tikzcd}
	\]
\end{prop}

\begin{proof}
	For $d:P\,b$, consider the constant map $\cst(d) :\jdeq \lambda x.d : B \to E$, where $E\defeq \totalty{P}$. The cocartesian lifts w.r.t.~$d:P\,b$ give rise to a natural transformation
	\[ \chi:\hom_{B \to E}(\cst(d), \yon d) \equiv \prod_{x:B} \hom_E(\cst(d,x), \yon(d,x)), \quad \chi_x :\jdeq \coliftarr{P}{\emptyset_x}{d}. \]
	Given $x,x':B$, for any arrow $u:\hom_B(x,x')$ we want to argue that the action of $\chi$ on $u$ yields a cocartesian arrow $\chi_u:\dhom^P_u(\yon dx,\yon dx')$.\footnote{We suppress the canonical $2$-cell witnessing composition.} Considering the naturality square
	\[
	\begin{tikzcd}
		d \ar[d, cocart, "{\chi_x}"] \ar[drr, cocart, "{\chi_{x'}}"] \ar[rr, equals] && d \ar[d, cocart, "{\chi_{x'}}"]  \\
		\yon(d,x) \ar[rr, "\yon du" swap] && \yon(d,x') \\
		b \ar[d,"\emptyset_x" swap, dashed] \ar[drr, "\emptyset_{x'}", dashed] & & \\
		x \ar[rr, "u" swap] & & x'
	\end{tikzcd}
	\]
	it follows that $\yon du$ is cocartesian by \cref{prop:cocart-arr-closure}(\ref{it:cocart-arr-cancel}).
\end{proof}

\begin{prop}[{\protect{\cite[Theorem 5.7.13]{RV}}}, {\protect\cite[Theorem 9.7]{RS17}}]\label{prop:evid-equiv}
	Let $B$ be a Rezk type, $b:B$ an initial object, and $P: B \to \UU$ a cocartesian family. Then evaluation at $b$
	\[ \ev_b : \Big(\prod_{B}^{\mathrm{cocart}} P \Big) \to P\,b \]
	is an equivalence.
\end{prop}

\begin{proof}
	By \cref{prop:yon-cocart-sec}, the map $\yon$ restricts to cocartesian sections. By \cref{prop:cocart-fam-evb-lari}, $\ev_b\circ \mathbf y = \id_{P\,b}$.

	We set $T :\jdeq  \prod_{B}^{\mathrm{cocart}} P$. For the converse direction, we define a natural transformation
	\begin{align*}
		\varepsilon : \hom_{T \to T}(\yon \circ \ev_b, \id_T) & \equiv \prod_{\sigma:T} \hom_T(\yon(\ev_b)(\sigma), \sigma) \\
		& \equiv  \prod_{\sigma:T} \prod_{x: B} \hom_{P\,x}(\yon(\sigma(b))(x), \sigma(x))
	\end{align*}
	as follows:
	For $x:B$, the action of $\sigma$ on $\emptyset_x:\hom_{B}(b,x)$ yields a cocartesian arrow $\sigma(\emptyset_x): \dhom^P_{\emptyset_x}(\sigma(b),\sigma(x))$.

	We define $\varepsilon_{\sigma,x}$ as the following filler:
	\[
	\begin{tikzcd}
		& & \sigma(x) \\
		\sigma(b) \ar[rru, cocart, "\sigma(\emptyset_x)"]  \ar[rr, cocart, "{\coliftarr{P}{\emptyset_x}{\sigma(b)}}" swap] & & \yon(\sigma(b),x) \ar[u, dashed, "{\varepsilon_{\sigma,x}}" swap]	\\
		b \ar[rr, "\emptyset_x"] && x
	\end{tikzcd}
	\]
	By \cref{prop:cocart-arr-closure}(\ref{it:cocart-arr-cancel}), $\varepsilon_{\sigma,x}$ is cocartesian. As a cocartesian lift of an identity, it is itself an identity. By~\cite[Proposition 10.3]{RS17}, $\varepsilon:\hom_{T \to T}(\yon \circ \ev_b,\id_T)$ is an identity.
\end{proof}

\subsection{Dependent and absolute Yoneda Lemma}\label{ssec:yoneda}

We now obtain analogues of the (dependent) Yoneda Lemma as in~\cite{RV}, \cite{RS17} for cocartesian families.

\begin{lem}[{\protect\cite[Lemma~9.8]{RS17}}]\label{lem:initial-obj-cocomma}
	Let $B$ be a Segal type. For any element $b:B$, the identity morphism $\id_b:b \downarrow B$ is an initial object.
\end{lem}

\begin{theorem}[Dependent Yoneda Lemma for cocartesian families, \cf~{\protect\cite[Theorem 5.7.2]{RV}}]\label{thm:dep-yon}
	Let $B$ be a Rezk type, $b:B$ any element, and $Q: b\downarrow B \to \UU$ a cocartesian family. Then evaluation at $\id_b$ as in
	\[ \ev_{\id_b} : \Big( \prod_{b \downarrow B}^{\mathrm{cocart}}Q \Big) \to Q(\id_b) \]
	is an equivalence.
\end{theorem}

\begin{proof}
Because of \cref{lem:initial-obj-cocomma} (cf.~\cite[Lemma~9.8]{RS17}) this follows from \cref{prop:evid-equiv}.
\end{proof}

\begin{theorem}[Yoneda Lemma for cocartesian families, \cf~{\protect\cite[Theorem 5.7.3]{RV}}  \& {\protect\cite[Theorem 9.1]{RS17}}]\label{thm:yon}
Let $B$ be a Rezk type, $b:B$ any element, and $P: B \to \UU$ a cocartesian family. Then evaluation at $\id_b$ as in
\[ \ev_{\id_b} : \Big( \prod_{b \downarrow B}^{\mathrm{cocart}} \partial_1^* P \Big) \to P\,b \]
is an equivalence, where
\[ \partial_1\defeq \lambda u.u(1): b \downarrow B \to B.\]
\end{theorem}

\begin{proof}
This is an instance of the Dependent Yoneda Lemma \cref{thm:dep-yon} for~$Q\defeq \partial_1^*P \jdeq \lambda u.P(u1):b \downarrow B \to B$.
\end{proof}

Noting that for a covariant family any section is cocartesian, we recover the discrete versions from~\cite{RV,RS17}.

\begin{cor}[Dependent Yoneda Lemma for covariant families, \cf~\protect{\protect\cite[Theorem 9.5]{RS17}}]\label{thm:dep-yon-cov}
	Let $B$ be a Rezk type, $b:B$ any element, and $Q: b\downarrow B \to \UU$ a covariant family. Then evaluation at $\id_b$ as in
	\[ \ev_{\id_b} : \Big( \prod_{b \downarrow B}Q \Big) \to Q(\id_b) \]
	is an equivalence.
\end{cor}

\begin{cor}[Yoneda Lemma for covariant families, \cf~\protect{\cite[Theorem~5.7.1]{RV}, \cite[Theorem~9.1]{RS17}, \cite[Proposition~2.1.7]{KV14}, \cite[Theorem~3.49]{RasYon}}]\label{thm:yon-cov}
	Let $B$ be a Rezk type, $b:B$ any element, and $P: B \to \UU$ a covariant family. Then evaluation at $\id_b$ as in
	\[ \ev_{\id_b} : \Big( \prod_{b \downarrow B} \partial_1^* P \Big) \to P\,b \]
	is an equivalence, where
	\[ \partial_1\defeq \lambda u.u(1): b \downarrow B \to B.\]
\end{cor}

\begin{appendices}

\section{Lax squares}\label{app:lax-sq}

Since Segal types form an exponential ideal, simplicial type theory in fact captures some of the $2$-dimensional theory of synthetic $\inftyone$-categories, cf.~\cite[Section 6]{RS17}. For Rezk types $A,B$ the type $B^A$ is again Rezk, with mapping spaces $\nat{A}{B}(f,g) \jdeq (f \Rightarrow g)$. This invites a treatment of \emph{lax squares}
% https://q.uiver.app/?q=WzAsNCxbMCwwLCJBIl0sWzAsMSwiQyJdLFsyLDEsIkQiXSxbMiwwLCJCIl0sWzAsMSwiayIsMl0sWzEsMiwiaCIsMl0sWzAsMywiZiJdLFszLDIsImciXSxbMSwzLCJcXHNpZ21hIiwwLHsibGV2ZWwiOjJ9XV0=
\[\begin{tikzcd}
	A && B \\
	C && D
	\arrow["k"', from=1-1, to=2-1]
	\arrow["h"', from=2-1, to=2-3]
	\arrow["f", from=1-1, to=1-3]
	\arrow["g", from=1-3, to=2-3]
	\arrow["\sigma", Rightarrow, from=2-1, to=1-3]
\end{tikzcd}\]
which are encoded by $2$-cells
\[ \sigma: hk \Rightarrow gf \simeq \prod_{a:A} (hk)(a) \to (gf)(a).\]
We are mainly interested in some basic results about mates, known from classical $2$-category theory, re-interpreted in the context of higher categories as in~\cite{RV} (cf.~\emph{op.~cit.}, Appendix B, for the classical setting). These occur in our study of cocartesian functors in \cref{ssec:cocart-fun}.

First, in \cref{app:ssec:pasting-lax} we explicitly define pasting operations for lax squares. Then we prove a simple pasting theorem for a specific pasting scheme, which is used in \cref{app:ssec:mates} to establish the mates correspondence for an adjunction.

Alas, in the absence of categorical type universes, a general pasting theorem, possibly along the lines of~\cite{ColPhD,JYpasting,infty2pasting}, seems out of reach. Therefore, we will only in an \emph{ad hoc} manner consider a few very specific pasting schemes relevant to our specific applications, and prove well-definedness of these pastings by manual pasting diagram chases, which is straightforward, though lengthy to spell out.

\subsection{Pasting of lax squares}\label{app:ssec:pasting-lax}

In this subsection, we introduce horizontal and vertical pasting of lax squares of functors between Segal types. We prove that these pasting operations are associative, and establish a certain pasting theorem which will be used to establish the mates correspondence subsequently in \cref{app:ssec:mates}.

\begin{defn}[Pasting of squares]
\begin{enumerate}
	\item The \emph{horizontal pasting} of a natural transformations $\alpha:\nat{A}{A'}(hu,u'k)$ with a natural transformation $\alpha':\nat{A'}{A''}(h'u',u''k')$ as indicated in
	% https://q.uiver.app/?q=WzAsMTIsWzAsMCwiQiJdLFswLDEsIkEiXSxbMSwxLCJBJyJdLFsxLDAsIkInIl0sWzIsMSwiQScnIl0sWzIsMCwiQicnIl0sWzMsMV0sWzMsMF0sWzQsMCwiQiJdLFs0LDEsIkEiXSxbNiwwLCJCJyciXSxbNiwxLCJBJyciXSxbMCwxLCJ1IiwyXSxbMSwyLCJoIiwyXSxbMCwzLCJrIl0sWzMsMiwidSciLDJdLFsxLDMsIlxcYWxwaGEiLDAseyJsZW5ndGgiOjcwLCJsZXZlbCI6Mn1dLFsyLDQsImgnIiwyXSxbNSw0LCJ1JyciXSxbMyw1LCJrJyJdLFsyLDUsIlxcYWxwaGEnIiwwLHsibGVuZ3RoIjo3MCwibGV2ZWwiOjJ9XSxbNyw2LCI9IiwxLHsic3R5bGUiOnsiYm9keSI6eyJuYW1lIjoibm9uZSJ9LCJoZWFkIjp7Im5hbWUiOiJub25lIn19fV0sWzgsOSwidSIsMl0sWzgsMTAsImsnayJdLFs5LDExLCJoJ2giLDJdLFsxMCwxMSwidScnIl0sWzksMTAsIlxcYWxwaGEnJyIsMCx7Imxlbmd0aCI6NzAsImxldmVsIjoyfV1d
	\[\begin{tikzcd}
		{B} & {B'} & {B''} & {} & {B} && {B''} \\
		{A} & {A'} & {A''} & {} & {A} && {A''}
		\arrow["{u}"', from=1-1, to=2-1]
		\arrow["{h}"', from=2-1, to=2-2]
		\arrow["{k}", from=1-1, to=1-2]
		\arrow["{u'}"', from=1-2, to=2-2]
		\arrow[Rightarrow, "{\alpha}", from=2-1, to=1-2, shorten <=3pt, shorten >=3pt]
		\arrow["{h'}"', from=2-2, to=2-3]
		\arrow["{u''}", from=1-3, to=2-3]
		\arrow["{k'}", from=1-2, to=1-3]
		\arrow[Rightarrow, "{\alpha'}", from=2-2, to=1-3, shorten <=3pt, shorten >=3pt]
		\arrow["{\rightsquigarrow}" description, from=1-4, to=2-4, phantom, no head]
		\arrow["{u}"', from=1-5, to=2-5]
		\arrow["{k'k}", from=1-5, to=1-7]
		\arrow["{h'h}"', from=2-5, to=2-7]
		\arrow["{u''}", from=1-7, to=2-7]
		\arrow[Rightarrow, "{\alpha' \hpaste \alpha}", from=2-5, to=1-7, shorten <=7pt, shorten >=7pt]
	\end{tikzcd}\]
	is defined as:
	\[ \alpha' \hpaste \alpha\defeq (\alpha' * \id_k) \circ (\id_{h'} * \alpha):\nat {B}{A''}(h'hu,u''k'k)\]

	\item The \emph{vertical pasting} of a natural transformations $\beta:\nat{A}{A'}(f'h,kf)$ with a natural transformation $\beta':\nat{A'}{A''}(f''h',k'f')$ as indicated in
		% https://q.uiver.app/?q=WzAsMTEsWzAsMCwiQSJdLFsxLDAsIkIiXSxbMCwxLCJcXGJ1bGxldCJdLFsxLDEsIlxcYnVsbGV0Il0sWzAsMiwiQScnIl0sWzEsMiwiQicnIl0sWzIsMSwiXFxyaWdodHNxdWlnYXJyb3ciXSxbMywwLCJBIl0sWzQsMCwiQiJdLFszLDIsIkEnJyJdLFs0LDIsIkInJyJdLFswLDEsImYiXSxbMCwyLCJoIiwyXSxbMiwzXSxbMiw0LCJoJyIsMl0sWzQsNSwiZicnIiwyXSxbMyw1LCJrJyJdLFs0LDMsIlxcYmV0YSciLDAseyJsZW5ndGgiOjcwLCJsZXZlbCI6Mn1dLFsyLDEsIlxcYmV0YSIsMCx7Imxlbmd0aCI6NzAsImxldmVsIjoyfV0sWzcsOCwiZiIsMl0sWzcsOSwiaCdoIiwyXSxbOCwxMCwiaydrIl0sWzksOCwiXFxiZXRhJyciLDAseyJsZW5ndGgiOjcwLCJsZXZlbCI6Mn1dLFsxLDMsImsiXSxbOSwxMCwiZicnIiwyXV0=
\[\begin{tikzcd}
		{A} & {B} && {A} & {B} \\
	{A'} & {B'} & {\rightsquigarrow} \\
	{A''} & {B''} && {A''} & {B''}
	\arrow["{f}", from=1-1, to=1-2]
	\arrow["{h}"', from=1-1, to=2-1]
	\arrow["{f'}",from=2-1, to=2-2]
	\arrow["{h'}"', from=2-1, to=3-1]
	\arrow["{f''}"', from=3-1, to=3-2]
	\arrow["{k'}", from=2-2, to=3-2]
	\arrow[Rightarrow, "{\beta'}", from=3-1, to=2-2, shorten <=3pt, shorten >=3pt]
	\arrow[Rightarrow, "{\beta}", from=2-1, to=1-2, shorten <=3pt, shorten >=3pt]
	\arrow["{f}"', from=1-4, to=1-5]
	\arrow["{h'h}"', from=1-4, to=3-4]
	\arrow["{k'k}", from=1-5, to=3-5]
	\arrow[Rightarrow, "{\beta' \vpaste \beta}", from=3-4, to=1-5, shorten <=6pt, shorten >=6pt]
	\arrow["{k}", from=1-2, to=2-2]
	\arrow["{f''}"', from=3-4, to=3-5]
\end{tikzcd}\]
	is defined as:
	\[  \beta' \vpaste \beta :\jdeq (\id_{k'}*\beta) \circ (\beta' * \mathrm{id}_h) : \nat{A}{B''}(f''h'h,k'kf)\]

\end{enumerate}

\end{defn}

\begin{figure}
	\centering
% https://q.uiver.app/?q=WzAsNixbMSwwLCJCIl0sWzMsMCwiQScnIl0sWzAsMF0sWzUsMCwiQSJdLFs0LDBdLFs3LDAsIkInJyJdLFswLDEsImgnaHUiLDEseyJjdXJ2ZSI6LTV9XSxbMCwxLCJoJ3UnayIsMV0sWzAsMSwidScnaydrIiwxLHsiY3VydmUiOjV9XSxbNCwzLCJcXGJldGEnXFx2cGFzdGUgXFxiZXRhXFxkZWZlcSIsMSx7InN0eWxlIjp7ImJvZHkiOnsibmFtZSI6Im5vbmUifSwiaGVhZCI6eyJuYW1lIjoibm9uZSJ9fX1dLFsyLDAsIlxcYWxwaGEnIFxcaHBhc3RlIFxcYWxwaGEgXFxkZWZlcSIsMSx7InN0eWxlIjp7ImJvZHkiOnsibmFtZSI6Im5vbmUifSwiaGVhZCI6eyJuYW1lIjoibm9uZSJ9fX1dLFszLDUsImsnZidoIiwxXSxbMyw1LCJmJydoJ2giLDEseyJjdXJ2ZSI6LTV9XSxbMyw1LCJrJ2tmIiwxLHsiY3VydmUiOjV9XSxbNiw3LCJcXG1hdGhybXtpZH1fe2gnfSpcXGFscGhhIiwxLHsic2hvcnRlbiI6eyJzb3VyY2UiOjE1LCJ0YXJnZXQiOjE1fX1dLFs3LDgsIlxcYWxwaGEnKlxcbWF0aHJte2lkfV9rIiwxLHsic2hvcnRlbiI6eyJzb3VyY2UiOjE1LCJ0YXJnZXQiOjE1fX1dLFsxMiwxMSwiXFxiZXRhJypcXGlkX2giLDEseyJzaG9ydGVuIjp7InNvdXJjZSI6MjAsInRhcmdldCI6MjB9fV0sWzExLDEzLCJcXGlkX3trJ30qXFxiZXRhIiwxLHsic2hvcnRlbiI6eyJzb3VyY2UiOjIwLCJ0YXJnZXQiOjIwfX1dXQ==
\[\begin{tikzcd}
	{} & B && {A''} & {} & A && {B''}
	\arrow[""{name=0, anchor=center, inner sep=0}, "{h'hu}"{description}, curve={height=-60pt}, from=1-2, to=1-4]
	\arrow[""{name=1, anchor=center, inner sep=0}, "{h'u'k}"{description}, from=1-2, to=1-4]
	\arrow[""{name=2, anchor=center, inner sep=0}, "{u''k'k}"{description}, curve={height=60pt}, from=1-2, to=1-4]
	\arrow["{\beta'\vpaste \beta\defeq}"{description}, draw=none, from=1-5, to=1-6]
	\arrow["{\alpha' \hpaste \alpha \defeq}"{description}, draw=none, from=1-1, to=1-2]
	\arrow[""{name=3, anchor=center, inner sep=0}, "{k'f'h}"{description}, from=1-6, to=1-8]
	\arrow[""{name=4, anchor=center, inner sep=0}, "{f''h'h}"{description}, curve={height=-60pt}, from=1-6, to=1-8]
	\arrow[""{name=5, anchor=center, inner sep=0}, "{k'kf}"{description}, curve={height=60pt}, from=1-6, to=1-8]
	\arrow["{\mathrm{id}_{h'}*\alpha}"{description}, shorten <=3pt, shorten >=3pt, Rightarrow, from=0, to=1]
	\arrow["{\alpha'*\mathrm{id}_k}"{description}, shorten <=3pt, shorten >=3pt, Rightarrow, from=1, to=2]
	\arrow["{\beta'*\id_h}"{description}, shorten <=4pt, shorten >=4pt, Rightarrow, from=4, to=3]
	\arrow["{\id_{k'}*\beta}"{description}, shorten <=4pt, shorten >=4pt, Rightarrow, from=3, to=5]
\end{tikzcd}\]
\caption{Horizontal and vertical pasting of squares}\label{fig:pasting-squares}
\end{figure}
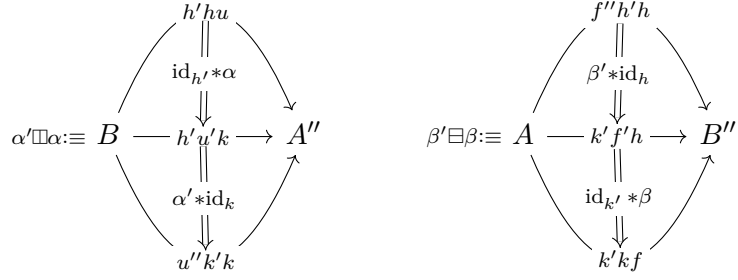
\begin{prop}[Associativity of horizontal pasting of squares]
For Segal types, consider $2$-cells as given in the diagram:
% https://q.uiver.app/?q=WzAsOCxbMCwwLCJCIl0sWzAsMSwiQSJdLFsxLDAsIkInIl0sWzEsMSwiQSciXSxbMiwwLCJCJyciXSxbMiwxLCJBJyciXSxbMywwLCJCJycnIl0sWzMsMSwiQScnJyJdLFswLDEsInUiLDJdLFswLDIsImsiXSxbMSwzLCJoIiwyXSxbMiw0LCJrJyJdLFszLDUsImgnIiwyXSxbNCw1LCJ1JyciLDJdLFs0LDYsImsnJyJdLFs1LDcsImgnJyIsMl0sWzYsNywidScnJyIsMl0sWzIsMywidSciLDJdLFsxLDIsIlxcYWxwaGEiLDAseyJsZW5ndGgiOjcwLCJsZXZlbCI6Mn1dLFszLDQsIlxcYWxwaGEnIiwwLHsibGVuZ3RoIjo3MCwibGV2ZWwiOjJ9XSxbNSw2LCJcXGFscGhhJyciLDAseyJsZW5ndGgiOjcwLCJsZXZlbCI6Mn1dXQ==
\[\begin{tikzcd}
	{B} & {B'} & {B''} & {B'''} \\
	{A} & {A'} & {A''} & {A'''}
	\arrow["{u}"', from=1-1, to=2-1]
	\arrow["{k}", from=1-1, to=1-2]
	\arrow["{h}"', from=2-1, to=2-2]
	\arrow["{k'}", from=1-2, to=1-3]
	\arrow["{h'}"', from=2-2, to=2-3]
	\arrow["{u''}"', from=1-3, to=2-3]
	\arrow["{k''}", from=1-3, to=1-4]
	\arrow["{h''}"', from=2-3, to=2-4]
	\arrow["{u'''}"', from=1-4, to=2-4]
	\arrow["{u'}"', from=1-2, to=2-2]
	\arrow[Rightarrow, "{\alpha}", from=2-1, to=1-2, shorten <=3pt, shorten >=3pt]
	\arrow[Rightarrow, "{\alpha'}", from=2-2, to=1-3, shorten <=3pt, shorten >=3pt]
	\arrow[Rightarrow, "{\alpha''}", from=2-3, to=1-4, shorten <=3pt, shorten >=3pt]
\end{tikzcd}\]
This defines a unique pasting $2$-cell $\alpha''':\hom_{B \to A'''}(h''h'hu,u'''k''k'k)$, so in particular there is an identification
\[ (\alpha'' \hpaste \alpha') \hpaste \alpha = \alpha'''= \alpha'' \hpaste (\alpha' \hpaste \alpha).\]
\end{prop}

\begin{proof}
We define $\beta : \hom_{B \to A''}(h'hu,u''k'k)$ as the composition $\beta\defeq \alpha' \hpaste \alpha$, \ie,~for all $b:B$:
	\[
	\begin{tikzcd}
		\beta_b : h'hu(b) \ar[rr, "h'\alpha_b"] & & h'u'k(b) \ar[rr, "\alpha'_{kb}"] && u''k'k(b)
	\end{tikzcd}
\]
Pasting $\alpha''$ on the right gives the natural transformation $\alpha'' \hpaste \beta \jdeq (\alpha'' * \id_{k'k}) \circ (\id_{h''} * \beta)$:
% https://q.uiver.app/?q=WzAsNCxbMCwwLCJoJydoJ3UoYikiXSxbMiwwLCJoJyd1JydrJ2soYikiXSxbMSwxLCJoJydoJ3UnayhiKSJdLFs0LDAsInUnJydrJydrJ2soYikiXSxbMCwxLCJoJydcXGJldGFfYiJdLFswLDIsImgnJ2gnXFxhbHBoYV9iIiwyXSxbMiwxLCJoJydcXGFscGhhJ197a2J9IiwyXSxbMSwzLCJcXGFscGhhJydfe2sna2J9Il1d
\[\begin{tikzcd}
	{h''h'hu(b)} && {h''u''k'k(b)} && {u'''k''k'k(b)} \\
	& {h''h'u'k(b)}
	\arrow["{h''\beta_b}", from=1-1, to=1-3]
	\arrow["{h''h'\alpha_b}"', from=1-1, to=2-2]
	\arrow["{h''\alpha'_{kb}}"', from=2-2, to=1-3]
	\arrow["{\alpha''_{k'kb}}", from=1-3, to=1-5]
\end{tikzcd}\]
By symmetry, one can show that for $\gamma : \hom_{B'\to A'''}(h''h'u',u'''k''k')$ defined as $\gamma :\jdeq \alpha'' \hpaste \alpha'$ the composite $\gamma \hpaste \alpha \jdeq (\gamma * \id_{k})\circ (\id_{h''h'}*\alpha)$ coincides with $\alpha'' \hpaste \beta$.
\end{proof}

\begin{prop}[Associativity of vertical pasting of squares]
	For Segal types, consider $2$-cells as given in the diagram:
% https://q.uiver.app/?q=WzAsOCxbMCwwLCJBIl0sWzAsMSwiQSciXSxbMSwwLCJCIl0sWzEsMSwiQiciXSxbMCwyLCJBJyciXSxbMSwyLCJCJyciXSxbMCwzLCJBJycnIl0sWzEsMywiQicnJyJdLFswLDEsImgiLDJdLFswLDIsImYiXSxbMSwzLCJmJyJdLFsyLDMsImsiXSxbMSw0LCJoJyIsMl0sWzQsNSwiZicnIl0sWzMsNSwiayciXSxbNCw2LCJoJyciLDJdLFs2LDcsImYnJyciXSxbNSw3LCJrJyciXSxbMSwyLCJcXGJldGEiLDAseyJsZW5ndGgiOjcwLCJsZXZlbCI6Mn1dLFs0LDMsIlxcYmV0YSciLDAseyJsZW5ndGgiOjcwLCJsZXZlbCI6Mn1dLFs2LDUsIlxcYmV0YScnIiwwLHsibGVuZ3RoIjo3MCwibGV2ZWwiOjJ9XV0=
\[\begin{tikzcd}
	{A} & {B} \\
	{A'} & {B'} \\
	{A''} & {B''} \\
	{A'''} & {B'''}
	\arrow["{h}"', from=1-1, to=2-1]
	\arrow["{f}", from=1-1, to=1-2]
	\arrow["{f'}", from=2-1, to=2-2]
	\arrow["{k}", from=1-2, to=2-2]
	\arrow["{h'}"', from=2-1, to=3-1]
	\arrow["{f''}", from=3-1, to=3-2]
	\arrow["{k'}", from=2-2, to=3-2]
	\arrow["{h''}"', from=3-1, to=4-1]
	\arrow["{f'''}", from=4-1, to=4-2]
	\arrow["{k''}", from=3-2, to=4-2]
	\arrow[Rightarrow, "{\beta}", from=2-1, to=1-2, shorten <=3pt, shorten >=3pt]
	\arrow[Rightarrow, "{\beta'}", from=3-1, to=2-2, shorten <=3pt, shorten >=3pt]
	\arrow[Rightarrow, "{\beta''}", from=4-1, to=3-2, shorten <=3pt, shorten >=3pt]
\end{tikzcd}\]
	This defines a unique pasting $2$-cell $\beta''':\hom_{B \to A'''}(f'''h''h'h,k''k'kf)$, so in particular there is an identification
	\[ (\beta'' \hpaste \beta') \hpaste \beta = \beta'''=\beta'' \hpaste (\beta' \hpaste \beta).\]
\end{prop}
The proof is completely analogous to the proof for the horizontal pasting.

To prove the mates correspondence we need the following pasting theorem involving $2$-cells with identity boundaries.

\begin{prop}[A pasting theorem]
For Segal types, consider the following configuration of natural transformations:
\[\begin{tikzcd}
	{A} && {B} && {B} && \\
	{A} && {A} && {B} && \\
	&& {A'} && {B'} && {B'} \\
	&& {A'} && {A'} && {B'}
	\arrow["{f'}" description, from=4-5, to=4-7]
	\arrow["{u'}"', from=3-5, to=4-5]
	\arrow[from=3-5, to=3-7, no head, equals]
	\arrow[from=3-7, to=4-7, no head, equals]
	\arrow[from=4-3, to=4-5, no head, equals]
	\arrow[from=3-3, to=4-3, no head, equals]
	\arrow["{f'}" description, from=3-3, to=3-5]
	\arrow["{h}"', from=2-3, to=3-3]
	\arrow["{f}" description, from=2-3, to=2-5]
	\arrow["{k}", from=2-5, to=3-5]
	\arrow["{u}"', from=1-3, to=2-3]
	\arrow[from=1-3, to=1-5, no head, equals]
	\arrow[from=2-1, to=2-3, no head, equals]
	\arrow[from=1-1, to=2-1, no head, equals]
	\arrow["{f}" description, from=1-1, to=1-3]
	\arrow[Rightarrow, "{\varphi}" description, from=2-1, to=1-3, shorten <=7pt, shorten >=7pt]
	\arrow[Rightarrow, "{\psi}" description, from=2-3, to=1-5, shorten <=7pt, shorten >=7pt]
	\arrow[Rightarrow, "{\beta}" description, from=3-3, to=2-5, shorten <=7pt, shorten >=7pt]
	\arrow[Rightarrow, "{\varphi'}" description, from=4-3, to=3-5, shorten <=7pt, shorten >=7pt]
	\arrow[Rightarrow, "{\psi'}" description, from=4-5, to=3-7, shorten <=7pt, shorten >=7pt]
	\arrow[from=1-5, to=2-5, no head, equals]
\end{tikzcd}\]
This gives rise to a unique pasting $2$-cell, in particular there is a homotopy
\[ (\psi' \hpaste \varphi') \vpaste \beta \vpaste (\psi \hpaste \varphi) = \psi' \hpaste (\varphi' \vpaste \beta \vpaste \psi) \hpaste \varphi.\]
% https://q.uiver.app/?q=WzAsMTgsWzcsMV0sWzcsMl0sWzgsMCwiQSJdLFsxMCwwLCJCIl0sWzgsMSwiQSJdLFsxMCwxLCJCIl0sWzgsMiwiQSciXSxbMTAsMiwiQiciXSxbOCwzLCJBJyJdLFsxMCwzLCJCJyJdLFs0LDIsIkEnIl0sWzQsMSwiQiciXSxbNiwxLCJCJyJdLFs2LDIsIkInIl0sWzIsMiwiQSJdLFsyLDEsIkIiXSxbMCwyLCJBIl0sWzAsMSwiQSJdLFswLDEsIj0iLDEseyJzdHlsZSI6eyJib2R5Ijp7Im5hbWUiOiJub25lIn0sImhlYWQiOnsibmFtZSI6Im5vbmUifX19XSxbMiwzLCJmIiwxXSxbMiw0LCIiLDEseyJzdHlsZSI6eyJoZWFkIjp7Im5hbWUiOiJub25lIn19fV0sWzQsNSwiZiIsMV0sWzMsNSwiIiwxLHsic3R5bGUiOnsiaGVhZCI6eyJuYW1lIjoibm9uZSJ9fX1dLFs0LDYsImgiLDJdLFs2LDcsImYnIiwxXSxbNSw3LCJrIl0sWzgsOSwiZiciLDFdLFs3LDksIiIsMSx7InN0eWxlIjp7ImhlYWQiOnsibmFtZSI6Im5vbmUifX19XSxbNCwzLCJcXHBzaSpcXHZhcnBoaSIsMSx7Imxlbmd0aCI6NzAsImxldmVsIjoyfV0sWzYsNSwiXFxiZXRhIiwxLHsibGVuZ3RoIjo3MCwibGV2ZWwiOjJ9XSxbOCw3LCJcXHBzaScqXFx2YXJwaGknIiwxLHsibGVuZ3RoIjo3MCwibGV2ZWwiOjJ9XSxbNiw4LCIiLDEseyJzdHlsZSI6eyJoZWFkIjp7Im5hbWUiOiJub25lIn19fV0sWzExLDEwLCJ1JyIsMl0sWzExLDEyLCIiLDEseyJzdHlsZSI6eyJoZWFkIjp7Im5hbWUiOiJub25lIn19fV0sWzEyLDEzLCIiLDEseyJzdHlsZSI6eyJoZWFkIjp7Im5hbWUiOiJub25lIn19fV0sWzEwLDEzLCJmJyIsMV0sWzEwLDEyLCJcXHBzaSciLDEseyJsZW5ndGgiOjcwLCJsZXZlbCI6Mn1dLFsxNCwxMCwiaCIsMV0sWzE1LDExLCJrIiwxXSxbMTUsMTQsInUiLDJdLFsxNiwxNCwiIiwxLHsic3R5bGUiOnsiaGVhZCI6eyJuYW1lIjoibm9uZSJ9fX1dLFsxNiwxNSwiXFx2YXJwaGkiLDEseyJsZW5ndGgiOjcwLCJsZXZlbCI6Mn1dLFsxNywxNiwiIiwxLHsic3R5bGUiOnsiaGVhZCI6eyJuYW1lIjoibm9uZSJ9fX1dLFsxNywxNSwiZiIsMV0sWzE0LDExLCJcXHZhcnBoaScqXFxiZXRhKlxccHNpIiwxLHsibGVuZ3RoIjo3MCwibGV2ZWwiOjJ9XV0=
\[\begin{tikzcd}
	&&&&&&&& {A} && {B} \\
	{A} && {B} && {B'} && {B'} & {} & {A} && {B} \\
	{A} && {A} && {A'} && {B'} & {} & {A'} && {B'} \\
	&&&&&&&& {A'} && {B'}
	\arrow["{=}" description, from=2-8, to=3-8, phantom, no head, equals]
	\arrow["{f}" description, from=1-9, to=1-11]
	\arrow[from=1-9, to=2-9, no head, equals]
	\arrow["{f}" description, from=2-9, to=2-11]
	\arrow[from=1-11, to=2-11, no head, equals]
	\arrow["{h}"', from=2-9, to=3-9]
	\arrow["{f'}" description, from=3-9, to=3-11]
	\arrow["{k}", from=2-11, to=3-11]
	\arrow["{f'}" description, from=4-9, to=4-11]
	\arrow[from=3-11, to=4-11, no head, equals]
	\arrow[Rightarrow, "{\psi\hpaste\varphi}" description, from=2-9, to=1-11, shorten <=7pt, shorten >=7pt]
	\arrow[Rightarrow, "{\beta}" description, from=3-9, to=2-11, shorten <=7pt, shorten >=7pt]
	\arrow[Rightarrow, "{\psi'\hpaste\varphi'}" description, from=4-9, to=3-11, shorten <=7pt, shorten >=7pt]
	\arrow[from=3-9, to=4-9, no head, equals]
	\arrow["{u'}"', from=2-5, to=3-5]
	\arrow[from=2-5, to=2-7, no head, equals]
	\arrow[from=2-7, to=3-7, no head, equals]
	\arrow["{f'}" description, from=3-5, to=3-7]
	\arrow[Rightarrow, "{\psi'}" description, from=3-5, to=2-7, shorten <=7pt, shorten >=7pt]
	\arrow["{h}" description, from=3-3, to=3-5]
	\arrow["{k}" description, from=2-3, to=2-5]
	\arrow["{u}"', from=2-3, to=3-3]
	\arrow[from=3-1, to=3-3, no head, equals]
	\arrow[Rightarrow, "{\varphi}" description, from=3-1, to=2-3, shorten <=7pt, shorten >=7pt]
	\arrow[from=2-1, to=3-1, no head, equals]
	\arrow["{f}" description, from=2-1, to=2-3]
	\arrow[Rightarrow, "{\varphi'\vpaste\beta\vpaste\psi}" description, from=3-3, to=2-5, shorten <=7pt, shorten >=7pt]
\end{tikzcd}\]
\end{prop}

\begin{proof}
Given a $2$-cell $\beta:f'h \Rightarrow kf$, we spell out the $2$-cells generated by the pastings on either side of the identification. Each time, this is done in two steps, first considering the diagram in the middle, and then pasting the cells on the sides. For the horizontal composite $H \defeq \psi' \hpaste (\varphi' \vpaste \beta \vpaste \psi) \vpaste \varphi$, we find
	% https://q.uiver.app/?q=WzAsNSxbMCwwLCJcXHZhcnBoaScqXFxiZXRhKlxccHNpOiJdLFsxLDAsImh1KGIpIl0sWzMsMCwidSdmJ2h1KGIpIl0sWzUsMCwidSdrZnUoYikiXSxbNywwLCJ1J2soYikiXSxbMSwyLCJcXHZhcnBoaSdfe2h1KGIpfSJdLFsyLDMsInUnXFxiZXRhX3t1KGIpfSJdLFszLDQsInUna1xccHNpX2IiXV0=
	\[\begin{tikzcd}
		{\alpha :\jdeq \varphi'\vpaste\beta\vpaste\psi:} & {hu(b)} && {u'f'hu(b)} && {u'kfu(b)} && {u'k(b)}
		\arrow["{\varphi'_{hu(b)}}", from=1-2, to=1-4]
		\arrow["{u'\beta_{u(b)}}", from=1-4, to=1-6]
		\arrow["{u'k\psi_b}", from=1-6, to=1-8]
	\end{tikzcd}\]
which yields the resulting cell:
% https://q.uiver.app/?q=WzAsMTIsWzAsMCwiXFxwc2knKlxcYWxwaGEqXFx2YXJwaGk6Il0sWzEsMCwiZidoKGEpIl0sWzMsMCwiZidodWYoYSkiXSxbNywwXSxbOSwwXSxbNCwyXSxbNiwyXSxbMywxXSxbMywyLCJmJ3UnZmh1ZihhKSJdLFs1LDAsImYndSdrZihhKSJdLFs1LDIsImYndSdrZnVmKGEpIl0sWzYsMCwia2YoYSkiXSxbMSwyLCJmJ2hcXHZhcnBoaV9hIl0sWzIsOSwiZidcXGFscGhhX3tmKGEpfSJdLFs4LDEwLCJmJ3UnXFxiZXRhX3t1ZihhKX0iLDJdLFsxMCw5LCJmJ3Una1xccHNpX3tmKGEpfSIsMl0sWzksMTEsIlxccHNpJ197a2YoYSl9Il0sWzIsOCwiZidcXHZhcnBoaSdfe2h1ZihhKX0iLDJdXQ==
\[\begin{tikzcd}
	{\psi'\hpaste\alpha\hpaste\varphi:} & {f'h(a)} && {f'huf(a)} && {f'u'kf(a)} & {kf(a)} \\
	&&& {} \\
	&&& {f'u'f'huf(a)} & {} & {f'u'kfuf(a)} & {}
	\arrow["{f'h\varphi_a}", from=1-2, to=1-4]
	\arrow["{f'\alpha_{f(a)}}", from=1-4, to=1-6]
	\arrow["{f'u'\beta_{uf(a)}}"', from=3-4, to=3-6]
	\arrow["{f'u'k\psi_{f(a)}}"', from=3-6, to=1-6]
	\arrow["{\psi'_{kf(a)}}", from=1-6, to=1-7]
	\arrow["{f'\varphi'_{huf(a)}}"', from=1-4, to=3-4]
\end{tikzcd}\]
The vertical composite $V\defeq (\psi' \hpaste \varphi') \vpaste \beta \vpaste (\varphi \hpaste \psi)$ is treated accordingly. In sum, we get the two composites:
% https://q.uiver.app/?q=WzAsMTIsWzAsMCwiTDpmJ2giXSxbMiwwLCJmJ3UnZidodWYiXSxbNCwwLCJmJ3Una2YiXSxbMCwxLCJSOmYnaCJdLFsyLDEsImYnaCJdLFs0LDEsImtmdWYiXSxbNSwxLCJrZiJdLFsxLDAsImYnaHVmIl0sWzMsMCwiZid1J2tmdWYiXSxbNSwwLCJrZiJdLFsxLDEsImYndSdmJ2giXSxbMywxLCJrZiJdLFswLDcsImYnaFxcdmFycGhpIiwwLHsibGV2ZWwiOjJ9XSxbMywxMCwiZidcXHZhcnBoaSdoIiwwLHsibGV2ZWwiOjJ9XSxbMTAsNCwiXFxwc2knZidoIiwwLHsibGV2ZWwiOjJ9XSxbNCwxMSwiXFxiZXRhIiwwLHsibGV2ZWwiOjJ9XSxbMTEsNSwia2ZcXHZhcnBoaSIsMCx7ImxldmVsIjoyfV0sWzUsNiwia2ZcXHBzaSBmIiwwLHsibGV2ZWwiOjJ9XSxbNywxLCJmJ1xcdmFycGhpJ2h1ZiIsMCx7ImxldmVsIjoyfV0sWzEsOCwiZid1J1xcYmV0YSB1ZiIsMCx7ImxldmVsIjoyfV0sWzgsMiwiZid1J2tcXHBzaSBmIiwwLHsibGV2ZWwiOjJ9XSxbMiw5LCJcXHBzaSdrZiIsMCx7ImxldmVsIjoyfV1d
\[\begin{tikzcd}
	{H:f'h} & {f'huf} & {f'u'f'huf} & {f'u'kfuf} & {f'u'kf} & kf \\
	{V:f'h} & {f'u'f'h} & {f'h} & kf & kfuf & kf
	\arrow["{f'h\varphi}", Rightarrow, from=1-1, to=1-2]
	\arrow["{f'\varphi'h}", Rightarrow, from=2-1, to=2-2]
	\arrow["{\psi'f'h}", Rightarrow, from=2-2, to=2-3]
	\arrow["\beta", Rightarrow, from=2-3, to=2-4]
	\arrow["kf\varphi", Rightarrow, from=2-4, to=2-5]
	\arrow["{k\psi f}", Rightarrow, from=2-5, to=2-6]
	\arrow["{f'\varphi'huf}", Rightarrow, from=1-2, to=1-3]
	\arrow["{f'u'\beta uf}", Rightarrow, from=1-3, to=1-4]
	\arrow["{f'u'k\psi f}", Rightarrow, from=1-4, to=1-5]
	\arrow["{\psi'kf}", Rightarrow, from=1-5, to=1-6]
\end{tikzcd}\]
We are to show that the composite $2$-cells $H$ and $V$ coincide.

Ultimately, this is established by commutation of the following diagram of natural transformations in $A \to B'$ (which can be checked pointwisely) where the outermost left composite is the $2$-cell $V$, and the outermost right composite is $H$:
% https://q.uiver.app/?q=WzAsMTEsWzEsMCwiZidoIl0sWzEsMSwiZid1J2YnaCJdLFszLDAsImYnaHVmIl0sWzMsMSwiZid1J2YnaHVmIl0sWzEsMiwiZid1J2tmIl0sWzMsMiwiZid1J2tmdWYiXSxbMSwzLCJrZiJdLFszLDMsImtmdWYiXSxbNSwyLCJmJ3Una2YiXSxbNSwzLCJrZiJdLFswLDEsImYnaCJdLFswLDEsImYnXFx2YXJwaGknaCIsMix7ImxldmVsIjoyfV0sWzAsMiwiZidoXFx2YXJwaGkiLDEseyJsZXZlbCI6Mn1dLFsxLDMsImYndSdmJ2hcXHZhcnBoaSIsMSx7ImxldmVsIjoyfV0sWzIsMywiZidcXHZhcnBoaSdodWYiLDAseyJsZXZlbCI6Mn1dLFsxLDQsImYndSdcXGJldGEiLDIseyJsZXZlbCI6Mn1dLFs0LDUsImYndSdrZlxcdmFycGhpIiwxLHsibGV2ZWwiOjJ9XSxbNCw2LCJcXHBzaSdrZiIsMix7ImxldmVsIjoyfV0sWzYsNywia1xcdmFycGhpIGYiLDIseyJsZXZlbCI6Mn1dLFs1LDcsIlxccHNpJ2tmdWYiLDEseyJsZXZlbCI6Mn1dLFszLDUsImYndSdcXGJldGEgdWYiLDAseyJsZXZlbCI6Mn1dLFs1LDgsImYndSdrXFxwc2kgZiIsMCx7ImxldmVsIjoyfV0sWzcsOSwia1xccHNpIGYiLDIseyJsZXZlbCI6Mn1dLFs4LDksIlxccHNpJ2tmIiwwLHsibGV2ZWwiOjJ9XSxbMCwzLCJmJyhcXHZhcnBoaSdoKlxcdmFycGhpKSIsMSx7ImxldmVsIjoyfV0sWzEsNSwiZid1JyhcXGJldGEgKiBcXHZhcnBoaSkiLDEseyJsZXZlbCI6Mn1dLFs0LDcsIlxccHNpJ2sqXFx2YXJwaGkgZiIsMSx7ImxldmVsIjoyfV0sWzUsOSwiXFxwc2knaypcXHBzaSBmIiwxLHsibGV2ZWwiOjJ9XSxbNiw5LCJrZlxccHNpIFxcY2lyYyBrZlxcdmFycGhpIiwyLHsiY3VydmUiOjUsImxldmVsIjoyfV0sWzEsMTAsIlxccHNpJ2YnaCIsMix7ImxldmVsIjoyfV0sWzEwLDYsIlxcYmV0YSIsMSx7ImN1cnZlIjo0LCJsZXZlbCI6Mn1dLFszMCwxNSwiKDEpIiwxLHsibGFiZWxfcG9zaXRpb24iOjQwLCJzaG9ydGVuIjp7InNvdXJjZSI6MjAsInRhcmdldCI6MjB9LCJzdHlsZSI6eyJib2R5Ijp7Im5hbWUiOiJub25lIn0sImhlYWQiOnsibmFtZSI6Im5vbmUifX19XSxbNywyOCwiKDIpIiwxLHsibGFiZWxfcG9zaXRpb24iOjQwLCJzaG9ydGVuIjp7InRhcmdldCI6MjB9LCJzdHlsZSI6eyJib2R5Ijp7Im5hbWUiOiJub25lIn0sImhlYWQiOnsibmFtZSI6Im5vbmUifX19XV0=
\[\begin{tikzcd}
	& {f'h} && {f'huf} \\
	{f'h} & {f'u'f'h} && {f'u'f'huf} \\
	& {f'u'kf} && {f'u'kfuf} && {f'u'kf} \\
	& kf && kfuf && kf
	\arrow["{f'\varphi'h}"', Rightarrow, from=1-2, to=2-2]
	\arrow["{f'h\varphi}"{description}, Rightarrow, from=1-2, to=1-4]
	\arrow["{f'u'f'h\varphi}"{description}, Rightarrow, from=2-2, to=2-4]
	\arrow["{f'\varphi'huf}", Rightarrow, from=1-4, to=2-4]
	\arrow[""{name=0, anchor=center, inner sep=0}, "{f'u'\beta}"', Rightarrow, from=2-2, to=3-2]
	\arrow["{f'u'kf\varphi}"{description}, Rightarrow, from=3-2, to=3-4]
	\arrow["{\psi'kf}"', Rightarrow, from=3-2, to=4-2]
	\arrow["{kf\varphi }"', Rightarrow, from=4-2, to=4-4]
	\arrow["{\psi'kfuf}"{description}, Rightarrow, from=3-4, to=4-4]
	\arrow["{f'u'\beta uf}", Rightarrow, from=2-4, to=3-4]
	\arrow["{f'u'k\psi f}", Rightarrow, from=3-4, to=3-6]
	\arrow["{k\psi f}"', Rightarrow, from=4-4, to=4-6]
	\arrow["{\psi'kf}", Rightarrow, from=3-6, to=4-6]
	\arrow["{f'(\varphi'h*\varphi)}"{description}, Rightarrow, from=1-2, to=2-4]
	\arrow["{f'u'(\beta * \varphi)}"{description}, Rightarrow, from=2-2, to=3-4]
	\arrow["{\psi'k*\varphi f}"{description}, Rightarrow, from=3-2, to=4-4]
	\arrow["{\psi'k*\psi f}"{description}, Rightarrow, from=3-4, to=4-6]
	\arrow["{\psi'f'h}"', Rightarrow, from=2-2, to=2-1]
	\arrow[""{name=2, anchor=center, inner sep=0}, "\beta"{description}, curve={height=24pt}, Rightarrow, from=2-1, to=4-2]
	\arrow["{(*)}"{description, pos=0.4}, Rightarrow, draw=none, from=2, to=0]
\end{tikzcd}\]
Here, the identity $(*)$ is established by:
% https://q.uiver.app/?q=WzAsNCxbMCwwLCJmJ3UnZidoIl0sWzIsMCwiZid1J2tmYSJdLFswLDEsImYnaCJdLFsyLDEsImtmIl0sWzAsMSwiZid1J1xcYmV0YSIsMCx7ImxldmVsIjoyfV0sWzAsMiwiXFxwc2knZidoIiwyLHsibGV2ZWwiOjJ9XSxbMiwzLCJcXGJldGEiLDIseyJsZXZlbCI6Mn1dLFsxLDMsIlxccHNpJ2tmIiwwLHsibGV2ZWwiOjJ9XSxbMCwzLCJcXHBzaSpcXGJldGEiLDEseyJsZXZlbCI6Mn1dXQ==
\[\begin{tikzcd}
	{f'u'f'h} && {f'u'kfa} \\
	{f'h} && kf
	\arrow["{f'u'\beta}", Rightarrow, from=1-1, to=1-3]
	\arrow["{\psi'f'h}"', Rightarrow, from=1-1, to=2-1]
	\arrow["\beta"', Rightarrow, from=2-1, to=2-3]
	\arrow["{\psi'kf}", Rightarrow, from=1-3, to=2-3]
	\arrow["{\psi*\beta}"{description}, Rightarrow, from=1-1, to=2-3]
\end{tikzcd}\]
In sum, as desired we have constructed a path between the two composites
\[ H \jdeq \psi' \hpaste (\varphi' \vpaste \beta \vpaste \psi) \vpaste \varphi = (\psi' \hpaste \varphi') \vpaste \beta \vpaste (\varphi \hpaste \psi) \jdeq V. \]
\end{proof}

\subsection{Mates}\label{app:ssec:mates}
We are now able to capture the mates correspondence of an adjunction in type theory.

\begin{prop}\label{prop:mates}
Consider adjunctions and morphisms between Rezk types as follows:
% https://q.uiver.app/?q=WzAsNCxbMCwwLCJCIl0sWzAsMSwiQSJdLFsxLDEsIkEnIl0sWzEsMCwiQiciXSxbMCwxLCJ1Il0sWzEsMiwiaCIsMl0sWzAsMywiayJdLFsxLDAsImYiLDAseyJjdXJ2ZSI6LTJ9XSxbMiwzLCJmJyIsMix7ImN1cnZlIjoyfV0sWzMsMiwidSciLDJdLFs3LDQsIiIsMSx7ImxldmVsIjoxLCJzdHlsZSI6eyJuYW1lIjoiYWRqdW5jdGlvbiJ9fV0sWzgsOSwiIiwxLHsibGV2ZWwiOjEsInN0eWxlIjp7Im5hbWUiOiJhZGp1bmN0aW9uIn19XV0=
\[\begin{tikzcd}
	{B} & {B'} \\
	{A} & {A'}
	\arrow["{u}"{name=0}, from=1-1, to=2-1]
	\arrow["{h}"', from=2-1, to=2-2]
	\arrow["{k}", from=1-1, to=1-2]
	\arrow["{f}"{name=1}, from=2-1, to=1-1, curve={height=-12pt}]
	\arrow["{f'}"{name=2, swap}, from=2-2, to=1-2, curve={height=12pt}]
	\arrow["{u'}"{name=3, swap}, from=1-2, to=2-2]
	\arrow["\dashv"{rotate=0}, from=1, to=0, phantom]
	\arrow["\dashv"{rotate=-180}, from=2, to=3, phantom]
\end{tikzcd}\]
There is a quasi-equivalence between types of natural transformations
\[ \Phi: \nat B {A'}(hu,u'k) \equiv \nat A {B'}(f'h,kf) : \Psi\]
 traditionally known as the \emph{mates correspondence}, \cf~\cite[Definition B.3.3]{RV}, given by pasting in the following way\footnote{The associativity theorems tell us that the diagrams are well-defined when read as pasting diagrams.}:

 \[\begin{tikzcd}
 	{B} & {B'} & {} &&& {B} & {B'} && {A'} \\
 	{A} & {A'} & {} & {B} && {A} & {A'} \\
 	{B} & {B'} & {} & {A} && {B} & {B'} \\
 	{A} & {A'} & { } & {} && {A} & {A'} && {B'} \\
 	&&&& {} &&&& {} \\
 	&&&&&&& {}
 	\arrow["{k}", from=3-1, to=3-2]
 	\arrow["{h}"', from=4-1, to=4-2]
 	\arrow["{u}"', from=3-6, to=4-6]
 	\arrow["{f}", from=3-4, to=3-6]
 	\arrow[Rightarrow, ""{name=0, inner sep=0}, from=3-4, to=4-6, shorten <=7pt, shorten >=7pt, no head]
 	\arrow["{k}", from=3-6, to=3-7]
 	\arrow["{h}"', from=4-6, to=4-7]
 	\arrow["{u'}"', from=3-7, to=4-7]
 	\arrow["{f'}"', from=4-7, to=4-9]
 	\arrow[Rightarrow, ""{name=1, inner sep=0}, from=4-9, to=3-7, shorten <=7pt, shorten >=7pt, no head]
 	\arrow[Rightarrow, "{\beta}", from=4-6, to=3-7, shorten <=3pt, shorten >=3pt]
 	\arrow["{:=}" description, from=3-3, to=4-3, phantom, no head]
 	\arrow[Rightarrow, "{\alpha_\beta}"', from=4-2, to=3-1, shorten <=3pt, shorten >=3pt]
 	\arrow["{f}", from=4-1, to=3-1]
 	\arrow["{f'}"', from=4-2, to=3-2]
 	\arrow["{u}"', from=1-1, to=2-1]
 	\arrow["{h}"', from=2-1, to=2-2]
 	\arrow["{u'}"' swap, from=1-2, to=2-2]
 	\arrow["{:=}" description, from=1-3, to=2-3, phantom, no head]
 	\arrow[Rightarrow, "{\beta_\alpha}", from=2-1, to=1-2, shorten <=3pt, shorten >=3pt]
 	\arrow["{u}"', from=2-4, to=2-6]
 	\arrow["{f}", from=2-6, to=1-6]
 	\arrow["{f'}"', from=2-7, to=1-7]
 	\arrow["{u'}", from=1-7, to=1-9]
 	\arrow[Rightarrow, ""{name=2, inner sep=0}, from=2-4, to=1-6, shorten <=7pt, shorten >=7pt, no head]
 	\arrow["{k}", from=1-6, to=1-7]
 	\arrow["{h}"', from=2-6, to=2-7]
 	\arrow[Rightarrow, "{\alpha}"', from=2-7, to=1-6, shorten <=3pt, shorten >=3pt]
 	\arrow[Rightarrow, ""{name=3, inner sep=0}, from=1-9, to=2-7, shorten <=7pt, shorten >=7pt, no head]
 	\arrow["{k}", from=1-1, to=1-2]
 	\arrow[Rightarrow, "{\varepsilon'}" description, from=4-7, to=1, shorten <=4pt, shorten >=4pt]
 	\arrow[Rightarrow, "{\eta}" description, from=0, to=3-6, shorten <=4pt, shorten >=4pt]
 	\arrow[Rightarrow, "{\varepsilon}" description, from=2-6, to=2, shorten <=4pt, shorten >=4pt]
 	\arrow[Rightarrow, "{\eta'}" description, from=3, to=1-7, shorten <=4pt, shorten >=4pt]
 \end{tikzcd}\]
Explicitly, the constructions are given by:
\[\begin{tikzcd}
	{} & {\beta_\alpha \defeq \Phi(\alpha) \defeq \Big(f'ha} && {f'hufa} && {f'u'kfa} && {kfa \Big)} \\
	& {\alpha_\beta \defeq \Psi(\beta) \defeq \Big( hub} && {u'f'hub} && {u'kfu(b)} && {u'kb\Big)}
	\arrow["f'\eta_{ha}", from=1-2, to=1-4]
	\arrow["{f'\alpha_{fa}}", from=1-4, to=1-6]
	\arrow["{\varepsilon'_{kfa}}", from=1-6, to=1-8]
	\arrow["{u'\beta_{ub}}", from=2-4, to=2-6]
	\arrow["{u'k\varepsilon_b}", from=2-6, to=2-8]
	\arrow["\eta'_{hub}", from=2-2, to=2-4]
\end{tikzcd}\]
\end{prop}

\begin{proof}[Proof of \cref{prop:mates}]
	Consider the maps
	\begin{align*}
		\Phi & : \nat B {A'}(hu,u'k) \to \nat B {A'}(f'h,kf), & \Phi(\alpha) :\jdeq \varepsilon' \hpaste \alpha \hpaste \eta \\
		\Psi & : \nat A {B'}(f'h,kf) \to \nat B {A'}(hu,u'k), & \Psi(\beta) :\jdeq \eta' \vpaste \beta \vpaste \varepsilon.
 	\end{align*}
 By the previous pasting theorem, we obtain for $\alpha : hu \Rightarrow u'k$ an identification
 \[ \Psi(\Phi(\alpha)) = \eta' \vpaste \Phi(\alpha) \vpaste \varepsilon = (\eta' \hpaste \varepsilon') \vpaste \alpha \vpaste (\eta \hpaste \varepsilon) = \alpha \]
 invoking witnesses for the triangle identities, via~\cite[Theorem 11.23]{RS17}.

The round trip in the converse direction is analogous.
\end{proof}

\begin{prop}[Invertibility of conjugates, {\protect\cite[Example B.3.5, Warning B.3.7]{RV}}]\label{prop:inv-conj}
	Given Rezk types $A,B$ with adjunctions $f \dashv u:B \to A$, $f' \dashv u':B \to A$, consider a $2$-cell $\alpha: u \Rightarrow u'$ and its mate $\alpha': f \Rightarrow f'$. Then $\alpha$ is invertible if and only if $\alpha'$ is:
% https://q.uiver.app/?q=WzAsMTAsWzAsMCwiQiJdLFswLDEsIkEiXSxbMSwwLCJCIl0sWzEsMSwiQSJdLFsyLDBdLFsyLDFdLFszLDAsIkIiXSxbMywxLCJBIl0sWzQsMSwiQSJdLFs0LDAsIkIiXSxbMCwxLCJ1IiwyXSxbMCwyLCI9Il0sWzEsMywiPSIsMl0sWzIsMywidSciXSxbNCw1LCJcXGxlZnRyaWdodHNxdWlnYXJyb3ciLDEseyJzdHlsZSI6eyJib2R5Ijp7Im5hbWUiOiJub25lIn0sImhlYWQiOnsibmFtZSI6Im5vbmUifX19XSxbNiw5LCI9Il0sWzcsNiwiZiJdLFs4LDksImYnIiwyXSxbMSwyLCJcXGFscGhhIiwwLHsibGVuZ3RoIjo3MCwibGV2ZWwiOjJ9XSxbOCw2LCJcXGFscGhhJyIsMix7Imxlbmd0aCI6NzAsImxldmVsIjoyfV0sWzcsOCwiPSIsMl1d
\[\begin{tikzcd}
	{B} & {B} & {} & {B} & {B} \\
	{A} & {A} & {} & {A} & {A}
	\arrow["{u}"', from=1-1, to=2-1]
	\arrow["{=}", from=1-1, to=1-2]
	\arrow["{=}"', from=2-1, to=2-2]
	\arrow["{u'}", from=1-2, to=2-2]
	\arrow["{\leftrightsquigarrow}" description, from=1-3, to=2-3, phantom, no head]
	\arrow["{=}", from=1-4, to=1-5]
	\arrow["{f}", from=2-4, to=1-4]
	\arrow["{f'}"', from=2-5, to=1-5]
	\arrow[Rightarrow, "{\alpha}", "{=}" swap, from=2-1, to=1-2, shorten <=3pt, shorten >=3pt]
	\arrow[Rightarrow, "{\alpha'}"', "{=}",  from=2-5, to=1-4, shorten <=3pt, shorten >=3pt]
	\arrow["{=}"', from=2-4, to=2-5]
\end{tikzcd}\]
\end{prop}

\begin{proof}
By precondition, we have the identification $\alpha:u = u'$. Recall the uniqueness of left adjoints and all the accompanying data~\cite[Theorem 11.23]{RS17}. Since both $f \dashv u$, $f' \dashv u'$, also $f' \dashv u$, but this implies there is an identification $f=f'$. Furthermore, the counit $\varepsilon'$ can be replaced by $\varepsilon$.

Thus, the mate of $\alpha$ at $a:A$ is homotopic to the composite
\[
\begin{tikzcd}
	f'a \ar[rr,"{f \eta_a}"] && fufa \ar[rr,equals, "{f\alpha_{fa}}"] && f'u'fa \ar[rr,"{\varepsilon_{fa}}"] && fa
\end{tikzcd}
\]
which is an isomorphism by one of the triangle identities.
\end{proof}

\begin{prop}[{\protect\cite[Exercise B.3.iii]{RV}}]\label{prop:inv-mates-composite-adjs}
	Consider adjunctions between Rezk types as in:
	% https://q.uiver.app/?q=WzAsNCxbMCwwLCJCIl0sWzAsMSwiQSJdLFsyLDEsIkEnIl0sWzIsMCwiQiciXSxbMSwyLCJrIiwxXSxbMCwzLCJrJyIsMV0sWzMsMiwidSciLDJdLFsxLDMsIlxcYWxwaGEiLDEseyJsZW5ndGgiOjcwLCJsZXZlbCI6Mn1dLFsxLDAsImYiLDAseyJjdXJ2ZSI6LTN9XSxbMywwLCJcXGVsbCciLDEseyJjdXJ2ZSI6M31dLFsyLDEsIlxcZWxsIiwxLHsiY3VydmUiOi0zfV0sWzIsMywiZiciLDIseyJjdXJ2ZSI6M31dLFswLDEsInUiXSxbOSw1LCIiLDEseyJsZXZlbCI6MSwic3R5bGUiOnsibmFtZSI6ImFkanVuY3Rpb24ifX1dLFsxMCw0LCIiLDEseyJsZXZlbCI6MSwic3R5bGUiOnsibmFtZSI6ImFkanVuY3Rpb24ifX1dLFs4LDEyLCIiLDEseyJsZXZlbCI6MSwic3R5bGUiOnsibmFtZSI6ImFkanVuY3Rpb24ifX1dLFsxMSw2LCIiLDEseyJsZXZlbCI6MSwic3R5bGUiOnsibmFtZSI6ImFkanVuY3Rpb24ifX1dXQ==
	\[\begin{tikzcd}
		{B} && {B'} \\
		{A} && {A'}
		\arrow["{k}"{name=0, description}, from=2-1, to=2-3]
		\arrow["{k'}"{name=1, description}, from=1-1, to=1-3]
		\arrow["{u'}"{name=2, swap}, from=1-3, to=2-3]
		\arrow[Rightarrow, "{\alpha}" description, from=2-1, to=1-3, shorten <=7pt, shorten >=7pt]
		\arrow["{f}"{name=3}, from=2-1, to=1-1, curve={height=-18pt}]
		\arrow["{\ell'}"{name=4, description}, from=1-3, to=1-1, curve={height=18pt}]
		\arrow["{\ell}"{name=5, description}, from=2-3, to=2-1, curve={height=-18pt}]
		\arrow["{f'}"{name=6, swap}, from=2-3, to=1-3, curve={height=18pt}]
		\arrow["{u}"{name=7}, from=1-1, to=2-1]
		\arrow["\dashv"{rotate=-90}, from=4, to=1, phantom]
		\arrow["\dashv"{rotate=90}, from=5, to=0, phantom]
		\arrow["\dashv"{rotate=0}, from=3, to=7, phantom]
		\arrow["\dashv"{rotate=-180}, from=6, to=2, phantom]
	\end{tikzcd}\]
Then the mate $\alpha':u \ell' \Rightarrow \ell u'$ w.r.t.~the horizontal adjunctions is invertible if and only if the mate $\alpha'':f'h \Rightarrow h'f$ w.r.t.~the vertical adjunctions is:
% https://q.uiver.app/?q=WzAsMTAsWzAsMCwiQiJdLFswLDEsIkEiXSxbMiwxLCJBJyJdLFsyLDAsIkInIl0sWzMsMF0sWzMsMV0sWzQsMCwiQiJdLFs0LDEsIkEiXSxbNiwxLCJBJyJdLFs2LDAsIkInIl0sWzAsMSwidSIsMl0sWzMsMiwidSciXSxbNCw1LCJcXGxlZnRyaWdodHNxdWlnYXJyb3ciLDEseyJzdHlsZSI6eyJib2R5Ijp7Im5hbWUiOiJub25lIn0sImhlYWQiOnsibmFtZSI6Im5vbmUifX19XSxbNyw4LCJoIiwxXSxbNiw5LCJoJyIsMV0sWzgsOSwiZiciLDJdLFs3LDYsImYiXSxbMywwLCJcXGVsbCciLDFdLFsyLDEsIlxcZWxsIiwxXSxbMCwyLCJcXGFscGhhJyIsMSx7Imxlbmd0aCI6NzAsImxldmVsIjoyfV0sWzgsNiwiXFxhbHBoYScnIiwxLHsibGVuZ3RoIjo3MCwibGV2ZWwiOjJ9XV0=
\[\begin{tikzcd}
	{B} && {B'} & {} & {B} && {B'} \\
	{A} && {A'} & {} & {A} && {A'}
	\arrow["{u}"', from=1-1, to=2-1]
	\arrow["{u'}", from=1-3, to=2-3]
	\arrow["{\leftrightsquigarrow}" description, from=1-4, to=2-4, phantom, no head]
	\arrow["{h}" description, from=2-5, to=2-7]
	\arrow["{h'}" description, from=1-5, to=1-7]
	\arrow["{f'}"', from=2-7, to=1-7]
	\arrow["{f}", from=2-5, to=1-5]
	\arrow["{\ell'}" description, from=1-3, to=1-1]
	\arrow["{\ell}" description, from=2-3, to=2-1]
	\arrow[Rightarrow, "{\alpha'}" swap,  "=", from=1-1, to=2-3, shorten <=7pt, shorten >=7pt]
	\arrow[Rightarrow, "{\alpha''}" swap,  "=", from=2-7, to=1-5, shorten <=7pt, shorten >=7pt]
\end{tikzcd}\]
\end{prop}

\begin{proof}
This becomes an instance of \cref{prop:inv-conj} after considering the composites of the adjunctions.
\end{proof}

\section{LARI and fibered adjunctions}\label{app:sec:adj}

We complement Riehl--Shulman's theory of adjunctions~\cite[Section 10]{RS17} by a treatment of \emph{left~adjoint right~inverse (LARI)~adjunctions}. Hereby, we assume the types involved to be Rezk.

Specifically, we provide a characterization result \cref{thm:char-lari} along the lines of~\cite[Theorem~11.23]{RS17}, as well as various closure properties which imply the closure properties in \cref{ssec:lari-closed,ssec:cocart-clos}. In particular, by fibrant replacement, we can identify LARI adjunctions between Rezk types with adjunctions whose left adjoint is a section. As for the stability properties, LARI adjunctions (between Rezk types) are always closed under dependent products, composition and pullback. This in turn implies the same closure properties for LARI fibrations \cref{ssec:lari-closed}.

Finally, we briefly discuss fibered adjunctions between isoinner families over a Rezk type. These occur in the characterization theorems of cocartesian families, \cref{thm:cocart-fam-intl-char-fib}, and functors, \cref{thm:cocart-fun-intl-char}.

\subsection{Left adjoint right inverse (LARI) adjunctions}

After providing a characterization of LARI adjunctions, similar to and relying on~\cite[Theorem~11.23]{RS17}, we prove a few closure properties. From these we derive corresponding closure properties of $j$-LARI fibrations (hence also cocartesian fibrations) in the main text, \cf~\cref{ssec:lari-closed}.

\subsubsection{Characterizations of LARI adjunctions}

\begin{defn}[Transposing LARI adjunction, cf.~{\protect\cite[Definition~11.1]{RS17}}]\label{def:transp-lari}
	A \emph{transposing left adjoint right inverse adjunction} (or \emph{LARI adjunction} for short) between Rezk types $A,B$ consists of:
	\begin{itemize}
		\item a transposing adjunction, consisting of functors $u:B \to A$, $f:A \to B$, and a family of equivalences $\varphi:\prod_{a:A, b:B} \hom_B(fa,b) \equiv \hom_A(a,ub)$
		\item such that for all $a:A$ the components $\eta_{\varphi,a} \defeq \varphi_{a,fa}(\id_{fa}) : \hom_A(a,ufa)$ of the unit are isomorphisms.
	\end{itemize}
Given this data, $f$ is called a \emph{transposing left adjoint right inverse} (or \emph{transposing LARI}).
\end{defn}

\begin{defn}[Bi-diagrammatic LARI adjunction, cf.~{\protect\cite[Definition~11.6]{RS17}}]
	A \emph{bi-diagrammatic left adjoint right inverse adjunction} (or \emph{bi-diagrammatic LARI adjunction}) between Rezk types $A$ and $B$ consists of:
	\begin{itemize}
		\item functors $u:B \to A$, $f:A \to B$
		\item a natural isomorphism $\eta:\iso_{A \to A}(\id_A, uf)$
		\item two natural transformations $\varepsilon, \varepsilon' : \hom_{B \to B}(fu,\id_B)$
		\item a path $\alpha:u\varepsilon \circ \eta u = \id_u$
		\item a path $\beta:\varepsilon' f \circ f \eta = \id_f$
	\end{itemize}
Given this data, $f$ is called a \emph{bi-diagrammatic left adjoint right inverse} (or \emph{bi-diagrammatic LARI}).
\end{defn}

\begin{defn}[Lifting LARI adjunction, cf.~{\protect\cite[Proposition~4.4.12]{RV2cat}}]
A \emph{lifting LARI adjunction} between Rezk types $A$ and $B$ consists of  functors $u:B \to A$, $f:A \to B$ with a homotopy $\sigma:u \circ f = \id_A$ such that the following holds: For all $a:A, b:B$ and arrows $\alpha:\hom_A(a,ub)$ there exists uniquely up to homotopy an arrow $\beta:\hom_B(fa,b)$ with $u\circ \beta = \alpha$:
\[
 \begin{tikzcd}
 	\partial \Delta^1 \ar[rr, "{[fa, b]}"] \ar[d] & & B  \ar[d, "u"] \\
 	\Delta^1 \ar[rr, "\alpha" swap] \ar[rru, dashed, "\beta" description]&& A
 \end{tikzcd}
 \]
Given this data, $f$ is called a \emph{lifting left adjoint right inverse} (or \emph{lifting LARI}).
\end{defn}

\begin{theorem}[Characterizations of LARI adjunctions, cf.~{\protect\cite[Theorem~11.23]{RS17}}]\label{thm:char-lari}
For a functor $u: B \to A$ between Rezk types the following types are equivalent propositions:
\begin{enumerate}
	\item\label{it:lari-transp} The type of transposing LARIs of $u$.
	\item\label{it:lari-lifting} The type of lifting LARIs of $u$.
	\item\label{it:lari-eta} The type of functors $f:A \to B$ together with natural isomorphisms $\eta:\iso_{A \to A}(\id_A,uf)$ such that $\varphi_\eta\defeq \lambda k.uk \circ \eta_a:\hom_B(fa,b) \to \hom_A(a,ub)$ is an equivalence for all $a:A$, $b:B$.
		\item\label{it:lari-bidiag} The type of bi-diagrammatic LARIs of $u$.
\end{enumerate}
\end{theorem}

\begin{proof}
Straightening $u$ to be the first projection of a family $P:A \to \UU$
(so that $B \equiv \totalty{P}$),
and strictifying all the occurring data, shows that the types in \cref{it:lari-lifting} and \cref{it:lari-eta} each are equivalent to the type
\[ \sum_{f:\prod_{a:A} P\,a} \prod_{\substack{x,y:A \\ e:P\,y}} \prod_{\alpha:x \to y} \isContr\big( fx \to_{\alpha}  e\big). \]

By~\cite[Theorem 11.23]{RS17}, the types of transposing and bi-diagrammatic left adjoints are equivalent to each other and to the type
\[ \sum_{f:A \to B} \sum_{\eta:\id_A \Rightarrow uf} \isEquiv(\varphi_\eta).\]
Hence, the type
\[ \sum_{f:A \to B} \sum_{\eta:\id_A \Rightarrow uf} \isEquiv(\varphi_\eta) \times \isIso(\eta)\]
is equivalent to the type of bi-diagrammatic LARIs of $u$, which establishes equivalences between the types from \cref{it:lari-eta,it:lari-bidiag,it:lari-transp}. Since the type of bi-diagrammatic left adjoints is a proposition by~\cite[Theorem~11.23]{RS17}, also the sub-type of bi-diagrammatic LARIs in \cref{it:lari-bidiag} is (and consequently, the type in \cref{it:lari-lifting} is as well).
\end{proof}

We remark that, in general, given two functors $u : B \to A$ and
$f : A \to B$, whether they determine an adjunction is extra
structure, and not just a proposition, but it is a proposition given
the data that $f$ is a section of $u$ (the proposition being: are the
corresponding transposing maps invertible?).
Similarly, whether $f$ is a section of $u$ is in general extra
structure, but this becomes a proposition when we have the data
of an adjunction (the proposition being: are the units invertible?).

\subsubsection{Closure properties of LARI  adjunctions}

\begin{prop}[(LARI) adjunctions are closed under products]\label{prop:lari-closed-under-pi}
	Let $I:\UU$ be a type, and $A,B:I \to \UU$ families with maps $r_i : B_i \to A_i$.
	If there is, for every $i:I$, a (LARI) adjunction
	\[
	\tikzset{%
		symbol/.style={%
			draw=none,
			every to/.append style={%
				edge node={node [sloped, allow upside down, auto=false]{$#1$}}}
		}
	}
	\begin{tikzcd}
		B(i) \ar[rr, bend right = 25, "r(i)" swap, ""{name=B, below}] && A(i) \ar[ll, bend right = 25, dotted, "\ell(i)" swap, ""{name=A, above}]
		\ar[from=B, to=A, symbol=\vdash]
	\end{tikzcd}
	\]
we get an induced (LARI) adjunction:
	\[
	\tikzset{%
		symbol/.style={%
			draw=none,
			every to/.append style={%
				edge node={node [sloped, allow upside down, auto=false]{$#1$}}}
		}
	}
	\begin{tikzcd}
		\prod_I B \ar[rr, bend right = 25, "\prod_I r" swap, ""{name=B, below}] && \prod_I A \ar[ll, bend right = 25, dotted, "\prod_I \ell" swap, ""{name=A, above}]
		\ar[from=B, to=A, symbol=\vdash]
	\end{tikzcd}
	\]
\end{prop}

\begin{proof}
By fibrant replacement we consider the straightenings $P_i:A_i \to \UU$, so that $B_i \simeq \widetilde{A_i} \simeq \sum_{a:A_i} P_i(a)$. Writing $\ell_i \defeq \pair{f_i}{L_i}$, the induced map is given by
\[ \prod_i \ell_i : \prod_i A_i \to \prod_i B_i, \quad (\prod_i \ell_i)(\alpha) \jdeq \lambda i. \pair{f_i \, \alpha_i}{L_i \, \alpha_i}.\]
With this, we obtain
\begin{align*}
	\hom_{\prod_i B_i}(\prod_i \ell_i \, \alpha, \pair{\alpha'}{\beta'}) & \simeq \prod_i \hom_{B_i}(\ell_i(\alpha_i), \pair{\alpha_i'}{\beta_i'}) \\
	& \simeq \prod_i \hom_{B_i}(\pair{f_i \alpha_i}{L_i\alpha_i}, \pair{\alpha_i', \beta_i'}) \\
	& \simeq \prod_i \hom_{A_i}(\alpha_i, \alpha_i') \tag{\text{since $\ell_i \dashv r_i$}} \\
	& \simeq \hom_{\prod_i A_i}(\alpha, \alpha').
\end{align*}
Thus, adjunctions are closed under taking dependent products. Clearly, this property descends to LARI adjunctions, \ie,~the case where $f_i:A_i \to A_i$ is the identity.
\end{proof}

\begin{prop}[LARI adjunctions are closed under pullback]\label{prop:lari-closed-under-pullback}
	LARI adjunctions between Rezk types are stable under pullback, \ie,~given a map $r:C \to A$ between Rezk types together with a LARI $\ell: A \to C$ and a map $j: B \to A$ where $B$ is a Rezk type, then the map $r':\jdeq j^*r$ has a LARI as well:
% https://q.uiver.app/?q=WzAsNCxbMCwwLCJEIl0sWzEsMCwiQyJdLFswLDEsIkIiXSxbMSwxLCJBIl0sWzAsMV0sWzAsMiwiciciXSxbMiwzLCJqIiwyXSxbMSwzLCJyIiwyXSxbMywxLCJcXGVsbCIsMix7ImN1cnZlIjoyLCJzdHlsZSI6eyJib2R5Ijp7Im5hbWUiOiJkb3R0ZWQifX19XSxbMiwwLCJcXGVsbCciLDAseyJjdXJ2ZSI6LTIsInN0eWxlIjp7ImJvZHkiOnsibmFtZSI6ImRvdHRlZCJ9fX1dLFswLDMsIiIsMSx7InN0eWxlIjp7Im5hbWUiOiJjb3JuZXIifX1dLFs4LDcsIiIsMix7ImxldmVsIjoxLCJzdHlsZSI6eyJuYW1lIjoiYWRqdW5jdGlvbiJ9fV0sWzksNSwiIiwwLHsibGV2ZWwiOjEsInN0eWxlIjp7Im5hbWUiOiJhZGp1bmN0aW9uIn19XV0=
\[\begin{tikzcd}
	{D} & {C} \\
	{B} & {A}
	\arrow[from=1-1, to=1-2]
	\arrow["{r'}"{name=0}, from=1-1, to=2-1]
	\arrow["{j}"', from=2-1, to=2-2]
	\arrow["{r}"{name=1, swap}, from=1-2, to=2-2]
	\arrow["{\ell}"{name=2, swap}, from=2-2, to=1-2, curve={height=12pt}, dotted]
	\arrow["{\ell'}"{name=3}, from=2-1, to=1-1, curve={height=-12pt}, dotted]
	\arrow["\lrcorner"{very near start, rotate=0}, from=1-1, to=2-2, phantom]
	\arrow["\dashv"{rotate=-180}, from=2, to=1, phantom]
	\arrow["\dashv"{rotate=0}, from=3, to=0, phantom]
\end{tikzcd}\]
\end{prop}

\begin{proof}
We fibrantly replace the square as follows. Denote by $P:A \to \UU$ the family of fibers of the map $j: B \to A$, so that $B \equiv \totalty{P}$, and likewise, $Q:A \to \UU$ the family associated to $r: C \to A$, so that $C \equiv \totalty{Q}$:
% https://q.uiver.app/?q=WzAsNCxbMCwwLCJcXHN1bV97YTpBfSBQXFwsYSBcXHRpbWVzIFFcXCxhIl0sWzEsMCwiXFxzdW1fe2E6QX1RXFwsYSJdLFswLDEsIlxcc3VtX3thOkF9IFBcXCxhIl0sWzEsMSwiQSJdLFswLDFdLFswLDIsInInIl0sWzIsMywiaiIsMl0sWzEsMywiciIsMl0sWzMsMSwiXFxlbGwiLDIseyJjdXJ2ZSI6Miwic3R5bGUiOnsiYm9keSI6eyJuYW1lIjoiZG90dGVkIn19fV0sWzIsMCwiXFxlbGwnIiwwLHsiY3VydmUiOi0yLCJzdHlsZSI6eyJib2R5Ijp7Im5hbWUiOiJkb3R0ZWQifX19XSxbMCwzLCIiLDEseyJzdHlsZSI6eyJuYW1lIjoiY29ybmVyIn19XSxbOCw3LCIiLDIseyJsZXZlbCI6MSwic3R5bGUiOnsibmFtZSI6ImFkanVuY3Rpb24ifX1dLFs5LDUsIiIsMCx7ImxldmVsIjoxLCJzdHlsZSI6eyJuYW1lIjoiYWRqdW5jdGlvbiJ9fV1d
\[\begin{tikzcd}
	{\sum_{a:A} P\,a \times Q\,a} & {\sum_{a:A}Q\,a} \\
	{\sum_{a:A} P\,a} & {A}
	\arrow[from=1-1, to=1-2]
	\arrow["{r'}"{name=0}, from=1-1, to=2-1]
	\arrow["{j}"', from=2-1, to=2-2]
	\arrow["{r}"{name=1, swap}, from=1-2, to=2-2]
	\arrow["{\ell}"{name=2, swap}, from=2-2, to=1-2, curve={height=12pt}, dotted]
	\arrow["{\ell'}"{name=3}, from=2-1, to=1-1, curve={height=-12pt}, dotted]
	\arrow["\lrcorner"{very near start, rotate=0}, from=1-1, to=2-2, phantom]
	\arrow["\dashv"{rotate=180}, from=2, to=1, phantom]
	\arrow["\dashv"{rotate=0}, from=3, to=0, phantom]
\end{tikzcd}\]
The section $\ell'$ is induced by the section $\ell$ by setting $\ell':\jdeq \lambda \pair{a}{d}.\angled{a,d,\ell(a)}$.\footnote{Here, in the notation we are identifying sections with their ``principal parts''.}
The adjunction $\ell \dashv r$ is given by an equivalence
\[ \sum_{\alpha:a \to a'} \ell(a) \to_\alpha e' \equiv \hom_C(\pair{a}{\ell(a)}, \pair{a'}{e'}) \stackrel{\Phi}{\equiv} \hom_A(a,a') \]
giving rise to an equivalence
\begin{align*}
	& \sum_{\alpha:a\to a'} (d \to_\alpha d') \times (\ell(a) \to e') \equiv \hom_D(\angled{a,d,\ell(a)}, \angled{a',d',e'}) \\
	 \stackrel{\Phi'}{\equiv} & \hom_B(\pair{a}{d},\pair{a'}{d'}) \equiv \sum_{\alpha:a \to a'} d \to_\alpha d',
\end{align*}
in sum establishing that $\ell'$ is a left adjoint right inverse of $r'$.
\end{proof}

\begin{prop}[LARI adjunctions are closed under composition]\label{prop:lari-closed-under-comp}
	Any two (LARI) adjunctions between Rezk types
	\[
	\tikzset{%
		symbol/.style={%
			draw=none,
			every to/.append style={%
				edge node={node [sloped, allow upside down, auto=false]{$#1$}}}
		}
	}
	\begin{tikzcd}
		C \ar[rr, bend right = 25, "r'" swap, ""{name=C, below}] && B \ar[ll, bend right = 25, dotted, "\ell'" swap, ""{name=B', above}]
		\ar[from=C, to=B', symbol=\vdash]
		&&
		B \ar[rr, bend right = 25, "r" swap, ""{name=B, below}] && A \ar[ll, bend right = 25, dotted, "\ell" swap, ""{name=A, above}]
		\ar[from=B, to=A, symbol=\vdash]
	\end{tikzcd}
	\]
	compose to a (LARI) adjunction:
	\[
	\tikzset{%
		symbol/.style={%
			draw=none,
			every to/.append style={%
				edge node={node [sloped, allow upside down, auto=false]{$#1$}}}
		}
	}
	\begin{tikzcd}
		C \ar[rr, bend right = 25, "r'r" swap, ""{name=C, below}] && A \ar[ll, bend right = 25, dotted, "\ell'\ell" swap, ""{name=A, above}]
		\ar[from=C, to=A, symbol=\vdash]
	\end{tikzcd}
	\]
\end{prop}

\begin{proof}
	By fibrant replacement, we take $r$ and $r'$, resp., to be the projections of families $P:A \to \UU$ and $Q: B \to \UU$, where $B :\jdeq \widetilde{P}$ and $C\defeq \widetilde{Q}$:
	% https://q.uiver.app/?q=WzAsNSxbMCwwLCJDIl0sWzIsMCwiQiJdLFs0LDAsIkEiXSxbMCwxLCJcXHN1bV97XFxzdWJzdGFja3thOkEgXFxcXCBkOlBcXCxhfX0gUVxcLGEgXFwsIGQiXSxbMiwxLCJcXHN1bV97YTpBfSBQXFwsYSJdLFswLDEsInInIiwyXSxbMSwyLCJyIiwyXSxbMCwzLCJcXHNpbWVxIiwxLHsic3R5bGUiOnsiYm9keSI6eyJuYW1lIjoibm9uZSJ9LCJoZWFkIjp7Im5hbWUiOiJub25lIn19fV0sWzEsNCwiXFxzaW1lcSIsMSx7InN0eWxlIjp7ImJvZHkiOnsibmFtZSI6Im5vbmUifSwiaGVhZCI6eyJuYW1lIjoibm9uZSJ9fX1dLFsxLDAsIlxcZWxsJyIsMix7ImN1cnZlIjoyfV0sWzIsMSwiXFxlbGwiLDIseyJjdXJ2ZSI6Mn1dLFs5LDUsIiIsMix7ImxldmVsIjoxLCJzdHlsZSI6eyJuYW1lIjoiYWRqdW5jdGlvbiJ9fV0sWzEwLDYsIiIsMix7ImxldmVsIjoxLCJzdHlsZSI6eyJuYW1lIjoiYWRqdW5jdGlvbiJ9fV1d
	\[\begin{tikzcd}
		{C} && {B} && {A} \\
		{\sum_{\substack{a:A \\ d:P\,a}} Q\,a \, d} && {\sum_{a:A} P\,a}
		\arrow["{r'}"{name=0, swap}, from=1-1, to=1-3, curve={height=4pt}]
		\arrow["{r}"{name=1, swap}, from=1-3, to=1-5, curve={height=4pt}]
		\arrow["{\equiv}" description, from=1-1, to=2-1, phantom, no head]
		\arrow["{\equiv}" description, from=1-3, to=2-3, phantom, no head]
		\arrow["{\ell'}"{name=2, swap}, from=1-3, to=1-1, curve={height=12pt}]
		\arrow["{\ell}"{name=3, swap}, from=1-5, to=1-3, curve={height=12pt}]
		\arrow["\dashv"{rotate=-90}, from=2, to=0, phantom]
		\arrow["\dashv"{rotate=-90}, from=3, to=1, phantom]
	\end{tikzcd}\]
We denote the actions of the maps $\ell$ and $\ell'$ in the following way:
\begin{align*}
	\ell \jdeq  \lambda a. &\pair{\ell_0 a}{\ell_1 a} \\
	\ell' \jdeq  \lambda \pair{a}{d}.&  \angled{\ell_0' \,a \,d, \ell_1' \,a \,d, \ell_2' \,a \,d}
\end{align*}
The adjunctions $\ell \dashv r$ and $\ell' \dashv r'$ are given by equivalences
\begin{align}
	\sum_{\alpha:\ell_0 a \to a'} (\ell_1 a \to_\alpha d') \equiv \hom_B(\ell(a), \pair{a'}{d'}) \stackrel{\Phi}{\equiv} \hom_A(a,a') \label{eq:adj-comp1}
\end{align}
and
\begin{align}
		& \sum_{\substack{\alpha:\ell'_0(a,d) \to a' \\ \beta:\ell_1'(a,d) \to_\alpha d'}}  (\ell'_2(a,d) \to_{\pair{\alpha}{\beta}} e') \equiv \hom_C(\ell'(a,d), \angled{a',d',e'})  \nonumber \\
	\stackrel{\Phi'}{\equiv} &  \hom_B(\pair{a}{d}, \pair{a'}{d'}) \equiv \sum_{\alpha:a \to a'} (d \to_\alpha d'), \label{eq:adj-comp2}
\end{align}
resp.
The composite of the left adjoints acts as:
\[ \ell''\jdeq \ell' \circ \ell \jdeq \lambda  a. \angled{\ell_0'(\ell_0 a, \ell_1 a),  \ell_1'(\ell_0a, \ell_1a), \ell_2'(\ell_0a, \ell_1a)} \]
This gives rise to an equivalence witnessing the composite adjunction:
\begin{align*}
	& \hom_C(\ell''(a), \angled{a',d',e'}) \equiv \sum_{\substack{\alpha:\ell'_0(\ell_0a, \ell_1a) \\ \beta:\ell_1'(\ell_0a, \ell_1a) }} \ell_2'(\ell_0a, \ell_1a) \to_{\pair{\alpha}{\beta}} e' \\
	\stackrel{\eqref{eq:adj-comp2}}{\equiv} & \sum_{\alpha:\ell_0a \to a'} \ell_1 a \to_\alpha d' \stackrel{\eqref{eq:adj-comp1}}{\equiv} \hom_A(a,a')
\end{align*}
Furthermore, clearly if both $\ell$ and $\ell'$, resp., happen to be sections of $r$ and $r'$, resp., then $\ell' \circ \ell$ is a section of $r' \circ r$ as can be seen from the above terms defining the functions.
\end{proof}

\subsubsection{Initial elements in a LARI adjunction}

The following lemma is useful in our considerations of cocartesian arrows, in particular in proving their uniqueness up to homotopy (\wrt~a fixed source vertex), \cref{prop:cocart-lifts-unique-in-isoinner-fams}.

Recall first the discussion in \cref{rem:lari-ff}.
\begin{lem}\label{lem:LARI-initial}
	Suppose $F \dashv U : E \to B$ is a LARI adjunction of Segal types,
	and let $b:B$. Then $b$ is initial in $B$ if and only if $F\,b$
	is initial in $E$.
\end{lem}
\begin{proof}
	The implication from left to right is clear, since
	left adjoints preserve initial objects.%
	\footnote{We expect there to be a nice general theory of (co)limits
		in simplicial type theory, but here we only need this specific case.}

	Conversely, for any $b':B$, since $F$ is fully faithful
	we have $\hom_B(b,b') \equiv \hom_E(F\,b, F\,b') \equiv \unit$,
	assuming $F\,b$ is initial in $E$.
\end{proof}

\subsection{Fibered adjunctions}

We give a brief treatment of fibered adjunctions which appear in the characterization theorems of cocartesian families and functors, resp., \cref{thm:cocart-fam-intl-char-fib,thm:cocart-fun-intl-char}. We furthermore show that fibered adjunctions over a common base pull back along arbitrary functors, which is used in the main text to show that, for a fibered adjunction between cocartesian fibrations, the fibered left adjoint is always cocartesian, cf.~\cref{prop:cocart-fun-fib-ladj}.

Throughout the subsection, we assume all types to be Rezk.

\begin{defn}[Fibered natural transformation]
	Let $\pi:E \fibarr B$ and $\xi: F \fibarr A$ be maps. If $\Phi\defeq\pair{k}{\varphi}$ and $\Psi \defeq \pair{m}{\psi}$ each are fibered functors from $\xi$ to $\pi$, then a \emph{fibered natural transformation $\Phi$ to $\Psi$} consists of a pair of natural transformations $\mu: k \Rightarrow m$ and $\vartheta: \varphi \Rightarrow \psi$ as indicated in
	% https://q.uiver.app/?q=WzAsNCxbMCwwLCJGIl0sWzIsMCwiRSJdLFswLDIsIkEiXSxbMiwyLCJCIl0sWzAsMSwiXFx2YXJwaGkiLDAseyJjdXJ2ZSI6LTJ9XSxbMCwxLCJcXHBzaSIsMix7ImN1cnZlIjoyfV0sWzAsMiwiXFx4aSIsMix7InN0eWxlIjp7ImhlYWQiOnsibmFtZSI6ImVwaSJ9fX1dLFsyLDMsIm0iLDIseyJjdXJ2ZSI6Mn1dLFsxLDMsIlxccGkiLDAseyJzdHlsZSI6eyJoZWFkIjp7Im5hbWUiOiJlcGkifX19XSxbMiwzLCJrIiwwLHsiY3VydmUiOi0yfV0sWzQsNSwiXFx2YXJ0aGV0YSIsMix7InNob3J0ZW4iOnsic291cmNlIjoyMCwidGFyZ2V0IjoyMH19XSxbOSw3LCJcXG11IiwyLHsic2hvcnRlbiI6eyJzb3VyY2UiOjIwLCJ0YXJnZXQiOjIwfX1dXQ==
	\[\begin{tikzcd}
		F && E \\
		\\
		A && B
		\arrow[""{name=0, anchor=center, inner sep=0}, "\varphi", curve={height=-12pt}, from=1-1, to=1-3]
		\arrow[""{name=1, anchor=center, inner sep=0}, "\psi"', curve={height=12pt}, from=1-1, to=1-3]
		\arrow["\xi"',from=1-1, to=3-1, two heads]
		\arrow[""{name=2, anchor=center, inner sep=0}, "m"', curve={height=12pt}, from=3-1, to=3-3]
		\arrow["\pi", from=1-3, to=3-3, two heads]
		\arrow[""{name=3, anchor=center, inner sep=0}, "k", curve={height=-12pt}, from=3-1, to=3-3]
		\arrow["\vartheta"', shorten <=3pt, shorten >=3pt, Rightarrow, from=0, to=1]
		\arrow["\mu"', shorten <=3pt, shorten >=3pt, Rightarrow, from=3, to=2]
	\end{tikzcd}\]
together with a family of paths
\[ \prod_{a:A, d:Q\,a} \pi(\vartheta_{d}) = \mu_{\xi_{d}}.\]
\end{defn}
By fibrant replacement, writing $P\defeq \St_B(\pi)$ and $Q \defeq \St_A(\xi)$, any fibered natural transformation can be presented by the data
\begin{itemize}
	\item $\mu:\hom_{A \to B}(k,m)$
	\item $\vartheta:\prod_{\pair{a}{d}:F} \varphi(d) \longrightarrow_{\mu_a}^P \psi(d)$.
\end{itemize}
We also write shorthand $\vartheta: \varphi \Rightarrow_\mu \psi$.

Over a common base, we obtain that a natural transformation as given in
% https://q.uiver.app/?q=WzAsMyxbMCwwLCJGIl0sWzIsMCwiRSJdLFsxLDIsIkIiXSxbMCwxLCJcXHZhcnBoaSIsMCx7ImN1cnZlIjotMn1dLFswLDEsIlxccHNpIiwyLHsiY3VydmUiOjJ9XSxbMCwyLCJcXHhpIiwyLHsic3R5bGUiOnsiaGVhZCI6eyJuYW1lIjoiZXBpIn19fV0sWzEsMiwiXFxwaSIsMCx7InN0eWxlIjp7ImhlYWQiOnsibmFtZSI6ImVwaSJ9fX1dLFszLDQsIlxcdmFydGhldGEiLDIseyJzaG9ydGVuIjp7InNvdXJjZSI6MjAsInRhcmdldCI6MjB9fV1d
\[\begin{tikzcd}
	F && E \\
	\\
	& B
	\arrow[""{name=0, anchor=center, inner sep=0}, "\varphi", curve={height=-12pt}, from=1-1, to=1-3]
	\arrow[""{name=1, anchor=center, inner sep=0}, "\psi"', curve={height=12pt}, from=1-1, to=1-3]
	\arrow["\xi"', from=1-1, to=3-2, two heads]
	\arrow["\pi", from=1-3, to=3-2, two heads]
	\arrow["\vartheta"', shorten <=3pt, shorten >=3pt, Rightarrow, from=0, to=1]
\end{tikzcd}\]
is fibered if and only if all components of $\vartheta$ are vertical arrows (cf.~\cref{def:vert-arr}).

\begin{defn}[Fibered adjunctions, {\cite[Definition 3.6.5]{RV}}]\label{def:fib-adj}
Let $\xi:F \fibarr B$, $\pi:E \fibarr B$ be maps between Rezk types, and fibered functors $\varphi: \xi \to_B \pi$, $\psi: \pi \to_B\xi$. A \emph{fibered adjunction (with fibered left adjoint $\psi$ and fibered right adjoint $\varphi$)} is an adjunction $\psi \dashv \varphi$ such that for all $\pair{b}{e}:E$, the components $\eta_{b,e}$ of the unit are vertical arrows. We indicate this by writing $\psi \dashv_B \varphi$ or the following diagram:
% https://q.uiver.app/?q=WzAsMyxbMCwwLCJGIl0sWzIsMCwiRSJdLFsxLDIsIkIiXSxbMCwxLCJcXHZhcnBoaSIsMCx7ImN1cnZlIjotMn1dLFsxLDAsIlxccHNpIiwwLHsiY3VydmUiOi0yfV0sWzAsMiwiXFx4aSIsMix7InN0eWxlIjp7ImhlYWQiOnsibmFtZSI6ImVwaSJ9fX1dLFsxLDIsIlxccGkiLDAseyJzdHlsZSI6eyJoZWFkIjp7Im5hbWUiOiJlcGkifX19XSxbMyw0LCIiLDAseyJsZXZlbCI6MSwic3R5bGUiOnsibmFtZSI6ImFkanVuY3Rpb24ifX1dXQ==
\[\begin{tikzcd}
	F && E \\
	\\
	& B
	\arrow[""{name=0, anchor=center, inner sep=0}, "\psi", curve={height=-12pt}, from=1-1, to=1-3]
	\arrow[""{name=1, anchor=center, inner sep=0}, "\varphi", curve={height=-12pt}, from=1-3, to=1-1]
	\arrow["\xi"', from=1-1, to=3-2, two heads]
	\arrow["\pi", from=1-3, to=3-2, two heads]
	\arrow["\dashv"{anchor=center, rotate=-90}, draw=none, from=0, to=1]
\end{tikzcd}\]
\end{defn}

\begin{prop}[Base change of fibered adjunctions, cf.~{\protect\cite[Lemma~3.6.6(i), Exercise~5.3.i]{RV}}]\label{prop:fib-adj-pb}
Given a fibered adjunction $\psi \dashv_B \varphi: F \to_B E$ between isoinner fibrations $\xi:F \fibarr A$, $\pi:E \fibarr B$, where $P\defeq \St_B(\pi)$, $Q \defeq \St_B(\xi)$, for any map $k:A \to B$ the induced adjunction $\psi' \dashv_B \varphi': k^*F \to_B k^*E$ is again a fibered adjunction:
% https://q.uiver.app/?q=WzAsNyxbMywwLCJGIl0sWzQsMSwiRSJdLFszLDIsIkIiXSxbMSwxLCJrXipFIl0sWzAsMCwia14qRiJdLFswLDIsIkEiXSxbMSwwXSxbMCwxLCJcXHZhcnBoaSIsMix7ImN1cnZlIjoxfV0sWzAsMiwiXFx4aSIsMSx7ImxhYmVsX3Bvc2l0aW9uIjozMCwic3R5bGUiOnsiaGVhZCI6eyJuYW1lIjoiZXBpIn19fV0sWzEsMiwiXFxwaSIsMSx7InN0eWxlIjp7ImhlYWQiOnsibmFtZSI6ImVwaSJ9fX1dLFszLDFdLFs0LDMsIlxcdmFycGhpJyIsMix7ImN1cnZlIjoxLCJzdHlsZSI6eyJib2R5Ijp7Im5hbWUiOiJkYXNoZWQifX19XSxbNCw1LCJrXipcXHhpIiwxLHsic3R5bGUiOnsiaGVhZCI6eyJuYW1lIjoiZXBpIn19fV0sWzUsMiwiayIsMl0sWzQsMF0sWzMsNSwia14qXFxwaSIsMSx7InN0eWxlIjp7ImhlYWQiOnsibmFtZSI6ImVwaSJ9fX1dLFszLDIsIiIsMSx7InN0eWxlIjp7Im5hbWUiOiJjb3JuZXIifX1dLFsxLDAsIlxccHNpIiwyLHsiY3VydmUiOjJ9XSxbMyw0LCJcXHBzaSciLDIseyJsYWJlbF9wb3NpdGlvbiI6NDAsImN1cnZlIjoyLCJzdHlsZSI6eyJib2R5Ijp7Im5hbWUiOiJkYXNoZWQifX19XSxbNiwyLCIiLDAseyJzdHlsZSI6eyJuYW1lIjoiY29ybmVyIn19XSxbMTcsNywiIiwxLHsibGV2ZWwiOjEsInN0eWxlIjp7Im5hbWUiOiJhZGp1bmN0aW9uIn19XSxbMTgsMTEsIiIsMSx7ImxldmVsIjoxLCJzdHlsZSI6eyJuYW1lIjoiYWRqdW5jdGlvbiJ9fV1d
\[\begin{tikzcd}
	{k^*F} & {} && F \\
	& {k^*E} &&& E \\
	A &&& B
	\arrow[""{name=0, anchor=center, inner sep=0}, "\varphi"', curve={height=6pt}, from=1-4, to=2-5]
	\arrow["\xi"{description, pos=0.3}, two heads, from=1-4, to=3-4]
	\arrow["\pi"{description}, two heads, from=2-5, to=3-4]
	\arrow[""{name=1, anchor=center, inner sep=0}, "{\varphi'}"', curve={height=6pt}, dashed, from=1-1, to=2-2]
	\arrow["{k^*\xi}"{description}, two heads, from=1-1, to=3-1]
	\arrow["k"', from=3-1, to=3-4]
	\arrow[from=1-1, to=1-4]
	\arrow["{k^*\pi}"{description}, two heads, from=2-2, to=3-1]
	\arrow["\lrcorner"{anchor=center, pos=0.125}, draw=none, from=2-2, to=3-4]
	\arrow[""{name=2, anchor=center, inner sep=0}, "\psi"', curve={height=12pt}, from=2-5, to=1-4]
	\arrow[""{name=3, anchor=center, inner sep=0}, "{\psi'}"'{pos=0.4}, curve={height=12pt}, dashed, from=2-2, to=1-1]
	\arrow["\lrcorner"{anchor=center, pos=0.125}, draw=none, from=1-2, to=3-4]
	\arrow["\dashv"{anchor=center, rotate=-135}, draw=none, from=2, to=0]
	\arrow["\dashv"{anchor=center, rotate=-138}, draw=none, from=3, to=1]
	\arrow[from=2-2, to=2-5, crossing over]
\end{tikzcd}\]
\end{prop}

\begin{proof}
Suppose given a fibered adjunction $\psi \dashv_B \varphi$ as indicated with unit
\[ \pair{\id_b}{\eta_{e}}: \pair{b}{e} \to \pair{b}{(\varphi \psi)_b(e)},\]
for $b:B$, $e:P\,b$.
By assumption, for $b':B$, $d':Q\,d'$, this induces an equivalence
\[ \Phi_\eta : \hom_F(\pair{b}{\psi_b(e)}, \pair{b'}{d'}) \to \hom_E(\pair{b}{e}, \pair{b'}{\varphi_b'(d')})\]
by
\[ \Phi \defeq \Phi_\eta \defeq \lambda u,g.\pair{u}{\varphi_u(g) \circ \eta_e}. \]
For the unit of the adjunction over $A$ we take
\[ \pair{\id_a}{\eta_{e}} : \pair{a}{e} \to \pair{a}{(\psi \varphi)_{ka}(e)}\]
for $a:A$, $e:P(ka)$. This gives rise to the transposing map
\[ \Phi'_\eta : \hom_{k^*F}(\pair{a}{\psi_{ka}(e)}, \pair{a'}{d'}) \to \hom_{k^*E}(\pair{a}{e}, \pair{a'}{\varphi_b'(d')})\]
for $a':A$, $d':Q(ka')$.
Contractibility of the fiber of a pair $\pair{v:a \to a'}{h:e\to_{kv}^P \varphi_b'(d')}$ demands the unique existence of an arrow $g_h:\psi_{ka}(e) \to^P_{v} d'$ such that
\[ \varphi_{kv}(g_h) \circ \eta_{e} = h,\]
which follows from the respective condition for the original map $\Phi$.
\end{proof}

\end{appendices}

%% Add the References bookmark
\addcontentsline{toc}{section}{Acknowledgements}

\section*{Acknowledgements}
The first author acknowledges the support of the Centre~for~Advanced~Study~(CAS) at the Norwegian~Academy~of~Science~and~Letters in Oslo, 
Norway, which funded and hosted the research project Homotopy Type Theory and Univalent Foundations during the academic~year~2018/19. The second author acknowledges the support of CAS and the hospitality of Bj{\o}rn~Ian~Dundas and Marc Bezem on the occasion of several guest visits of the HoTT-UF project during which parts of this work have been carried out. The second author also acknowledges the support of ARO under MURI Grant W911NF-20-1-0082 during final stages of this work.

We wish to greatly thank Steve Awodey, Paolo Capriotti, Denis-Charles Cisinski, Dan Licata, Michael Shulman, Thomas Streicher, and Matthew Weaver for many helpful discussions and feedback. We are also thankful to Bastiaan Cnossen who pointed out a flaw in~\cref{def:transp-lari} as well as a typo. Furthermore, we express gratitude to the anonymous referee for a very careful and detailed review.

We are particularly grateful to Emily Riehl for numerous continued insightful conversations, and for her hospitality hosting the second author multiple times to discuss this work.

\phantomsection%
\nocite{Ras18model,AFfib,GHT17,RV2cat,RVexp,RVscratch,RVyoneda,rasekh2021cartesian,RSSmod,LiBint,clementino2020lax,hermida1992fibred,rezk2017stuff,kock2013local,BorHandb2,BarwickShahFib,JoyQcat}

\printbibliography[heading=bibintoc]

\end{document}